\documentclass{amsart}
\usepackage{amsthm,amsmath,amssymb,amscd,graphics,epsfig,epic,eepic}
\usepackage{blkarray}

\usepackage[T1]{fontenc}

\newcommand{\entrylabel}[1]{\mbox{\textsf{{\rm c}1}}\hfil}
{\end{list}}
{
   \newtheorem{theorem}{Theorem}[subsection]
   \newtheorem{proposition}[theorem]{Proposition}
   \newtheorem{lemma}[theorem]{Lemma}

   \newtheorem{corollary}[theorem]{Corollary}
   \newtheorem{conjecture}[theorem]{Conjecture}

}
{\theoremstyle{definition}
   
   \newtheorem{example}[theorem]{Example}
   \newtheorem{definition}[theorem]{Definition}
}
{\theoremstyle{remark}
   \newtheorem*{remark}{Remark}
   \newtheorem*{remarks}{Remarks}
}
\newcommand{\Hi}{\cite{Hir}}

\newcommand{\Vi}{\cite{Vi}}

\newcommand{\RR}{{\mathbb{R}}}

\newcommand{\CC}{{\mathbb{C}}}
\newcommand{\NN}{{\mathbb{N}}}

\newcommand{\PP}{{\mathbb{P}}}
\newcommand{\ZZ}{{\mathbb{Z}}}

\newcommand{\bbA}{{\mathbb{A}}}

\newcommand{\bA}{{\mathbf{A}}}

\newcommand{\cC}{{\mathcal C}}
\newcommand{\cD}{{\mathcal D}}
\newcommand{\cE}{{\mathcal E}}
\newcommand{\cF}{{\mathcal F}}
\newcommand{\cG}{{\mathcal G}}
\newcommand{\cH}{{\mathcal H}}
\newcommand{\cI}{{\mathcal I}}

\newcommand{\cJ}{{\mathcal J}}

\newcommand{\cM}{{\mathcal M}}
\newcommand{\cam}{{\it m}}
\newcommand{\cN}{{\mathcal N}}
\newcommand{\cO}{{\mathcal O}}

\newcommand{\cR}{{\mathcal R}}

\newcommand{\supd}{{\operatorname{supd}}}
\newcommand{\supp}{{\operatorname{supp}}}
\newcommand{\cosupp}{{\operatorname{cosupp}}}

\newcommand{\expe}{{\operatorname{exp}}}

\newcommand{\adj}{{\operatorname{adj}}}
\newcommand{\mon}{{\operatorname{mon}}}

\newcommand{\inn}{\operatorname{in}}

\newcommand{\Ann}{\textup{Ann}}
\newcommand{\Res}{\operatorname{Res}}

\newcommand{\res}{{\operatorname{res}}}

\newcommand{\Ver}{{\operatorname{Vert}}}

\newcommand{\rank}{{\operatorname{rank}}}

\newcommand{\spa}{\operatorname{span}}

\newcommand{\Sub}{\operatorname{Sub}}

\newcommand{\Ga}{\operatorname{Gal}}

\newcommand{\Spec}{\operatorname{Spec}}

\newcommand{\Sing}{\operatorname{Sing}}

\newcommand{\gr}{{\operatorname{gr}}}

\newcommand{\ord}{{\operatorname{ord}}}

\newcommand{\codim}{\operatorname{codim}}

\newcommand{\id}{{\operatorname{id}}}

\numberwithin{equation}{section}

\begin{document}


\thanks{The author was supported in part by NSF grant DMS-0100598 and BSF grant 2014365}

\address{J. W\l odarczyk, Department of Mathematics\\Purdue University\\West
Lafayette, IN 47907\\USA}
\email{wlodarcz@purdue.edu, wlodar@math.purdue.edu}

\author{Jaros\l aw W\l odarczyk}
\title[Singular implicit and inverse  function theorems. Strong resolution with normally flat centers]{Singular implicit and inverse  function theorems. \\Strong resolution with normally flat centers}

\date{\today}
\begin{abstract}
Building upon ideas of Hironaka, Bierstone-Milman, Malgrange and others we generalize the inverse and implicit function theorem (in differential, analytic and algebraic setting) to  sets of functions of larger multiplicities (or ideals).
This allows one to describe singularities  given by a finite set of generators or by ideals in a simpler  form. In the special Cohen-Macaulay case we obtain a singular  analog of the inverse function theorem.
The singular implicit function theorem is closely related to a (proven here) extended version of 
 the Weierstrass-Hironaka-Malgrange division and preparation theorems.  
 The primary motivation for this paper  comes from the desingularization problem. As an illustration of the techniques used, 
we give some applications of our theorems to desingularization extending some results on  Hironaka normal flatness, the Samuel stratification and the Hilbert-Samuel function. The notion of the  standard basis along Samuel stratum introduced in the paper (inspired by the Bierstone-Milman and Hironaka constructions)  allows us to describe singularities along the Samuel stratum in a relatively simple way. 

It leads to a canonical reduction of the strong Hironaka desingularization with normally flat centers to a so called resolution of marked ideals. Moreover, in  characteristic zero, the  standard basis along Samuel stratum generates a unique canonical Rees algebra along Samuel startum giving a  straightforward proof of the strong desingularization in algebraic  and analytic cases.
\end{abstract}

\maketitle

\tableofcontents
\addtocounter{section}{-1}

\section{Introduction} 

Suppose that $f(t,x)$ is an analytic  function of $t\in \CC$ and $x=(x_1,\ldots,x_n)\in {\CC}^n$ near the origin, and let $k$ be the smallest integer such that
$$
f(0,0)=0,\quad \frac{\partial{f}}{\partial{t}}(0,0)=0,\ldots,\frac{\partial^{k-1}{f}}{\partial{t^{k-1}}}(0,0)=0,\quad
\frac{\partial^{k}{f}}{\partial{t^{k}}}(0,0)\neq 0.
$$ 
Then  the {\bf Weierstrass preparation theorem} states that near the origin, $f$ can be written uniquely as the product of an analytic function $c$ that is nonzero at the origin, and an analytic function that as a function of $t$ is a polynomial of degree $k$. In other words,
$$
f(t,x)=c(t,x)(t^k+a_{k-1}(x)t^{k-1}+\ldots+a_0(x))
$$
where the functions $c$ and $a_i$ are analytic and $c$ is nonzero at the origin.
The theorem can  also  be understood
as a direct generalization of the implicit function theorem in  the analytic situation for a single function. The solution of the equation is given in  implicit polynomial form. If $k=1$ we obtain  the solution in the explicit
form $t=a(x)$ or equivalently $f(a(x),x)=0$.

The Weierstrass preparation theorem is closely related to and follows from the {\bf Weierstrass division theorem} which says that if $f$ and $k$ satisfy the conditions above and $g$ is an analytic function near the origin, then we can write
$$
g=qf+r,
$$
where $q$ and $r$ are analytic, and as a function of $t$, $r=\sum_{j=0}^{k-1} t^jr_j(x) $ is a polynomial of degree less than $k$. 
 The theorem plays an important role in  analytic and  differential geometry. 
Recall that the analog of the classical Weierstrass preparation  theorem for smooth real functions was conjectured by  Thom 
and proved by   Malgrange (Weierstrass preparation) and Mather (Weierstrass division). The smooth case is much more difficult and deeper, and  of  fundamental importance  in the theory of singularities.

On the other hand, Hironaka in his proof of desingularization theorem, and many others like Aroca, Vincente, Brian\c con, Jalambert, Teissier (\cite{AHV}, \cite{Briancon}, \cite{LJ})  studied different versions of  division for multiple generators (respectively ideals). The so called {\bf Weierstrass-Hironaka formal division theorem} says that for a set of formal analytic functions $f_1,\ldots,f_r$ and the corresponding set of leading exponents there exists division with remainder
 $$
g=\sum h_if_i+r,
$$
where $h_i$, $r$ satisfy some combinatorial conditions.

In \cite{Hir2} Hironaka proved  a weaker form of so called Henselian division  for algebraic functions. 
A  somewhat different approach and language was used in Bierstone-Milman's papers on desingularization (\cite{BM2}, \cite{BM22}). It allows a better control over singularities exploiting the (simplest)  formal analytic version of division and  requires passing to formal coordinate charts.

In this paper we use the ideas of Hironaka and Bierstone-Milman and many others.
One of our main goals is to establish a generalized version of  Weierstrass-Hironaka division for algebraic, analytic, and differentiable functions. As a consequence, we prove a generalized preparation theorem, and the existence of a standard basis (with respect to a monotone order) satisfying some  stronger conditions.


One can look at the classical Weierstrass division and preparation from two different perspectives.
It can be regarded as an extension of the one-variable implicit function theorem and also  as   division or preparation with respect to a dominating monomial. The second approach requires introduction of a certain monomial and was exploited by the Hironaka generalization.

The method of  monomial order may be  easier from the computational point of view  but it has severe limitations - as the resulting division and the dominating monomials dramatically change when passing from a point to its neighborhood.

On the other hand, the classical implicit function theorem for several variables allows one to describe subvarieties
(submanifolds) in terms of differential (coherent) conditions, and uses 
no monomial order.

In this paper we generalize the implicit function theorem to
the case of several functions of  higher multiplicities and to ideals. By using a ``resultant-type'' extension of standard  Jacobian matrices,   we get a coherent polynomial description of singularities. The approach gives a great flexibility in selecting ``dominating'' monomials  and controlling the forms of the describing equations in a neighborhood, since no monomial order is used.

In particular, in the simplest situation of singularities defined by 
$k$ equations of multiplicities $d_1,\ldots,d_k$ which are in general position, one can choose the ``leading monomials'' to be  powers of independent coordinates $x_1^{d_1},\ldots,x_k^{d_k}$. This means that the generic intersection of hypersurfaces is somewhat similar to the intersections of  transversal hyperplanes with multiplicites, at least in terms of the common zeroes (counted with multiplicities). This gives a certain analogy to B\'ezout's theorem, already observed by Macaulay. In the case of homogeneous polynomials the resulting matrices were used by Macaulay in his definition of resultant. 

Also in  the Cohen-Macaulay  case one can prove a certain analog of the inverse function theorem (in algebraic, analytic and differential cases). It is closely related, in the analytic case, to the Grauert-Remmert theorem on finite morphisms.

Upon establishing a generalized implicit function theorem we define a more general notion of  standard basis (and pre-basis) along Samuel stratum. 

Existence of a  standard basis quite immediately implies Hironaka's normal flatness theorem and Bennett's theorems on upper semicontinuity of the Hilbert-Samuel functions (also in the differential setting). 

As a consequence we  show how to reduce in a canonical way a  Hironaka desingularization with smooth normally flat  centers to  resolution of so called marked ideals (or Hironaka's idealistic exponents) which leads to  very straightforward proofs of Hironaka resolution theorems for algebraic varieties in its strongest form. (The same method can be also applied to analytic spaces).

The paper is organized as follows. 
In the first  chapter we  discuss some extensions of Weierstrass division for coherent sheaves in the algebraic and analytic category. We also show ``neighborhood versions'' of classical  division theorems in the algebraic and analytic cases. 
 We use some results of Grauert and Remmert in the analytic category and some Hironaka's ideas in the algebraic situation.   As a consequence we  show a certain   Cohen-Macaulay analog of the inverse function theorem (in the algebraic and  analytic setting ; Theorems  \ref{I3},\ref{I1}, \ref{I4})

In the second chapter we review  the classical results by Malgrange and Mather. We also prove  stronger ``neighborhood versions'' of classical  division theorems of Weierstrass and Malgrange-Mather in the differential setting. Also by combining the Grauert-Remmert approach in the analytic situation with the Malgrange-Mather technique in the differential setting    we  prove a certain   Cohen-Macaulay analog of the inverse function theorem in the differential cases; Theorem  \ref{I3}, \ref{I4}). We also define a category  of smooth objects. It allows us to treat similarly algebraic, analytic and smooth functions in the subsequent chapters.

In Chapter 3, following the ideas of Hironaka, Galligo, Grauert  and Bierstone-Milman  we prove a somewhat extended version of Weierstrass-Hironaka-Grauert-Galligo division in the simplest case of formal analytic functions (for any monomial order) (Theorem \ref{formal}). We also review basic properties of monotone diagrams and diagrams of finite type,  used in the Hironaka and Bierstone-Milman papers. In our paper we combine the quite different approaches of Hironaka and Bierstone-Milman.
Then we show some   natural decompositions of the diagrams which play an essential role in  generalized division and preparation theorems proven in Chapter 5 (Corollary \ref{imp33}).

In Chapter 4 we introduce a notion of filtered Stanley decomposition (Definitions \ref{Stanley}, \ref{Stanley2}).
Recall that Stanley decomposition was constructed and studied originally by Stanley, and is a fundamental tool in homological algebra.
In this paper we introduce  a stronger version of   filtered Stanley decomposition which turns out to be  a critical tool in the proofs of the implicit function theorem and division-preparation theorem. We prove existence of the filtered Stanley decomposition for any graded modules over polynomial rings (over an infinite field) and for modules over smooth rings, together with some of  their direct consequences (Theorems \ref{free1}, \ref{free2}).
 The language of filtered Stanley decomposition allows us  to reduce, by using the Malgrange theorem for modules, the division of functions to the division of their initial forms in graded rings.  The constructions are motivated by  ideas of both Hironaka's  Henselian division and Bierstone-Milman's stabilization theorems for monotone diagrams (Theorem \ref{Stab}).

At the beginning of  Chapter 5 we establish generalized versions of the Weierstrass-Hironaka division  and preparation theorems as a consequence of the results proven in Chapter 4 (Theorems \ref{main0},  \ref{main00},  \ref{main2}). In particular, the division theorem   in the algebraic setting
extends and strengthens the Hironaka Henselian division theorem.
On the other hand, since the results are  proven also in differential settings, they extend  the 
Malgrange-Mather preparation and division theorems to the case of multiple generators.
As a consequence we give the construction of a standard basis with respect to the monomial order (in either setting). 

In the second part of Chapter 5 we prove  a generalized implicit function theorem
and its consequences (Theorems \ref{main01}, \ref{main011}, \ref{main011}). 
Upon,  introduced here generalized Jacobian conditions, one can represent the singularities in a simpler "reduced" form. The very basic idea goes back to Bierstone-Milman stabilization Theorem and their approach to standard basis (\cite{BM1}).
In the particular case of complete intersections  related
algebraic conditions were studied by Macauley in the context of his notion of resultant. \cite{Macaulay}


In Chapters 6-9 we show some consequences of the singular implicit function theorem proving existence of Hironaka canonical  desingularization of algebraic varieties in its strongest version, by the sequence of blow-ups of smooth normally flat centers. Unlike in the  weaker version where a nonembedded resolution is
achieved merely by a birational projective morphism in Hironaka's original approach  the desingularization uses a sequence of smooth normally flat centers (see \cite{Hir2},\cite{BM1}, \cite{Villamayor}). The condition of normal flatness means that the normal cones of the varieties along the centers are flat. This gives a certain geometric control over the process.

In Chapter 6 we briefly discuss the notion of resolution of marked ideals.
Then using the singular implicit function   we construct   a coherent notion of   standard basis of any ideal along Samuel stratum (Definition \ref{Dweak2}), proving 
Hironaka's normal flatness theorem and Bennett's theorems on the Hilbert-Samuel function (also in the differential setting). The notion of   standard basis along Samuel stratum is a counterpart of Hironaka's distinguished data (as in \cite{Hir2}) and Bierstone-Milman's semicoherent presentation of  ideals (as in \cite{BM2},\cite{BM22}).  In our approach it is merely an extension
of the classical implicit function theorem  to a more complex case of singular subspaces (see also Example \ref{Weak3}).
The conditions for the standard basis are closely related to the ones for the Bierstone-Milman's standard basis relative to a diagram for formal analytic functions. On the other hand the constructions are obtained in the language which stems from Hironaka's Henselian division.
The notion allows us to   
 reduce strong resolution of singularities  to resolution of  marked functions with assigned multiplicities (Theorems \ref{weak}, \ref{weak2}, \ref{Flat}).
 
 In Chapter 7 we show that the constructed  standard basis, though not unique, defines a unique associated (multiple) marked ideal- so called canonical Rees Algebra controlling Hilbert-Samuel function (see Theorem \ref{Rees}). 
 
 In Chapter 8 we formulate Hironaka's resolution theorems
 in algebraic setting and show that (in characteristic zero) they can be deduced via canonical
 Rees algebra to the canonical resolution of marked ideals (see Theorem \ref{Red1}).
 
 In Chapter 9 we briefly show existence of canonical resolution of marked ideals in order to complete the proof of the desingularization theorems.
 
 Although one of our main foci  is given by the differential perspective of the Malgrange-Mather approach, and the algebraic perspective of  Cohen-Macaulay rings, the primary motivation for this paper  comes from the desingularization problem. Division theorems   are a tool of fundamental importance when studying singularities, especially in positive characteristic. 
 While, in characteristic zero, the technique allows one to prove  Hironaka desingularization in its strongest form, it  seems to be indispensable in positive characteristic.
In fact, different particular versions of  Weierstrass division for multiple generators were applied  in many recent papers on desingularization (see for instance \cite{BM3}, \cite{Hir3}, \cite{Vill}, \cite{KM}, \cite{Wlod}).
\subsection{Acknowledgements} The  author  thanks  Professors Kenji Matsuki an Orlando Villamayor for helpful comments.

\section{Weierstrass division for sheaves over algebraic and analytic functions}
\label{lag}

In this section we shall study generalized versions of Weierstrass division for algebraic and analytic functions which will be exploited in the remaining part of the paper.
\subsection{Hironaka's Hensel's Lemma and Weierstrass division for algebraic functions}

A particular case of the following theorem (for $k=1$) was used by Hironaka in \cite{Hir2} (with proof omitted) and was referred to as ``Hensel's lemma''.  Here we  also prove its  general ``Malgrange'' form (originally proved by Malgrange for smooth functions) and  extend it to sheaf ``Grauert-Remmert'' versions (similar to the versions proven by Grauert-Remmert for holomorphic functions).
In the next sections we shall also study the analytic and smooth versions of the theorems proven below and their generalizations.

The Hironaka  Hensel's Lemma can be derived directly from the properties of Henselian rings via a version of Zariski's  main theorem.

\begin{theorem}[Hironaka's Hensel's Lemma]
Let  $R$ be  a  local Henselian and Noetherian ring with maximal ideal $m$, and let
$S=R\langle z_1,\ldots,z_k\rangle$ denote the Henselianization of the localization $R[z_1,\ldots,z_k]_{m_k}$ at the maximal ideal $m_k:=m+(z_1,\ldots,z_k)$.
Suppose   $M$ is a finite  $R\langle z_1,\ldots,z_k\rangle$-module and
 $M/(m\cdot M)$ is a finite-dimensional vector space over $K=R/m$.
 Then $M$ is finite over  $R$.
 \end{theorem}
Before proving this theorem we shall need a simple Lemma
\begin{lemma} \label{map1}Let $R\to S$ be  a map (homomorphism) of noetherian local rings. Let $m\subset R$ be a maximal ideal. Suppose   $M$ is a finite  $S$-module, and $I:=Ann(M)\subset S$ is the annihilator of $M$.  
Then the following conditions are equivalent
\begin{enumerate}
 
 \item $M/(m\cdot M)$ is a finite-dimensional vector space over $K=R/m$.
 
\item $S/(I+m\cdot S)$ is finite over $K= R/m$
 \end{enumerate}

\end{lemma}

\begin{proof} Recall that by the support of $S$-module $M$ we mean
$$\supp(M):=\{p\in \Spec(S)\mid M_p\neq 0\},$$
(where $M_p$ is the localization of $M$ at the prime ideal $p\subset S$).
Then, it follows (\cite[Tag 080S]{stacks-project}) that  
$\supp(M)=V(I)$, where $$V(I)=\{p\in \Spec(S)\mid p\supset I\}$$ is the {\it vanishing locus of $I$} . Also $$\supp(M/(m\cdot M))=V(I)\cap V(m\cdot S)=V(I+m\cdot S)=\supp(S/(I+m\cdot S).$$
On the other hand $\supp(M/(m\cdot M))=V(J)$, where $J\subset S$ is the annihilator of $M/(m\cdot M)$. Note that $J\supseteq I+m\cdot S$, and we have a map $R/m\to  S/J$. 

Thus $S/J$ acts faithfully on $M/(m\cdot M)$. Since $M/(m\cdot M)$
is a finite over $K= R/m$ we get that $S/J$ is finite over $K$.
(Since $S/J$ embeds into a finite $K$-vector space of endomorphisms $End_K(M/(m\cdot M)$).

On the other hand $V(J)=V(I+m\cdot S)$, and thus $(I+m\cdot S)\supset J^k$. Since $S/\cJ^k$ has  a filtration $\cJ^i$ with factors $\cJ^i/\cJ^{i+1}$ finite over $S/\cJ$ we get that each factor is finite over $K=R/m$, and finally $S/\cJ^k$ and its factor $S/(I+m\cdot S)$ is finite over $K= R/m$. 

Conversely $M/(m\cdot M)$ is finite over $S/(I+m\cdot S)$.

\end{proof}

\begin{proof} Let $\cI\subset R\langle z_1,\ldots,z_k\rangle$ denote the annihilator of $M$ as before. 
Consider an \'etale affine neighborhood $U$ of $\bA^k_R:=\Spec(R[z_1,\ldots,z_k])$ containing the generators of $\cI$ and preserving the residue field $K$ so that $\cO(U)\subset R\langle z_1,\ldots,z_k\rangle$. Denote by $\cI_U\subset \cO(U)$ the ideal determined by these   generators.  
Note that the map $R\to \cO(U)$ is of finite type. 

Let $x\in U$ denote the point corresponding to the ideal $m_k$, let $\cO(U)_x$ be the localization at $x$, set $\cI_x:=\cI_U\cdot \cO(U)_x\subset \cO(U)_x$, and let $m_x\subset \cO(U)_x$ be the maximal ideal of $x$. The Henselianization of $\cO(U)_x/\cI_x$ is equal to $R\langle z_1,\ldots,z_k\rangle/\cI$ since Henselianization commutes with quotients. Likewise  the Henselianization of $\cO(U)_x/((\cI_U+m)\cdot \cO(U)_x)$ is  $R\langle z_1,\ldots,z_k\rangle/(\cI+m)$, which is a finite-dimensional vector space over $K$. Thus $\cO(U)_x/((\cI_U+m)\cdot \cO(U)_x)$ is a finite-dimensional vector space itself. In other words, $R\to \cO(U)/\cI_U$ is a map of finite type and the localization of the   fiber over $m$ is  finite.
Thus by shrinking $U$ we can assume that the fiber of $R\to \cO(U)/\cI_U$ over $m$ is  irreducible and its reduced structure coincides with $\{x\}$, and $m\cdot \cO(U)\supset m_x^\ell$ for a certain $\ell$.

By a version of Zariski's main  theorem due to Raynaud \cite{Raynaud} one can factor this map  as $$R\to \cO(U')/\cI'_U\to \cO(U)/\cI_U$$ into a composition of a finite map $\phi^*: R\to \cO(U')/\cI_U'$ followed by the map $\cO(U')/\cI'_U\to \cO(U)/\cI_U$ defining and open inclusion $\Spec(\cO(U)/\cI_U)\subset \Spec(\cO(U')/\cI'_U)$.
To be more precise, there exists $f\in \cO(U')\subset \cO(U)$, with $f\not \in m_x$, such that $(\cO(U)/\cI_U)_{f}=(\cO(U')/\cI_U')_f$ 
with $\cO(U')/\cI_U'$  a  finite $R$-module.  The latter is, by  Hensel's lemma, isomorphic to  a finite product of local Henselian rings $R_i$:
$$\cO(U')/\cI_U'\simeq \bigoplus_{i=1}^k R_i$$
 
 The point $x\in U\subset U'$ defines a maximal ideal of, say, $R_1$ in the fiber of $R\to \cO(U')/\cI_U'\simeq \bigoplus_{i=1}^k R_i$  over $m$.
 Consider the function $g\in \cO(U')$ such that its class in $\cO(U')/\cI_U'$ defines the element $(1,0,\ldots,0)\in \bigoplus_{i=1}^k R_i$.
 Then $g\not\in m_x$  and $\cO(U')_g/\cI_U'=(\bigoplus_{i=1}^k R_i)_g=R_1$.
Moreover since $f\not\in m_x$ we get that $\cO(U')_{g,f}/\cI_U'=(\bigoplus_{i=1}^k R_i)_{g,f}=(R_1)_f=R_1$
 
 In other words by shrinking $U$ around $x$ to $U_{{g,f}}=U'_{{g,f}}$ we can assume that  $\cO(U)/\cI_U\simeq R_1$ is  local Henselian itself. Consequently it is isomorphic to its Henselianization $R\langle z_1,\ldots,z_k\rangle/\cI$. Since $M$ is finite over $R\langle z_1,\ldots,z_k\rangle$ and $\cI$ acts trivially on $M$, we find that $M$ is finite over $R\langle z_1,\ldots,z_k\rangle/\cI$. But $R\langle z_1,\ldots,z_k\rangle/\cI\simeq \cO(U)/\cI_U$ is finite over $R$, and thus  $M$ is finite over $R$.
\end{proof}
The following more general version of Hironaka's Hensel's Lemma shall be understood as an 
algebraic counterpart of Malgrange division for modules over
smooth functions.

\begin{theorem}
Let $f: R\langle x_1,\ldots,x_k\rangle \to R\langle y_1,\ldots,y_m\rangle $ be a homomorphism of Noetherian $R$-algebras where $R$ is a  local Henselian and Noetherian ring.
Suppose that  $M$ is a finite  $R\langle y_1,\ldots,y_m\rangle$-module and
 $M/(f(m_k)\cdot M)$ is a finite-dimensional vector space over $K=R\langle y_1,\ldots,y_k\rangle /m_k$.
 Then $M$ is finite over  $R\langle x_1,\ldots,x_k\rangle$. 
\end{theorem}

\begin{proof} One can factor $f$ as the composition of  the natural inclusion $i:=R\langle x_1,\ldots,x_k\rangle \hookrightarrow R\langle x_1,\ldots,x_k,\allowbreak y_1,\ldots,y_m\rangle$ followed by  the projection $\pi: R\langle x_1,\ldots,x_k,y_1,\ldots,y_m\rangle\to R\langle y_1,\ldots,y_m\rangle$, $\pi(x_i)=f(x_i)$, $\pi(y_j)=y_j$. Note that $M$ is a finitely generated $R\langle x_1,\ldots,x_k,y_1,\ldots,y_m\rangle$-module, and  $M/(f(m_k)\cdot M)=M/(i(m_k)\cdot M)$ is of finite dimension. By Hironaka's Hensel's Lemma, $M$ is a finite $R\langle x_1,\ldots,x_k\rangle$-module.	
\end{proof}

\begin{lemma} \label{c11} Let $K$ be a field.  Consider the natural inclusion map $$f: R:=K\langle z_{1},\ldots,z_{k}\rangle \to S=K\langle z_{1},\ldots,z_{k+n}\rangle.$$  Let  $\cI\subset S$ be an ideal such that
the map
$R\to S/\cI$  is  finite.
Then there exist  smooth affine schemes   $U_1$ and $U_2$ over $\Spec(K)$ of dimension  $k$ and $n+k$ respectively  and  a smooth map $\phi: U_2\to U_1$ of dimension $n$  compatible with $f$
for which:
 \begin{enumerate}
\item  $K[z_1,\ldots,z_k]\subset \cO(U_1)\subset K\langle z_1,\ldots,z_k\rangle$,
\item  $K[ z_1,\ldots,z_{k+n}]\subset \cO(U_2)\subset K\langle z_1,\ldots,z_{k+n}\rangle$,
\item $\phi^*: \cO(U_1)\to \cO(U_2)$ is the restriction  of $f$.
 
\item $\cO(U_1)\to \cO(U_2)/\cI_2$ is finite, where $\cI_2\subset \cO(U_2)$ is an ideal  for which \\$\cI=\cI_2\cdot K\langle z_1,\ldots,z_{k+n}\rangle$.
\item The fiber  $\cO(U_2)/(m\cdot\cO(U_2)+\cI_2)$ irreducible, where $m=(z_1,\ldots,z_k)\subset \cO(U_1)$.

\end{enumerate}
\end{lemma}

\begin{proof} 

Let $b_1,\ldots, b_s$ be generators of the module $S/\cI=\sum R\cdot b_i$ over $R$, and  $f_1,\ldots,f_k\in \cI$ be generators of $\cI$.  Consider a nonsingular  affine $V_2$  \'etale  neighborhood of  $0\in \bA^{n+k}=\Spec K[z_1,\ldots,z_{k+n}]$, preserving residue field and with $\cO(V_2)$ containg $b_i,f_j$. Moreover by shrinking $V_2$ around $0$ we van assume that the fiber $K[z_1,\ldots,z_{n}]\to \cO(V_2)/(f_1,\ldots,f_k)$ over $m=(z_1,\ldots,z_k)$ is irreducible.

Denote by  $d_1,\ldots, d_r$  generators of the $R$-algebra $\cO(V_2)$ over $K$. Write 
$$
b_ib_j\equiv \sum_{l=1}^s c_{ijl}b_l\quad (\textup{mod}\ \cI),\quad\quad d_i\equiv \sum_{l=1}^r d_{il}b_l\quad (\textup{mod}\,\, \cI)
$$  
where $c_{ijl}, d_{il}\in R$. 

We shall assume that $U_1$ is an affine \'etale neighorhood of $0\in Spec(K[z_1,\ldots,z_k])$  with the ring of regular functions $\cO(U_1)\subset R=K\langle z_1,\ldots,z_k\rangle$  containing $c_{ijl}$ and $d_{il}$. Now let $U_2$ be a component in $V_2\times_{\bA^k}U_1$ dominated by  $\Spec(S)$. Then  $U_2\to U_1$ is smooth, and by the universal property of the component of the product we get that $\cO(U_2)$ is equal to the subring of $S$ generated by $\cO(U_1)$ and $\cO(V_2)$:
$$
\cO(U_2)=\cO(U_1)\cO(V_2)=\cO(U_1)[d_1,\ldots,d_r]\subset S
$$
and  it contains  $b_i$ and $f_j$.  Then for $\cI_{2}= \cI\cap \cO(U_2)$ we have $\cI_2\cdot S=\cI$ and  
$$
\cO(U_2)/\cI_{2}=\sum_{i=1}^k \cO(U_1)b_i.$$
Indeed the $\cO(U_1)$-submodule $$\sum_{i=1}^k \cO(U_1)b_i\subset \cO(U_2)/\cI_2=\cO(U_1)[d_1,\ldots, d_r]/\cI_2=\cO(U_2)/\cI_2$$ is, in fact, a  subring since it is closed under multiplication by the relation on $b_ib_j$.
Moreover it contains all the generators  $d_i$ over $\cO(U_1)$ which implies the equality.

\end{proof}

\begin{corollary}\label{coherent2}
Let $f:X\to Y$ be a morphism of  schemes of finite type over a  field $K$ with  $x\in X$ and $y=f(x) \in Y$. Let $\cF$ be a coherent sheaf of  $\cO_X$-modules, with a stalk $\cF_x$ (which is a finitely generated $\cO_{x,X}$-module). Suppose that the vector space $\cF_x/(m_y\cdot \cF_x)$ is of  finite dimension over the field $K=f^*(\cO_{y,Y}/m_{y,Y})$.
 Then there exist an \'etale neighborhood   $Y'\subset Y$ of $y$  and $\alpha: X'\to X$ of $x$, preserving the residue fields with
the induced morphism $f': X'\to Y'$, and the induced coherent sheaf $\cF'=\alpha^*(\cF)$ such that: 
\begin{enumerate}
\item $f'_{x*}({\cF'}_x)\simeq {\cF'}_x$.

\item $f'_*(\cF')$ is  a coherent $\cO_{Y'}$-module.
\item In particular 
%
$f'_*(\cO_{X'}/\Ann(\cF'))$ 
is  a coherent $\cO_{Y'}$-module (where $\Ann(\cF')$ denotes the annihilator  of the sheaf $\cF'$).

 \item If $c_1,\ldots, c_k\in \cF_x$ generate $\cF_x/(m_y\cF_x)$ over the  field $K=\cO_y/m_y$ then $X'$ and $Y'$ can be chosen such that $c_1,\ldots, c_k$ are in $\cF'(X')$ and generate the  sheaf $f'_*(\cF')$  over  $\cO_{Y'}$.
\end{enumerate}	
\end{corollary}

\begin{proof} 
(2) \& (3) One can assume that $X$ and $Y$ are affine schemes and  find  affine spaces $\bA^n$, $\bA^m$ over $K$ containing $X$ and $Y$ respectively with inclusions $i_X:X\hookrightarrow \bA^n$, $i_Y:Y\hookrightarrow \bA^m$, and 
 an extension  $f_A: \bA^n\to \bA^m$ of the morphism  $f:X\to Y\subset \bA^m$.   
 
 Factor $f_A: \bA^n\to \bA^m$ into the composition of the closed immersion $i:\bA^n\to \bA^{n+m}$ defined by the graph of $f_A$ followed by the natural projection $\pi: \bA^{n+m}\to \bA^{m}$. 
The coherent sheaf $\cF$ on $X$ defines the coherent sheaf $\cF_A=j_*(\cF)$ on $\bA^{n+m}$ via the inclusion $j:X\hookrightarrow X\times Y\subset \bA^{n+m}$. The annihilator of $\cF_A$ contains $\cI_{j(X)}\subset \cO_{\bA^{n+m}}$ , and  is supported on the closed subset of $j(X)\subset \bA^{n+m}$. Moreover, by comparing stalks we see that $\cO_{\bA^{n+m}}/\Ann(\cF_A)=j_*(\cO_X/\Ann(\cF))$. The latter implies the equality of the vanishing loci $j(V(\Ann(\cF))=V(\Ann(\cF_A))$.

By Lemma \ref{map1}, the restriction  
$$
\overline{f}:V(\Ann(\cF)):=\Spec(\cO(X)/\Ann(\cF(X)))\to Y
$$ 
of $f$ to the vanishing locus $V(\Ann(\cF))$ has a  
finite fiber $\overline{f}^{-1}(x)$, 
and likewise the restriction 
$$\overline{\pi}:V(\Ann(\cF_A)):=
\Spec(K[z_1,\ldots, z_{n+m}/\Ann(\cF_A(\bA^{n+m}))\to \bA^{m}.
$$
 The latter defines
  a map 
$$
\overline{\pi}^*: K\langle z_1,\ldots,z_n\rangle \to K\langle z_1,\ldots,z_{n+m}\rangle/((\Ann(\cF_A)\cdot K\langle z_1,\ldots,z_{n+m}\rangle)
$$ 
with finite fiber over the maximal ideal $(z_1,\ldots,z_{n+m})$.
 Then, by Lemma \ref{coherent2},  there exist \'etale affine neighborhoods   $j_1:U_1\to  \bA^n$ of $y$  and $j_2: U_2\to \bA^{n+k}$ of $x$, with induced $\overline{\pi}: U_2\to U_1$ such that the map of rings $$\cO(U_1)\to \cO(U_2)/(\Ann(\cF_A)\cdot \cO(U_2)) =\cO(U_2)/\Ann(\cF_{U_2})$$ is finite, where $\cF_{U_2}:=j^*(\cF_A)$ is the induced coherent sheaf, and $\Ann(\cF_{U_2})=\Ann(\cF_A)\cdot \cO(U_2)$ denote its annihilator.
 
  The \'etale affine neighborhoods   $U_1$, and $U_2$  induce   \'etale neighborhoods $X'\subset U_1$ of $X$ and $Y'\subset U_2$ of $Y$.

 The coherent sheaf $\cF$ on $X$ defines  a unique coherent  sheaf $\cF'$ on $X'$. Moreover since $X'\to X$ is \'etale, $\Ann(\cF')=\Ann(\cF)\cdot \cO(X')$.
By the above,   the restriction of $U_2\to U_1$ to the support of the annihilator,
%
$$\overline{f}': V(\Ann(\cF'))
\simeq V(\Ann(\cF_{U_2}))\to Y'\subset U_1,
$$
 is finite. 

 Since $\cF'$ is annihilated by $\Ann(\cF')$, we see that the closed immersion $i:V(\Ann(\cF'))\subseteq X'$ and the coherent sheaf $\cF'$ on $X'$ define a coherent sheaf $\cF''$ on the scheme 
$V(\Ann(\cF'))$,
such that $\cF'=i_*(\cF'')$, which implies that 
$$f'_*(\cF')=f'_*i_*(\cF'')=(\overline{f}')^*(\cF'')
$$ is a coherent sheaf of $\cO_{Y'}$-modules. This proves  (2) and thus (3). 
 
 By Lemma \ref{c11}, we can assume that the fiber of $\overline{f}'$ over $x$ is irreducible, and thus 
%
$$
f'_*(\cF')_x=((\overline{f}')^*(\cF''))_x=\cF''_x\simeq \cF'_x.
$$
 
 To prove (4) we  shrink  $X$ so that $\cF$ is generated by $c_1,\ldots,c_k\in \cF(X)$ over $\cO(X)$. This implies that they also generate $\cF'(X')$ over $\cO(X')$. By the Nakayama lemma, the  stalk 
$$
(f'_*(\cF'))_x=\cF'_x=\cF_x\otimes\cO_{X',x'}
$$  of the coherent sheaf $f'_*(\cF')$ is generated over $\cO_{Y',y'}$ by $c_1,\ldots,c_k$. This implies that, after shrinking  $X'$  and $Y'$ further, we can assume that
$f'_*(\cF')$ is generated over $\cO_{Y'}$ by $c_1,\ldots,c_k$.
\end{proof}

Both theorems generalize Weierstrass division either  locally or in a neighborhood.

\begin{definition} $f(t,x)$ is \emph{$d$-regular with respect to $t$} (where $x=(x_1,\ldots,x_n)$) if 
  $$
{f(0,0)}= {\frac{\partial{f}}{\partial t}((0,0)}=\ldots ={\frac{\partial^{d-1}f}{\partial t^{d-1}}(0,0)}= 0,\quad {\frac{\partial^{d}f}{\partial t^{d}}(0,0)}\neq 0.
$$
\end{definition}

\begin{theorem}[Weierstrass division of algebraic  functions in a neighborhood] \label{Mi2}
Let $X$ be an affine scheme which is \'etale over $\bbA_K^{n+1}=\Spec{K[t,x]}$,
 where $K$ is a field.
 Let $g(t,x)\in \cO(X)$ be a $d$-regular function at $\overline{x}\in X$ over a closed point in $\Spec(K)$.
Then there is an \'etale neighborhood $U_2\to X$ of $\overline{x}$ preserving the residue field of $\overline{x}$ such that:
\begin{enumerate}

\item There is  a smooth morphism $\pi: U_2\to U_1$ of dimension one onto an affine scheme $U_1$ which is smooth over $\Spec(K)$ and a closed embedding $i: U_1\to U_2$ such that $\pi\circ i=\id_{U_1}: {U_1}\to U_2$,
 with $\cO({U_1})\subset K\langle x\rangle$ and $\cO(U_2)\subset K\langle t, x\rangle$.
\item Weierstrass division by $f$ with remainder exists in the ring $\cO(U_2)$: For any $g\in \cO(U_2)$ there exist $q\in \cO(U_2)$  and  $r=\sum_{i=0}^{d-1} r_i(x)t^i\in \cO(U_2)$ with $r_i\in \cO({U_1})$ such that $g=q\cdot f+r$.
\item  Weierstrass division by $f$ exists for any \'etale neighborhood   $U_1'\to {U_1}$  of $\overline{x}$ preserving the residue field of $\overline{x}$ and for $U_2':=U_2\times_{U_1}U_1'$.
\end{enumerate}
\end{theorem}

\begin{proof} (1)  We can assume that $X$ is affine with  $\cO(X)\subset K\langle t, y_1,\ldots,y_n\rangle$. Then the quotient ring 
%
$R:=\cO(X)/(t)$ 
can be identified  with a subring of $ K\langle  y_1,\ldots,y_n\rangle\simeq K\langle t, y_1,\ldots,y_n\rangle/(t)$.
Let 
$$
U_1:=\Spec (R)\simeq \{\overline{y}\in U\mid t(\overline{y})=0\},\quad U_2:=\Spec (R\cdot\cO(X)).
$$
Let   $\pi: U_2\to {U_1}$ be the projection defined by the inclusion $R\subset R\cdot\cO(X)$. The schemes $U_2$ and ${U_1}$ are smooth over $\Spec(K)$ at $\overline{x}$ and $\pi(\overline{x})$.

By shrinking $U_2$ and ${U_1}$ if necessary, we can assume that
$\pi$ is smooth and $U_2\to X$ is \'etale.

(2) Consider the coherent sheaf $\cF:=\cO_{U_2}/(f)$  and the projection $\pi:  U_2\to U_1$.  Then $\cF_{x}/(m_{y}\cF_{x})$ is spanned by $1,t,\ldots,t^{d-1}$. Consequently,  by using Corollary \ref{coherent2}, after passing to \'etale neighborhoods  of $U_1$ and $U_2$, the sheaf 
$\pi_*(\cO_{ U_2}/(f))$  
is coherent and generated over $\cO_{U_1}$ by the same set of 
polynomials $1,t,\ldots,t^{d-1}$. Since we can assume that all the subschemes are affine, we deduce that the ring of global sections, 
%
$\cO({ U_2})/(f)$, is the $\cO(U_1)$-module generated by $1,t,\ldots,t^{d-1}$. This implies existence of Weierstrass division.

(3) The same property is valid when passing to \'etale neighborhoods.
	\end{proof}
\begin{corollary}[Singular ``inverse function'' theorem (algebraic version)] \label{I3}
Let $f:X\to Y$ be a morphism of  schemes of finite type over a  field $K$, and let  $x\in X$ and $y=f(x) \in Y$ be $K$-rational points. Assume that $\cO_{X,x}/f^*(m_y)$ is of finite dimension $d$  over $K=\cO_y/m_y$ generated by  $c_1,\ldots,c_d\in \cO_{X,x}$. Then there exist  \'etale neighborhoods   $Y'\to  Y$ of $y$  and $X'\to X$ of $x$, preserving the residue field $K$, with
the induced finite morphism $f': X'\to Y'$ of degree $d$.
Moreover, if $X$ and $Y$ are of the same dimension, and $X$ is  Cohen-Macaulay and $Y$ is regular,  then: 
\begin{enumerate}
\item $f': X'\to Y'$ is a finite, flat and surjective morphism (a $d$-branched cover).
\item There is an isomorphism of $\cO_{Y'}$-modules $\cO_{Y'}^d \buildrel{\phi}\over{\simeq} f'_*(\cO_{X'}) $, $\phi(a_1,\ldots,a_d)=a_1c_1+\ldots+a_dc_d$.
\item  The point 
$y=f'(x)\in Y'$ is in the 
%
ramified locus of $f'$ of maximal index  $d$.
\end{enumerate}
\end{corollary}

\begin{proof} (1) To show finiteness of $f'$ we apply Theorem \ref{coherent2} for $\cF=\cO_X$. 

(2)  $f'_*(\cO_{X'})$ is a finite $\cO_{Y'}$-module generated by $c_1,\ldots,c_d$. The fact that $\phi$ is an isomorphsim  follows from 
Theorem \ref{Rem2}, 
or a result by Eisenbud that local Cohen-Macaulay rings which are are finite 	over regular rings are free modules (\cite{E}). 
We can represent $X\to Y$ as the composition of the closed immersion $i: X\to X\times Y$ followed by the projection $\pi:X\times Y\to Y$.
Then, by Theorem \ref{Rem2}, 
$$
f'_*(\cO_{X'})=\pi_*(\cO_{X'\times Y'}/\cI_{i(X')})\simeq \cO_{Y'}^d
$$ 
is a free $\cO_{Y'}$-module in a neighborhood of $y$. By shrinking $Y'$ (and  $X'$) one can assume that  $c_1,\ldots,c_d$ is a basis of $f'_*(\cO_{X'})$. 

(3) follows from the fact that $\{x\}=(f')^{-1}(y)$ is irreducible.
\end{proof}

\begin{remark} 
Note that the degree of $f$ is usually greater than $d$ so the result is not valid in the Zariski topology.
	\end{remark}

\subsection{The Grauert-Remmert theorem for finite holomorphic maps}\label{lag2}

Recall that a map of topological spaces (in particular complex analytic  differentiable spaces) is \emph{finite}  if it has finite fibers and is closed.

\begin{lemma} 
Let $\cF$ be a coherent sheaf on a complex space $X$. Let $f:X\to Y$ be a holomorphic map of analytic spaces  and consider points $x\in U$ and $y=f(x) \in V$. The following conditions are equivalent:
\begin{enumerate}
\item $x$ is an isolated point in $f^{-1}(y)\cap V(\Ann(\cF))$. 
\item $\cF_x/(m_y\cdot \cF_x)$ is a finite-dimensional vector space.
	\end{enumerate}
\end{lemma}
	
\begin{proof} 
One can shrink $U$ so that $\{x\}=\cF_x/(m_y\cdot \cF_x)$; this does not affect condition (2).
Consider the coherent sheaf of ideals $\cI:=\Ann(\cF)$. Then  condition (1) is equivalent to 
$m_y\cdot\cO_X+\cI \supset m_x^d$ for some $d$, or 
%
$\cO_{X,x}/(m_y\cdot\cO_X+\cI)$ 
is of finite dimension. But the module $\cF_x/(m_y\cdot \cF_x)$ is finitely generated over	$\cO_{X,x}$, and thus over 
%
$\cO_{X,x}/(m_y\cdot\cO_X+\cI)$, which in turn is of finite dimension.
\end{proof}

\begin{remark} Condition (2) was used in particular by Malgrange, for  Malgrange-Mather division of  moduli over smooth functions, and is very convenient, especially when combined with Nakayama's lemma. Both conditions are equivalent in the algebraic and analytic situation but are different for smooth functions.
	\end{remark}

The following result  is equivalent to the  Grauert-Remmert theorem for finite morphisms of complex spaces \cite[Theorem~2, p.~54]{GR}. (We use here the Malgrange condition on the stalk $\cF_x$.)

\begin{corollary}\label{GR} (\cite{GR})
 Let $f:X\to Y$  be a holomorphic   map of analytic spaces , and consider points $x\in U$ and $y=f(x) \in V$. Let $\cF$ be a sheaf of  $\cO_X$-modules which is coherent , with a stalk $\cF_x$. Suppose $\cF_x/(m_y\cdot \cF_x)$ is a finite-dimensional vector space over $\CC$. Then there exist  neighborhoods $V'\subset V$ of $y$  and $U'$ of $x$ such that:
\begin{enumerate}

\item  The sheaf $(f_{|U'})_*(\cF)$ is a coherent $\cO_{V'}$-module.
\item  The   restriction of $f$ to  $V(\Ann(\cF))\cap U'$ defines a finite map
$V(\Ann(\cF))\cap U'\to V'$ .
\item $\cF_x\simeq ((f_{U'})_*(\cF))_y$.
\item If $c_1,\ldots, c_k$ generate $\cF_x/(m_y\cF_x)$ over the field $K=\cO_y/m_y$ then $(f_{|U'})_*(\cF)$ is generated over  $\cO_{V'}$ by $c_1,\ldots, c_k\in \cO(V')$.
\item The module $\cF(U')$ is finitely generated over $\cO(V')$.
\end{enumerate}	
\end{corollary}

\begin{proof} The corollary is a consequence of the Grauert-Remmert result  and Nakayama's lemma. By the previous lemma we can replace  condition (1) with the equivalent condition (2). For part (3) we choose the neighborhood $U$ containing only a single point of the fiber.
Part (4) follows from the Nakayama lemma and~(1). Part (5) is again a consequence of 
(4). 
\end{proof}

The corollary  also follows from the methods discussed in the next sections.
 As a particular case of the above  we get a neighborhood version of Weierstrass division (which can also be  deduced from a result by H\"ormander \cite[Theorem 6.1.1]{Hormander}).

\begin{theorem}[Weierstrass division of complex analytic functions in a neighborhood]\label{Mi1}
 Let $\cO(U)$ denote the ring of holomorphic functions on an open subset $U\subset \RR^1\times \RR^{n}$. Let $f(t,x)\in \cO(U)$ be a $d$-regular function at $(t_0,x_0)\in U$. Then Weierstrass division  by $f$ is possible in the ring $\cO(U_2)$ for  a certain  neighborhood $U_2=W\times U_1 \subset U$ of $x$. 
That is, for every $g(t,x)\in \cO(U_2)$ there are $q(t,x)\in \cO(U_2)$ and $r=\sum_{i=0}^{d-1} r_i(x)t^i\in \cO(U_2)$ with $r_i(x)\in \cO(U_1)$ such that $g=q\cdot f+r$.
This holds for any open subsets $U_1'\subset U_1$ and $U_2'=W\times U_1'$.
\end{theorem}

\begin{proof}
 The proof is similar to the proof of the algebraic version (Theorem \ref{Mi2}). We use Theorem 1 of Grauert-Remmert \cite[p.~52]{GR}. 
\end{proof}

\begin{corollary}[Singular  ``inverse function'' theorem (analytic version)] \label{I1}
Let $f:X\to Y$ be a holomorphic map  of complex  analytic  spaces, and let $x\in X$ and $y=f(x) \in Y$. Assume  $\cO_{X,x}/f^*(m_y)$ is of finite dimension $d$  over $\CC=\cO_y/m_y$ generated by  $c_1,\ldots,c_d\in \cO_{X,x}$. Then there exist open neighborhoods   $Y'\subset  Y$ of $y$  and $X'\subset X$ of $x$ such that the induced finite morphism $f': X'\to Y'$ is of degree $d$. Moreover, if  $X$ is a Cohen-Macaulay complex space and $Y$ is a manifold with $\dim(X)=\dim(Y)$ then: 
\begin{enumerate}
\item $f': X'\to Y'$ is a finite, flat, open and surjective morphism (a  $d$-branched cover).
\item There is an isomorphism of $\cO_{Y'}$-modules $\cO_{Y'}^d \buildrel{\phi}\over{\simeq} f'_*(\cO_{X'}) $, $\phi(a_1,\ldots,a_d)=a_1c_1+\ldots+a_dc_d$.
\item There is an isomorphism of $\cO({Y'})$-modules $ \phi_{Y'}:\cO({Y'})^d\simeq \cO(X')$.
\item   The point 
$y=f'(x)\in Y'$
 is in the ramified locus of maximal index  $d$.
\end{enumerate}
\end{corollary}

\begin{proof} The proof is identical to that for the algebraic version. We use Theorem \ref{GR}.
\end{proof}

\begin{corollary}[Singular ``inverse function'' theorem 2] \label{I4}
Let $f:X\to Y$ be a morphism of smooth schemes of finite type over a  field $K$ (respectively a map between   complex manifolds) of the same dimension,  and let  $x\in X$ and $y=f(x) \in Y$ be $K$-rational points. Let $u'_1,\ldots,u'_n$ be a coordinate system on $U'$ at $f(y)=x\in V$, and $u_1,\ldots,u_n$ be a coordinate system at $x$. Suppose that $f$ is given by a finite set of functions $f_i=f^*(u'_i)$ which form a 
Cohen-Macaulay regular sequence 
at $x$ (\cite{Matsumura}). Then there are  \'etale (respectively open) neighborhoods $X'\to X$ of $x$ and $Y'\to Y$ of $y$ such that the induced   morphism ${f}': X'\to Y'$ is finite,  and there is an isomorphism  of $\cO_{Y'}$-modules $\cO_{Y'}^d\to f'_*(\cO_{X'}) $.
\end{corollary}

\begin{example}  Let $f:X\to Y$ be a morphism or a map as above which is given by a finite set of functions $f_i=f^*(u'_i)$ with
 multiplicities $d_i$ and suppose the initial forms
$\inn_x(f_1),\ldots, \inn_x(f_k)$ 
define 
a regular sequence. 
Then the induced morphism ${f}': X'\to Y'$ is finite of degree $d_1\cdot\ldots\cdot d_n$.
  \end{example}
  
\begin{example} If $f_1,\ldots,f_n$ is a regular sequence in $\bA^n$, respectively $\CC^n$,
then the induced map from the zero set $V(f_1,\ldots,f_k)\to \bA^{n-k}$ given by $f_{k+1},\ldots, f_n: \bA^n\to \bA^{n-k}$
defines a finite morphism $f:W\to V$ for some neighborhoods $W\subset V(f_1,\ldots,f_k)$ and $U\subset \bA^n$, and there is a Weierstrass isomorphism $$
\cO_V^d\to f_*(\cO_W/(f_1,\ldots,f_k)).
$$
\end{example}

\section{Malgrange-Mather preparation and the inverse function theorem}\label{lag3}

The proofs of analogous theorems for  smooth functions are  more involved and will be given in the next few sections of this chapter. For the  most part we use the strategy of Malgrange to prove neighborhood Weierstrass division, which is then combined with some ideas of Grauert-Remmert   applied  to a sheaf version ( \cite{GR}). The proof of   neighborhood Malgrange special division,  which is the key technical result,  
  is essentially identical with  Milman's proof  (\cite{Milman}) of the Malgrange special division. As a consequence of the methods applied we give  an analog of the inverse function theorem (Section \ref{lag4}).

We also extend  the classical constructions mostly due to Malgrange and  Mather   in the  language of smooth objects (Section \ref{lag5}).

%

\subsection{The neighborhood version of the Malgrange special division for smooth functions}

The following theorem can be considered as a neighborhood version of Malgrange special division, which is the key technical result in  Malgrange's strategy (\cite{Malgrange1}). We shall extend the method of Milman, slightly modifying his original proof (\cite{Milman}). It turns out that in order to show the generalized Malgrange-Mather division it suffices to consider only the Malgrange special division by a generic polynomial.

\begin{theorem} \label{N1} 
Let  $U\subset \RR^{n+1}$  be  an open neighborhood of  $0\in \RR^{n+1}$.   Let $\cO(U)$ be the ring of smooth ($C^\infty$) functions  over 
%
$U$.
There exist  open neighborhoods $V^1_1\subset \RR^{1}$,  $V^d_2\subset \RR^d$ and  $W^n\subset \RR^{n}$  of $0$
 for which  $V^1_1\times W^n\subset U$, and such that
 for any $f(t,x_1,\ldots,x_n)\in \cO(V^1_1\times W^n)$ and the ``generic polynomial'' 
$$
P^d(t,y_1,\ldots,y_d):=t^d+y_1t^{d-1}+\ldots+ y_d \in \cO(\RR\times \RR^{d})
$$
 there exists  ``Malgrange special  division with remainder'': 
$$
f(t,x)=q^d\cdot P^d + r^d$$ where 
$$
r^d= r^d_{d-1}(x,y)t^{d-1}+\ldots+r^d_0(x,y),
$$
and 
$$
q^d(t,x,y)\in \cO(V^1_1\times V^d_2\times  W^n),\quad  r^d_i=r_i(x,y)\in \cO(V^d_2\times  W^n)\subset \cO(V^1_1\times V^d_2\times  W^n).
$$
Moreover the division exists for any open subset $V^1_1\times V^d_2\times  (W')^n$, with $(W')^n\subset W^n$  an open convex subset of $W^n$, which is a neighborhood of $0$. 
	\end{theorem}
	
\begin{proof} The proof can be obtained by a  cosmetic modification of  Milman's proof of the Malgrange division theorem. Here we briefly sketch the proof following Milman's original argument and referring to his paper \cite{Milman} for additional details.

Observe that  division by $\overline{P}^1:=t-y$ is very 
simple. We find neighborhoods $V^1_1\subset \RR$ and $W^n\subset \RR^n$ such that $V^1_1\times W^n\subset U$. Then we set $V_2^1:=V_1^1$ and write 
\begin{equation} \label{E0} 
f(t,x)=\frac{f(t,x)-f(y,x)}{t-y}\cdot (t-y)+f(y,x)
=q_1(t-y)
+r^1,
\end{equation} 
where $q_1:=\frac{f(t,x)-f(t,y)}{t-y}\in \cO(V^1_1\times V^1_2\times W^n)$ and $r^1=r^1_1=f(y,x)\in \cO(V_2\times V_2^1)$.
This can be equivalently translated into division by  $P^1(t,y)=t+y=\overline{P}^1(t,-y)$ on the same neighborhoods.

Suppose that division by $P^{d-1}(t,y_1,\ldots,y_k)$ is defined on $V^1_1\times V^{d-1}_2\times W^n\subset \RR^{n+d}$ in the form  
$$
f(t,x)=q^{d-1}(t,x,y)\cdot P^{d-1} + r^{d-1}.
$$ 
 Applying division by $t-y_{d}$ to $q^{d-1}=(q')^{d}(t-y_d)+(r'')^d$,  one defines division 
by $P^{d-1}(t-y_{d})$:
\begin{equation} 
\label{E1} f(t,x)=(q')^{d}\cdot P^d(t-y_{d}) +(r')^d ,
\end{equation}
where 
\begin{equation} \label{E2} 
(r')^d:=(r'')^{d}\cdot P^d+r^{d-1}=(r')^d_{d-1}(x,y)t^{d-1}+\ldots+(r')^d_0(x,y),
\end{equation}		
which is defined on the open subset $(V^1_1\times  (V^{d-1}_2\times V^1_2)\times W^n)\subset \RR^{n+d+1}$ (where $V^1_2=V^1_1$).
The actual division by $P^{d}$ is done by passing from 
$P^{d-1}(t-y_{d})$ to $P^{d}$. Consider the mapping 
\begin{gather*}
\psi_{d-1,n}: \RR^{d-1}\times \RR\times \RR^{n}\to \RR^{d}\times \RR^{n},
\\
(y_1,\ldots,y_{d-1},y_{d},x_1,\ldots,x_n)\mapsto (y'_1,\ldots,y'_{d},x_1,\ldots,x_n),
\end{gather*} 
given by the equality of the generic polynomials 
$$
P^{d-1}(t,y)\cdot(t-y_{d})=P^{d}(t,y').
$$
This map can be conveniently  represented as a combination of two maps 
$$
\RR^{d+n}\buildrel{\phi_1}\over\to V^d\times \RR^n \buildrel{\pi_d}\over\to \RR^{d}\times \RR^{n}
$$
where $V^d\subset \RR^{d+1}$ is the set of zeroes of  $P^{d}(t,y)$ in $\RR^{d+1}\times \RR^n$, and the 
isomorphism  $\phi_1$ is given by  
$$
(y_1,\ldots,y_{d-1},y_{d},x_1,\ldots,x_n)\mapsto
(t=y_d,y'_1,\ldots,y'_{d},x_1,\ldots,x_n),
$$
and is followed by the projection 
%
$\pi_d(t,y'_1,\ldots,
y'_{d-1},y'_{d},x_1,\ldots,x_n)=(y'_1,\ldots,y'_{d-1},y'_{d},x_1,\ldots,x_n).$

We shall use a modification of  Milman's lemma:

\begin{lemma}[Milman](\cite{Milman}
 Let $\cO_{\pi_d}(V^d\times \RR^n)$ be the subspace of 
%
$\cO(V^d\times \RR^n)$ 
consisting of all functions which are constant on the fibers of $\pi_d$.
There exists a continuous linear  operator
$J:\cO_{\pi_d}(V^d\times \RR^n)\to \cO(\RR^d\times \RR^n)$ such that for any $f(t,y,x)\in \cO_{\pi_d}(V^d\times \RR^n)$ there is a function
 $$
J(f)(y,x)=f(t,y,x)\in \cO(\RR^d\times \RR^n).
$$ 
Likewise for	 any open convex  neighborhood of $0$, say of the form   
%
$U^d\times W^n\subset \RR^d\times \RR^n$, there exists a linear operator 
$J_d: \cO_{\pi_d}(\pi_d^{-1}(U^d\times W^n))\to \cO(U^d\times W^n)$ for which
$J_d(f)(y,x)=f(t,y,x)$. Moreover the operator $J_d$ exists for $U^d\times W^n$ if $W^n$ is replaced with any  open convex subset which is a neighborhood of $0$.  
\end{lemma}

\begin{proof}
The proof remains the same. Roughly speaking, we show that the function $f(t,y,x)$ which is constant on the fibers can be ``pushed down'' to  a unique differentiable function $(\pi_d)_*f(t,y,x)$ on the  closed subset 
%
$Z:={\pi_d}(\pi_d^{-1}(U^d\times W^n))$ 
of $U^d\times W^n$. 
Since $\pi_d$ is closed, the function is continuous. It is also differentiable at the points where the Jacobian of $\pi_d$ does not vanish. It turns out that the Jacobian is equal to $P_d(t)$.
By an inductive argument, one can extend  differentiability of  $(\pi_d)_*f$ to the set where $P_d$ has a root $t$ of multiplicity exactly $k$, and finally show that it is differentiable on $Z$.

In the complex analytic situation 
the mapping is surjective and we are done: the function $J_d(f)=(\pi_d)_*f$ is holomorphic on $Z=U^d\times W^n$. In the differential setting, if the dimension is even the mapping is not surjective and we need to extend it.
(This causes nonuniqueness of  Malgrange-Mather division.)
We extend the function using Stein's (or Whitney's) extension theorem. We use the fact that $(U^d\times W^n)\setminus Z$ is convex. This set can be interpreted as corresponding to the monic polynomials with coefficients in  $U^d\times W^n$ having no roots 
(being positive), and is convex.
\end{proof} 

We briefly sketch Milman's  (adjusted) strategy for the remaining part of the proof.
We choose a convex $W^n$ and sufficiently small convex $V_2^d$ such that  there is 
an  embedding $\phi_1^{-1}: {\pi_d}^{-1}(V_2^d\times W^n) \to V_2^1\times V_2^{d-1}\times W^n$. 
It follows from Milman's lemma  that there is a linear operator 
$$
J_{d}: \cO_\psi(V_2^1\times V_2^{d-1}\times W^n)\to \cO(V_2^d\times W^n).
$$
We define  the remainder of division of $f$ by $P^d$ on $V_1^1\times V_2^d\times W^n$ to be $r^d=\sum r^d_it^i$, where  
$$
r^d_i:=J_{d}((r')^d_i)\in \cO(V_2^{d}\times W_n)\quad \text{for}\quad(r')^d_i\in \cO(V_2^1\times V_2^{d-1} \times W_n),
$$ 
and show (as in Milman's proof) that  $f-r_d$ is in fact divisible by $P^d$. (Roughly speaking, the polynomial $P^d=\ldots+y_d$ is a coordinate and $f-r_d$ vanishes along its zero locus.) Then there exists a quotient $q^d\in \cO(V_1\times V_2^d\times W_n)$ such that
$f=q^dP^d+r_d$ on $U_d\times W_n$. The details are the same as in Milman's proof (\cite{Milman}). Apart from few mentionned modifications the proof is identical with Milman's one.
\end{proof}

In the real analytic situation the above reasoning proves a local version of  Malgrange-Mather special division in the ring $\cE_n$ of germs of real analytic functions. It is not clear whether the corresponding neighborhood version remains valid for real analytic functions.

\subsection{Malgrange-Mather generalized division} \label{01}
In Sections \ref{01}, \ref{02}, \ref{03} we shall mean by $\cE_n$ the local ring of germs of smooth functions on $\RR^n$ at the origin $0$.

\begin{definition} 
$f \in  \cE_{n+1}(t,x)$ is \emph{$d$-regular with respect to $t$} (where $x=(x_1,\ldots,x_n)$) if 
  $$
{f(0,0)}= {\frac{\partial{f}}{\partial t}((0,0)}=\ldots ={\frac{\partial^{d-1}f}{\partial t^{d-1}}(0,0)}= 0,\quad {\frac{\partial^{d}f}{\partial t^{d}}(0,0)}\neq 0.
$$
\end{definition}

\begin{theorem}[Malgrange-Mather preparation and division theorem] 
\begin{enumerate}
\item Let  $f(t,x)\in  \cE_{n+1}$  be a 
$d$-regular function. There exists a Weierstrass polynomial 
%
$$
P^d=t^d+c_1(x)\cdot t^{d-1}+\ldots+c_d(x) 
$$ 
and an invertible function $\alpha\in \cE_{n+1}$ such that  $f(t,x)=P^d\cdot\alpha$.

\item Let $f(t,x),g(t,x) \in  \cE_{n+1} $ with $f$ 
$d$-regular. Then  there exist  $q,r \in  \cE_{n+1}$ and $h_j \in \cE_n$, $j=0,\ldots,d-1$, such that
$$
g=qf+r,\quad \text{where} 
\quad r=r(t,x)= \sum^{d-1}_{j=0}h_j(x)t^j.
$$
\end{enumerate}

\end{theorem}
The above theorem is a consequence of generic Malgrange division via the
``Malgrange trick''.
\begin{proof} 
Let $P^d(y,t)=t^d+\ldots+y_d \in  \cE_{n+1}$ be the generic polynomial. 
Write 
%
$$
f(t,x)=h(t,x,y)P(y,t)+r(t,x,y),\quad g(t,x)=h_1(t,x,y)p(y,t)+r_1(t,x,y),
$$
where 
$$
r(t,x,y))=\sum^d_{i=0} A_i(x,y )t^i,\quad r_1(t,x,y))=\sum^d_{i=0} B_i(x,y)t^i
$$
with $A_i, B_i\in \cE_{n+1} $.
One can easily see that  $\Psi:=(A_1,\ldots,A_d,x_1\ldots,x_k)$ is invertible since the Jacobian of $\Psi$ at $0$ is upper triangular.
Its inverse is given by $\Psi^{-1}=(\phi_1,\ldots,\phi_k,x_1\ldots,x_k)$. Then 
$$
\Psi^{-1}(0,x)=(\phi_1(0,x),\ldots,\phi_k(0,x),x_1\ldots,x_k),\quad
A_j(\phi_j(0,x),x)=0.
$$
Consequently,  
%
$$
f(t,x)=h(t,x,\phi(x))p(\phi,t)+r(t,x,\phi(x))
=h(t,x,\phi(x))p(\phi,t)
$$
 with $h(0,0)=c\neq 0$ and $h(t,x,\phi(x))$ invertible.
 Also 
%
$$
g(t,x)=h_1(t,x,\phi)p(y,\phi)+r_1(t,x,\phi)
=(h_1(t,x,\phi)h^{-1}(t,x,\phi(x)))f(t,x)+r_1(t,x,\phi)
$$
(see the details in 
%
\cite{Malgrange1}).
\end{proof}

The  Weierstrass preparation and division theorem was proved by Weierstrass for holomorphic functions. 
Its extension to convergent power series over a valued field was  proved in \cite{Na62}, and for formal power series over a  field in \cite{Bo65}. The algebraic case was proven in \cite{Laf65}.

\begin{corollary} 
Let $f(t,x)\in \cE_n[t]$ be a  polynomial in $t$, of degree $k$, 
which is $d$-regular in $t$, where $d\leq k$.
 Then  there exists  a factorization into polynomials
 $$
f(t,x)=g_1(x,t)g_2(x,t),
$$
 where 
$$
g_1=t^d+A_1(x)t^{d-1}+\ldots+A_k(x),\quad
g_1(t,0)=t^k,
$$ 
and $g_2(x,t)$ is of degree $k-d$ with $g_2(0,0)\neq 0$.
\end{corollary}

\begin{proof} Let $P^d(y,t)=t^d+\ldots+y_k \in  \cE_n[t]$ be the generic polynomial. Consider  division with  remainder in the polynomial ring $E(2n)[t]$:
 $$
f(t,x)=h(t,x,y)p(y,t)+r(x,y,t),
$$ 
where $r(x,y,t)$ is a polynomial in $t$ of degree $d-1$.
As before we  find functions $\phi_1(x),\ldots,\phi_k(x)$ such that
after substitution, $r(t,x,\phi)=0$. The reasoning is the same as in the previous proof.
\end{proof}

As a consequence, all the rings $\cE_n$ are Henselian: 

\begin{corollary}[Hensel's lemma] Let $f(t,x)\in \cE_n[t]$ be a  monic polynomial in $t$ of degree $k$  such that
$$
{f(t,0)}=(t-c_1)^{k_1}\cdot\ldots \cdot(t-c_r)^{k_r}
$$
for some $c_i\in {K}$. 
Then there exists a factorization 
$$
f=g_1\cdot\ldots\cdot g_r
$$
 into monic polynomials $g_1,\ldots,g_r$ of degree $k_1,\ldots, k_r$ such that 
$$
g_i(t,0)= (t-c_i)^{k_i}.
$$
\end{corollary}

\begin{proof} The function $f^{c_1}(t):={f(t+c_1,x)}$ is $k_1$-regular. Hence  by the previous result there exists a factorization into the product of polynomials 
$$
f^{c_1}(t,x)={g_1^{c_1}(x,t)}\cdot {f_1^{c_1}(x,t)}
,
$$
and 
$$
f(t,x)=g_1(x,t)f_1(x,t),
$$
where $g_1(x,t):={g_1^{c_1}(x,t-c_1)}$ and $f_1(x,t):={f_1^{c_1}(x,t-c_1)}$ are polynomials of degree $k_1$ and $k-k_1$.
Then  by uniqueness $g_1(t,0)= (t-c_1)^{k_1}$ and $f_1(x,0)=(t-c_2)^{k_1}\cdot\ldots\cdot (t-c_r)^{k_r}$ and we can proceed by induction.
\end{proof}

\begin{remark} 
The fact that the local rings of smooth functions are Henselian was proven in \cite{Moe}.
\end{remark}

\subsection{Neighborhood version of  Malgrange-Mather division}\label{02}

Let $X$ be a topological space and $\cO_X$ be a sheaf of rings. We call  $\cO_X$ the \emph{structural sheaf} and the pair $(X,\cO_X)$ a \emph{ringed space.} It is a \emph{locally ringed space} if all the stalks 
$\cO_{X,x}$ are local rings.
 
 We say that a sheaf of $\cO_X$-modules $\cF$ is \emph{of finite type} if for any $x\in X$ there is a neighborhood $U\subset X$ and a surjective morphism  of sheaves $\cO^k_{|U}\to \cF_{|U}$.

A locally ringed space $(Y,\cO_Y)$ is called a closed subspace of $X$ if $Y$ is a closed subset  of $X$, there is an ideal sheaf of finite type $\cI_Y\subset \cO_X$ for which $Y$ is is the support of  the sheaf $\cO_x/\cI_Y$, and $\cO_Y$ is the restriction of $\cO_X/\cI_Y$ to $Y$. By a {\it differentiable space}  mean a locally ringed space $(Y,\cO_Y)$ which is locally a closed   subspace of open subset of $(\RR^n,\cO_{\RR^n}$ with smooth functions sheaf $\cO_{\RR^n}$.

The definition of the sheaf of $\cO_X$-modules  of finite type implies immediately that if the stalk $\cF_x$ of the sheaf of finite type is 
an $\cO_x$-module  finitely generated by sections $s_1,\ldots,s_k\in \cF_x$ then there is a neighborhood $U$ of $x$ such that $s_1,\ldots,s_k\in \cF(U)$ generate the $\cO(U)$-module
$\cF(U)$.

We shall use the following lemmas:

\begin{lemma} \label{ft} 
Let  $X$ be an open subset of $\RR^n$ and $\cO_X$ be a sheaf of smooth functions on $X$. Let $\cF$ be a sheaf of
%
$\cO_X$-modules of finite type generated over $X$ by finitely many sections $s_1,\ldots,s_k\in \cF(X)$. Then for an  open $U\subset X$, the 
$\cO(U)$-module $\cF(U)$ is generated by $s_1,\ldots,s_k$ over $\cO(U)$.
\end{lemma}

\begin{proof}
If $t\in \cF(U)$, then for any $x\in U$ the germ of $t$ at $x$ can be written as a finite sum $t_x=\sum a_{ix}s_{ix}$, where $a_{ix}\in \cO(V_x)\subset \cO_x$, for a sufficiently small neighborhood 
$V_x$ of $x$. One can consider a locally finite subcover $\{V_j\}$ of $\{V_x\}$ and a subordinate partition of unity $1=\sum b_j$.
Then consider $1\cdot t= (\sum b_j)\cdot t$, and $ t_{|V_j}= (\sum_{V_j\subset U_i} b_ja_{i})_{|V_j}s_{i|V_j}$. But $c_i:=(\sum_{V_j\subset U_i} b_ja_{i})\in \cO(U)$ is a function defined on $U$, and the sections $t$ and $\sum c_is_i$ coincide on each $V_j$ and thus  $t=\sum c_is_i$.
\end{proof}

\begin{lemma} \label{ft2} Let  $X$ be an open subset of $\RR^n$.
If $\cF$ is a sheaf of $\cO_X$-modules  of finite type generated by global $s_1,\ldots, s_k$, and  $\cG\subset \cF$ is its $\cO_X$-submodule also of finite type generated by global sections$t_1,\ldots,t_s$, then 
$$
(\cF/\cG)(U)=\cF(U)/\cG(U)
$$ 
for any open subset $U$.
	\end{lemma}
 
\begin{proof}
Consider the quotient morphism of sheaves $f:\cF\to \cF/\cG$. The sheaf 
$\cF/\cG$ is of finite type generated by $s_1,\ldots,s_k$. Then $(\cF/\cG)(U)$ is also generated by $s_1,\ldots,s_k$, and $f_U:\cF(U)\to (\cF/\cG)(U)$ is an epimorphism. On the other hand, $\cG(U)$ is in the kernel $K(U)$ of $f_U$.  For any $t\in K(U)\subset \cF(U)$  the germs of $t$ are generated by finitely many sections $t_1,\ldots,t_s$ of $\cG(X)$. Then by the same reasoning
$t=\sum c_is_i$.
 \end{proof}

Also it follows from the definition that
\begin{lemma} \label{ft3} If $f:Y\to X$ is a closed embedding of differentiable spaces then for any sheaf of finite type $\cF$, the direct image $f_*(\cF)$ is of finite type
\end{lemma}

\begin{proof} Suppose $\cF$ is generated by  sections $s_1,\ldots,s_k\in \cF(U)$, where $U\subset Y$. For any $y\in U$ and $f\in \cO(U)$ exist $g_x\in \cO_{X,x}$ which restricts to $f_x\in \in \cO_{U,x}$.
Each $g_x$ is defined on open neighborhood $V_x\subset X$.
By using the partition of unity argument we extend $f\in \cO(U)$ to a certain $g\in \cO(V)$ which is an open subset $V\subset Y$ such that $V\cap X=U$, and which is common for all $f$. Thus $\cO(V)\to \cO(U)$ is the surjection, and $s_1,\ldots,s_k\in f_*(\cF)(V)$ generate $f_*(\cF)$ on $V$ over $\cO_V$.
	
\end{proof}

\begin{theorem}[Generalized Malgrange-Mather division in a neighborhood]\label{MM1}
 Let $\cO(U)$ denote the ring of smooth  functions on an open subset $U\subset \RR^1\times \RR^{n}$. Let $f(t,x)\in \cO(U)$ be a $d$-regular function at $y=(t_0,x_0)\in U$. Then  there is  Malgrange-Mather division  by $f$ in the ring $\cO(U_2)$ for  a certain  neighborhood $U_2=V_1\times U_1 \subset U$ of $x$. 
That is, for every $g(t,x)\in \cO(U_2)$ there are $q(t,x)\in \cO(U_2)$ and $r=\sum_{i=0}^{d-1} r_i(x)t^i\in \cO(U_2)$ with $r_i(x)\in \cO(U_1)$ such that $g=q\cdot f+r$.
\end{theorem}

\begin{proof} 
First, by applying  Malgrange preparation and shrinking $U$, we can assume that $f(t,x)=\alpha(t,x)\cdot \overline{P}^d(t,x)$ on $U$ where  $\overline{P}^d(t,x)=t^d+c_1(x)\cdot t^{d-1}+\ldots+c_k(x)$. Then there is  generic division by $P^d(t,y)=t^d+y_1\cdot t^{d-1}+\ldots+y_k$ in some neighborhood  $V_1\times V_2\times W$, where $W=U_1$.  Consider the maps 
$$
c:  U_1\to  \RR^d,\quad c_1:=(c,id_{U_1}): U_1\to  \RR^d\times U_1\quad \text{and}\quad c_2:=\id_{V_1}\times c_1: V_1\times U_1\to V_1\times \RR^d\times U_1, \quad t\mapsto t,\,\, y_i\mapsto c_i.
$$ By shrinking  $U_1$ we can assume that $c( U_1)\subset V_2$ and thus  $$ c_2(V_1\times U_1)\subset  V_1\times V_2\times U_1.$$ Applying  special  division in $V_1\times V_2\times U_1$  to   $g\in \cO(V_1\times U_1)$ we can write 
$$
g=q^dP^d+r^d
$$
with $q^d,r^d\in \cO(V_1\times V_2\times U_1)$ and $r^d_i\in \cO( V_2\times U_1)$.
Then  $\overline{P}^d(t,x)={P}^d(t,x)\circ c_2 \in \cO(V_1\times U_1)$ and for  
$\overline{r}(t,x):=r^d\circ{c_2}\in \cO(V_1\times U_1)$, and $\overline{q}:=q^d\circ{c_2} \in \cO(V_1\times U_1)$ we have the division 
$$
f=\overline{q}\overline{P}^d+\overline{r},
$$
which implies $g=q\cdot f+r$ with $q=\overline{q}/\alpha$ and $r=\overline{r}$.
\end{proof}

The  results below are extensions of the Grauert-Remmert  ``Weierstrass isomomorphism''   \cite[Theorem 1.2.3]{GR} in the holomorphic case to the differential setting with a similar proof.
\begin{lemma}\label{Rem6} If $f: X\to Y$ is finite map of subsets in $\RR^n$ then for any $y\in Y$ with the fiber $f^{-1}(y)=\{x_1,\ldots,x_k\}$ and $\epsilon > 0$ there exists $\delta>0$ such that if $f^{-1}(B_\delta(y))\subset \bigcup B_\epsilon(x_i)$. (Here $B_\delta(y)$ is an open ball of radius $\delta$ with the center $x$.)
\end{lemma}
\begin{proof} Suppose otherwise. Then for each $n\in \NN$, there exists $z_n\not\in \bigcup B_\epsilon(x_i)$ with $f(x_n)-y<1/n$. If the sequence $(z_n)$ has an accumulation point $z=\lim z_{n_i}$ then $f(z)=\lim(f(z_{n_i})=y$. But $z\not \in \{x_1,\ldots,x_k\}=f_(y)$ which is a contradiction. If $(z_n)$ does not  have an accumulation points and defines a closed subset of $X$ with non closed image $\{f(z_n)\}$ with accumulation point $y=\lim f(x_n)$. First prove the lemmas
	
\end{proof}
\begin{lemma}\label{Rem5}
	If $f:X\to Y$ is a finite 
 map  of subsets in $\RR^n$ with the fiber $f^{-1}(y)=\{x_1,\ldots,x_k\}$ of $y\in Y$, and $\cF$ is a sheaf on $X$ then there is an isomorphism of  $\cO_{Y,y}$-modules 
$(f_*(\cF))_{y}\simeq \bigoplus_i \cF_{x_i}$

\end{lemma}
\begin{proof} By the previous lemma for $n>>0$ and  $U_n=B_{1/n}(y)$ the preimage is a union of 
disjoint neighborhoods $W_{ni}$ of $x_i$  with $\bigcap_n W_ni=\{x_i\}$. We conclude that 
$$
\pi_{*}(\cF)_{y}=\lim_{{U_n}\ni y} \cF(\pi_i^{-1}({U_n})) =\bigoplus \lim_{W_{ni}\ni x_i}\cF(W_{ni})= \bigoplus  \cF_{y_i}.
$$

\end{proof}

\begin{corollary}[Grauert-Remmert  ``Weierstrass isomorphism'']\label{Rem}  
Let $\pi: \RR^{n+1}\to \RR^{n}$  be the standard projection. For any open subset  $U\subset \RR^{n+1}$ 
denote by $\cO_U$ the sheaf of smooth   functions on $U$.
Consider a  function $f(t,x)\in \cO(U)$   which is  $d$-regular for $\overline{x}\in U$. Then there is a convex open neighborhood ${U_2}:=W\times U_1$ of $\overline{x}$ such that: 
\begin{enumerate}
	\item  
%
$\pi_{0*}(\cO({U_2})/(f(t,x)))$ is a  finitely generated $\cO(U_1)$-module,  where $\pi_0: {U_2}=W\times U_1\to U_1$ is the natural projection. 
	\item There are surjections of $\cO_{U_1}$-modules 
%
$$
\phi: \bigoplus_{i=0}^{d-1}\cO_{U_1} t^i \to \pi_*(\cO_{U_2}/(f(t,x))),
$$ 
	and of the rings of global sections
%
$$
\phi_{U_2}: \bigoplus_{i=0}^{d-1}\cO({U_1}) t^i\to \cO({U_2})/(f(t,x)).
$$ 
	\item The restriction  $\pi_1: V(f)\to {U_1}$ to the zero set $V(f)$ of $f$ on ${U_2}$ is a finite (thus closed and proper) map. 	 
  \item   If $\cF$ is a  $\cO_U$-sheaf  of finite type over $U$  which is annihilated by the function $f$ and generated by the global sections on ${U_2}$ then 
%
$\pi_{1*}(\cF_{|{U_2}})$ is of finite type on $U$.
\item  $\cF_x\simeq (\pi_{1*}(\cF_{|U_2}))_x$.
\item $U_1$ can be replaced, in particular, with any open convex subset $U_1'$ containing $\pi(x)$, and the above conditions will be satisfied.
\end{enumerate}
\end{corollary}

\begin{proof}
(3) Let us first show that the map  $\pi_1: V(f)\to {U_1}$ is closed  and has finite fibers.
One can replace $f$ with the Weierstrass polynomial
$P_d=t^d+c_1t^{d-1}+\ldots+c_d$ defined on the open set $U\subset \CC^k$ (or $\RR^n$). Then the points in the fibers correspond exactly to the roots of the polynomial $P_d$, and thus are finite and of cardinality $\leq d$. Now, if the sequence of points $y_n:=\pi_1(x_n)$ converges to $y$ then it defines a converging sequence of polynomials $P_{d,y_i}$ with coefficients $c_i(y_n)\mapsto c_i(y)$. This  implies that the coefficients of the polynomials $P_{d,y_i}$ are bounded, as also are their roots (see computation below). Thus there is a convergent subsequence $x_{n_k}\to x$, and $y:=\pi(x)$ is the limit of $y_n$. This shows that $\pi_1$ is closed and finite.
(See proof of \cite[Theorem 1.2.3]{GR} for details)

(5) By Lemma \ref{Rem5}, for any $p\in {U_1}$ there is an isomorphism of  stalks 
$$
\pi_*(\cO_{U_2}/(f))_p\simeq \bigoplus_{y\in \pi_1^{-1}(p)} (\cO_{U_2}/(f))_y,
$$ 
and in general if $\cF$ is annihilated by $f$  then it can be considered as a direct image of a sheaf $\cF'$ on the vanishing locus $V(f)\subset U_2$, and again using  Lemma \ref{Rem5}, we get 
%
$$
\pi_*(\cF)_p\simeq\bigoplus_{y\in \pi_1^{-1}(p)} \cF_y. 
$$

(1) \& (2) By  Theorem \ref{MM1}, there is an epimorphism 
%
$$
\phi_{U_1}: \bigoplus_{i=0}^{d-1}\cO({U_1}) t^i \to \cO({U_2})/(f(t,x)
$$
of the rings of global sections.
By a partition of unity argument, any germ of $\cO_{{U_2},y}$ or $(\cO_{U_2}/(f))_y$ at a point extends to a global section over ${U_2}$. 
Using the epimorphism of global sections we find  that any element in $(\cO_{U_2}/(f))_y$ is in the image of $\bigoplus_{i=0}^{d-1}\cO_{U_1} t^i$. Likewise any element in $\bigoplus_{y\in f^{-1}(p)} (\cO_{U_2}/(f))_y\simeq \pi_*(\cO_{U_2}/(f))_p$ extends to a global section in $\cO_{U_2}/(f)$, and thus in $\bigoplus_{i=0}^{d-1}\cO_{U_1} t^i$.
This implies that $\phi$ is an  epimorphism on the sheaves.

 (4) The  restriction $\cF_{|U_2}$ of $\cF$ to $U_2$ can be considered  as a sheaf of  $\cO_{U_2}/(f)$-modules of finite type. The module  
%
$\pi_{1*}(\cO_{U_2}/(f))$ is a finite $\cO_{U_1}$-module. We need to show that $\pi_{1*}(\cF)$ is a finite $\pi_{1*}(\cO_{U_2}/f)$-module. The stalk $\pi_*(\cF)_y$ is, as before, isomorphic to $\bigoplus_{y\in f^{-1}(p)} \cF_y$. The epimorphism of sheaves $\cO_{U_2}^k\to \cF$ factors through  $\cO_{U_2}^k/(f)\to \cF$, which gives
an epimorphism of stalks $(\cO_{U_2}^k/(f))_y\to \cF_y$.

Since   $\pi_{1*}(\cF)_p\simeq\bigoplus_{y\in f^{-1}(p)} \cF_y$ and $\pi_{1*}(\cO_{U_2}^k/(f))_p\simeq\bigoplus_{y\in f^{-1}(p)} (\cO_{U_2}^k/(f))_y$, we get  a surjection on the stalks $\pi_{1*}(\cO_{U_2}^k/(f))_y\to (\pi_{1*}(\cF))_y$.	
\end{proof}


\subsection{Malgrange preparation for modules} \label{03}
 
 The Malgrange-Mather division theorem generalizes to
 the following theorem on moduli (see also \cite{Diff}).
 
 \begin{theorem}[Malgrange]\label{Mal1} 
Denote by $\cE_k:=C^\infty_0(\RR^k)$ the local ring of smooth functions at $0\in \RR^k$. Let 
 $\phi_k^* : \cE_k\to \cE_m$ be any homomorphism induced by a smooth map $\phi_k: U\to \RR^k$ defined in a neighborhood $U$ of $0\in \RR^m$, and  $M$ be 
 any  finitely generated $\cE_n$-module. Then  the following conditions are equivalent:
 \begin{enumerate}
\item $M$ is finitely generated over $\cE_k$.

\item  The dimension  of the vector space  
%
$M/(\phi_k^*(\cam_k)\cdot M)$ 
over ${{\RR}}=\cE_k/{m_k}$ is finite. 
\end{enumerate}
\end{theorem}

\begin{proof}  The proof is identical to Malgrange's  original proof. As we need some elements of the proof later, we shall briefly present it here.

$(1)\Rightarrow(2)$. Obvious.

$(2)\Rightarrow(1)$.
Consider the case where  $n=k+1$ and $\phi_k^*=\pi_{n,k}^*:\cE_{k}(x)\to \cE_n(t,x)$  is  defined by the natural projection onto the second factor, and for simplicity of notation let us identify $\cE_k(x)$ with the subring of $\cE_{k+1}(t,x)$.
Let $a_1,\ldots,a_r$ be  elements   of $M$ whose  classes modulo ${\cam}_n$ form a basis of the vector space $M/({\cam}_n\cdot M)$ over $\RR=\cE_n/{\cam}_n$. Then, by Nakayama, $a_1,\ldots,a_r$ generate $M$ over $\cE_n$ and their  classes modulo $\cam_k$ generate $M/(\cam_k\cdot M)$ over $\RR=\cE_k/{m}_k$.

Write  
$$
ta_j=\sum \alpha_{ij}a_i +\sum_{s\in S_j} \beta'_{sj}\cdot b_{sj},\quad \text{where}\quad \alpha_{ij} \in \RR,\quad  \beta'_{sj}\in {\cam_k},\quad  b_{sj}\in M, 
$$
and the sets of indices are finite.
Since $a_1,\ldots,a_r$ generate $M$, we have 
$$
b_{sj}=\sum \gamma_{sij}a_j
$$ 
with 
%
$\gamma_{sij} \in \cE_{n}=\cE_{k+1}$. This yields
$$
ta_j=\sum (\alpha_{ij}+\beta_{ij})a_i ,
$$
where $\beta_{ij}=\sum_i\sum_{s} \beta'_{sj}\gamma_{sij}\in {m}_k\cdot\cE_n$.

Consider the square matrices $A:=[\alpha_{ij}], B=[\beta_{ij}]$ and the vector $\overline{a}:=(a_1 , \ldots, a_r)$.
Then we write the above in  matrix form
$$ 
(tI-A-B)\overline{a}=0.
$$
Then
$$
\adj(tI-A-B)\cdot (tI-A-B)\overline{a}=\det(tI-A-B)\overline{a}=0.
$$
This implies that  $\Delta:=\det(tI-A-B)\in \cE_n$ annihilates $M$, that is, $\Delta\cdot M=0$.

Then $\Delta(t,0)=t^k+c_1(0) t^{k-1}+\ldots+c_k(0) \in \RR[t]$ is a monic polynomial in $t$ of degree $k$. 
 Let $q\leq k$ denote the multiplicity of ${\Delta(t,0)}$ at $0$. Then we can write ${\Delta(t,0)}$ as a product of two polynomials ${\Delta(t,0)}=t^q\cdot \alpha$, where ${\alpha(0,0)}\neq 0$ and $\Delta$  is $q$-regular at $(t,0)$.

The module $M$  is  a finitely generated  
%
$\cE_{k+1}/(\Delta\cdot \cE_{k+1}$)-module. 
By  the  division theorem, 
%
$\cE_{k+1}/(\Delta\cdot\cE_{k+1})$ is generated by $\{1,t,\ldots,t_{q-1}\}$ over  $\cE_k$. This implies that $M$ is finitely generated  over $\cE_k$.

Assume now that $k\geq n$ is arbitrary, and $\rank(f)=n$.
Then in certain   coordinates, $f$ is an embedding onto the first $n$ coordinates, 
$f: \RR^n\to \RR^k$, with
 $$
f: (x_1,\ldots, x_n)\mapsto (x_1, \ldots, x_n, 0, \ldots, 0),
$$
 and 
  $f^* : \cE_k\to \cE_n$ is surjective, that is, $f^*(\cE_k)=\cE_n$ and $f^*(m_k)=m_n$, and the conclusion is valid.
 In  general  any   map $f: \RR^n\to \RR^k$ is the composite of the immersion
  $$
(id,f): \RR^n\to \RR^k\to \RR^{n+k}
$$
 followed by a sequence of $n$ projections
$$ 
g:\RR^{n+k}\to  \RR^{k}
$$ 
as in the first case. It suffices to notice that  if the theorem is true for two maps $f$ and $g$ then it is true for their composition.
\end{proof}

The above theorems can be extended to their neighborhood versions by using Weierstrass division proven before. The following results can be related to  \cite[Theorem 1.3.1]{GR}.

\begin{lemma} \label{projection} 
Let  $V\times U\subset \RR^{k}\times\RR^n$  be an open neighborhood of the point $z=(x,y)$, and $\pi: V\times U \to U$ be the restriction of the standard projection $\pi: \RR^{k+n}\to \RR^k$. 
Let $\cF$ be a sheaf of  $\cO_{V\times U}$-modules  of 
 finite type, with a stalk $\cF_z$. 
 Suppose $\cF_x/(m_{y}\cdot \cF_z)$ is a finite-dimensional vector space over $\RR$. Then there exists an open convex  neighborhood $V'\times U'$ of $z=(x,y)$ such that: 
 \begin{enumerate}
 \item	The sheaf $(\pi_{|V'\times U'})_*(\cF_{|V'\times U'})$ is of finite type (in the differential setting). 
 \item The restriction of $\pi$ to the vanishing locus of the annihilator 
%
$V(\Ann(\cF))\cap U'$ is a finite map (thus closed).

%
 \item $((\pi_{|U'})_*(\cF))_y \simeq \cF_z$. 
 \item  If $U''$ is a convex open neighborhood of  $y$ then the above  conditions are satisfied for $V'\times U''$.
 \end{enumerate}
	\end{lemma}

\begin{proof}
The first part of the reasoning is the same as in the proof of Theorem \ref{Mal1}. 
We settle the case of $k=1$. 
Let $a_j\in \cF_z/(m_y\cdot\cF_z)$ be generators as in the proof of Theorem \ref{Mal1}.
Find convex subsets $V_1\times U_1\subset U$ with $V_1\subset \RR$, $U_1\subset \RR^n$ such that:
\begin{enumerate}

\item The elements $a_j$ generate $\cF(V_1\times U_1)$ over $\cO(V_1\times U_1)$.
\item $ta_j=\sum (\alpha_{ij}+\beta_{ij})a_i $,
where $\beta_{ij}\in m_y\cdot\cO(V_1\times U_1)$, $\alpha_{ij}\in \RR$, and  $m_y\subset \cO(U_1)$ is the maximal ideal of $y$.
\item There is a $q$-regular function $\Delta\in \cO(V_1\times U_1)$ at $(0,0)$ which  annihilates  $\cF_{|V_1\times U_1}$.
\item $\cO(V_1\times U_1)/\Delta$ is a finitely generated $\cO(U_1)$-module.
\end{enumerate}
The polynomial $\Delta$  annihilates the generators $a_i$ of the module  $\cF(V_1\times U_1)$. Hence  $\cF(V_1\times U_1)$ is a finite $\cO(V_1\times U_1)/(\Delta)$-module and thus a finite $\cO(U_1)$-module. Thus the case follows from Theorem \ref{Rem}.
Also if we replace $U_1$ with an open  neighborhood  $U_1'$ of $y$ then the above conditions will be  satisfied.

Then we show the general case by  induction on $n$.
We can find  an open convex neighborhood  $V=V_1\times\ldots \times V_k\subset \RR^k$ of $(x_1,\ldots,x_k)$ such that 
  $\pi$ is a composition of the codimension one projections 
$$
\pi_i: U_i:=V_i\times \ldots\times V_{k}\times U\to U_{i+1}:=V_{i+1}\times \ldots\times V_{k}\times U.
$$
  Set inductively $\overline{x_i}:=(x_i,\ldots,x_k)$ and
  $\cF_0:=\cF_{U_1}$, $\cF_{i+1}=\pi_{i|U_i}(\cF_i)$.
   By the above we can assume that:
   \begin{enumerate}
   	\item For any projection $\pi_i:U_i=V_i\times U_{i+1}\to U_{i+1}$
   	there is a $q_i$-regular function $\Delta_i\in \cO(U\times W)$ at $z_i:=(\overline{x}_i,y)$ which  annihilates  $\cF_i$.
   	\item 
%
$\cO(U_i)/\Delta_i$ is a finitely generated $\cO(U_{i+1})$-module.
   	\item The restriction of the projection $\pi_i$ to $V(\Delta_i)\subset U_i$ defines a finite map $V(\Delta_i)\to U_{i+1}$, 
likewise its restriction $V(\Ann(\cF_i))\to V(\Ann(\cF_{i+1}))$ to $V(\Ann(\cF_i))$ is a finite map .
   	\item  $\cF_i$ is  of finite type.
   	\item  
%
$(\cF_i)_{z_i}\simeq (\cF_{i+1})_{z_{i+1}}$ 
is an isomorphism of $\cO_{z_{i+1},U_{i+1}}$-modules.
   \end{enumerate}
\end{proof}

The following result  is very similar to Theorem \ref{GR}. 

\begin{corollary}\label{GR2}
 Let $f:X\to Y$ be a differentiable  map of 
%
differentiable  spaces, and consider points $x\in U$ and $y=f(x) \in V$. Let $\cF$ be a sheaf of  $\cO_X$-modules which is  of finite type, with a stalk $\cF_x$. 
 Suppose $\cF_x/(m_y\cdot \cF_x)$ is a finite-dimensional vector space over $\RR$ with a basis defined by $c_1,\ldots, c_k\in \cF_x$.
Then there exist  neighborhoods $V'\subset V$ of $y$  and $U'$ of $x$ such that:
\begin{enumerate}

\item  The sheaf $(f_{|U'})_*(\cF)$ is an  $\cO_{V'}$-module  of finite type.
\item  The   restriction of $f$  to $V(\Ann(\cF))\cap U'\to V'$ is a finite map.
\item $\cF_x\simeq ((f_{U'})_*(\cF))_y$.
\item The sheaf $(f_{|U'})_*(\cF)$ is generated over  $\cO_{V'}$ by $c_1,\ldots, c_k\in \cO(V')$
\item The module $\cF(U')$ is  generated over $\cO(V')$ by $c_1,\ldots, c_k\in \cO(V')$.
\end{enumerate}
\end{corollary}

\begin{proof} 
(1) and (2) The situation is local so we may assume that all subsets are closed 
subspaces of domains in $\RR^n$. 
Let us first consider the case when $Y$ is a domain in  $\RR^n$, and $X$ is a subspace of a  domain $B\subset \RR^k$.
The map $f:X\to Y\subset \RR^n$ can be written as the composition $f=\pi\circ\alpha$ of the closed embedding 
$$
\alpha:=(id, f):X\subset B\times V
$$ 
followed by the  projection 
$$
\pi: B\times V.
$$ 
Let $z=\alpha(x)$ with $\pi(z)=y$.
Note that $m_z\supset \pi^*(m_y)$ and $\alpha^*(m_{z})\supset f^*(m_y)$.
Then  
%
$\cF_x/(\alpha^*(m_{z}))$ is of finite dimension, and 
the sheaf $\cF':=\alpha_*(\cF)$  is of finite type, with  $\cF'_z=\cF_x$, and thus 
%
$\cF_x/f^*(m_y)=\cF'_z/\pi^*(m_{y})$ is finitely generated.
By the previous result we can find neighborhoods $V$ of $x$ and $U$ of $y$ such that for the restrictions $f_{|V}: V\to U$ and $\pi_{|V\times U}: V\times U\to U$, $f_{|V*}(\cF)=(\pi_{|V\times U})_*(\cF')$ is of finite  type. Moreover $\alpha(V(Ann((\cF))\subset(V(Ann(\alpha_*(\cF))\subset B\times V$. By Theorem \ref{Mal1}, the restriction of projection $\pi: B\times V\to V$ to $(V(Ann(\alpha_*(\cF))$ is finite. Since  $\alpha$ defines  an inclusion of $V(Ann((\cF)$ into $(V(Ann(\alpha_*(\cF))$ and thus is finite we get that the restriction of the composition $f=\pi\alpha$ of two finite maps to $V(Ann((\cF)$ is finite

 Now in the general case if $ Y\subset D$ is a complex subspace of the domain $D\subset \RR^n$, then we consider the induced map $\overline{f}:X\to D$. Then   $\overline{f}_{|U*}(\cF_{|U})$ is an $\cO_D$-module of finite type supported on $Y$, and annihilated by $\cI_Y$. It is the trivial extension of the sheaf ${f}_{|U*}(\cF_{|U})$. This implies that the sheaf ${f}_{|U*}(\cF_{|U})$ is an $\cO_{V\cap Y}$-module  of finite type. Again the annihilator of $\cF$ has the same vanishing locus as the annihilator of ${f}_{|U*}(\cF_{|U})$, and by the special case considered before the restriction of $f$ to the vanishing locus  $V(Ann(\cF)=V({f}_{|U*}(\cF_{|U})$ is finite.
 
(3) Since $\alpha$ is an embedding $\alpha_*(\cF_x)\simeq \cF_x$, and $\alpha_*(\cF)$ is of finite type by Lemma \ref{ft3}. Also, by Lemma \ref{projection} and the above: $$f_*(\cF_x)=\pi_*(\alpha_*(\cF_x))=\pi_*(\alpha_*(\cF))_x\simeq (\alpha_*(\cF))_x\simeq \cF_x$$
 (4) Since $\cF_x$  is finitely generated over $\cO_y$ and  $c_1,\ldots, c_k$ generate $\cF_x/(m_y\cF_x)$,  by the Nakayama lemma they generate it over $\cO_y$. It follows that they generate $(f_{|U'})_*(\cF)$ over $\cO_U$ in a certain neighborhood of $y$.

(5) The natural map $i: \cF(U')\to (f_{|U'})_*(\cF_{|U'})$ is injective since, by Lemma \ref{Rem5}, $((f_{|U'})_*(\cF_{|U'}))_y=\bigoplus \cF_{x_i}$. Thus if $i(s)_y=0$ for  $s\in \cF(U')$, and all  $y\in V'$ then $s_x=0$ for all $x\in U'$, and thus $s=0$. By shrinking $U'$ we can assume that $c_1,\ldots, c_k \in \cF(U')$ which is $\cO(V')$-submodule  of the module of the global sections $\Gamma((f_{|U'})_*(\cF_{|U'})$.  By (4) we see that they generate all the stalks $(f_{|U'})_*(\cF_{|U'})$. Lemma \ref{ft} implies that $c_1,\ldots, c_k$ generate
the $\cO(U)$-module  $\Gamma((f_{|U'})_*(\cF_{|U'})$ and thus its submodule $\cF(U')=\Gamma((f_{|U'})_*(\cF_{|U'})$.
\end{proof}

\subsection{Singular inverse function theorem for differentiable maps}\label{lag4}
If $f:X\to Y$ is a finite differentiable  map of 
differential spaces then we define the locus $Y_d\subset Y$ (resp. $Y_{\geq d}\subset Y$ of  points where $f$ has degree $d$, that is, 
$$
Y_d:=\Big\{y\in Y
\mid \sum_{x\in f^{-1}(y)}\dim(\cO_x/m_y)=d\Big\}.
$$ 
The set $Y_{\geq d}$ is closed, as we will see below. For any subset $Z\subset X$ denote by $m_Z^\infty$ the ideal of all the functions flat along $Z$ such that  ${\partial^k}f_{|Z}=0$.

\begin{corollary}[Singular ``inverse function'' theorem (differential version)] \label{I2}
Let $f:X\to Y$ be a differentiable  map of 
%
differentiable spaces, and let $x\in X$ and $y=f(x) \in Y$. Assume  $\cO_{X,x}/f^*(m_y)$ is of finite dimension $d$  over $\RR=\cO_y/m_y$. Then there exist neighborhoods   $Y'\subset  Y$ of $y$  and $X'\subset X$ of $x$ such that the induced finite thus proper and closed morphism $f': X'\to Y'$ is of degree $d$. 
If $\cO_{X,x}/f^*(m_y)$ is generated by $c_1,\ldots,c_d$ over $\RR$ then the sections generate $f'_*(\cO_{X'})$ in the neighborhood of $X'$ over $\cO_{Y'}$, defining an
epimorphism of sheaves of $\cO_{Y'}$-modules
$$
\phi: \cO_{Y'}^d\to f'_*(\cO_{X'}), \quad \phi(a_1,\ldots,a_d)=a_1c_1+\ldots+a_dc_d,
$$
and the corresponding epimorphism of  $\cO(Y')$-modules 
$$
\phi_{Y'}: \cO({Y'})^d\to \cO({X'}).
$$
Moreover, if $X$ is Cohen-Macaulay and  $Y$ is a manifold  of the same dimension then the kernels of $\phi$ and $\phi_{Y'}$ are contained in $m_{Y'_d}^\infty\cdot \cO_{Y'}^d$ (respectively in $m_{Y'_d}^\infty\cdot \cO({Y'})^d$). Moreover  the point 
$y=f'(x)\in Y'$ is in the ramified locus of  maximal index  $d$, and $y\in Y'_d$.
\end{corollary}

\begin{proof} The proof is similar to one before. We apply Corollary \ref{GR2}  to $\cF=\cO_X$. This shows surjectivity of $\phi$ in a neighborhood of $x$.

For the ``moreover'' part, for any $y'\in Y'_d$  consider the fiber $f^{-1}(y')=\{x_1,\ldots,x_k\}$ with $k\leq d$, $d_i=\dim(\cO_{x_i}/m_y)$, and $\sum d_i=d$.
As before we  represent $X'\to Y'$ as the composition of the closed immersion $i: X'\to X'\times Y'$ followed by the projection $\pi:X'\times Y'\to Y'$.
 The sheaf
$$
f'_*(\cO_{X'})=\pi_*(i_*(\cO_{X'}))=\pi_*(\cO_{X'\times Y'}/\cI_{i(X')})
$$ 
can be written by Theorem \ref{Rem2}, as the image of the free $\cO_{Y'}$-module $\cO_{Y'}^d$ in a neighborhood of $y$. By shrinking $Y'$ (and $X'$) one can assume that  $c_1,\ldots,c_d$ generate the $\cO_{Y'}$-module $f'_*(\cO_{X'})$. 
 
 Since $X'\to Y'$ is finite, by Lemma \ref{Rem5}, there is an isomorphism of  $\cO_{Y,y'}$-modules 
 $$
f_*(\cO_{X'})_{y'}\simeq \pi_*(\cO_{X'\times Y'}/\cI_{i(X')})_{y'}\simeq \bigoplus (\cO_{X'\times Y'}/\cI_{i(X')})_{x_i}\simeq \bigoplus \cO_{X,x_i}.
$$ 
Applying 
%
Theorem \ref{Rem4} to each point $x_i$ we obtain  surjections $\cO_{Y,y'}^{d_i}\to \cO_{X,x_i}$ with  kernel contained in $m_{y'}^\infty$. Thus there is  such a surjection $$\psi_{y'}: \cO_{Y,y'}^d=\bigoplus_i \cO_{Y,y'}^{d_i}\to f'_*(\cO_{X'})_{y'}=\bigoplus_i \cO_{X,x_i}$$ with 
%
the  kernel contained in $m_{y'}^\infty$. Consider the induced generators $e_1:=\psi_{y'}(1,0,\ldots,0),\ldots,e_d:=\psi_{y'}(0,\ldots,0,1)$ of $f'_*(\cO_{X'})_{y'}$. Since $c_1,\ldots,c_d$ is another set of generators we  can represent the generators $e_i$ as $e_i=\sum a_{ij}c_j$ for some  invertible matrix $[a_{ij}(y')]$. The matrix $[a_{ij}]$ defines an isomorphism
$\alpha$ of $\cO({Y'})^d$ in a neighborhood of $y'$. Locally we get the relation $\phi_{y'}=\alpha_{y'}\psi_{y'}: \cO_{Y',y'}^d\to f'_*(\cO_{X'})_{y'}$ with kernel contained in $m_{y'}^\infty$.

The homorphsim of $\cO({Y'})$-modules $$
\phi_{Y'}: \cO({Y'})^d\to \cO({X'}).
$$ is an epimorphism since the sections $c_1,\ldots, c_d\in \cO({X'})$ generate all the stalks 

%
%
\end{proof}

\begin{remark}
Observe that in a neighborhood of each  $y\in Y_d$, the number $\sum_{x\in f^{-1}(y)}\dim(\cO_x/m_y)$ is at most $d$, since there is an epimorphism
%
of vector spaces 
$$
(\cO_y/m_y)^d\to\bigoplus_{x\in f^{-1}(y)} \cO_{x_i}/m_y.
$$
In the holomorphic or algebraic setting, $Y'_d=Y'$ and $m_{Y'_d}^\infty=0$.
\end{remark}

\subsection{Smooth objects}\label{lag5}

Summarizing the results from the previous sections we  introduce the category of \emph{smooth objects} $\cR^n$ over a  field  $K$  modeled, in particular, on  the local rings  of smooth 
functions on $\RR^n$ over $\RR$. This approach allows us to treat algebraic, analytic, and smooth functions in the same way.

Each $\cR^n$ is a triple 
$$
\cR^n=(\cE_n,\{x_1,\ldots,x_n\},\cam_n),
$$ 
such that:
\begin{enumerate}
\item $\cE_0=K$, $m_0=(0)$. 
\item $\cE_n$ is a local ring with  maximal ideal ${m}_n$, and containing its residue field $K$. 
\item  $x_1,\ldots,x_n$ are elements of $\cam_n$ and their classes in  $\cam_n/\cam_n^2$ form a basis of a  free module over $K$.
\item There exists a homomorphism 
$$
T_n: \cE_n\to \widehat{\cE_n}=\lim_k {\cE_n}/m_n^k\simeq  K[[x_1,\ldots,x_n]]
$$
 of $K$-algebras, whose kernel is equal to $\cam_n^\infty$, and which transforms $x_i\in \cE_n$ to 
%
$x_i\in K[[x_1,\ldots,x_n]]$.

A map $f: \cR^n\to \cR^m$ is given by any sequence of functions $f_1,\ldots, f_m\in m_n$, which defines a  unique ring homomorphism
$$
f^*: \cE_m\to \cE_n,\quad f^*(x_i)=f_i,
$$
commuting with $T_i$.
We denote  the map defined by the sequence $(f_i)$ by 
%
$f=(f_1,\ldots,f_m)$. By abuse of notation we 
write 
$$
f^*(g)=g(f_1,\ldots,f_m)
$$
for any $g\in \cE_m$. We assume the following conditions hold:

\item
 The map $\sigma:=(x_1,\ldots,x_n): \cR^n\to \cR^n$ defines the identity map, that is,
the endomorphism $\sigma^*:\cE_n\to \cE_n$ is the identity.
\item
 The composition of maps $\phi_1: \cR^n\to \cR^m$ and $\phi_2: \cR^m\to \cR^k$ is a map 

$\phi_2\circ\phi_1: \cR^n\to  \cR^k$, given by the composition of the ring homomorphisms $
\phi_1^*\phi_2^*: \cE_k\to \cE_n$.
\item
  There exists a differentiation $\frac{\partial}{\partial{x_i}}$ of $\cE_n$ commuting with $T_n$ and defining the standard
derivation $\frac{\partial}{\partial{x_i}}$ on $K[[x_1,\ldots,x_n]]$.
\item
 For any $k\leq n$ there is  the natural projection $p_{n,k}:\cR^n\to \cR^k$, given by $(x_1,\ldots,x_k)$. It defines the
inclusions $\cE_k\subset  \cE_n$ and  $K[[x_1,\ldots x_k]]\subset K[[x_1,\ldots x_n]]$ commuting with the $T_n$.
Moreover 

$$
\cE_k=\bigg\{f\in \cE_n\mid \frac{\partial}{\partial{x_i}}(f)=0\ \text{for}\ i>k\bigg\}
$$
and for $i\leq n$, the restriction of the differentiation $\frac{\partial}{\partial{x_i}}$ on $\cE_n$ to 

$\cE_k$ coincides with that on $\cE_n$.
\item
 For any $k\leq n$  the map 

$(x_1,\ldots,x_k,0,\ldots,0)=i_{k,n}: \cR^k\to \cR^n$ defines a ring surjection $\cE_n\to \cE_k$ whose kernel is the ideal $(x_{k+1},\ldots, x_n)$.

\item
 (Inverse function theorem) For any functions $u_1,\ldots,u_n\in \cE_n$ for which ${\det[\frac{\partial{u_i}}{\partial{x_j}}](0)}\neq 0$ the
map $(u_1,\ldots,u_n): \cR^n\to \cR^n$ is  invertible.

\item
 (Malgrange-Mather special division) Let  $\cE_{n+k+1}$  be the ring of smooth objects with coordinates $(t,x,y):=(t,x_1,\ldots,x_n,y_1,\ldots,y_k)$.
For any $g(t,x_1,\ldots,x_n)\in \cE_{n+1}$ and the ``generic polynomial'' 
$$
P^d_:=t^d+y_1t^{d-1}+\ldots+ y_d \in \cE_{n+d+1}
$$
 there exists ``special Malgrange-Mather division'': 
$$
f(t,x)=h^d(t,x,y)\cdot P^d + r^d,
$$ 
where 

$$
r^d=\sum r^d_{d-1}(x,y)t^{d-1}
+\ldots+r^d_0(x,y),
$$
and 
%
$$
h^d(t,x,y)\in \cE_{n+d+1},\quad r^d_i=r^d_i(x,y)\in \cE_{n+d}.
$$
\item
 If $m_n^\infty =0$ then $(\cE_n)$ will be called \emph{reduced.} 
\end{enumerate}

\begin{remark}
In positive characteristic we  assume  the existence of Hasse derivatives $\frac{\partial}{\partial{x_i^{p^j}}}$ commuting with $T_n$ in condition (4).
Also   condition 
(8) is slightly modified: 
$$
\cE_k=\bigg\{f\in \cE_n\mid \frac{\partial}{\partial{x_i^{p^j}}}(f)=0\ \text{for}\ i>k, j>0\bigg\}.
$$
\end{remark}

\begin{example} 

Examples of categories of smooth objects include: 

\begin{enumerate}
\item  The germs $C_x^\infty(\RR^{n})$ of smooth functions on $\RR^{n}$ over $K=\RR$. 

\item The germs $\cO_x(K^{n})$ of analytic functions on $K^{n}$ over $K=\CC$ or $K=\RR$.

\item The germs $K\langle x_1,\ldots,x_n\rangle$ of algebraic functions on smooth algebraic varieties of dimension $n$ over a field $K$. (Here $K\langle x_1,\ldots,x_n\rangle $ denotes the Henselianization of the localization $K[x_1,\ldots,x_n]_{(x_1,\ldots,x_n)}$ of $K[x_1,\ldots,x_n]$ at the maximal ideal $(x_1,\ldots,x_n)$.)  In this case $K\langle x_1,\ldots,x_n\rangle$ 
is the direct image of \'etale extensions of $K[x_1,\ldots,x_n]$, preserving the residue field $K$, and it is a subring of the Henselian ring $K[[x_1,\ldots,x_n]]$.

\item The rings of formal analytic functions $K[[u_1,\ldots,u_k]]$ over a  field  $K$. (This follows in particular from more general ``Hironaka formal division'') (Theorem \ref{formal}).
\end{enumerate}
\end{example}

 
Malgrange's strategy presented in the previous sections utilizes  only the algebraic properties  defined for the category of smooth functions, and thus can be extended to the more general situation. In particular this implies the following:

 \begin{theorem}[Malgrange]\label{Mal2} 
Let  $\phi_k^* : \cE_k\to \cE_m$ be any homomorphism (in the smooth category), and let $M$ be 
 any  finitely generated $\cE_n$-module. Then  the following conditions are equivalent:
 \begin{enumerate}
\item $M$ is finitely generated over $\cE_k$.

\item  The dimension  of the vector space  
%
$M/(\phi_k^*(\cam_k)\cdot M)$ 
over ${{K}}=\cE_k/{m_k}$ is finite. 
\end{enumerate}
\end{theorem}

\begin{proof} 
The proof is the same as the proof of 
Theorem \ref{Mal1}. 
\end{proof}

We are going to use the following theorems due to Malgrange (with some modifications) (see also \cite{Diff}).

\begin{corollary}[Malgrange]\label{CMal2} 
Given a  finite $\cE_n$-module $M$ and  a homomorphism $f^*: \cE_k\to \cE_n$,
the set  $\{b_1, \ldots, b_r \}$  generates $M$ as an $\cE_k$-module if  it generates the $K$-vector space $M/(f^*({\cam}_k)·M)$. 
\end{corollary}
 
\begin{proof}($\Rightarrow$)
 By  Malgrange preparation, 
%
$M/(f^*(\cam_k)\cdot M) 
= \langle b_1, \ldots, b_r \rangle R$  is finite over $\cE_k$.
Hence  
$$
M = \langle b_1, \ldots, b_r \rangle \cE_k+{\cam}_k\cdot A.
$$ 
By the  Nakayama lemma we get  $M=\langle b_1, \ldots, b_r \rangle \cE_k$.
\end{proof}

For  a finite  $\cE_n$-module $M$ define its \emph{completion}  
$$
\widehat{M}=\lim_{\leftarrow} M/(m_n^i\cdot M),
$$ 
which is a module over the 
%
completion of the ring 

$\widehat{\cE_n}:=\lim_{\leftarrow} \cE_n/{m_n}^i={K}[[x_1,\ldots,x_n]]$.
These rings define the category of smooth objects over $K$.
There is a natural homomorphism $M\to \widehat{M}$ with kernel defined by $m^\infty\cdot M$.
Any ring homomorphism $\phi: \cE_k\to \cE_n$ induces a unique homomorphism $\widehat{\phi}: \widehat{\cE_k}\to \widehat{\cE_n} $.

\begin{corollary}[Preparation theorem (in  Malgrange form)]\label{Malg}
Given a finite  $\cE_n$-module $M$ and a map $f: \cR^k\to \cR^n$,
  the following statements are equivalent for $a_1,\ldots,a_r\in M$:
\begin{enumerate}
\item $a_1,\ldots,a_r\in M$ generate $M$  as an $f^*(\cE_k)$-module. 

\item $a_1,\ldots,a_r\in M$ span  $M/(f^*({m_k})\cdot M)$  as a ${K}=f^*(\cE_k/{m_k})$-vector space.
\item  $\widehat{a_1},\ldots,\widehat{a_r}\in \widehat{M}$ generate $\widehat{M}$ as  an $\widehat{f}^*(\widehat{\cE_k})$-module. 
\item  $\widehat{a_1},\ldots,\widehat{a_r}\in \widehat{M}$ 
generate $\widehat{M}/(\widehat{f}^*(\widehat{\cE_k}/\widehat{m_k})\cdot \widehat{M})$ as  a $K=\widehat{f}^*(\widehat{\cE_k}/\widehat{m_k}))$-vector space.
\end{enumerate}
\end{corollary}

\begin{proof} 
The rings 
%
$\cE_n$ and $\widehat{\cE_n}$ define categories of smooth objects.
Then by  Malgrange preparation and Nakayama's lemma we get the equivalences 
$
(1)\Leftrightarrow (2)$ and $(3)\Leftrightarrow (4).$

$(2)\Rightarrow (4)$. There exists a natural epimorphism 
$$ 
M/({m_k}\cdot M) \to \widehat{M}/({m_k}\cdot \widehat{M})=\lim_{\leftarrow} M/({m_k}\cdot M+m_n^d\cdot M).
$$

$(4)\Rightarrow (2)$. By the assumption, ${a_1},\ldots,{a_r}$ generate the finite-dimensional vector space 
$$
\widehat{M}/({m_k}\cdot \widehat{M})=\lim_{\leftarrow} M/({m_k}\cdot M+m_n^d\cdot M).
$$ 
Then there exists $d_0$ such that for $d\geq d_0$ 
we have
the natural  epimorphism 
$$
M/({m_k}\cdot M+m_n^d\cdot M)\to M/({m_k}\cdot M+m_n^{d_0}\cdot M),
$$ 
which
implies that 
$$
{m_k}\cdot M+m_n^{d_0}\cdot M={m_k}\cdot M+m_n^{d_0+1}\cdot M=\ldots.
$$
Let  
$$
A:={m_k}\cdot M+m_n^{d_0}\cdot M,\quad  B:={m_k}\cdot M.
$$
Then by the above
$$
A=B+m_n\cdot A,
$$
which by Nakayama's lemma yields $A=B$, that is, 
$$
{m_k}\cdot M={m_k}\cdot M+m_n^{d_0}\cdot M,
$$
and consequently 
%
$$
\widehat{M}/({m_k}\cdot \widehat{M})=\lim_{\leftarrow} M/({m_k}\cdot M+m_n^d\cdot M)= 
M/({m_k}\cdot M).
$$
\end{proof}

\section{Diagrams of initial exponents}

\subsection{Weierstrass-Hironaka division for formal analytic functions}
\label{one}

Consider the ring $$K[[u]]=K[[u_1,\ldots,u_n]]$$ of formal power series over any field  $K$. The monomials $u^\alpha$ can be naturally identified with the elements of $\NN^n$, where $\NN$ denotes the set of natural numbers and zero.
For any  nonzero function $f\in  K[[u]]$, $f=\sum c_\alpha u^{\alpha}$,
 define the   \emph{support} of $f$ to be 
$$
\supp(f):=\{\alpha\in \NN^n\mid c_\alpha\neq 0\}.
$$ 
 By the \emph{differential support} of $f$ we mean
 $$
\supd(f):=\{\alpha \mid D_{u^\alpha}(f)\neq 0\},
$$
 where $D_{u^\alpha}=\frac{1}{\alpha!}\frac{\partial}{\partial{u}^\alpha}$. Note that the latter makes sense in positive characteristic and is called the \emph{Hasse derivative.}

This notion of differential support can be extended to regular functions on a smooth variety. It also better reflects properties of the functions. It  is coherent and is not defined pointwise like  support, thus  allowing one to control  singularities in a neighborhood.

\begin{example} 
If $\operatorname{char}(K)=0$ and $f=u^d$ is a function on $\bA^1$ then $\supp(f)= \{d\}$ 
%
at $0$, 
while $\supd(f)=\{0,1,\ldots,d\}$. However, in the neighborhood of $0$, $f=(u+t)^d=u^d+tu^{d-1}+\ldots+t^d$ has the same  support and differential support. If $\operatorname{char}(K)=p$ and  $f=u^p$ then $\supp(f)= \{p\}$ at $0$, while $\supp(f)=\{0,p\}$ for $u\neq 0$, and   $\supd(f)=\{0,p\}$.
\end{example}

The following lemma is an  immediate consequence of the definition:

 \begin{lemma}\begin{enumerate} \label{support}
 \item $\supp(f)\subseteq {\supd}(f)$.
 \item If $\alpha\in {\supd}(f)$ then $\alpha+\beta\in {\supp}(f)$ for some $\beta\in \NN^n$. 
 
  \end{enumerate} 
 \end{lemma}

 The definition immediately yields:
 
\begin{lemma} 
Let $X$ be a smooth scheme over a field $K$ (respectively an analytic or differentiable manifold), with a coordinate system $u_1,\ldots,u_n$, and let $\Gamma\subset \NN^n$. Then the sheaf 
$$
\cO_X^\Gamma:=\{f\in\cO_X \mid \supd(f)\subset \Gamma\}=\{f\in\cO_X \mid D_{u^\alpha}(f)\equiv 0, \alpha\not\in \Gamma\}\subset \cO_X
$$ 
is a subsheaf of groups of $\cO_X$.	
\end{lemma}

 For any $n$-tuple $\alpha =(a_1,\ldots,a_n)$ of nonnegative integers set $|\alpha|:=a_1+\ldots+a_n$.
Then the \emph{multiplicity} of $f=\sum c_\alpha u^{\alpha}$ is defined as
$$
\ord(f)= \min \{|\alpha| \mid c_\alpha\neq 0\}.
$$
It follows immediately from the definition that
$$
\ord(f_1\cdot f_2)\geq \ord(f_1)+\ord(f_2), \quad \ord(u^\alpha\cdot f)=|\alpha|
%
+\ord(f).
$$
Any ordered set of  exponents $\alpha^1,\ldots,\alpha^k \in \NN^n$  defines a unique decomposition of 
$$ 
\NN^n= \Gamma\cup \Delta_1\cup\ldots\cup\Delta_k,
$$
where 
%
\begin{align*}
\Delta_1 &:=a_1+\NN^n,\ldots,  \Delta_j:=a_j+\NN^n\setminus \bigcup_{i=1}^{j-1} a_i+\NN^n=a_j+\NN^n\setminus \bigcup_{i=1}^{j-1} \Delta_i,
\\
\Gamma=\Gamma_0 &:=\NN^n\setminus \bigcup_{i=1}^k a_i+\NN^n=a_i+\NN^n\setminus \bigcup_{i=1}^k \Delta_i,\quad \Delta=\bigcup_{i=1}^k \Delta_i=\bigcup_{i=1}^k a_i+\NN^n
\end{align*}
For 
$i=1,\ldots, k$, 
$\Gamma_i:=
%
\Delta_1
-\alpha^i\subset \NN^n$ is defined to be the set satisfying  $$\Delta_i=a_i+\Gamma_i,$$ 
Note that $\Gamma\cup \bigcup_{i=1}^j \Delta_i=\Gamma + \Delta= \NN^n$ and  $\Delta+{\NN^n}\subseteq \Delta$.
\begin{definition} We call $\Delta=\bigcup_{i=1}^k a_i+\NN^n$ {\it the diagram}  defined for the set $\{a_1,\ldots,a_k\}$.
	
\end{definition}

\begin{lemma} 
If $\alpha \not\in \Gamma_i$ then $\alpha+\beta\not \in \Gamma_i$ for all $\beta\in \NN^n$. In particular the following conditions are equivalent:
\begin{enumerate}
\item $\supp(f)\subset \Gamma_i$.
\item $\supd(f)\subset \Gamma_i$.
\end{enumerate}
\end{lemma}

\begin{proof} 
The condition  $\alpha \not\in \Gamma_i$ is equivalent to  $\alpha+\alpha_i\not\in \Delta_i=(\alpha_i+\NN^n)\setminus \bigcup_{j<i} (\alpha_j+\NN^n)$. The latter can be stated as $\alpha+\alpha_i\in \bigcup_{j<i} (\alpha_j+\NN^n)$. Obviously if $\alpha+\alpha_i\in \bigcup_{j<i} (\alpha_j+\NN^n)$ then $\alpha+\beta +\alpha_i\in \bigcup_{j<i} (\alpha_j+\NN^n)$.

Now suppose $\supp(f)\subset \Gamma_i$ and  $\alpha \in \supd(f)$. This implies that   $\alpha +\beta \in \supp(f)\subset \Gamma_i$ for some $\beta\in \NN^n$ and  $\alpha \in \Gamma_i$. Consequently, 	$\supd(f)\subset \Gamma_i$. The other implication is obvious: $\supp(f)\subseteq  \supd(f)\subset \Gamma_i$.
 \end{proof}

\begin{definition}
We call a linear form   
$$
L=a_1x_1+\ldots +a_nx_n: \RR^n\to \RR
$$
 \emph{positive} (respectively \emph{nonnegative}) if $a_i>0$ (resp. $a_i\geq 0$) for all~$i$.
Any $k$-tuple   $\overline{T}=(T_1,\ldots,T_k)$ of  nonnegative linear forms is called \emph{positive} if for any $a\in \NN^n\subset \RR^n$ there exists at least one $i$ such that $T_i(a)>0$.

Any positive $k$-tuple $\overline{T}$ defines the  
\emph{monomial grading} $\overline{T}: \NN^n\to \RR^k$, of $\NN^n$ and thus a (partial) \emph{monomial order} on $\NN^n$ induced by the lexicographic order on $\RR^n$:
$$
\alpha\leq_{\overline{T}}\beta\quad \mbox{ if}\quad \overline{T}(\alpha)\leq_{\rm lex} \overline{T}(\beta).
$$
We shall call this grading \emph{total} if $\overline{T}:\NN^n\to \RR^k$ is injective. $\overline{T}$ will be called \emph{normalized} if  $T_1=x_1+\ldots+x_n$.
\end{definition}

 The definition  immediately yields
\begin{lemma}\label{infinite}
\begin{enumerate}
\item If $\overline{T}$ is positive then for 
 any increasing sequence $\overline{T}(\alpha_1)<\overline{T}(\alpha_2)<\ldots$, the sequence of $|\alpha_i|$, where $i=1,2,\ldots$, diverges to infinity, and consequenly $\{\overline{T}(\alpha_i)\}\subset \RR^n$ is not bounded.
 
And  vice versa, if  $\{\overline{T}(\alpha_i)\}$ is bounded then $(\alpha_i)$ is finite.

\item If $\overline{T}$ is total, the order defined by $\overline{T}$ on $\NN^k$ is total.
\end{enumerate}
\end{lemma}

\begin{proof} 
It follows that there exists $i$ such that the $i$-th component $T_i(\alpha_j)$ diverges to  infinity, as 
does    $\alpha_i$.
	\end{proof}

For any $f\in R[[u]]$, 
%
$f=\sum c_\alpha u^{\alpha}$, we call 
$$
\beta=\expe_{\overline{T}}(f)
$$ 
its \emph{initial exponent} if $\beta=\min_{\overline{T}}(\supp(f))$ is a unique minimal element with respect to the $\overline{T}$-order.
Also define the \emph{$\overline{T}$-multiplicity} of any $f=\sum c_\alpha u^{\alpha}$ to be  
$$
\ord_{\overline{T}}(f)=\min\{{\overline{T}}(\beta)\mid\beta \in \supp(f)\}.
$$
Observe that we have

\begin{lemma}
\begin{enumerate}
\item If   $\overline{T}$  is normalized then $\ord_{\overline{T}}(f)\leq \ord_{\overline{T}}(g)$ implies 
$\ord(f)\leq \ord(g)$. In particular if $\alpha=\expe_{\overline{T}}(f)$ then $\ord(f)=|\alpha|$.
\item If $T$ is total then $\expe_{\overline{T}}(f)$ exists for any $f$.
\end{enumerate}
\end{lemma}

The following theorem extends  Weierstrass-Hironaka formal division in  Grauert-Galligo form to any monomial order.

\begin{theorem}[Weierstrass-Hironaka formal division theorem \cite{AHV}, \cite{Galligo}, \cite{BM2}] \label{formal} 
Consider any monomial order defined by a positive $r$-tuple  $\overline{T}=(T_1,\ldots,T_r)$  on $\NN^n$.
Let

$f_1,\ldots, f_k\in  K[[u,v]]=K[[u_1,\ldots,u_n,v_1,\ldots,v_m]] $ 
be formal analytic functions.  
Assume there exist exponents  
$$
\alpha_1:=\expe_{\overline{T}}(f_1(u,0)),\ldots, \alpha_r:=\expe_{\overline{T}}({f_r(u,0)}) \in  \NN^n.
$$
%
Let $\Delta$ be the  diagram  defined for the exponents $\alpha_1,\ldots, \alpha_k \in  \NN^n$.
Then for every $g\in  K[[u,v]]$, there exist  unique $h_i\in  K[[u,v]] $ and $r(g)\in  K[[u,v]] $ such that  ${\supd}(h_i)\subset \Gamma_i$, 
${\supd}(r(g))\subset  \Gamma$, and
$$
g=\sum h_if_i+ r(g).
$$
Moreover, if $\ord(f_i)=|\alpha_i|$ for any $i$ then
 $$
\ord(r(g))\geq \ord(g),\quad \ord(h_i)\geq \ord(g)- |\alpha_i|.
$$ 
\end{theorem}

\begin{proof}  First note that the order $\overline{T}$ defined on $\NN^n$ extends to an order on $\NN^{n+k}$.
 
 We can assume that the coefficient 
$c_{i\alpha_i}$ of $f_i$ is  $1$, by replacing  $f_i$ with $c_{i\alpha_i}^{-1}f_i$ if necessary.

Observe that any function $g\in  K[[u]]$ can be uniquely written as $g=\sum h_iu^{\alpha_i}+r(g)$, where $\supd(h_i)\subset \Gamma_i $ and $\supd(r(g))\subset \Gamma_i$.  Define a $K$-linear transformation $\Phi:  K[[u]] \to  K[[u]]$ as 
$$
g=\sum h_iu^{\alpha_i}+r(g)\mapsto  \Phi(g)= \sum h_if_i+r(g).
$$

We show that $\Phi$ is invertible. We can write $\Phi=I+U$, where 
$$
U(g)= \sum h_i(f_i-u^\alpha_i)
$$
with
$$
\ord_{\overline{T}}(f_i-u^\alpha_i)>\ord_{\overline{T}}(u^{\alpha_i}),
\quad
\ord_{\overline{T}}(g)\leq \ord_{\overline{T}}(h_iu^{\alpha_i}),\quad  \ord_{\overline{T}}(g)\leq\ord_{\overline{T}}(r(f)).
$$
This implies that 
%
$\ord_{\overline{T}}(U(g))
>\ord_{\overline{T}}(g)$, and we get an increasing sequence $\ord_{\overline{T}}(U^i(g))$, which implies, by  Lemma \ref{infinite}, that $\ord(U^i(g))\to +\infty$, and  $\Phi^{-1}=I-U+U^2+ \ldots$ is well defined for any $g$. Also $\ord_{\overline{T}}(\Phi(g))=\ord_{\overline{T}}(g)$, and if 
 $g=\sum h_iu^{\alpha_i}+r(g)$ then 
$$
\ord_{\overline{T}}(g)= \min\{\ord_{\overline{T}}(h_i)+\ord_{\overline{T}}(\alpha_i)\}.
$$
Analogous considerations imply the ``moreover'' part of the theorem.
\end{proof}

\subsection{Diagrams of initial exponents}

Let $\cI$ be any ideal in $K[[x_1,\ldots,x_n]]$ (or  any homogeneous ideal in $K[x_1\ldots,x_n]$). Consider a total normalized  order $\overline{T}$ on $\NN^n$. 
Then 
 the corresponding 
%
diagram 
$$
\Delta=\Delta(\cI)=\expe_{\overline{T}}(\cI):=\{\expe_{\overline{T}}(f)\mid f\in\cI\}\subset \NN^n
$$  
will be called the \emph{diagram of initial exponents} of $\cI$ with respect to $\overline{T}$.
 Again we see that $\exp(\cI)+{\NN^n}\subseteq \exp(\cI)$ and $\exp(\cI)$ is finitely generated in the sense that $\exp(\cI)=\bigcup_{i=1}^n \{\alpha^i+\NN^n\}$ for a certain finite set of exponents $\alpha_1,\ldots, \alpha_k$ called \emph{vertices} of  $\exp(\cI)$, characterized by the property 
$$
\alpha_i\not \in \Delta\setminus ({\alpha_i}+\NN^n).
$$

\begin{lemma}[\cite{BM2}]
If  $f_i\in \cI$ for which $\exp(f_i)=\alpha_i$ are vertices of $\exp(\cI)$   then the elements  $f_i$ generate $\cI$.
Moreover, there is a uniquely determined set of generators $\overline{f_i}:=x^{\alpha_i}+r_i$,  called\emph{the Hironaka standard basis},  such that
$\supp(r_i)$ contained in $\Delta$ and $\mon(\overline{f_i})=x^{\alpha_i}$.
\end{lemma}

\begin{proof}  By  Hironaka formal division, for any  $g \in \cI$ we can write $g=\sum h_if_i+ r(g)$. Since $g, f_i \in \cI$, we get
 $r(g)\in\cI$ and consequently 
%
$\supp(r(g))\in \exp(\cI)= \Delta$. But again by the Hironaka division theorem, $\supp(r(g))\in \Gamma$. Both conditions imply that 
%
$\supp(r(g))\subset \Gamma\cap \Delta=\emptyset$. 
Thus $r(g)$ is constant and eventually $0$, and $g=\sum h_if_i$. 

For the second part apply the formal division algorithm directly to the set of functions $x^{\alpha_i}$ to get functions $r_i$ with $\supp(r_i)\subset \Gamma$.
Set $\overline{f_i}:=x^{\alpha_i}+r_i$.

\end{proof}

Denote by $H_\cI$ the Hilbert-Samuel function of $K[[x_1,\ldots,x_n]]/\cI$, defined as  
%
$$
H_\cI(s)=\dim_K(K[[x_1,\ldots,x_n]]/(\cI+m^{s+1})), \quad s\in \NN.
$$
Similarly for any diagram $\Delta$ we set 
$$
H(\Delta)(s):=\{\alpha\not\in \Delta\mid |\alpha|\leq s\}.
$$

\begin{corollary}  
There is a natural isomorphism of vector spaces over $K$   given by remainder  and preserving filtration by $(m^s)$, where $m=(x_1,\ldots,x_n)$,
%
$$
\overline{r}: K[[x_1,\ldots,x_n]]^{\Gamma}=\{f\in K[[x_1,\ldots,x_n]]\mid \supp(f)\subset \Gamma\}\to  K[[x_1,\ldots,x_n]]/\cI.
$$
In particular the following functions are equal:
$$
H_\cI=H(\Delta(\cI)).
$$
\end{corollary}

\subsection{Diagrams of finite type}

\begin{definition}[\cite{Galligo}, \cite{Hir2}, \cite{BM2}, \cite{BM22}]
 A  diagram  $\Delta$ of initial exponents is  \emph{monotone} if  for any $i<j$ and any element  $\alpha=(\alpha_1,\ldots,\alpha_i,\ldots,\alpha_j,\ldots,\alpha_n)\in \Delta$  we have that the element
$
R_{ij}(\alpha):=(\alpha_1,\ldots, \alpha_i+\alpha_j,\ldots, 
%
0_{j},
\ldots, \alpha_n)$ is in $\Delta.
$
\end{definition}

Monotone diagrams were introduced by Hironaka in his proof of the Henselian division theorem \cite{Hir2}. They also played an important role in the
 Bierstone-Milman proof of the Hironaka strong desingularization theorem \cite{BM2}. The monotonicity of  diagrams will be considered as an analog to the regularity condition in the Weierstrass preparation and division theorems.
As observed by Weierstrass, by the generic change of coordinates  one can make any analytic function into a $d$-regular one. A similar approach was used by Galligo and Grauert who  proved the analogous result for  diagrams in generic coordinates  \cite{Galligo}.
One should mention that the initial condition for the ``generic diagrams'' considered by Grauert and Galligo was stronger than the Hironaka condition above, but was valid only in characteristic zero. Still the Galligo argument works for monotone diagrams in any characteristic.
 
\begin{definition} 
We shall call a monomial order $\overline{T}$ \emph{monotone} if it is total,  normalized  and $\overline{T}(R_{ij}(\alpha))\leq \overline{T}(\alpha)$ for any $\alpha\in \NN^n$ and $i<j$.
\end{definition}

\begin{example} 
It follows from the definition that the monomial order
$$
\overline{T}=(x_1+\ldots +x_n,x_2+\ldots +x_n,x_3+\ldots +x_n,\ldots,x_n)
$$
is total, normalized and monotone.
\end{example}

Let $\cI$ be any ideal in $\cE_n$ (or  any homogeneous ideal in $K[x_1\ldots,x_n]$). For a coordinate system $u_1,\ldots, u_n\in \cE_n$ and a total monotone monomial order $\overline{T}$ consider
 the corresponding diagram 
$$
\Delta=\Delta(\cI)=\expe_{\overline{T}}(\cI)=\{\expe_{\overline{T}}(f)\mid f\in\cI\}.
$$ 
 Then $\Delta(\cI)$ clearly  
 depends upon the choice of the coordinate system.
One can represent the diagram as an infinite sequence $\alpha(\Delta)=(\alpha_1,\ldots, \alpha_k,\ldots)$ of all its elements $\alpha_i\in \Delta$ put in ascending order.
We can then order the set of all possible diagrams corresponding to the ideal $\cI$ by introducing the lexicographic order on the set of values $\alpha(\Delta)$.
Fix a total, normalized and monotone order $\overline{T}$.

\begin{theorem}[Galligo-Grauert](\cite{Galligo})
If $K$ is an infinite field then
the minimal diagram $\Delta(\cI)$ corresponding to $\cI$ with respect to a monotone order and a generic coordinate system is unique and  monotone.\qed
\end{theorem}
\begin{proof} The proof in \cite{Galligo} can be easily adapted to monotone diagrams.

\end{proof}

\begin{corollary}\label{etale} 
Let $K$ be any field. Then there exists a finite separable extension $K'$ of $K$ and a coordinate change defined  over $K'$ such that $\Delta(\cI)$ is monotone.
\end{corollary} 

\begin{proof} 
Let $K^s$ be a separable closure of $K$. Then $K^s$ is infinite and we can apply the previous theorem.
 \end{proof}


The definition of monotone diagram is closely related  to a generic linear transformation which takes an element $x^ay^b$ to the polynomial with initial exponent  $x^{a+b}$. It gives however unnecessary constraints on the diagrams. 

 \begin{definition} \label{S1}
A  diagram  $\Delta$ of initial exponents is  \emph{of finite type} if  for any $i<j$ and any element  $\alpha=(\alpha_1,\ldots,\alpha_i,\ldots,\alpha_j,\ldots,\alpha_n)\in \Delta$  there exists an element 
$$
S_\alpha(\alpha):=(\alpha_1,\ldots, \alpha_i+\alpha'_i,\ldots, 0_j,\ldots, \alpha_n)\in \Delta
$$ 
for some $\alpha'_i\in \NN$.
A diagram is \emph{finite} if  $\Gamma=\NN^n\setminus\Delta$ is finite.
\end{definition}

It follows from the definition that finite  and monotone diagrams are of   of finite type. However, not all finite diagrams are monotone.

Now, for any $i\leq n$ let us identify the set $\NN^i$ with the subset $\NN^i\times \{0\}$ of  $\NN^n$ of elements with the last $n-i$ components  zero. Similarly by $\NN^{*i,n}$ will mean the subset $\{0\} \times \NN^i$ of $\NN^n$ of elements with the first $n-i$ components zero.

 For any $i<j\leq n$ denote by $\pi_{ji}$ the natural projection $\pi_{ji}: \NN^j\to \NN^i$.
It follows from the definition that 
$$
\pi_{ji}\pi_{kj}=\pi_{ki}.
$$

\begin{lemma}
\begin{enumerate} \label{c0}
\item  If $\Delta\subset \NN^n$ is of finite type  then so is $\Delta\cap \NN^i\times\{0\}\subset \NN^i\times\{0\}$. 
\item  If $\Delta\subset \NN^n$ is of finite type  then so is  $\pi_i(\Delta)\subset \NN^i\times\{0\}$.
\end{enumerate}
\end{lemma}

\begin{proof} 
The  properties are simple consequences of the definition.
  \end{proof}

Let $\Gamma^0:=\Gamma=\NN^n\setminus \Delta$. Set $\overline{\Gamma}^i:=\NN^i\setminus \pi_{ni}(\Delta)$ and $
 {\Gamma}^i:=\overline{\Gamma}^i\times \NN^{n-i}\subset \Gamma $. 
This defines the natural filtration of  sets
$$
\Gamma^0:=\emptyset\subset \Gamma^1\subset\ldots\subset\Gamma^i\subset\ldots\subset  \Gamma^n=\Gamma.
$$
Then 
%
$$
\Gamma^{i}\setminus \Gamma^{i-1}=(\overline{\Gamma_i}\setminus 
(\overline{\Gamma_{i-1}}\times \NN))\times \NN^{n-i}=A_i\times \NN^{n-i},
$$
where the set 
%
$$
A_i:=\overline{\Gamma_i}\setminus (\overline{\Gamma_{i-1}}\times \NN)
=\pi_{i,i-1}^{-1}(\pi_{n,{i-1}}(\Delta) \setminus  \pi_{ni}(\Delta)
$$
is a subset of  $\NN^i\setminus \NN^{i-1}$.

Here is a slight enhancement of an important observation of Hironaka.
 
\begin{lemma}[\cite{Hir2}]\label{c1}  
For any diagram $\Delta$ there is a finite decomposition

$$
\Gamma=\bigcup_{i=0}^n A_i\times \NN^{n-i},
$$ 
where $A_i\subset \NN^i\setminus \NN^{i-1}$.  Moreover  $\Delta$ is of finite type if and only if   all $A_i$ are finite.
\end{lemma}

\begin{proof}
By the above, 
$$
\bigcup_{i=0}^n A_i\times \NN^{n-i}=\bigcup_{i=0}^n\Gamma^i\setminus \Gamma^{i-1}=\Gamma^n=\Gamma.
$$
Suppose now $\Delta$ is of finite type.
 It suffices to show the conclusion for $i=n-1$ and $\pi_i=\pi$. Then since $\pi_i(\Delta)\in \NN^i$ is of finite type, we reduce the case of $\pi_i$ to the situation of codimension one.  Denote $\overline{\Delta}_s:=\alpha^s+\NN^n\subset \Delta$. Then $\pi(\Delta)=\bigcup \pi(\overline{\Delta_s})$.
We show that 

$(\pi(\overline{\Delta_s})\times \NN)\setminus \Delta$ is finite. 
Then the set 

$$
A_n= \pi^{-1}\pi(\Delta) \setminus \Delta=\bigcup_s ((\pi(\overline{\Delta_s})\times \NN)\setminus \Delta) \subset \NN^n\setminus \Delta=\Gamma 
$$ 
is finite as well. To this end,  write  $\alpha^s=(\alpha_1,\ldots,\alpha_n)$ and $\pi(\alpha^s)=(\alpha_1,\ldots,\alpha_{n-1},0)$.
If $\Delta$ is of finite type then for any $i=1,\ldots,n-1$, 
$$
\pi(\alpha^s)+\alpha'_ie_i =(\alpha_1,\ldots,\alpha_i+\alpha'_i, \alpha_{n-1},0)
$$ 
are in $\Delta$ for suitable $\alpha'_i$. In other words, 

$(\pi(\overline{\Delta_s})\times \NN)\setminus \Delta$ contains only vectors of the form $\gamma=\pi(\alpha^s)+(\beta_1,\ldots,\beta_n)$, where $\beta_i<\alpha'_i$ for $i=1,\ldots,n-1$ and $\beta_n<\alpha_n$. This implies the finiteness of each set 

$(\pi(\overline{\Delta_s})\times \NN)\setminus \Delta$.

Conversely, assume all $A_i$ are finite. Let $\alpha=(\alpha_1,\ldots,\alpha_i,\ldots,\alpha_j,\ldots,\alpha_n)\in \Delta$, and suppose all $(\alpha_1,\ldots,\alpha_i+k,\ldots,0_j,\ldots,\alpha_n)$ are in $\Gamma=\bigcup_{i=0}^n A_i\times \NN^{n-i}$ for any natural $k$. Then it follows from finiteness of $A_i$ that for sufficiently large $k$ the elements $(\alpha_1,\ldots,\alpha_i+k,\ldots,0_j,\ldots,\alpha_n)$ are in $A_{i_0}\times \NN^{n-j}$ for $i_0<i$. But then $(\alpha_1,\ldots,\alpha_{i_0},0\ldots,0)\in A_j$  and $\alpha\in A_j\times \NN^{n-j}\subset \Gamma$,  which contradicts  the assumption.
\end{proof}

\subsection{Decomposition of  diagrams of finite type}\label{D}

Our next goal will be to find a similar finite decomposition for the set $\Delta$.
For that purpose consider the reverse lexicographic order $<_r$ on $\NN^n$, that is, the one corresponding to the linear map 
$$
\overline{T}_r=(x_n,x_{n-1},\ldots,x_1).
$$
Note that if $\alpha \in \NN^{i+1}\setminus \NN^i$ and $\beta \in \NN^{j+1}\setminus \NN^j$ with $i<j$ then $\alpha<_r\beta$.

\begin{remark}
We are going to use the reverse lexicographic order only to determine the desired subdivision of $\Delta$ and not to determine the initial exponents of functions. 
\end{remark}

\begin{lemma}
If $\Delta=\bigcup^k_{s=1} \Delta_s$ is of finite type then for any $ r\leq k$ the diagram $\Delta^s=\bigcup^s_{r=1} \Delta_r$ is also of finite type.
\end{lemma}

\begin{proof}  Note that if $\alpha\in \Delta_s=\alpha^s+\NN^n$ then $S_\alpha(\alpha)<_r\alpha$.
 On the other hand,  $S_\alpha(\alpha)\in \Delta_{s'}=\alpha_{s'}+\NN^n$, where $\alpha_s'\leq S_\alpha(\alpha_s)\leq_r \alpha_s$. Finally, $S_\alpha(\alpha)\in \Delta^{s'}\subset \Delta^s$. 
\end{proof}

Using the reverse lexicographic order, we rewrite the  sequence of vertices of  the  diagram $\Delta$ of finite type as 
$$
\alpha_1<\alpha_{2,1}<\ldots <\alpha_{2,k_2}<\alpha_{3,1}<\ldots< \alpha_{r,k_r}=\alpha_k,
$$
where $\alpha_{i,j}\in \NN^i\setminus \NN^{i-1}$.
Then as in Section  \ref{one}, we define a subdivision of $\NN^n$ into the disjoint union of the sets 
$$
\Delta_{i,j}=\alpha_{i,j}+\NN^n\setminus \bigcup_{\alpha_{i,j'}<_r\alpha_{i,j}}(\alpha_{i,j'}+\NN^n)=\alpha_{i,j}+\NN^n\setminus\bigcup_{\alpha_{i,j'}<_r\alpha_{i,j}}\Delta_{i,j'},
$$
where 
$$
\Delta_{i,j}=\alpha_{i,j}+\Gamma_{i,j}
,\quad
\Gamma=\NN^n\setminus \Delta.
$$


Consider the filtration of subsets of $\Delta$ analogous to the above:
 $$
\overline{\Delta}^0:=\{0\}\subset \overline{\Delta}^1:=(\Delta\cap \NN^1)\times \NN^{n-1}\subset \overline{\Delta}^2:=(\Delta\cap \NN^2)\times \NN^{n-2}\subset\ldots\subset \overline{\Delta}^n:=\Delta\cap \NN^{n}=\Delta,
$$
and similarly set
$$
\Delta^i:=\overline{\Delta}^i\times \NN^{n-i}
$$
to get the filtration
$$
\Delta^0=\{0\}\subset \Delta^1\subset \ldots\subset \Delta^n=\Delta.
 $$
Then 
 $$
\Delta^i\setminus \Delta^{i-1}=(\Delta\cap \NN^i)\setminus ((\Delta\cap \NN^{i-1})\times \NN)\times \NN^{n-i}=\overline{B}_i\times
\NN^{n-i},
$$ 
where
$$
{B}_i:=(\Delta\cap \NN^i)\setminus ((\Delta\cap \NN^{i-1})\times \NN)
$$ 
is a subset of $\NN^i\setminus \NN^{i-1}$.
So we have 
$$
\Delta=\bigcup B_i\times \NN^{n-i}.
$$
Note however that the sets $B_i$ are usually not finite and thus they will be subsequently decomposed.

\begin{lemma}\label{c3} 
$\pi_{n,i-1}(\Delta)=\pi_{i,i-1}(\Delta\cap \NN^i)$.
\end{lemma}

\begin{proof} Let $\overline{a}=\overline{a}^n:=(a_1,\ldots,a_n)\in \Delta$. By the finite type assumption
the elements 
\begin{align*}
\overline{a}^{n-1} &:=(a_1,\ldots,a_{n-1}+a_n',0)\in \NN^{n-1},
\\
\overline{a}^{n-2} &:=(a_1,\ldots,a_{n-2}+a'_{n-1}+a'_n,0,0)\in \NN^{n-2},\quad\ldots,
\\
\overline{a}^{i} &:=(a_1,\ldots,a'_i+\ldots+a'_n,0,\ldots,0)\in  \NN^{i}
\end{align*}
  are all in $\Delta$ for some $\alpha_i'\in \NN$.
But $\pi_{i,i-1}(\overline{a}^i)=\pi_{n,i-1}(\overline{a})$, which finishes the proof.
\end{proof}

Write 
$$
C_1:=\NN^1  \setminus  ( A_1 \cup B_1),\quad
C_1\times \NN:=\NN^2  \setminus (A_1\times N^1 \cup B_1\times N^1).
$$
By definition  $A_2, B_2\subset C_1\times \NN$. Set
$$
C_2:= (C_1\times N)  \setminus  (A_2\cup B_2).
 $$
Again $A_3, B_3\subset C_2\times \NN$, and we set
$$
C_3:= (C_2\times N)   \setminus  (A_3\cup B_3).
$$
We define  $C_k$ recursively by
$$
C_n:= (C_{n-1}\times N)  \setminus   (A_n \cup B_n).
 $$
This gives us the decomposition of $\NN^i$ into a union of disjoint subsets:
$$
\NN^i=\Big(\bigcup_{j=1}^i A_j\times \NN^{i-j}\Big) \cup\Big(\bigcup_{j=1}^i B_j\times \NN^{i-j}\Big) \cup  C_i,
$$
where  $A_j, C_j \subset \Gamma\cap (\NN^j\setminus \NN^{j-1})$,  $B_j\subset \Delta\cap (\NN^j\setminus \NN^{j-1})$ are finite, $\pi_{j,j-1}(A_j)\subset \Gamma$, $\pi_{j,j-1}(C_j)\subset \Delta$, and 
$$
\Big(\bigcup_{j=1}^i A_j\times \NN^{i-j}\Big)\cup C_i= \Gamma\cap \NN^i,\quad
\bigcup_{j=1}^i B_j\times \NN^{i-j}= \Delta \cap \NN^i.
$$
In particular
$$
\NN^n=\Big(\bigcup_{j=1}^n A_j\times \NN^{n-j}\Big)  \cup  \Big(\bigcup_{j=1}^n B_n\Big) \cup \, C_n=\Gamma  \cup  \Delta \cup C_n=\NN^n \cup C_n.
$$
 So $C_n=\emptyset$, and we get
 \begin{equation}\label{S2}
\NN^n=\Big(\bigcup_{j=1}^n A_j\times \NN^{n-j}\Big) \cup  \Big(\bigcup_{j=1}^n B_j\times \NN^{n-j}\Big).
\end{equation}

\begin{lemma}
\begin{enumerate}
\item $C_i$ is finite.
\item $\pi_{i+1,i}: B_{i+1}\to C_i$ is surjective.
\item $B_{i+1}=\bigcup_j \Delta_{i+1,j}\cap \NN^{i+1}$.
\item $B_{i+1}\times \NN^{n-i-1}= \bigcup_j \Delta_{i+1,j}$.
\item The sets $C_{ij}:=\pi_{i+1,i}(\Delta_{i+1,j})$ are disjoint and define a subdivision of $C_i=\bigcup C_{ij}$.
\item Write $C_{ij}=:\pi_{i+1,i}(\alpha_{i+1,j})+\overline{\Gamma}_{i+1,j}$. Then $\Gamma_{i+1,j}=\overline{\Gamma}_{i+1,j}\times \NN^{n-i}$.
\end{enumerate}
\end{lemma}

\begin{proof} (1) By Lemma \ref{c0}, the diagram $\Delta\cap \NN^i$ is monotone. Then, by Lemma \ref{c1},
 $$
\pi_{i,i-1}^{-1}\pi_{i,i-1}(\Delta\cap \NN^i)\setminus (\Delta\cap \NN^{i})
$$ 
is finite.
Intersecting this set with $\NN^{i-1}$ we see that
 $$
C_i:=\pi_{i,i-1}(\Delta\cap \NN^i)\setminus (\Delta\cap \NN^{i-1})
$$ 
is finite.

(2) We have $B_{i+1}=(\Delta \cap \NN^{i+1})\setminus \pi_{i+1,i}^{-1}
 (\Delta_{i+1}\cap \NN^{i+1})$ 
 and thus 
$$
\pi_{i+1,i}(B_{i+1})=\pi_{i+1,i}(\Delta\cap \NN^{i+1})
 \setminus (\Delta\cap \NN^i).
$$
By Lemma \ref{c3}, this gives $\pi_{i+1,i}(B_{i+1})=\pi_{n,i}(\Delta)\setminus B_i=\NN^i\setminus A_i \setminus B_i= C_i$.

(3) We have
$$
B_{i+1}=(\Delta \cap \NN^{i+1})\setminus \pi_{i+1,i}^{-1}
 (\Delta_{i+1}\cap \NN^{i+1})= \bigcup_{i'\leq i,j} \Delta_{i',j}\cap \NN^i\setminus \bigcup_{i'< i,j} \Delta_{i',j}\cap \NN^i=\bigcup_j \Delta_{i,j}\cap \NN^i.
$$

(4) Follows from (3).

(5) Suppose $\overline{\beta}=(\beta_1,\ldots,\beta_{i})\in \pi_{i+1,i}(\Delta_{i+1,j_1})\cap \pi_{i+1,i}(\Delta_{i+1,j_2})$, where $j_1 < j_2$.
Let $\beta=(\beta_1,\ldots,\beta_{i+1})$ be the smallest element in $\Delta_{i+1,j_1}\cap \NN^{i+1}$, with respect to the reverse lexicographic order, such that $\pi_{i+1,i}(\beta)=\overline{\beta}$. Then $\beta=\alpha_{i+1,j_1}+\beta'$, where $\beta'\in \NN^{i}\subset \NN^{i+1}$. Since $\alpha_{i+1,j_1}<_r\alpha_{i+1,j_2}$, and the last $n-i-1$ coordinates (in $\NN^n$) of both are zero the $i+1$-coordinate $(\alpha_{i+1,j_1})_{i+1}$ of $\alpha_{i+1,j_1}$ is not greater than the one of $\alpha_{i+1,j_2}$. 

Also by definition, $\beta_{i+1}=(\alpha_{i+1,j_1})_{i+1}+\beta'_{i+1}=(\alpha_{i+1,j_1})_{i+1}$, and all the elements $(\beta_1,\ldots,\beta_i,\beta'_{i+1})$ with $\beta'_{i+1}\geq \beta_{i+1}$ are in $\Delta_{i+1,j_1}\cap \NN^{i+1}$. (I)

 On the other hand, if $\beta''\in \Delta_{i+1,j_2}$, and
$\pi_{i+1,i}(\beta'')=\overline{\beta}$ then $\beta''$ has a form $\beta''=(\beta_1,\ldots,\beta_i,\beta_{i+1}'')$ with  $\beta''\geq  \alpha_{i+1,j_2}$. This implies
$$
\beta_{i+1}''\geq (\alpha_{i+1,j_2})_{i+1}\geq (\alpha_{i+1,j_1})_{i+1}\geq \beta_{i+1}.
$$
By (I) we conclude that  
$\beta''\in \Delta_{i+1,j_1}$ and thus is not an element of $\Delta_{i+1,j_2}$, which contradicts the definition of  $\beta''$.

(6) follows from (5).
\end{proof}

 Let 
$$
\overline{B}_{i,j}:=\{\alpha_{i,j}+\overline{\Gamma}_{i,j}\},\quad 
\overline{B}_{i}:=\bigcup_j \overline{B}_{i,j}.
$$ 
It follows from the above considerations that the following decompositions of  sets are finite:

\begin{corollary}
\label{12}
\begin{enumerate}
\item $\Delta_{i,j}\cap \NN^{i}=\bigcup_{\beta\in \overline{B}_{i,j}} (\beta+\NN^{*1,n-i})),$
\item $\Delta_{i,j}=\bigcup_{\beta\in \overline{B}_{i,j}} (\beta+\NN^{*n-i+1,n})),$ 

\item $B_{i}=\bigcup_{\beta\in \overline{B}_{i}} (\beta+\NN^{*1,n-i})),$
\item $B_{i}\times \NN^{n-i}=\bigcup_{\beta\in \overline{B}_{i}} (\beta+(\NN^{*n-i+1,n})).$
\end{enumerate}

\end{corollary}

Summarizing we get from formula \ref{S2} and the above: 
\begin{equation}\label{S3}
\NN^n=\Gamma+\Delta=\Big(\bigcup_{j=1}^n A_j\times \NN^{n-j}\Big) \cup \bigcup_{j=1}^n\bigcup_{\beta\in \overline{B}_{j}} (\beta+(\NN^{*n-j+1,n})).
\end{equation}
Consider the natural filtration of  rings 
$$
R_0=K\subset R_1=K[x_n]\subset\ldots\subset R_n=R=K[x_1,\ldots,x_n].
$$

We rewrite some of 
the relations between the sets $A_i$, $B_j$, and $C_s$ in the language of modules. For $i=1,\ldots,n,$ set $M^A_0=0$, $M^A_{n+1}=M^B_{n+1}=0$ and 
$$ 
M^A_i:=\bigoplus_{\alpha\in A_i}  R_{n-i}\cdot x^\alpha,\quad 
M^C_i:=\bigoplus_{\alpha\in C_i}  R_{n-i}\cdot x^\alpha ,  \quad  M^B_i:= \sum_{\beta\in \overline{B}_{i}} R_{n-i+1}\cdot x^\beta.
$$

There exists a decomposition of $R_n$ as a group into modules
%
$$
R_n=(M^A_1\oplus M^C_1)\oplus M^B_1
$$
$$
(M^A_2\oplus M^C_2)\oplus M^B_2 = M^C_1
$$
$$
\ldots
$$
$$(M^A_n\oplus M^C_n)\oplus M^B_{n} = M^C_{n-1}
$$
or alternatively
$$
R_n=(M^A_1\oplus M^C_1)\oplus (M^A_0\oplus M^B_1)
$$
$$
(M^A_2\oplus M^C_2)\oplus (M^A_1\oplus M^B_2) = (M^A_1\oplus M^C_1)
$$
$$
\ldots
$$
$$
(M^A_n\oplus M^C_n)\oplus (M^A_{n-1}\oplus M^B_{n}) = (M^A_{n-1}\oplus M^C_{n-1})
$$
$$
\{0\}\oplus (M^A_{n}\oplus \{0\}) = (M^A_{n}\oplus 0),
$$
where  $M^C_{n}=M^B_{n}=0$.
The latter sequence yields the following 
\begin{corollary} There is a decomposition \label{imp33}
\begin{equation} \label{22}
R_n=\bigoplus_{i=1}^{n+1} (M^A_{i-1}\oplus M^B_i)=R_n^\Delta\oplus R_n^\Gamma
\end{equation}
into a direct sum of free $R_{n-i+1}$-modules $M^A_{i-1}\oplus M^B_i$, where 
$$
R_n^\Delta=\{f\in R_n\mid \supp(f)\subset \Delta\}=\bigoplus_{i=1}^{n+1}  M^B_i,
$$ 
$$
R_n^\Gamma=\{f\in R_n\mid \supp(f)\subset\Gamma\}=\bigoplus_{i=1}^{n+1} M^A_{i-1}.
$$
\end{corollary}

\section{Filtered Stanley decomposition}

\subsection{Properties of  filtered Stanley decomposition}

The classical Weierstrass theorem
 and its Malgrange-Mather extension allow us
to write 
the $\cE_n$-module $\cE_n$ as the image of a direct sum of groups (more specifically modules over different rings),namely
\begin{equation}
\cE_n\cdot f\oplus \bigoplus_{i=0}^{k-1} \cE_{n-1}\cdot x^i\to \cE_n
\end{equation}
for any $k$-regular function $f$. The above (group) homomorphism is an isomorphism in a reduced category.
 This decomposition can be viewed as a modification or perturbation of the decompostion
\begin{equation}\label{2}
\cE_n=\cE_n\cdot x^k\oplus \bigoplus_{i=0}^{k-1} \cE_{n-1}\cdot x^i.
\end{equation}
We are going to extend this approach to several functions and ideals.
Note that the generalization of decomposition (\ref{2}) for   graded rings  has been worked out in the previous section in  formula (\ref{22}).

\begin{definition}\label{free44}  
A homomorphism of finite $\cE_n$-modules 
$$
\Psi: M\to N
$$ 
will be called a \emph{quasi-isomorphism} if $\Psi$ is an epimorpism and its kernel  is contained in $ \cam_n^\infty \cdot M$. We will call 

a finite $\cE_n$-module $M$ \emph{quasi-free}  if there exists a quasi-isomorphism $\cE_n^k\to M$.
\end{definition}

\begin{definition} \label{Stanley}  Let $M$ be a finitely generated  $\cE_n$-module with filtration $(M_s)$ of $\cE_n$-modules satisfying 
\begin{equation} \label{filtration} 
m_n^r\cdot M_s\subset M_{r+s}.
\end{equation}
 We say that the $\cE_n$-module $M$ admits  a  \emph{filtered Stanley decomposition} over  $(\cE_i)^n_{i=0}$ if there exist   free finitely generated $\cE_i$-modules  $N_i$ for $i=0,\ldots,n$ and
a  homomorphism of $K$-spaces 
$$
\Psi:\bigoplus N_i\to M
$$ 
such that:
\begin{enumerate}
\item The restriction $\Psi_{|N_i}: N_i\to M$ of $\Psi$ to $N_i$ is a homomorphism of $\cE_i$-modules.
\item $\Psi$ is surjective.
\item The kernel of $\Psi$ is contained in 
%
$\bigoplus \cam_i^\infty \cdot N^i$.
\item  There exist  bases $\{e_{i1},\ldots,e_{ik_i}\}$  of $N_i$ for $i=0,\ldots,n$ such that 

$$
\Psi^{-1}(M_s)=\bigoplus m_i^{s-d_{ij}}e_{ij},
$$ 
where 
%
$d_{ij}=\ord(\Psi(e_{ij}))$.
\end{enumerate}

The  set  $\{b_{ij}\}:=\{\Psi(e_{ij})\}_{i,j}$ will be called a \emph{basis}  of $M$ over 
$(\cE_i)^n_{i=0}$. The homomorphism $\Psi$   will be 
called a \emph{quasi-isomorphism over} $(\cE_i)^n_{i=0}$.
\end{definition}

Any $\cE_n$-module $M$ admits a filtration $m_n^{s+1}\cdot M$, which we refer to as the \emph{standard filtration.} While this filtration is of  primary concern, it would be useful to consider in certain situations a  more general, positive linear grading $T:\NN^n\to \RR$ with real values, given by a single positive linear form $T$. Observe that in that case, by Lemma \ref{infinite},
 each subset of  the set of values $S_n:=T(\NN^{n})$   that   is bounded above  is finite. We consider the natural filtration on $\cE_n$ defined by $T$, indexed by $a\in S_n=T(\NN^n)$:
$$
m_{n,T,a}:=\{f\in \cE_n \mid \ord_T(f)\geq a\}.
$$ 
Then  
%
$\ord_T(f):=\max\{a \in  S_n \mid  f\in m_{n,T,a}\}$.
We shall call an $\cE_n$-module \emph{$T$-filtered} if it satisfies the condition
\begin{equation}\label{filtration2}
 m_{n,T,r}\cdot M_s\subset M_{r+s}.
 \end{equation}
 
\begin{definition} We say that any $T$-filtered $\cE_n$-module $M$ with  $T$-filtration $M_s$  admits a \emph{$T$-filtered Stanley decomposition} over 
$(\cE_i)^n_{i=0}$ if there is a homomorphism $\phi$ as in the definition above satisfying conditions (1)--(3) and  the condition 
\begin{enumerate}
\item[$(4')$]  There exist  bases $\{e_{i1},\ldots,e_{ik_i}\}$  of $N_i$ for $i=0,\ldots,n$ such that $$
\Psi^{-1}(M_s)=\bigoplus m_{i,T,s-d_{ij}}e_{ij},
$$ 
where 
%
$d_{ij}=\ord_T(\Psi(e_{ij}))$.
\end{enumerate}
\end{definition}

One can extend this immediately to graded modules over graded rings. Observe
that the graded ring of $\cE_n$ defined by the filtration $m_n^k\subset \cE_n$ is equal to 
$$
R:=\gr(\cE_n)=\gr(\widehat{\cE_n})=\gr(K[[x_1,\ldots,x_n]])=\bigoplus_i m^{i}/m^{i+1}=K[x_1,\ldots,x_n].
$$

For any ring $K$ consider the natural filtration of  rings 
$$
R_0=K\subset R_1=K[x_n]\subset R_2=K[x_n,x_{n-1}]\subset\ldots\subset R_n=R=K[x_1,\ldots,x_n].
$$ 

For any  module $\cE_n$-module  $M$ with filtration $(M_s)$ satisfying  condition (4) from the definition above, we consider the associated graded $R_n$-module $\gr(M):=\bigoplus M_s/M_{s+1}$.

Similarly for any positive linear grading $T:\NN^n\to \RR$ one can consider a ring $R_T=K[x_1,\ldots,x_n]_T$ with gradation defined by $T$, and the filtration $R_{iT}$.

\begin{definition} \label{Stanley2} 
Let $K$ be any (commutative) ring with $1$, and $M$ be a finitely generated (or simply finite) graded $K[x_1,\ldots,x_n]$-module. By a \emph{filtered Stanley decomposition}, or simply a \emph{Stanley decomposition}, of $M$  over $(R_i)_{i=0}^n$ we mean  a graded group decomposition 
$$
M=\bigoplus N_i,
$$ 
where all  $N_i$ are free finite $R_i$-modules. 


 We shall call a finite set
$\{g_{ij}\}_{i=1,\ldots,n,j\in S_j}$ of homogeneous elements in $M$ a \emph{generating system} of $M$   if $\{g_{ij}\}_{j\in S_j}\in N_i$ generate the $R_i$-module $N_i$ for any $i$ and  
$$
\sum N_i=M.
$$  
We shall call a set
$\{g_{ij}\}_{i=1,\ldots,n,j\in S_j}$ of homogeneous elements  in $M$
\emph{independent} over $(R_i)_{i=0}^n$ if 

$$
\sum c_{ij}g_{ij}=0,\quad c_{ij}\in R_i,
$$
implies that $c_{ij}=0$.
A set
$\{g_{ij}\}_{i=1,\ldots,n,j\in S_j}$ of homogeneous elements is a \emph{Stanley basis}, or simply a \emph{basis}, of $R$ if $M=\bigoplus N_i$ 
and each $N_i$ is a free $E_i$-module generated by a basis $\{g_{ij}\}_{j\in S_j}$.
\end{definition}

 Stanley decomposition provides an effective tool for computing Hilbert (or Hilbert-Samuel) functions for  finite $\cE_n$-modules.

Let us define a function 
%
$\phi(n,k):\ZZ\to \NN$ by
\[ \phi(n,k)= \left\{ \begin{array}{ll}
       \binom{n+k}{k}& \text{if $n \geq 0$};\\
        \quad 0 & \text{if $n < 0$}.\end{array} \right. \] 

Denote by $K[x_1,\ldots,x_n]_s$ the  $s$-gradation of the ring $K[x_1,\ldots,x_n]_s$. Then

$$
\rank_K (K[x_1,\ldots,x_n]_s)=\phi(n-1,s),\quad 
\rank_K \Big(\bigoplus_{i\leq s}
K[x_1,\ldots,x_n]_i\Big)=\phi(n,s).
$$

\begin{corollary} \label{grading22} 
Let $M$ be any finite graded $R$-module over $R$ with a (filtered) Stanley decomposition. 
Let $g_{ij}$ be its Stanley basis, and let $d_{ij}\in \NN$ denote  the degrees of the generators.
Then each module $M_s/M_{s+1}$  and

$M/M_{s+1}$ is free over $K$, and  
the Hilbert function
$$
H_M(s)=\rank_K(M_s/M_{s+1})
$$ is 
equal to 
$$
H_M(s)=\sum_{i,j} \phi (s-d_{ij},i).
$$
In particular $H_M(s)$ is 
a polynomial 
for 
$$
s\geq d(M)=\max \{d_{ij}\}.
$$
\end{corollary}

\begin{proof}
We have
$$
M_s/M_{s+1}\simeq \psi^{-1}(M_s)/\psi^{-1}(M_{s+1})\simeq \bigoplus m_i^{s-d_{ij}}/m_i^{s+1-d_{ij}}e_{ij}.
$$
\end{proof}

\begin{remark}
Existence of a Stanley decomposition was proven by Stanley and it is a tool of fundamental importance in homological algebra.  In this paper we show existence of a filtered Stanley decomposition of any graded ring over an infinite field, likewise over the smooth category of rings $\cE_n$ over $K$. The stronger conditions imposed on Stanley decompositions are critical for this paper, in view of, in particular, the stabilization theorem for graded rings
 (Theorem \ref{free2}).
\end{remark}

The relation between $(\cE_i)_{i=0}^n$-modules and their graded $(R_i)_{i=0}^n$-modules is useful in view of the following observation.

\begin{theorem} \label{imp3}  
Let $M$ be a finitely generated  $\cE_n$-module with filtration $M_i$ such that 
 $m_n^j\cdot M_i\subset M_{i+j}$, and with the associated finite graded $R_n$-module $\gr(M)=\bigoplus M_k/M_{k+1}$. (Respectively let $M$ be any $T$-filtered module with the associated finite $T$-graded $R_n$-module $\gr(M)$.) Denote by $\widehat{M}:=\lim_{\leftarrow}M/M_s$ the completion of $M$, which is an $\widehat{\cE_n}$-module. Let 
$b_{11},\ldots 
, b_{n,k_n}\in M$ be  a finite set of elements.
Then the following conditions are equivalent:
\begin{enumerate}
\item  $b_{11},\ldots ,
 b_{n,k_n}\in M$  is a basis of $M$ over $(\cE_i)_{i=0}^n$.

\item  $\widehat{b}_{11},\ldots ,
 \widehat{b}_{n,k_n}\in \widehat{M}$  is a basis of $\widehat{M}$ over $(\widehat{\cE_i})_{i=0}^n$.

\item  $\inn(b_{11}),\ldots ,
 \inn(b_{n,k_n})$  form a basis of  $\gr(M)$ over $(R_i)_{i=0}^n=(\gr(\cE_i))$.
\end{enumerate}
\end{theorem}

\begin{proof}
For simplicity of notation we consider only the case of the standard filtration. The case of the $T$-filtration is identical.

Write $\gr(M)=\bigoplus^k_{i=0} \gr(N_i)$, where $\gr(N_i)$ are free $R_i$-modules generated by $b_{ij}$, $j\in A_i$. 
We shall induct on  $n$. If $n=0$, $M=\widehat{M}=\gr(M)=\gr(N_0)$ is a $K$-vector space of finite dimension generated by $b_1,\ldots,b_k$. So the three conditions are equivalent. The inductive assumption will be only used for the implication $(2)\Rightarrow (1)$.


$(3)\Rightarrow (2)$.  
Suppose that 
  $\sum c_{ij}b_{ij}\in M_s$  defines an element of order $s$.  If 
%
$\min (\ord(c_{ij}b_{ij}))=s_0<s$ 
then 
$$
\ord\Big(\sum_{\ord(c_{ij}b_{ij})=s_0} c_{ij}b_{ij}\Big)>s_0 
$$ 
and %
$$
\sum_{\ord(c_{ij}b_{ij})=s_0} \inn(c_{ij}) \inn (b_{i,j})=0,
$$
which is impossible since $\inn(b_{i,j})$ is a basis of $\gr(N)$.
Thus $\ord(c_{ij})=s-d_{ij}$. In particular if $\sum c_{ij}b_{ij}\in M_\infty:=\bigcap^\infty_{s=1} M_s$ then $c_{ij}\in m_i^\infty$. Also, if $\sum c_{ij}b_{ij}=0$ then $c_{ij}\in m_i^\infty$.
 This  implies that there exists a gradation preserving epimorphism
\begin{equation} \label{completion4} 
\bigoplus m_i^{s-d_{ij}}\cdot b_{ij}\to M_s\cap \sum \cE_i\cdot b_{ij} 
\end{equation}
and its kernel is contained in $\bigoplus m_i^\infty\cdot b_{ij}$.

By the assumption for any $a\in M_s$ 
 one can find $c_{ij}\in m_i^{s-d_{ij}}$ such that    $a=\sum {c_{ij}}{b_{ij}}\  ( {\rm mod}\ M_s)$. Then consider $a_1:=a-\sum {c_{ij}}{b_{ij}}\in M_{s+1}$ and repeat the procedure.
 This allows  representing any element in $M_s/M_{s+d}$  as $\sum {c_{ij}}{b_{ij}}$, where the classes of the elements $c_{ij}\in m_i^{s-d_{ij}}/m_i^{s+d-d_{ij}}$ are uniquely determined. This implies that  
 \begin{equation}
 \label{completion}
M/M_s=\Big(\bigoplus {\cE_i}{b_{ij}}+M_s\Big)/M_s=\bigoplus ({\cE_i}/m_k^{s-d_{ij}})\cdot{b_{ij}}
\end{equation}
Consequently, 
\begin{equation} \label{completion1} 
\widehat{M}=\bigoplus_{i\leq n} \widehat{\cE_i}{b_{ij}},\quad \widehat{M_s}= \bigoplus_{i\leq n} \widehat{m}_k^{s-d_{ij}}\cdot b_{ij}.\end{equation} 

$(2)\Rightarrow (3)$. Obvious.

$(2) \Rightarrow (1)$.
Let $N:=N_n$ be the submodule generated over $\cE_n$ by $b_{n,j}$. Then $\overline{M}:= M/N$ is an $\cE_n$-module. The graded module $\gr(\overline{M})=\gr(M/N)
=\gr(M)/\gr(N)$ is generated by  a basis $\inn(b_{ij})$, $i\leq n-1$.

  The module  $M/N$ is a finitely generated $\cE_n$-module. By  the formula
(\ref{completion}),
\begin{equation}
 \label{completion2}
M/(M_s+N)=\Big(\bigoplus_{i<n} {\cE_i}{b_{ij}}+M_s+N\Big)/(M_s+N)=\bigoplus_{i<n} ({\cE_i}/m_k^{s-d_{ij}})\cdot{b_{ij}}
\end{equation}
 and 
 $$\widehat{M/N}=\bigoplus_{i< n} \widehat{\cE_i}{b_{ij}}=\widehat{M}/\widehat{N}$$ is finitely generated over $\widehat{\cE}_{n-1}$ by %
$(b_{ij})_{i<n}$. By  Malgrange division (Corollary \ref{Malg}), the module $M/N$ is a finitely generated $\cE_{n-1}$-module.
By the inductive assumption 
$\{b_{ij}\}_{i\leq n-1}$ 
is a Stanley's basis of $M/N$. In particular, $\bigoplus^{n-1}_{i=0} N_i\to M/N$ is surjective, as also is $\phi: \bigoplus^{n}_{i=0} N_i\to M$. Since $\widehat{\phi}:\bigoplus^{n}_{i=0} \widehat{N_i}\to \widehat{M}$ is an isomorphism, the kernel of $\phi$ is contained in $\bigoplus^{n}_{i=0} m_i^\infty \cdot N_i$. Moreover, since  $M=\sum_{i\/leq n} \cE_i\cdot b_{ij}$ by the formula \ref{completion4}, there exists
a surjection
\begin{equation} \label{completion3} 
\bigoplus m_i^{s-d_{ij}}\cdot b_{ij}\to M_s\cap \sum \cE_i\cdot b_{ij}=M_s 
\end{equation}
and its kernel is contained in $\bigoplus m_i^\infty\cdot b_{ij}$.

 $(1)\Rightarrow (3)$.  Let $b_{ij}$ be a basis of $M$. Then for any $a=\sum c_{ij}b_{ij}$ we get a unique decomposition $\inn(a)=\sum \inn(c_{ij})\inn(b_{ij})$, where $\inn(c_{ij})\in R_i$ (as before). That is, 
%
$(\inn(b_{ij})$ 
is a basis of 
%
$\gr(M)$ over $(R_i)$.
\end{proof}

\begin{lemma}\label{restriction2}   
Assume $M$ is a finitely generated graded $R_{n}=K[x]$-module.   Let 
$b_{11},\ldots ,
b_{n,k_n}\in M$  be homogeneous elements in $M$.
  Then  the following conditions are equivalent:
\begin{enumerate}
\item  $b_{11},\ldots ,
b_{n,k_n}\in M$  is a basis of $M$ over $(R_{i})_{i=0}^n$.

\item  $b_{11},\ldots ,
b_{n,k_n}\in M$  is a basis of $M\otimes_{R_n}R_{n+m} $ over $(R_{i+m})_{i=0}^n$ with respect to any grading $T$ on $M\otimes_{R_n}R_{n+m}$ extending the standard grading. 
 \end{enumerate}
Moreover   condition $(1)$ implies 
\begin{enumerate}  
\item[(3)] $b_{11},\ldots ,
b_{n,k_n}\in M$  is a basis of $M\otimes_{K}K' $ over $(R_{i}\otimes_{K}K')_{i=0}^n$, where $K'\supset K$ is a commutative ring with $1$.
\end{enumerate} 
\end{lemma}

\begin{proof}
This follows from the definition of basis. 

$(1)\Rightarrow (2)$. If $M=\bigoplus_{i,j} R_i\cdot b_{ij}$ then 
$M\otimes_{R_n} R_{n+m}=M\otimes_{K}R_{m}=\bigoplus_{i,j} R_{i+m}\cdot b_{ij}$. 

$(2)\Rightarrow (1)$. If $M\otimes_{R_n}R_{n+m}=M\otimes_{K}R_{m}=\bigoplus_{i,j} R_{i+m}\cdot  b_{ij}$ with $b_{ij}\in M$ then there is an epimorphism 
$$
\bigoplus_{i,j} R_i\cdot b_{ij}\to 
%
\bigoplus_{i,j} R_{i+m}/((x_{n+1},\ldots,x_{n+m})\cdot b_{ij}) \to 
(M\otimes_{K}R_{m})/(x_{n+1},\ldots,x_{n+m})\simeq M.
$$
On the other hand, $\bigoplus_{i,j} R_i\cdot b_{ij}\to M$ is also a monomorphism since it is a restriction of $\bigoplus_{i,j} R_{i+m}\cdot b_{ij}\to M\otimes_{R_n}R_{n+m}\supset M$.

$(1)\Rightarrow (3)$. Obvious. 
\end{proof}

\begin{lemma}\label{restriction3}   
Consider the graded $R_{n+m}$-module  $M=R_{n+m}=K[x,y]$ (with the standard gradation).   Let 
$b_{11},\ldots ,
b_{n,k_n}\in M$  
be homogeneous elements in $M$ with $\deg(b_{ij})=\deg(b_{ij})(x,0)$.
  If 
$b_{11}(x,0),\ldots ,
b_{n,k_n}(x,0)\in R_n\subset R_{n+m}$  is a basis of $M$ over $(R_{i+m})_{i=0}^n$ then 
  $b_{11},\ldots ,
b_{n,k_n}\in M$  
is a basis of  $M$ over $(R_{i+m})_{i=0}^n$.
  \end{lemma}
  
\begin{proof} 
We  apply  double induction on $m$ and the grading $s$. Suppose the theorem is valid for $m-1$ and any $n$. Let  ${b^1_{ij}}\in R^{n+m}$ denote the restrictions of the elements  $b_{ij}$ to $R_{n+1}=R_{n+m}/(y_2,\ldots,y_m)$.

We show, by induction on the degree $s$, that $({b^1_{ij}})$ generate $M$ and are independent. Note that ${b^1_{ij}}(x,0)=b_{ij}(x,0)$.
Consider any element $a(x,y)$ in the $s$-grading of $R_{n+1}$.
Then by the assumption we can write  $a(x,0)=\sum c_{ij}b_{ij}(x,0)$, where $c_{ij}\in R_{i}=K[x_{n-i+1},\ldots,x_n]$. Consequently, 
%
$a-\sum c_{ij}{b^1_{ij}}=y_1a'$. By  induction $a'=\sum c'_{ij}b^1_{ij}$, where $c'_{ij}\in R_{i}[y_1]$  and 
$$
a=\sum (c_{ij}+y_1c'_{ij}){b^1_{ij}}
$$ 
with  $y_1c'_{ij}\in R_i[y_1]$.

Now suppose $\sum c_{ij}{b^1_{ij}}=0$ with $c_{ij}\in R_{i}[y_1]$ is a relation of the smallest degree.  
Taking  it   modulo $y_1$ we get $\sum {c_{ij}}(x,0)b^1_{ij}(x,0)=0$. By the assumption $b^1_{ij}(x,0)=b_{ij}(x,0)$ is a basis of $R_n$.
This implies $c_{ij}(x,0)=0$ and $c_{ij}=y_1c'_{ij}$.
Consequently, $\sum c'_{ij}b_{ij}=0$, with the smallest degree of $c'_{ij}$. We have shown that the restrictions of ${b_{ij}}$ to $R_{n+1}$ are a basis of $R_{n+1}$. By the inductive assumption 
on $m-1$ (or by repeating the argument $m-1$ times) the elements ${b_{ij}}$ form a basis of  $R_{n+1+(m-1)}=R_{n+m}$. 
\end{proof}

\begin{lemma}\label{restriction} 
Consider the $\cE_{n+m}$-module  $M=\cE_{n+m}(x,y)$-module with standard filtration $M_i=m_n^j\cdot \cE_{n+m}$ over ${\cE_{m}}$ such that 
 $m_n^j\cdot M_s\subset M_{s+j}$, and  each $M_s/M_{s+1}$ is a free $\cE_m(y)$-module. 
 If 
$$
\inn({b}_{11})(x,0),\ldots ,
\inn({b}_{n,k_n})(x,0)\in \gr(M/m_m)=R_n 
$$  
is a basis of $\gr(M/m_m)$ over $({R_i})_{i=0}^n$ 
then   
$$
b_{11},\ldots , 
b_{n,k_n}\in M
$$  
is a basis of $M$ over $(\cE_{i+m})_{i=0}^n$ with respect to a certain filtration $M_T$ on $\cE_{n+m}(x,y)$ extending the standard filtration on $\cE_n$.

Moreover, if $\deg(\inn({b}_{ij})(x,0))=\ord(b_{ij})$ then 
$
b_{11},\ldots , 
b_{n,k_n}\in M
$  
is a basis of $M$ over $(\cE_{i+m})_{i=0}^n$ with respect to the standard filtration on $\cE_{m+n}$.
\end{lemma}

\begin{proof} 
There exists a monomial grading $T=T_M$, where 
$$
T_M(x_1,\ldots,x_n,y_1,\ldots,y_m)=x_1+\ldots+x_n+My_1+\ldots+My_m,
$$
for sufficiently large $M\gg 0$ such that $\inn_T(b_{ij})=\inn_T(b_{ij})(x,0)=\inn(b_{ij})(x,0)$. Then $\inn_T(b_{ij})=\inn(b_{ij})(x,0)$ is a basis of $R_n$ over $(R_i)^n_{i=0}$, and,  by the previous lemma, it is a basis 
of $R_{n+m}$  over $(R_{i+m})^n_{i=0}$. Consequently, 
%
$\inn_T(b_{ij})=\inn_T(b_{ij})(x,0)$ is a basis of $R_{n+m}$ over $(R_{i+m})^m_{i=0}$, and by Theorem \ref{imp3},  $b_{ij}$ is a basis of $\cE_{m+n}$ with respect to  the $T$-filtration. 

For the ``moreover'' part, observe that 
%
if $\inn(b_{ij})(x,0)$ is a basis of $R_{n+m}$ over $(R_{i+m})^m_{i=0}$ then $\inn(b_{ij})$ is a basis of $R_{n+m}$ over $(R_{i+m})^m_{i=0}$ by Lemma \ref{restriction3}.  
Thus,  by Theorem \ref{imp3}, $b_{ij}$ is a basis of $\cE_{m+n}$ with respect to  the standard filtration. 
\end{proof}

\begin{corollary}\label{imp} 
Let  $M$ be any filtered $\cE_n$-module,  $M_0\subset M$ an $\cE_n$-module with induced filtration, and $\gr(M)$  the induced graded module with graded submodule $\gr(M_0)$.
Assume there exists a basis $(\overline{b}_{ij})_{(i,j)\in S}$ of $\gr(M)$ and a decomposition $S=S_1\cup S_2$ such that $(\overline{b}_{ij})_{(i,j)\in S_1}$ is a basis of $\gr(M_0)$.
Then there exists a basis $({b}_{ij})_{(i,j)\in S}$ of $M$ 
%
such that: 
\begin{enumerate}
%
\item $\inn(b_{ij})
=\overline{b}_{ij}$.

\item $({b}_{ij})_{(i,j)\in S_1}$ is a basis of $M_0$. 

\item $({b}_{ij})_{(i,j)\in S_2}$ is a basis of $M/M_0$.
\end{enumerate}
\end{corollary}

\begin{proof} 
It suffices to find elements $({b}_{ij})_{(i,j)\in S}$  in   $M$ such that $\inn(b_{ij})=\overline{b}_{ij}$, and ${b}_{ij}\in M_0$ whenever $(i,j)\in S_1$, and apply Theorem \ref{imp3}.
\end{proof}

\begin{definition} 
Let $M$ be a finite graded module and let $a_{ij}$ be its basis over $(R_i)_{i=0}^n$. We say that $a_{ij}$ \emph{majorizes} (respectively \emph{majorizes up to degree $d$}) a set of homogeneous elements $b_{ij}\in M$  if there is a bijective correspondence $a_{ij}\leftrightarrow b_{ij}$, preserving degrees $\deg(a_{ij})=\deg(b_{ij})$ and  multiplicaton rings $R_i$, such that the elements $b_{ij}$ generate $M$ over $(R_i)_{i=0}^n$ (respectively generate all $M_s$, where $s\leq d$).
\end{definition}

The following result  generalizes the stabilization theorem
for monotone diagrams \cite{BM2}, which plays a critical role in the Bierstone-Milman approach to desingularization and their use of the Hilbert-Samuel function.

\begin{theorem}[Stabilization for modules]\label{Stab} 
Let $R$ be a ring of polynomials over a commutative ring $K$ with identity and let $M=\bigoplus_{s\in \NN} M_s$ be any finite  $R$-module. Assume $M$ is generated over $(R_i)_{i=0}^n$ by a finite set of homogeneous 
elements $g_{ij}$ of degrees  $d_{ij}\in \NN$. Let $d(M)=\max\{ d_{ij}\}$.
Let $s_{ij'}\in M$ be any set of homogeneous elements of   degree $d_{ij}$ which generates $M_s$, where $s\leq d(M) +1$,
over $(R_i)_{i=0}^n$.  Then the elements  $s_{ij}\in M$  generate $M$ over $(R_i)_{i=0}^n$.
 
  Moreover, if 
   $g_{ij}$ is a basis of $M$ over $(R_i)_{i=0}^n$, and it majorizes $s_{ij}$ up to degree $d(M)+1$, then $s_{ij}$ is a basis of $M$ over $(R_i)_{i=0}^n$.
\end{theorem}

\begin{proof} 
We can label the set of generators as  $\{g_{ij}\mid j\in J_i,\, i=0,\ldots,n\}$, where $\{g_{ij}\mid j\in J_i\}$ are multiplied by elements in the rings $R_i=K[x_n,\ldots,x_{n-i+1}]$ for $i\geq 1$, and  $R_0=K$.

We use double induction: on the number of variables $n$ and on the degree. If $n=0$ then $M$ has only finitely many grades as it is generated over $K$ by finitely many elements of degree $\leq d$. For  $n\geq 1$ the elements $g_{ij}$ define a generating set for the $K[x_1,\ldots,x_{n-1}]$-module $\overline{M}:=M/(x_n\cdot M)$ over the rings 
$$
\overline{R_{i-1}}=R_i/(x_n)=K[x_{n-1},\ldots,x_{n-i+1}]
$$ 
for $i\geq 2$ and $\overline{R_0}=K$. More precisely, we need to relabel the classes of elements $g_{ij}$ accordingly. 
We set $\overline{J}_0:=J_0\cup J_1$, and $\overline{g_{0j}}=g_{0j}$ for $j\in J_0$ and $\overline{g_{0j}}=g_{1j}$ for $j\in J_1$.  For  $i\geq 1$ we set $\overline{g_{ij}}=g_{i-1,j}$ for $j\in \overline{J}_i=J_{i+1}$. It follows that the elements $g_{ij}$ generate $\overline{M}$ over $\overline{R_i}$.
By the inductive assumption  on the number of variables we conclude that the elements $s_{ij}$ generate $\overline{M}$ over $\overline{R_i}$.
Then we show by induction on the degree $s\in \NN$ that the elements $s_{ij}\in M$ generate $M_s$. By the assumption it is true for $s\leq  d(M)+1$.
Suppose it is true for all $s'<s$, where  $s>d(M)+1$. We can write any element $a\in M_s$  as 
$$
a=\sum b_{ij}s_{ij}+x_na',
$$ where $b_{ij}\in (R_i)_{s-d'_{ij}}$, $i=1,\ldots,n$. 
But $a'\in M_{s-1}$, so by the inductive assumption we can write  $a'=\sum b'_{ij}s_{ij}$, where again $b'_{ij}\in (R_i)_{s-1-d'_{ij}}$, and  since ${s-1-d'_{ij}}\geq s-1-d(M)> 0$ the elements $b'_{ij}$ are indexed by subscripts with  $i>0$. (The ring $R_0=K$ has only gradation $0$.)
Consequently, $x_na'=\sum x_nb'_{ij}s_{ij}$, where  $x_nb'_{ij}\in (R_i)_{s-d'_{ij}}$ with $i>0$, which yields 
$$
a=\sum (b_{ij}-x_nb'_{ij}) s_{ij}.
$$

Now if $g_{ij}$ is a basis which majorizes  a generating set $s_{ij}$ then by Corollary \ref{grading22} each $M_s$ is a free $K$-module and
$$
H_M(s)=\sum_{i,j} \phi (s-d_{ij},i)=\sum_{i,j} \phi (s-d'_{ij},i).
$$
But $M$ is the image of the direct sum of free $R_i$-modules $\bigoplus N_i$ generated by elements of degree $d'_{ij}=d_{ij}$.

Since both graded modules are free over $K$ and have the same $K$-ranks in each gradation then the epimorphism defines an isomorphism. (The relevant square matrix is invertible over $K$.)

Consequently,  $s_{ij}$ is a free basis of $M$ over $(R_i)_{i=0}^n$.
\end{proof}

\begin{lemma}\label{T}
   Let $M$ be a finitely generated graded $R$-module (with the standard gradation)  and let $\overline{T}:\NN^n\to \RR^k$ be a normalized grading. If
 $A:=\{a_{11},\ldots,a_{n,j_n}\}$ is a basis of $M$ over $(R_i)_{i=0}^n$, and $B:=\{b_{11},\ldots,b_{n,j_n}\}$ be a set in  bijective correspondence with $A$, such that  
$$
\deg(a_{ij})=\deg(b_{ij}),\quad \deg_{\overline{T}}(a_{ij}-b_{ij})>\deg_{\overline{T}}(a_{ij})
$$ 
for all $i,j$ then  $A$ majorizes $B$ over   $(R_i)_{i=0}^n$.
\end{lemma}

\begin{proof} Consider the grade preserving $K$-algebra homomorphism
$$
\Psi:\sum c_{ij}a_{ij}\to \sum c_{ij}b_{ij},
$$ 
where $c_{ij}\in R_{i}$.
Then for any $x$ in the $s$-th gradation we have 
%
$\deg_{\overline{T}}(\Psi-I)(x)>\deg_{\overline{T}}(x)$, and $(\Psi-I)^n(x)=0$  for $n\geq \dim(R)_s$. Then $\Psi$ is a $K$-linear  isomorphism with  inverse $I-(\Psi-I)+(\Psi-I)^2+\ldots$ and defines the majorization.
\end{proof}

\begin{lemma}\label{open} Let $X$ be an affine variety over a field $K$ (repectively an analytic space or an open subset in $\RR^m$), and let $\cO(X)$ be the ring of regular functions on $X$. Consider the map evaluation at $x\in X$, 
$$
\pi_x:\cO_X[x_1,\ldots,x_k] \to K[x_1,\ldots,x_n].
$$
Let ${b_{ij}}\in \cO_X[x_1,\ldots,x_n]$ be  homogeneous elements whose evaluations at $x\in X$,  $\pi_x(b_{ij})=\overline{b_{ij}}$, form a basis of $K[x_1,\ldots,x_n]$ over $(K[x_i,\ldots,x_n])_{i=1}^n$. Then there is an open subset $U\subset X$ such that  
%
${b_{ij|U}}$ is 
a basis of  $\cO(U)[x_1,\ldots,x_n]$ 
over  $(\cO_U[x_i,\ldots,x_n])_{i=1}^n$.
\end{lemma}

\begin{proof}  By Lemma \ref{restriction2}, the set $(\overline{b_{ij}})$ can be thought of as a (Stanley's) basis of $\cO_X[x_1,\ldots,x_k]$. Then 
for any $\beta\in \NN^{i}$, and the function  

$x^\beta\cdot b_{ij}$, where $\deg(b_{ij})+|\beta|=s\leq d+1$, one can write 
$$
x^{\beta}\cdot b_{ij}=\sum_{\beta'\in \NN^{i'}, \deg(b'_{i'j'})+|\beta'|=s,} A_{\beta,ij:\beta'i'j'}\cdot x^{\beta'}\cdot \overline{b_{i'j'}}.
$$
Consider the matrix 
$$
A_s=\big[A_{\beta,ij:\beta'i'j'}\big]_{\deg(b_{ij})+|\beta|=s}.
$$
The evaluation  $A_s(x)$ at $x$ is the identity.
Thus  the subset  
$$
U:=\{y\in X\mid \det(A_s)(y)\neq 0\ \text{for}\  s\leq d+1\}
$$ 
is a  nonempty open subset such that $b_{ij}$ majorizes ${b_{ij|U}}$ up to degree $d+1$. Consequently,  ${b_{ij|U}}$ is a basis of 
$\cO(U)[x_1,\ldots,x_n]$.
\end{proof}
\subsection{Neighborhood versions of Stanley decomposition}

\begin{lemma}\label{T1} 
Let $\cF$ be a coherent  module of finite type  on a regular scheme $X$ over a field $K$ of dimension $n$, with a coordinate system $u_1,\ldots,u_n$ at a $K$-rational point $\overline{x}$. Let $\cO^h_{\overline{x}}\simeq K\langle u_1,\ldots,u_k\rangle$ denote the Henselianization of the local ring $\cO_{\overline{x}}$.
Suppose the induced $\cO^h_{\overline{x}}$-module 
$$
\cF_x^h:=\cF_x\otimes_{\cO_{\overline{x}}}\cO^h_{\overline{x}}\simeq (\cO^h_{\overline{x}})^d
$$ 
is free and finite. Then there exists  an \'etale neighborhood $X'\to X$ of $\overline{x}$ preserving the residue field $K$ such that the induced coherent sheaf $\cF$ on $X$ is free and finite over $\cO_{\overline{x},X'}$.
\end{lemma}

\begin{proof} The module $\cF_x^h$ is generated by finitely many sections. Passing to an \'etale neighborhood $X'\to X$ we can assume that the generators are in $\cF(X')$. Then there is a surjective morphism $\phi:\cO_{X'}^d\to \cF_x$ of sheaves and the induced epimorphism  $\phi_x: \cO_{x,X'}^d\to \cF_x$ of stalks. Since the homomorphism $\cO_{x,X'}^d\to (\cO^h_{x,X'})^d$ is injective and  $(\cO^h_{x,X'})^d\to \cF_x^h$ is an isomorphism it follows that their  composition $\cO_{x,X'}^d\to \cF_x^h$ is injective. Thus  $\phi_x$ is also injective and is an isomorphism of stalks. The latter implies that $\phi$ is an isomorphism of sheaves in a Zariski neighborhood of $\overline{x}$.
	
\end{proof}

\begin{theorem}\label{NS1} 
Let $\cF$ be a coherent module on a smooth scheme $X$ over a field $K$ of dimension $n$, with a coordinate system $u_1,\ldots,u_n$ at a $K$-rational point $\overline{x}$  defining subschemes $X_i:=V(u_1,\ldots,u_{n-i+1})\ni \overline{x}$, $X_0=\{\overline{x}\}\subset X_1\subset\ldots\subset X_n=X$. 
Denote by $\cO_{\overline{x},X}$ the local sheaf at  $\overline{x}$ and by ${\cO^h_{\overline{x},X}}= K\langle u_1,\ldots,u_n\rangle$ its Henselianization.
Suppose  there exists  a Stanley basis $b_{ij}\in\cF_x$	of
$\cF_x^h:=\cF_x\otimes_{\cO_{\overline{x},X}}\cO_{\overline{x},X}^h$ over  ${\cO^h_{\overline{x},X}}$.

Then  there exist an \'etale  neighborhood $X'\to X$ of $\overline{x}$ preserving the residue field $K$ with induced subschemes $X'_0\subset X'_1\subset \ldots\subset X'_n=X'$ with $\cO(X'_i)\subset K\langle u_{n-i+1},\ldots,u_n\rangle$ and (smooth) projections $\pi_i: X'\to X'_i$ such that
$$
\cF(X')=\bigoplus_{ij} \cO(X'_i)b_{ij}.
$$
\end{theorem}

\begin{proof}
The elements $b_{nj}$ are defined in a certain affine \'etale neighborhood $X'$ of $X$, and they generate a coherent free submodule $\cG$ for ``sufficiently small'' $X'$, that is $$\cG(X')=\bigoplus \cO(X')b_{ni}.$$
Then $\cF_{n-1}:=\cF/\cG$ is a coherent  $\cO_{X'}$-module.
By 
%
Theorem \ref{Mi2}, one can modify $X'$ so that there exists the natural projection $\pi_{n-1}: X'\to X'_{n-1}$ on the subscheme $X'_{n-1}$ (defined by $u_1=0$) induced by the inclusions $\cO(X'_{n-1})\subset K\langle u_2,\ldots,u_n\rangle\subset K\langle u_1,\ldots,u_n\rangle$.
Then 
%
\begin{align*}
(\cF_{n-1}^h)_{\overline{x}}/((u_2,\ldots,u_n)\cdot (\cF_{n-1})_{\overline{x}})^h)
&=(\cF_{n-1})_{\overline{x}}\otimes_{\cO_{\overline{x},X}}({\cO_{\overline{x},X}}^h/((u_2,\ldots,u_n)\cdot {\cO_{\overline{x},X}}^h)\\
&=
(\cF_{n-1})_{\overline{x}}\otimes_{\cO_{\overline{x},X}}({\cO_{\overline{x},X}}/((u_2,\ldots,u_n)\cdot {\cO_{\overline{x},X}})\\
&=(\cF_{n-1})_{\overline{x}}/((u_2,\ldots,u_n)\cdot (\cF_{n-1})_{\overline{x}})=\cF_{n-1}/\pi_{n-1}^*(m_{\overline{x},X'_{n-1}})
\end{align*}
is a  finitely generated  vector space over $K$.  Thus, by 
%
Theorem \ref{coherent2}, 
after  possibly passing to an \'etale neighborhood of $X'$, the sheaf $\overline{\cF}_{n-1}:=\pi_{1*}(\cF_{n-1})_{|X'_{n-1}}$  is coherent 
 on $X_{n-1}'$ with $$\overline{\cF}_{n-1}(X'_{n-1})=\cF_{n-1}(X')={\cF}(X')/G(X').$$ Then by the inductive assumption applied to $\overline{\cF}_{n-1}$, after passing to an affine \'etale neighborhood $X''_{n-1}\to X'_{n-1}$, and the induced \'etale neighborhood 
$X''\to X'$, 
we have the natural isomorphism of $K$-vector spaces 
$$
\phi':\bigoplus_{i=0}^{n-1} \bigoplus_j\cO(X''_i)b_{ij}\to \overline{\cF}_{n-1}(X''_{n-1})=\cF_{n-1}(X'')=\cF(X'')/\cG(X''), 
$$ 
	which implies the existence of a natural  isomorphism of $K$-vector spaces
	$$
\phi: \cG(X'')\oplus \bigoplus_{i=0}^{n-1}\bigoplus_j \cO(X'_i)b_{ij}\simeq \bigoplus_{i=0}^{n}\bigoplus_j \cO(X''_i)b_{ij}\to \cF(X'').
$$ 
	The fact that $\phi$ is surjective follows from  surjectivity of $\phi'$, while  injectivity follows from injectivity of  $\phi_{|\cG(X'')}:\cG(X'')\to \cF(X'')$.
\end{proof}

A similar result is valid in the holomorphic setting (with the same proof).

\begin{theorem}\label{NS2} 
Let $\cF$ be a coherent module on a domain $U=U_1\times\ldots \times U_n$ in $\CC^n$ containing $0$, where $U_i\subset \CC$. 
Consider the natural projection  $\overline{\pi}_i: U\to \overline{U}_i:=U_i\times\ldots\times U_n$. If the stalk $\cF_0$	 admits a Stanley basis $b_{ij}$ then there exist   neighborhoods $V_i\subset U_i$ containing $0$ such that  for $\overline{V}_i:=V_{i}\times\ldots\times V_n$, $i=1,\ldots,n$,  $\overline{V}_0=\{0\}$ there is  an isomorphism

$$
\cF(\overline{V}_n)=\bigoplus_{i=0}^{n}\bigoplus_{j} \cO(\overline{V}_i)b_{ij}.
$$
\end{theorem}

\begin{proof}
As before, we  find a neighborhood $V_1\times \overline{V}_{n-1}$ with the projection $\pi_{n-1}: V_1\times \overline{V}_{n-1}\to \overline{V}_{n-1}$  such that the submodule  generated by $b_{nj}$ is free, and 
$$
\cG(V_1\times \overline{V}_{n-1})=\bigoplus \cO(V_1\times \overline{V}_{n-1})b_{ni}
$$ 
and the induced sheaf $\cF_{n-1}=\cF/\cG$ is coherent, with the finite-dimensional vector space  $\cF_{n-1,\overline{x}}/((u_2,\ldots,u_n)\cdot \cF_{n-1,\overline{x}})$.  Moreover by using Theorem\ref{GR}, we can assume that  
the direct image 
$\overline{\cF}_{n-1}:=\pi_{n-1*}(\cF_{n-1})$ of $\cF_{n-1}$ is coherent on $\overline{V}_{n-1}$, and then use the inductive assumption. Note that by 
 the necessary modification of $\overline{V}_{n-1}$ does not affect the previously constructed  isomorphisms.	
\end{proof}

The following lemma is an immediate consequence of the definition.

\begin{lemma}\label{B1} Let $\cF$ be a  module of finite type on a domain $U$ containing $x$. If there exists a quasi-isomorphism $\cO^d_x\to \cF_x$ then there exists an open neighborhood $U'\subset U$ and a surjection of  $\cO^d(U)$-modules
	$\cO^d(U)\to \cF(U)$ with  kernel contained in $m_x^\infty\cdot \cO^d(U)$.
\end{lemma}

In the differential setting one can prove a weaker result.

\begin{theorem} \label{NS3} 
Let $\cF$ be a  module of finite type on a domain $U=U_1\times\ldots \times U_n$ in $\RR^n$ containing $0$, where $U_i\subset \RR$. 
Consider the natural projection  $\overline{\pi}_i: U\to \overline{U}_i:=U_i\times\ldots\times U_n$. If $\cF_x$	 admits a Stanley basis $b_{ij}$ then there exist neighborhoods $V_i\subset U_i$ containing $0$ such that  for $\overline{V}_i:=V_{i}\times\ldots\times V_n$, $i=1,\ldots,n$,  $\overline{V}_0=\{0\}$ there is  an epimorphism 
%
$$
\bigoplus_{i=0}^{n}\bigoplus_j \cO(\overline{V}_i)\cdot b_{ij}\to \cF(\overline{V}_n)
$$
with kernel contained  in $\bigoplus_{i=0}^{n} m_{0,V_i}^\infty\cdot \cO(\overline{V}_i)b_{ij}$.
\end{theorem}

\begin{proof}
The only difference here is that we consider the sheaf $\cF$  of finite type over the sheaf of differentiable functions. We use the previous lemma to get an epimorphism
 $$
\bigoplus \cO(V_1\times \overline{V}_{n-1})\cdot b_{ni}\to \cG(V_1\times \overline{V}_{n-1})
$$
with kernel in $\bigoplus m_{0,V}^\infty\cdot\cO(V_1\times \overline{V}_{n-1})\cdot b_{ni}$.
By Lemma \ref{ft2},  the sheaf $\cF_{n-1}=\cF/\cG$ is of finite type. Moreover  it has the finite-dimensional vector space  $\cF_{n-1,\overline{x}}/(u_2,\ldots,u_n)$.  Thus, by using Corollary \ref{GR2}, we can assume that 
the direct image 
$\overline{\cF}_{n-1}:=\pi_{n-1}^*(\cF_{n-1})$ is of finite type on $\overline{V}_{n-1}$, and then use the inductive assumption as in the algebraic or analytic cases. 
\end{proof}

\subsection{Diagrams associated with functions}
\label{beta}

Let $\overline{T}$ be  any normalized linear grading.
Let $f_{ij}(u,v)\in \cE_{n+m}$ be a finite set of functions  such that the initial exponents 
 $$
\alpha_1=\expe_{\overline{T}}(f_1(u,0))<_r \alpha_{2,1}=\expe(f_{2,1}(u,0)<\ldots
$$  
define a  diagram of initial exponents $\Delta\subset \NN^n$ of finite type and its decomposition into  disjoint subsets $\Delta_{ij}$. 
 For simplicity we assume that the coefficients $c_{\alpha_{ij}}(f_{ij})$ of $f_{ij}$ are all equal to one.

By the above  any $\beta\in \overline{B}_i=\bigcup \overline{B}_{i,j}$ can be written as $\beta=\alpha_{ij}+\gamma ,$
where  $\gamma\in \overline{\Gamma}_{i+1,j}$ so we set 
$$
f_\beta:=u^\gamma f_{i,j}, \quad F_\beta:=\inn(f_\beta).
$$

By Corollary \ref{imp33} the set 
%
$$
S_1:=\Big\{x^{\beta}\mid \beta \in \bigcup_{i=1}^{n+1}(A_{i}\cup \overline{B}_{i-1})\Big\}
$$ 
is a basis over $(R_i)_{i=0}^n$ of the graded $R_n$-module $R_n$. 

Since $\expe(F_\beta)=\beta$, by Lemma \ref{T}, the set $S_1$ majorizes  the set 

$$
S_2:=\Big\{x^{\beta}\mid \beta \in \bigcup_{i=1}^{n+1} A_{i}\Big\}\cup
\Big\{F_\beta\mid  
\beta\in\bigcup_{i=0}^{n}\overline{B}_i\Big\},
$$
 which is thus  a basis of $R_n$. Consequently,  by Theorem \ref{imp33}, 
$$
\Big\{x^{\beta}\mid \beta \in\bigcup A_{i}\Big\}\cup
\Big\{f_\beta\mid  
 \beta \in\bigcup_{i=1}^{n}\overline{B}_i\Big\}
$$
 is a (Stanley's) basis
of  $\cE_n$ 

over $(\cE_i)_{i=0}^n$.

This leads to the following 

\begin{theorem} \label{imp2}
There is a quasi-isomorphism  over $(\cE_n,\ldots,\cE_1)$ (preserving filtration by the powers $(m_n^k)$) 
$$ 
\phi: \bigoplus_{j=1}^n \bigoplus_{\beta\in \overline{B}_j} \cE_{n-j+1}\cdot f_\beta \oplus \bigoplus_{j=1}^n\bigoplus_{\alpha\in A_j} \cE_{n-j}\cdot x^\alpha\to  \cE_n.
$$
\end{theorem}

Set 
$$
\cE_n(\Gamma):=\{f\in \cE_n\mid \supd(f)\subset \Gamma\}.
$$
The following theorem  extends the original \emph{Hironaka Henselian division theorem} \cite{Hir2} for
algebraic functions to any smooth category.

\begin{theorem}[Hironaka Henselian division theorem] 
Let $\cI\subset \cE_n$ be an ideal generated by  the functions $f_{ij}$. There exists an   epimorphism  
$$
\cE_n(\Gamma)=\bigoplus_{j=1}^n\bigoplus_{\alpha\in A_j} \cE_{n-j}\cdot x^\alpha\to \cE_n/\cI.
$$
\end{theorem}

\begin{proof} 
By the above there are  epimorphisms 
$$
\cE_n(\Gamma)=\bigoplus_{j=1}^n\bigoplus_{\alpha\in A_j} \cE_{n-j}\cdot x^\alpha\to \cE_n/M(\cI)\to  \cE_n/\cI,
$$
where 
$M(\cI):=\bigoplus_{j=1}^n \bigoplus_{\beta\in \overline{B}_j} \cE_{n-j+1}\cdot f_\beta \subset \cI. $
\end{proof}

\begin{remark} 
The subspace $M(\cI)\subsetneq \cI$ is usually much smaller than $\cI$.	
\end{remark}

\begin{lemma} \label{66}
Let $I\subset K[x_1,\ldots,x_n]$ be a homogeneous ideal, where $K$ is a field, and such that the diagram $\Delta:=\expe_{\overline{T}}(\cI)$ is of finite type.
Consider any set of homogeneous elements $F_i$ in $I$ such that
$\expe(F_i)=\alpha_i$ are vertices of $\Delta$. Then there is a decomposition
$$
\bigoplus_{j=1}^n \bigoplus_{\beta\in \overline{B}_j} R_{n-j+1}\cdot \overline{F}_{n,\beta} \oplus \bigoplus_{j=1}^n\bigoplus_{\alpha\in A_j} R_{n-j}\cdot x^\alpha=  R_n,
$$
where 
$$
I=\bigoplus_{j=1}^n \bigoplus_{\beta\in \overline{B}_j} R_{n-j+1}\cdot \overline{F}_{n,\beta}.
$$
\end{lemma}

\begin{theorem}[Existence of a filtered  Stanley decomposition] \label{free2} 
If $K$ is an infinite field that any finite graded $K[x_1,\ldots,x_n]$-module $M$  has a filtered Stanley decomposition  over $(K[x_i,\ldots,x_n])_{i=0}^{n-1}$  (possibly after a generic linear change of coordinates). That is, $M$ can be written as (a vector space) 
$$
M=\bigoplus_{i=0}^{n} N_i,
$$ 
where $N_i\subset M$ are free finite graded $R_i$-submodules. If $K$ is a finite field, such a decomposition exists over a certain finite extension  of $K$.
\end{theorem}

\begin{proof} 
Write $M$ as the quotient $M=(R_n)^k/ N$, where $N=N_k\subset (R_n)^k$ is a submodule. We prove the theorem by induction on $k$. Consider the projection $\pi_k:(R_n)^k\to R_n$ to the last coordinate,  let $\pi_k(N_k)=\cI_k\subset R_n$, and let $\Delta_k$ be a monotone diagram defined for $\cI_k$ and generic coordinates. Then there exists a standard basis  $\pi_k(F_{k,\beta})=\overline{F}_{k,\beta}\in \cI_k$ of $\cI_k$, as in 
Lemma \ref{66}, such that
 $$
\phi_k: \bigoplus_{j=1}^n \bigoplus_{\beta\in \overline{B}_j} R_{n-j+1}\cdot \overline{F}_{k,\beta} \oplus \bigoplus_{j=1}^n\bigoplus_{\alpha\in A_j} R_{n-j+1}\cdot x^\alpha\to  R_n
$$ 
is an isomorphism.

 Set  $\overline{B}_{k,j}:=\overline{B}_{j}$, and define 
 inductively an isomorphism 
$$ 
\id_{k-1}\oplus\phi_k: R_n^{k-1} \oplus \bigoplus_{j=1}^n \bigoplus_{\beta\in \overline{B}_{k,j}} R_{k-j+1}\cdot {F}_{k,\beta} \oplus \bigoplus_{j=1}^n\bigoplus_{\alpha\in A_{k,j}} R_{k-j}\cdot (0,\ldots,0,x^\alpha)\to  R_n^k.
$$

Now consider the module 

$N_{k-1}=(R_n)^{k-1}\cap N_k$ and repeat the procedure by induction.
We eventually get
an isomorphism
$$ 
\phi:=\phi_1\oplus\ldots\oplus\phi_n: \bigoplus_{s=1}^k\bigoplus_{j=1}^n \bigoplus_{\beta\in \overline{B}_{s,j}} R_{n-j+1}\cdot {F}_{s,\beta} \oplus \bigoplus_{j=1}^n\bigoplus_{\alpha\in A_{s,j}} R_{n-j}\cdot (0,\ldots,x^\alpha(s),0,\ldots, 0)\to  R_n^k
$$
for the relevant $\overline{B}_{s,j}$ occurring in the process.
The latter induces an isomorphism 
$$ 
\overline{\phi}: \bigoplus_{s=1}^n\bigoplus_{j=1}^n\bigoplus_{\alpha\in A_{s,j}} R_{n-j}\cdot (0,\ldots,x^\alpha(s),0,\ldots, 0)\to  R_n^k/\phi\Big(\bigoplus_{s=1}^n\bigoplus_{j}^n \bigoplus_{\beta\in \overline{B}_{s,j}} R_{n-j}\cdot {F}_{s,\beta}\Big)\to R_n^k/N=M. 
$$
Setting $N^j:=\bigoplus_{s=1}^n\bigoplus_{\alpha\in A_{s,j}} R_{n-j}\cdot (0,\ldots,x^\alpha(s),0,\ldots, 0)$ we 
get an isomorphism $\overline{\phi}:\bigoplus N^j\to M$.
\end{proof}

\begin{theorem}[Existence of a filtered  Stanley decomposition 2]  \label{free1}
Let $\cE_n$ be a smooth category over an infinite field $K$.
For any finite filtered $\cE_n$-module $M$, with filtration $(M_i)$ satisfying $m_n^i\cdot M_j\subset M_{i+j}$, there exists  (after a generic linear change of coordinates in $\cE_n$) a filtered Stanley decomposition, that is, 
there exist free finite $\cE_i$-modules $N^i=\cE_i^{k_i}$ for $0\leq i \leq n$, and  a quasi-isomorphism over $(\cE_0,\cE_1,\ldots,\cE_n)$
$$
\overline{\phi}: \bigoplus_{j=0}^n  N^j \to M.
$$  
 In particular, if the category $(\cE_n)$ is reduced then any finite $\cE_n$-module $M$  admits a decomposition 
$$
M=\bigoplus N^i
$$
into a sum of  free  finite $\cE_i$-modules $N^i=\cE_i^{k_i}$ for $0\leq i \leq n$.  
If $K$ is finite then such a decomposition exists after
 passing to a finite extension $K'$ of $K$.
 \end{theorem} 
 
 \begin{proof} 
As before, one can write  $M$ as the quotient module $M=\cE_n^k/M_0$. By the proof of Theorem \ref{free1} there exists a basis $(\overline{b}_{ij})_{(i,j)\in S}$ of the graded module    
  $R_n^k=\gr(\cE_n^k)$, whose part (for $S_1\subset S$) is a basis $(\overline{b}_{ij})_{(i,j)\in S_1}$ of the submodule $\gr(M_0)$, and $(\overline{b}_{ij})_{(i,j)\in S\setminus S_1}$ defines a basis of 
$\gr(M)=\gr(\cE_n^k)/\gr(M_0)$. Then, by Lemma \ref{imp}, there is a basis  $({b}_{ij})_{(i,j)\in S}$ of $\cE_n^k$ such that $({b}_{ij})_{(i,j)\in S_1}$ is a basis of  $M_0$ and $({b}_{ij})_{(i,j)\in S_1}$ defines a basis of  $M=\cE_n^k/M_0$. We proceed by induction.
 \end{proof}

\section{Implicit function and  Weierstrass-Hironaka division}

\subsection{Generalized Weierstrass-Hironaka division and preparation}

Consider any normalized grading  $\overline{T}$ on $\NN^n$. 
Let $f_{ij}(u,v)\in \cE_{n+m}$ be a finite set of functions  such that the initial exponents (with respect to $\overline{T}$) are defined and generate a diagram $\Delta$ of finite type. Without loss of generality we assume that
$\alpha_{i,j}=\expe(f_{i,j}(u,0))\in \NN^i\setminus \NN^{i-1}$ are ordered with respect to the reverse lexicographic order
 $$
\alpha_1<_r \alpha_{2,1}<_r\ldots<_r \alpha_{2,k_2}<_r\alpha_{3,1}<_r\ldots 
$$ 
and  define a decomposition of  $\Delta\subset \NN^n$ into  disjoint subsets $\Delta_{ij}=\alpha_{ij}+\Gamma_{ij}$. 
By Corollary \ref{12}:
\begin{enumerate}
\item Each $\Delta_{ij}$ can be written as
$\Delta_{ij}=\overline{B_{ij}}+\NN^{*n-i,n}$, where $\overline{B_{ij}}$ is a finite subset of $\NN^{i}\setminus \NN^{i-1}$ and $\NN^{*n-i,n}=\{(0,\ldots,0,x_{i+1},\ldots, x_n\mid x_j\in \NN\}$.
\item (Hironaka) The set $\Gamma=\NN^n\setminus \Delta$ decomposes uniquely as  the union of the sets $A_i\times \NN^{n-i}$, where all  subsets $A_i\subset \NN^{i}\setminus \NN^{i-1}$ are finite.
\end{enumerate}

\begin{theorem}[Generalized Weierstrass-Hironaka division theorem] \label{main0} 
For any  $g\in  \cE_{n+m}$, there exist   $h_{ij}\in  \cE_{n+m}$ and $r(g)=h_{00} \in  \cE_{n+m}$ such that 
$$
g=\sum h_{ij}f_{ij}+ r(g)$$

where  $\supd(h_{ij})\subset \Gamma_{ij}\times \NN^{m}$, 
%
$\supd(r(g))
\subset  \Gamma\times \NN^{m}$. 

Moreover:
\begin{enumerate}
\item If $(\cE_n)$ is reduced the decomposition is unique.
\item If $\ord(f_{ij})=|\alpha_{ij}|$ then $\ord(h_{ij})\geq \ord(g)-|\alpha_{ij}|$.
\item If $\overline{T}$ is total then  $\expe(h_{ij}(u,0))+\alpha_{ij}\geq \expe(g(u,0))$.
\item     $\ord(h_{ij}(u,0))\geq \ord(g(u,0))-|\alpha_{ij}|$.
\end{enumerate}
\end{theorem}

\begin{proof} This follows from Theorem \ref{imp2}. Uniqueness in (1) follows from uniqueness of formal division (Theorem \ref{formal}), and the fact that there is  a monomorphism
$\cE_n\to K[[u,v]]$.
\end{proof}
\begin{remark} More precisely, one can describe the coeficients in the division theorems as 
$$h_{ij}=\sum_{\beta \in \overline{B_{ij}}} c_\beta\cdot u^\beta,\quad  r(g)=\sum_{j=1}^n\sum_{\alpha\in A_j}c_\alpha\cdot u^\alpha, 
$$  
and
where  $c_\beta(u_i,\ldots,u_n,v) \in \cE_{n-i}$
 and $c_\alpha\in \cE_{n-j}(u_{j+1},\ldots,u_n,v)$,

\end{remark}

\begin{theorem}[Generalized Weierstrass-Hironaka division  theorem 2]\label{main00} 
Let $X$ be a smooth scheme  of finite type over a field $K$ (or a $\CC$-analytic/differentiable manifold) of dimension $n+m$ with a given coordinate system at a  $K$-rational point $x\in X$.  Let $f_{ij}(u,v)\in \cO(X)$ be a finite set of functions  such that the initial exponents (with respect to a certain  normalized  grading $\overline{T}$) form a diagram  $\Delta$ of finite type  at $x\in V$.
Then there exists an \'etale neighborhood  $X'\to X$
preserving the residue field $K$ (respectively an open neighborhood) such that
for any  $g\in  \cO(X')$, there exist   $h_{ij},r(g)\in  \cO(X')$  such that  $\supd(h_{i,j})\subset \Gamma_{i,j}\times \NN^{m}$, 
%
$\supd(r(g))\subset  \Gamma\times \NN^{m}$, and
$$
g=\sum h_{ij}f_{ij}+ r(g).
$$
Moreover this presentation is unique in the algebraic and the analytic situations, and  conditions (1) through (4) of the previous theorem are satisfied at  $x$.
\end{theorem}

\begin{proof} 
For existence we apply  Theorem \ref{imp2} together with Theorems \ref{NS1}, \ref{NS2}, \ref{NS3} respectively. Uniqueness  follows from uniqueness of extensions in the algebraic and analytic situations, and uniqueness of division in the local ring.	
\end{proof}

\begin{theorem}[Generalized preparation  theorem] \label{main2}
Let $\overline{T}$ be any  normalized order on $\NN^n$.
Let $\{f_1(u,v),\allowbreak\ldots,f_k(u,v)\}$ be a finite set of functions in  $ \cE_{n+m}$ for which the initial exponents with respect to $\overline{T}$ exist and
 $\expe(f_i(u,0))=\alpha^i$ are 
%
the
vertices of 
 a  diagram $\Delta$ of finite type in $\NN^n$. 
Then there is a   set of generators of the form $$\overline{f_i}:=u^{\alpha^i}+r_i$$ of the ideal $(f_1,\ldots,f_k)$ such that 
%
$\expe(\overline{f_i}(u,0))=\alpha^i$, and
${\supd}(r_i(u,0))$ is  contained in $\Gamma\times \NN^n$. 
 Moreover  each $\overline{f_i}$
can be written
 as a finite sum 
 \begin{equation}
 \overline{f_i}=u^{\alpha^i}+\sum^n_{j=1}\sum_{\alpha\in A_j}c_\alpha\cdot u^\alpha, 
 \end{equation}
 where   $c_\alpha\in \cE_{n +m-j}(u_{j+1},\ldots,u_n,v)$ for $j=1,\ldots,n$. Furthermore:
\begin{enumerate}
\item If $(\cE_n)$ is reduced the decomposition is unique.
\item If $\ord(f_{i})=|\alpha_{i}|$ then $\ord(r_i)\geq |\alpha_{i}|$.
\item $\expe(\overline{f}_{i}(u,0))=\alpha_{ij}$.
\end{enumerate}  
 \end{theorem}

\begin{proof} 
We construct $\overline{f_i}$ and show by induction
that all the sets 
$$
\overline{f_1},\ldots,\overline{f_{i-1}},\overline{f_i},f_{i+1},\ldots, f_k
$$ 
generate the same ideal.
This is true for $i=0$. Suppose it is valid for $i-1$. Consider the division  with remainder of $u^{\alpha_i}$ by $\overline{f_1},\ldots,\overline{f_{i-1}},{f_i},f_{i+1},\ldots, f_k$:
$$
u^{\alpha^i}=\sum_{j=1}^{i-1} h_j\overline{f_j}+\sum_{j=i}^{k} h_jf_j + r(u^{\alpha^i}),
$$
where 
%
${\supd}(r(u^{\alpha^i}(u,0)))
\subset  \Gamma$.
Then set 
$$
\overline{f_i}:=u^{\alpha^i}-r(u^{\alpha^i})=\sum_{j=1}^{i-1} h_j\overline{f_j}+\sum_{j=i}^{k} h_jf_j.
$$
Note that $\expe(h_j\overline{f_j}(u,0))$, and  $\expe(h_jf_j(u,0))$ are in $\Delta_j$, and thus all are distinct for distinct $j$. We have
%
$$
\alpha_i=\expe(\overline{f_i}(u,0))
=\min_j \expe(h_jf_j(u,0))=\expe(h_if_i(u,0))=\expe(h_i(u,0))+\expe(f_i(u,0)).
$$
Consequently,  
$\expe(h_i(u,0))=0$, 
and the functions   $h_i(u,0)$ and  $h_i(u,v)$ are invertible, and thus the ideals  generated by $\overline{f_1},\ldots,\overline{f_{i-1}},\overline{f_i},f_{i+1},\ldots, f_k$ and 
$\overline{f_1},\ldots,
\overline{f_{i-1}},
f_{i},\ldots, f_k$ 
are the same. The other properties follow from the previous theorem.
\end{proof}

Theorem \ref{main2} generalizes the Malgrange- Weierstrass preparation
theorem for a single variable.

We can consider another particularly simple situation of the above theorem.

\begin{corollary} 
Let $\overline{T}$ be any  monomial order on $\NN^n$.
Let  $f_1,\ldots,f_n\in  \cE_{n+m}$ be 
a  
set of functions  for which 
 $\expe(f_i(u,0))=k_i\cdot e_i$, where $i=1,\ldots,n$, $k_i\in \NN$,  $\{e_1,\ldots, e_n\}$ is the standard basis of $\NN^n$. Then:
 \begin{enumerate}
 \item There exists a   set of generators $\overline{f_i}$ of  the ideal $(f_1,\ldots,f_k)$ of the form
 \begin{equation}
\overline{f_i}:=u_i^{k_i}+\sum_{\alpha_i<k_i} c_\alpha(v)\cdot u_1^{\alpha_1}\cdot\ldots\cdot u_k^{\alpha_n},
\end{equation} 
where $c_\alpha\in \cE_{m}(v)$.

\item  For any $g\in \cE_n$ there exist $h_{i}\in  \cE_n$ and $r(g)\in  \cE_n$ such that:  
\begin{enumerate} 
\item $g=\sum h_{ij}f_{ij}+ r(g)$.
\item $r(g)=\sum_{\alpha_i<k_i} c_\alpha(v)\cdot u_1^{\alpha_1}\cdot\ldots \cdot u_k^{\alpha_k}$.
\item $h_i=\sum_{\alpha_i<k_i} c_{i\alpha}(u_{i+1},\ldots,u_k,v)\cdot u_1^{\alpha_1}\cdot\ldots \cdot u_i^{\alpha_i}$.

\end{enumerate}

\end{enumerate}
\end{corollary}

\begin{proof} 
In that case 
$\Gamma=A_n=[0,k_1-1]\times\ldots\times [0,k_n-1]$ is finite  and $c_\alpha\in \cE_{m}(v)$, and we apply the generalized preparation and division theorems.
\end{proof}

\subsection{Hironaka standard basis for algebraic, analytic and smooth functions}

The following theorem extends existence of the Hironaka standard basis theorem for formal analytic functions 
\cite{BM2}. We  note that  Henselian Hironaka-Weierstrass division in \cite{Hir2} gives, in general, no control on the multiplicities  of the remainders.  On the other hand the Hironaka standard basis in the analytic case is convergent (see \cite{Hir2}, \cite{BM}).

\begin{theorem}[Existence of a standard basis]
\label{standard basis}
Assume $(\cE_n)$ is a  smooth category over an infinite  field $K$, and let  $\cI\subset\cE_n$ be any ideal. Consider a monotone  grading $\overline{T}$ on $\NN^n$, and  let  
%
$\Delta=\Delta(\cI)=\{\expe_{\overline{T}}(f)\mid f\in\cI\}$ be   the monotone diagram of initial exponents  
defined for a generic coordinate system. 
  Let $\alpha_1,\ldots,\alpha_k$ be the set of vertices of $\Delta$ ordered by using the reverse lexicographic order.
 Then there exists a standard basis of $\cI$ with respect to $\overline{T}$, that is,  a  set of  functions ${f_i}:=u^{\alpha^i}+r_i\in \cI$, where 
 $\supd(r_i)$ is contained in $\Gamma$ for ${i=1},\ldots,k$, with $\exp(f_i)=\alpha_i$ such that: 
 \begin{enumerate}
 \item  
 $\mon({f_i}):=u^{\alpha_i}$ and
$\ord({f_i})=|\alpha_i|$.
 \item  The elements  $\inn({f_i})$ generate $\inn(\cI)$.
 \item Any function $f\in \cI$ can be  represented 
as $f= \sum h_i{f_i}=r(f)$, where   $r(f)\in m_n^{\infty}$, $\supd(r(f))\subset \Gamma$,   $\supd(h_i)\subset \Gamma_i$, and the functions $h_i\in \cE_n$  satisfy the conditions as in Theorem 
\ref{main0}, and are uniquely defined modulo $m_n^{\infty}$.
\item In the reduced category (of algebraic and analytic functions) the ideal $m_n^{\infty}$ is $0$ and condition $(3)$ can be stated as the equality $f= \sum h_i{f_i}$. In particular $\cI$ is generated by ${f_i}$. Moreover the standard basis is uniquely determined by ${f_i}$.
\end{enumerate}
\end{theorem}

\begin{proof}
 Let $\overline{f}_i$ be any functions with $\exp(\overline{f}_i)=\alpha_i$. Write $u^{\alpha^i}=\sum h_i\overline{f}_i+r_i$. Then ${f_i}:=u^{\alpha^i}+r_i$ satisfies $\expe(f_i)=\alpha_i$.
This shows that the elements $\expe(f_i)$ generate $\Delta(\cI)$, and the $\inn(f_i)$ generate $\inn(\cI)$. If $f\in \cI$ then by the division theorem we can write $f=\sum h_if_i+r(f)$, where $f\in \cI$, $\supd(r(f))\subset \Gamma$.
Then  we get $\expe(r(f))\subset \Gamma\cap \Delta=\emptyset$ and  $r(f)\not\in m^\infty$. This argument also implies uniqueness of the standard basis in the reduced category. 
\end{proof}

 In a nonreduced category the standard basis need not  generate the ideal $\cI$. One can remedy this  under  stronger assumptions by using the preparation theorem.

\begin{theorem} \label{standard basis 2} 
Let $\Delta$ be a monotone diagram of the initial exponents with respect to a certain total normalized monomial grading $\overline{T}$. Consider the ideal $\cI$ generated by functions $f_i$ such that the elements $\alpha_i:=\expe(f_i)$ are 
%
the 
vertices of $\Delta$. Then there exists a standard basis ${f_i}$ satisfying conditions {\rm (1)--(3)} and generating $\cI$.
\end{theorem}

\begin{proof}
The proof and construction of ${f_i}$ are the same. To show that the elements  ${f_i}$ generate $\cI$ we apply the preparation theorem. 
\end{proof}

\begin{theorem}[Existence of a standard basis 2]
\label{standard basis 3}
Let $X$ be a regular scheme  of finite type over a field $K$ (or a complex analytic/differentiable manifold) with a given coordinate system, and $x\in X$ be a $K$-rational point.  Let $\cI$ be a sheaf of ideals  on $X$ of finite type.  Consider a monotone  grading $\overline{T}$ on $\NN^n$, and  let  
$\Delta=\Delta(\cI)=\{\expe_{\overline{T}}(f)\mid f\in\cI\}$ be  the diagram of initial exponents (which is monotone and thus of finite type for a generic coordinate system). 
  Let $\alpha_1,\ldots,\alpha_k$ be the set of vertices of $\Delta$ ordered according to the reverse lexicographic order.
 Then there is an \'etale  neighborhood $X'\to X$ of $x$ preserving the residue field at $x\in X$ (respectively an open neighborhood) and  a standard basis of $\cI$ on $X'$, which is a uniquely determined set of  functions $f_1,\ldots,f_k\in \cI(X')$, of the  form 
 $$
{f_i}:=u^{\alpha^i}+r_i\in \cI,
$$ 
where  
 $\supd(r_i)\subset \Gamma$ for ${i=1},\ldots,k$ and 
 $\ord_x({f_i})=|\alpha_i|$. Moreover (in the algebraic and analytic setting)
  any function $f\in \cI(X')$ can be uniquely represented 
as $f= \sum h_i{f_i}$,  where   $h_i\in \cO(X')$ with $\supd(h_i)\subset \Gamma_i$. (In the differential category 
the function $f\in \cI$ can be uniquely represented 
as $f\equiv \sum h_i{f_i}\ ({\rm mod} \ m_n^{\infty})$.)
\end{theorem}

\begin{proof} 
Follows from  Theorems \ref{NS1},  \ref{NS2}, \ref{NS3} and  \ref{standard basis}.	
\end{proof}

\subsection{Implicit function theorems}

Let $\Delta\subset \NN^n$ be a diagram of finite type with vertices $\alpha_i$ ordered reverse-lexicographically as in   Section \ref{D}. 

Consider the finite decomposition \ref{S3} after Corollary \ref{12}, and write 

$$
\Gamma=\bigcup_{i=1}^n A_i\times \NN^{n-i},\quad \Delta =\bigcup_{i=1}^n  \overline{B}_i+\NN^{*n-i+1,n},
$$ 
Set 
$$
d(\Delta):=\max\Big( \bigcup_i (A_i\cup \overline{B}_i)\Big),\quad \Delta(s):=\{\alpha\in \Delta\mid |\alpha|=s\}.
$$
Let $\{f_1(u,v),\ldots,f_k(u,v)\}$ be a finite set of functions in  $ \cE_{n+m}=\cE_{n+m}(u_1,\ldots,u_n,v_1,\ldots,v_m)=\cE_{n+m}(u,v)$ which is in set-theoretic bijective correspondence with vertices 
$$
f_i\mapsto \alpha_i,\quad \ord(f_i(u,0))
=|\alpha_i|.
$$  
Any $\beta\in \Delta_i$ can be written as $\beta=\alpha_{i}+\gamma_i$,
where  $\gamma_i\in {\Gamma}_{i}$.   Set   
$$
f_\beta:=u^\gamma f_{i},\quad i(\beta):=i. 
$$  
We  define \emph{generalized Jacobians} as
%
$$ 
J^s(f_1,\ldots,f_k: u^{\alpha_1},\ldots,u^{\alpha_k})
:=\det\big[D_{u^{\alpha}}{(f_{\beta})}\big]_{\alpha,\beta\in \Delta(s)}
=\det\big[D_{u^{\alpha-\beta+\alpha_{i(\beta)}}}{(f_{i(\beta)})}\big]_{\alpha,\beta\in \Delta(s)}.
$$
Recall that $D_{u^\alpha}=\frac{1}{\alpha!}\frac{\partial^{|\alpha|}}{\partial u^\alpha}$, and
note that there is an  obvious equality 
$$
D_{u^{\alpha}}{(f_{\beta})}=D_{u^{\alpha}}(u^{\beta-{\alpha_{i(\beta)}}}f_{i(\beta)})=D_{u^{\alpha-\beta+{\alpha_{i(\beta)}}}}{(f_{i(\beta)})},
$$
 which allows us to compute Jacobians in two different ways.
 
\begin{remark}
Here 
%
$D_{{u^{\alpha-\beta+\alpha_{i(\beta)}}}}$ is assumed to be $0$ if $\alpha-\beta+\alpha_{i(\beta)}$ contains a negative component.
\end{remark}

In the case of functions of multiplicity one this notion coincides with the standard Jacobian 
$$ 
J^1(f_1,\ldots,f_k: u_1,\ldots,u_k):=\det\big[D_{u_j}{(f_i)}\big].
$$

\begin{theorem}[Generalized division theorem 3]\label{main01}
If for all  $s\leq d(\Delta)+1$, the Jacobians
$$ 
J^s(f_1,\ldots,f_k: u^{\alpha_1},\ldots,u^{\alpha_k})
$$ 
are invertible
then for any  $g\in  \cE_n$ there exist   $h_{i}\in  \cE_n$ and $r(g)\in  \cE_n$ such that  $\supd(h_{i})\subset \Gamma_{i}$, 
%
$\supd(r(g))
\subset  \Gamma$, 
and
$$
g=\sum h_{i}f_{i}+ r(g).
$$
 Also 
$$
r(g)=\sum_{j=1}^n\sum_{\alpha\in A_j}c_\alpha\cdot u^\alpha, 
$$
 where $c_\alpha\in \cE_{n-j}(u_{j+1},\ldots,u_n,v)$ 

Moreover, if 
%
$\ord(f_{i})=ord(f_i(u,0))
=|\alpha_{i}|$ then $\ord(h_{i})\geq \ord(g)-|\alpha_{i}|$.
\end{theorem}

\begin{proof} The theorem is a consequence of Theorem \ref{imp3} and Lemma \ref{restriction}. Consider the rings of polynomials $R_i$ over   $K$.  By Corollary \ref{imp33}, the  set of monomials 
$$
\Big\{u^\alpha \mid \alpha \in \bigcup  (A_{i-1}\cup  \overline{B}_i)\Big\}
$$ 
is a basis of the module $R_n$ over $(R_i)_{i=0}^n$.
 The condition of invertibility of $J^s(f)(0)$ means exactly that  
%
$$
\{\inn(f_\beta(u,0))\mid  \beta\in \Delta(s)\}
 \cup 
\{u^\alpha\mid  \alpha \in  \Gamma(s)\}
$$ 
is a  basis of the $s$ gradation $(R_n)_s$ of $R_n$ (over $K$).

 Consequently,  by the stabilization theorem  (Theorem \ref{imp}), the set  
%
$$
\Big\{\inn(f_\beta(u,0))\mid  \beta\in \bigcup \overline{B}_i\Big\} 
\cup  \Big\{u^\alpha\mid  \alpha \in \bigcup  A_{i-1}\Big\}
$$ 
is a basis of $R_n$   because it generates the module $R_n$ over $(R_i)_{i=0}^n$  up to  degree $d(\Delta)+1$.

Thus, by Lemma \ref{restriction3}, the set  
%
$$
\Big\{\inn(f_\beta(u,v)) \mid  \beta\in \bigcup \overline{B}_i\Big\} 
\cup   \Big\{u^\alpha \mid  \alpha \in \bigcup  A_i\Big\}
$$
is a basis of $R_{n+m}$   because it generates the module $R_{n+m}$ over $(R_i)_{i=0}^{n}$. 

Finally, by Theorem \ref{imp3},
the set  
$$
\Big\{f_\beta(u,v)\mid  \beta\in \bigcup \overline{B}_i\Big\}
 \cup   \Big\{u^\alpha \mid  \alpha \in \bigcup  A_i\Big\}
$$
is a basis of $\cE_{n+m}$ over $(\cE_{i+m})_{i=0}^n$.  
 In the ``moreover'' part  we can use the standard filtration on $\cE_{n+m}$ and $R_{n+m}$ and   Lemma \ref{restriction3}.
\end{proof}

\begin{corollary}\label{main012} Under the previous assumption the set $$
\Big\{\inn(f_\beta(u,v)) \mid  \beta\in \bigcup \overline{B}_i\Big\} 
\cup   \Big\{u^\alpha \mid  \alpha \in \bigcup  A_i\Big\}=
\{\inn(f_\beta(u,v))\mid  \beta\in \Delta\}
 \cup 
\{u^\alpha\mid  \alpha \in  \Gamma\}
 $$
the set generates $R_{n+m}$ over $(R_i)_{i=0}^{n}$.

\end{corollary}

\begin{theorem}[Generalized division theorem 4]\label{main011}
Let $X$ be a smooth scheme  of finite type over a field $K$ (or a $\CC$-analytic/differentiable manifold) of dimension $n$ with a given coordinate system at a  $K$-rational point $x\in X$. Let   $\Delta\subset \NN^r\subset \NN^n$  be a diagram of finite type, 
 for some $r\leq n$, generated by $\alpha_i\in \NN^r$. Let $f_1,\ldots,f_k\in\cO(X)$ be functions in a bijective correspondence with $\alpha_1,\ldots,\alpha_k$, and  such that $\ord_x(f_i)=|\alpha_i|$.
Assume moreover that
for all  $s\leq d(\Delta)+1$, the Jacobians
$$ 
J^s(f_1,\ldots,f_k: u^{\alpha_1},\ldots,u^{\alpha_k})
$$ 
are invertible at  $x$. Then there exists an 
\'etale neighborhood  $X'\to X$ preserving the residue field $K$ (respectively an open neighborhood)
such that for any  $g\in  \cO(X')$ there exist   $h_{i}\in  \cO(X')$ and $r(g)\in  \cO(X')$ with  $\supd(h_{i})\subset \Gamma_{i}$, 
%
$\supd(r(g))\subset  \Gamma$, and
$$
g=\sum h_{i}f_{i}+ r(g).
$$
Moreover, if $\ord_x(f_{ij})=\ord_x(f_i(u,0)=|\alpha_{i}|$ then $\ord_x(h_{i})\geq \ord_x(g)-|\alpha_{i}|$.
\end{theorem}

\begin{proof} This follows directly from the previous theorem and
Theorems \ref{NS1}, \ref{NS2}, \ref{NS3}. 	
\end{proof}




Denote by $\Ver(\Delta)$  the set of vertices of $\Delta$, and set  
%
$\Ver(\Delta(s)):=\Ver(\Delta)\cap \Delta(s)$. 
We define  
$$
\Delta^0(s):=\Delta(s)\setminus \Ver(\Delta(s))
$$ 
if $\Ver(\Delta(s))\neq \emptyset$, otherwise $\Delta^0(s):=\emptyset$.
Set 
$$ 
J^0_s(f_1,\ldots,f_k: u^{\alpha_1},\ldots,u^{\alpha_k}):=\det\big[D_{u^{\alpha}}{(f_{\beta})}\big]_{\alpha,\beta\in \Delta^0(s)}=\det\big[D_{u^{\alpha-\beta+i(\beta)}}{(f_{i(\beta)})}\big]_{\alpha,\beta\in \Delta^0(s)}.
 $$
The Jacobians $J^0_s$ occur only for  functions of distinct multiplicities.

\begin{theorem}[Generalized implicit function theorem]\label{im1}
With the notation of Theorem \ref{main01}, assume that for all  $s\leq d(\Delta)+1$, the Jacobians 
$$
J^s(f_1,\ldots,f_k: u^{\alpha_1},\ldots,u^{\alpha_k})\quad \text{and}\quad J^0_s(f_1,\ldots,f_k: u^{\alpha_1},\ldots,u^{\alpha_k})
$$ 
are invertible.
Then there is a   set of generators $\overline{f_1},\ldots,\overline{f_k}$ of the ideal $(f_1,\ldots,f_k)$ of the form  
\begin{equation}
 \overline{f_i}:=u^{\alpha^i}+r_i=u^{\alpha^i}+\sum^n_{j=1}\sum_{\alpha\in A_j}c_\alpha\cdot u^\alpha, 
 \end{equation}
 where   $c_\alpha\in \cE_{n +m-j}(u_{j+1},\ldots,u_n,v)$ for $j=1,\ldots,n$, and ${\supd}(r_i(u))\subset\Gamma\times \NN^k$. 

Moreover, $\ord(\overline{f_i}(u,0))=\ord({f_i}(u,0))=|\alpha_i|$, and if
$\ord({f_i})=|\alpha_i|$ then $\ord(\overline{f_i})=|\alpha_i|$.
\end{theorem}

\begin{proof} The proof is similar to the proof of the preparation theorem.
We construct $\overline{f_i}$ with  multiplicities $|\alpha_i|$ and show by induction on the multiplicity $s$
that all the ideals   
$$
(\overline{f_i}\mid |\alpha_i|\leq s)+(f_i\mid |\alpha_i|> s)\quad  \text{and}\quad  (\overline{f_i}\mid |\alpha_i|\leq s')+(f_i\mid |\alpha_i|> s')
$$ 
are the same.

This is true for $s=0$. Suppose it is valid for $s$. Let $s'$ be the smallest integer $>s$ for which there exists a subscript $1\leq j\leq k$ such that $s'=|\alpha_j|$.
For  all $j$ such that $\alpha_j=s'$ consider the division  with remainder of $u^{\alpha_j}$  by ${f_1},\ldots, f_k$:
$$
u^{\alpha^i}=\sum_{j=1}^{k} h_{ij}{f_j}+ r(u^{\alpha^i}),
$$
where 
%
${\supd}(r(u^{\alpha^i}))
\subset  \Gamma\times \NN^m$ and $ \supd(h_{ij})\subset \Gamma_j\times \NN^m$.
Then set 
\begin{equation} \label{imm} 
\overline{f_i}:=u^{\alpha^i}-r(u^{\alpha^i})=\sum_{j=1}^{i-1} h_{ij}{f_j}=\sum_{|\alpha_i|< s'} h_{ij}f_j+\sum_{|\alpha_i|= s'} h_{ij}f_j+\sum_{|\alpha_i|> s'} h_{ij}f_j .
\end{equation}
Passing to the initial forms gives
$$
\inn(\overline{f_i}(u,0))=\sum_{|\alpha_i|< s'} \inn(h_{ij}(u,0))\inn(f_j(u,0))+\sum_{|\alpha_i|= s'} \inn(h_{ij}(u,0))\inn(f_j(u,0)).
$$
 We see that if $|\alpha_i|< s'$ then $\ord(h_{ij})>0$.
 (The summation in the above formula is taken  only over terms of degree $s'$.)
 
 The condition of $J^0_s$ being  invertible  can be written as invertibility of 
 $$
\det\big[D_{u^{\alpha}}{(f_{\beta})}\big]_{\alpha,\beta\in \Delta^0(s)},
$$
%
which
 implies that 
%
$$
\{\inn(\overline{f_{i}(u,0)})
\mid |\alpha_i|= s'\} \cup   \{\inn({f_\beta}(u,0))\mid \beta\in  \Delta_0({s'})\}
$$
 forms a basis of $\inn_{s'}(\cI(u,0))$.
 
The invertibility of  $J_s$  means that  
$$
\det\big[D_{u^{\alpha}}{(f_{\beta})}\big]_{\alpha,\beta\in \Delta(s)}
$$
is invertible and  that
$$
\{\inn({f_{i}}(u,0))\mid |\alpha_i|= s'\} \cup  \{\inn({f_\beta}(u,0))\mid \beta\in  \Delta_0({s'}) \}
$$
 is another  basis of $\inn_{s'}(\cI(u,0))$.
Consequently,  
the determinant  
$$
\det\big[\overline{h_{ij}}(0)\big]_{|\alpha_i|=|\alpha_j|= s'}
$$ 
is invertible.

 By the inductive assumption, for any $i'$ with $|\alpha_{i'}|<s'$ we  can write 
$$
f_{i'}=\sum_{|\alpha_{j}|<s'} g_{i'j}\overline{f_j}+\sum_{|\alpha_{j}|\geq s'} g_{i'j}f_j.
$$
 Substituting this formula for $f_{i'}$  in (\ref{imm}) gives 
\begin{align*}
\overline{f_i} &=\sum_{|\alpha_i|< s'} h_{ij'}\Big(\sum_{|\alpha_{j'}|<s'} g_{i'j'}\overline{f_{j'}}+\sum_{|\alpha_{j}|\geq s'} g_{i'j}f_j\Big)+\sum_{|\alpha_i|= s'} h_{ij}f_j+\sum_{|\alpha_i|> s'} h_{ij}f_j 
\\ 
&=\sum_{|\alpha_i|< s'} \overline{h_{ij}}\cdot \overline{f_j}+\sum_{|\alpha_i|= s'}\overline{h_{ij}}f_j+\sum_{|\alpha_i|> s'} \overline{h_{ij}}f_j,
\end{align*}
 where  the matrix  
$$
\big[\overline{h_{ij}}(0)\big]_{|\alpha_i|=|\alpha_j|= s'}=\big[{h_{ij}}(0)\big]_{|\alpha_i|=|\alpha_j|= s'}
$$ 
is invertible.
  The latter implies that  
%
$$
(\overline{f_i}\mid |\alpha|\leq s)+(f_i\mid |\alpha|> s)=(\overline{f_i}\mid |\alpha|\leq s')+(f_i\mid |\alpha|> s'),
$$ 
 of which completes the inductive step.
\end{proof}

Theorems \ref{main01} and \ref{main011} generalize the preparation and division theorems. 
One can easily see that the Jacobian condition generalizes the condition for the initial exponents with respect to a monomial order.

\begin{lemma}\label{easy} 
Let $K$ be any commutative ring. Let $F_1,\ldots, F_k\in K[x_1\ldots,x_n]$ be  forms. Let $\overline{T}$ be a normalized order on $\NN^n$ and  $\Delta$ be a diagram of finite type with vertices  $\alpha_i=\expe_{\overline{T}}(f_i)$. Assume all the highest coefficients are invertible.  Then all the Jacobians 
$$
J^s(F_1,\ldots,F_k: u^{\alpha_1},\ldots,u^{\alpha_k})\quad \text{and}\quad J^s_0(F_1,\ldots,F_k: u^{\alpha_1},\ldots,u^{\alpha_k})
$$ 
are invertible.
\end{lemma}

\begin{proof} 
Order the monomials according to $\overline{T}$. Then each generalized Jacobian
matrix is lower  triangular with invertible entries on the main diagonal.
\end{proof}

\begin{example}\label{Jacobian}
In the original implicit function theorem  the relevant diagram $\Delta$ is generated by the standard basis $e_1,\ldots,e_n\in\NN^n$, and $A_0=\{0\}$, $A_i=\emptyset$, $i>0$, $\overline{B}_i=\{ e_i\}$, so that we get $d(\Delta)=1$. The condition $J^s(f)(x)\neq 0$ in Theorem \ref{im1} is assumed only for $s=1
=d(\Delta)$. In the theorems above we need to verify the Jacobian conditions  for $s\leq d(\Delta)+1=2$. 

The generalized Jacobians depend only on the initial forms of the function, so we  use only the initial forms in the computations below.

Let $F_1=a_{11}x_1+a_{12}x_2$, $F_2=a_{21}x_1+a_{22}x_2$. Then 
$$
J^1=J^1(F_1,F_2:x_1,x_2)=\det\left( \begin{array}{cc} 
                a_{11}&a_{12} \\
                a_{21} & a_{22}  \end{array}\right)
$$ 
is the ordinary Jacobian and 
\[ 
J^2=\det\left( \begin{array}{ccc} 
                a_{11}&a_{12} & 0 \\
                0& a_{11}  &  a_{12}  \\
                0  &a_{21} & a_{22}  \end{array}\right) =a_{11}\cdot J^1 
\] 
We observe the occurrence  of the extra condition of invertibility of $a_{11}$ in the generalized implicit and division theorems when comparing to the condition $J^1\neq 0$ in the classical implicit function theorem. This condition is due to lack of symmetry of a Stanley basis (and division), and  it can be fulfilled by merely swapping coordinates. 
 In that sense the condition is not essential from the point of view of the formulation of the 
classical implicit function theorem  and can be dropped. 

\end{example}

The generalized implicit function theorem  allows us to choose the exponents and corresponding
 monotone diagrams $\Delta$ in many different ways as long as the multiplicities of generators and degrees of the corresponding monomials are equal and the Jacobian conditions hold. Among all those diagrams, particularly useful is the simplest one defined by the $k$ vectors which are multiples of the standard basis or its part. Again such a diagram can be applied to a generic set of $k$ functions.
 
 The invertibility of Jacobians   can be regarded as a generalization of the transversality condition in the smooth case  (see Theorem \ref{imp22}(1)).

Let  $f_1,\ldots,f_k\in  \cE_{n}(u_1,\ldots,u_k,v_1,\ldots,v_{n-k})$, where $k\leq n$, be a  set of functions such that 
%
$\ord(f_i(u,0))
=d_i$, where $d_1\leq \ldots\leq d_n$. Let $\Delta$ be the diagram generated by  $\alpha_i=d_i\cdot e_i$, where $i=1,\ldots,k$.  Then 
$$
\Delta=\{\alpha\in \NN\mid \exists i\  \alpha_i\geq d_i\},\quad\text{and}\quad  i(\alpha):=\min\{i\mid \alpha_i\geq d_i\}
\quad\text{ 
for any $\alpha\in \Delta$}.
$$

\begin{corollary}[Generalized implicit function theorem 2] \label{imp22} 
Let  $f_1,\ldots,f_k\in  \cE_{n}(u_1,\ldots,u_k,v_1,\ldots,v_{n-k})$, where $k\leq n$, be a  set of functions such that 
%
$\ord(f_i(u,0))
=d_i$, where $d_1\leq\ldots\leq d_n$. Denote by $\cI\subset \cE_n$ the ideal generated by $f_1,\ldots,f_k$.
If for all  $s\leq \sum d_i-k+2$,  the Jacobians
$$
\label{spe2} J^s(f_1,\ldots,f_k: u_1^{d_1},\ldots,u_k^{d_k})
$$ 
are invertible  then the following conditions are satisfied:


\begin{enumerate}
\item The module  $\cE_n/\cI$ is free (or quasifree for the nonreduced category) and finite over $\cE_{n-k}(v)$ with a basis  $ u_1^{\alpha_1}\cdot\ldots\cdot u_k^{\alpha_k}$, where ${\alpha_i<d_i}$, and 
$$
\rank_{\cE_{n-k}(v)}(\cE_n/\cI)=d_1\cdot\ldots\cdot d_k.
$$ 
 
\item Each  $f\in\cE_n$ can be written uniquely (up to flat functions) as 
$$
f=\sum c_if_i+r(f),
$$ 
where $\supd(r(f))\subset \Gamma\times \NN^{n-k}$ and  $\supd(c_i) \subset \Gamma_i\times \NN^{n-k}$, that is,
$$
r(f)=\sum_{\alpha_j<d_j} c_\alpha(v)\cdot u_1^{\alpha_1}\cdot\ldots \cdot u_i^{\alpha_j},\quad     
c_i=\sum_{\alpha_j<d_j} c_\alpha(u_{j+1},\ldots,u_n,v)\cdot u_1^{\alpha_1}\cdot\ldots\cdot u_i^{\alpha_j}.
$$ 
\item If additionally $f\in \cI$ then  $r(f)\in m_{n+m}^\infty$.

\item If $\ord(f_i)=d_i$ (or $m=0$) then the Hilbert function of $\cI$ is equal to 
$$
H_{\cI}=H(\Delta)=\sum_{0\leq a_i \leq d_i-1} \phi(n+|a|,d)=d_1\cdot\ldots\cdot d_n\cdot t^{n-k}+ a_{n-k-1}t^{n-k-1}+\ldots+a_0.
$$

\item If additionally 
$$
J_0^s(f_1,\ldots,f_k: u_1^{d_1},\ldots,u_k^{d_k})
$$ 
are invertible for $s\leq \sum d_i-n+2$, then  there exists a set of generators $\overline{f_i}$  of the ideal $\cI$ of the form 
$$
\overline{f_i}:=u_i^{d_i}+\sum_{\alpha_i<d_i} c_\alpha(v)\cdot u_1^{\alpha_1}\cdot\ldots u_k^{\alpha_n},
$$
where $c_\alpha\in \cE_{m}(v)$. 
\item 
%
$\ord(\overline{f_i}(u,0))
=\ord({f_i}(u,0))=|\alpha_i|$, and if
$\ord({f_i})=|\alpha_i|$ then $\ord(\overline{f_i})=|\alpha_i|$.
\item The generators $\overline{f_i}$ satisfy  condition $(2)$ above.

\end{enumerate}
\end{corollary}

\begin{proof}  In this situation 
$\Gamma=A_n=[0,d_1-1]\times\ldots\times [0,d_n-1]$ is finite,  $B_i=d_ie_i+\{(\alpha_1,\ldots,\alpha_{i-1},0,\ldots,0)\mid \alpha_i<d_i\}$ and $d(\Delta)=(\sum d_i)-n+1$.  
Conditions (2) and (5) follow from the previous theorem. 
The conditions on the Jacobians imply, by Theorem \ref{Macaulay5}, that the functions $f_i$ form a Cohen-Macaulay regular sequence.
Condition (1) is a consequence of the more general Theorem \ref{Rem4} below. The other conditions follow from Theorem \ref{imp2}.
\end{proof}

\begin{remark} 
The above theorems are valid with unchanged proofs for homogeneous (and nonhomogeneous) polynomials $F_1,\ldots, F_k$ in $K[x_1\ldots,x_n]$, where  $K$ is a commutative ring with $1$. Under the Jacobian conditions,  $K[x_1,\ldots,x_n]/(F_1,\ldots, F_k)$ is a free module over  $K[x_{k+1},\ldots,x_n]$ of rank $d_1\cdot\ldots\cdot d_n$. 
 
When $k=n$ in the theorem above, and $K$ is a field,  the integer 
%
$\dim_K (K[x_1,\ldots,x_n]/\cI)=d_1\cdot\ldots\cdot d_n$ can be understood as the number of zeroes with multiplicities of the polynomials $f_1,\ldots,f_n$ in the affine space $\bbA^n_K$, as in B\'ezout's theorem.
\end{remark}

\subsection{Cohen-Macaulay Weierstrass isomorphism}

Let us analyze the conditions imposed in Theorem \ref{imp22}, and consider
the following example.

\begin{example} Let $F_1=a_{11}x^2_1+a_{12}x_1x_2+a_{13}x^2_2$ and  $F_2=a_{21}x^2_1+a_{22}x_1x_2+a_{23}x^2_2$. Then $d(\Delta)=2$, and the  computation of Jacobians  will be carried out up to order $s\leq 4$ (starting from $s=2$).
We have
$$
J^2=J^2(F_1,F_2:x^2_1,x^2_2)=\det\left( \begin{array}{cc} 
                a_{11}&a_{13} \\
                a_{21} & a_{23}  \end{array}\right).
$$
The condition $\cJ^2\neq 0$ implies the linear independence of the pure quadratic parts  of the (initial) forms. 
Next,
\[ 
J^3=\det\left( \begin{array}{cccc} 
                a_{11}&a_{12} & a_{13} & 0  \\
                0& a_{11}  &  a_{12} & a_{13}  \\
                 a_{21} & a_{22}&a_{23}&0 \\
                 0 &a_{21} & a_{22}&a_{23}
                 \end{array}\right).  
\] 
Thus $\cJ^3$ is nothing but the resultant of the forms. Its invertibility  implies that the forms $F_1,F_2$ have no common linear factor. A simple computation shows  that $\cJ^4=a_{11}\cJ^3$, and $\cJ^5=a_{11}^2\cJ^3$  as before. The additional  conditions of invertibility of $a_{11}$ and  $\cJ_2$ are due to lack of symmetry of the construction and imply, together with other conditions,  certain asymmetry properties of filtered Stanley decompositions like the division theorem. Once
the forms are in general position, that is, $\cJ^3$ is  invertible, the conditions of invertibility of $a_{11}$ and $\cJ^2$ can be ensured by a  generic change of coordinates, and as such are not essential for many  properties.

Observe that the example above easily generalizes to the case of any two forms $F_1,F_2\in K[x,y]$ of degrees $d_1,d_2$. We get 
$\cJ^d(F_1,F_2)=\Res(F_1,F_2)$, where $d=d(\Delta)=
d_{1}+d_{2}
-k+1$.
\end{example}

In general, the determinants of the matrices $J^s(F_1,\ldots,F_k:x_1^{d_1},\ldots,x_k^{d_k})$ were  first considered by Macaulay, and used in his definition of (projective) resultant for homogeneous polynomials.
He proved in particular the following results:

\begin{theorem}[Macaulay \cite{Macaulay}, \cite{Kalorkoti}] \label{Macaulay5}
\begin{enumerate}

\item $J^s(F_1,\ldots,F_k)= \Res(F_1,\ldots,F_k)\cdot \Delta(F_1,\ldots,F_k,s)$  for all $s\geq d:=\sum d_i-k+1$,   
where $\Delta(F_1,\ldots,F_k,s)$ is 
,so called, {\it an extraneous factor}.
\item The square-free part of $\Delta(F_1,\ldots,F_k,s)$ is equal (up to sign) to the
minor of $M(F_1,\ldots,F_n;s)$ obtained by deleting all rows and columns that are indexed by any
power product, that is, divisible by exactly one $x_{d_i}$ for $1\leq i \leq  n$.
\end{enumerate}
\end{theorem}

Moreover it follows from his theorem and 
his
considerations that $\Delta(F_1,\ldots,F_k,s)$ can be made invertible by some coordinate change.

Let $K$ denote a field and $F_1,\ldots, F_k\in R=K[x_1,\ldots,x_k]$ be forms of degree $d_1,\ldots,d_k$ respectively. Set $I=(F_1,\ldots, F_k)$, and denote by $I_s$ and $R_s$ the forms in $I$ and $R$ of degree $s$. Let $t$ be an indeterminate over 
$\ZZ$ 
and  write the Hilbert function as the formal power series 
$$
H(R/I)=\sum_{s=0}^\infty (R_s/I_s)t^s.
$$
The following result is essentially due to Macaulay.

\begin{theorem}[Macaulay] 
The following conditions are equivalent:
\begin{enumerate}
\item The forms $F_1,\ldots, F_k$  form a regular sequence.
\item  The resultant $\Res(F_1,\ldots, F_k)$ does not vanish.
\item  The vector space $K[x_1,\ldots,x_k]/(F_1,\ldots, F_k)$  has  finite dimension.
\item The ideal $I:=(F_1,\ldots, F_k)$ is $(x_1,\ldots, x_k)$-adic.	
\item The forms $F_1,\ldots, F_k$ have no nontrivial solutions over the algebraic closure of $K$.
\end{enumerate}
Moreover, if these conditions are satisfied then:
\begin{enumerate}
\item	
%
$\dim_K(K[x_1,\ldots,x_n]/(F_1,\ldots, F_k))=d_1\cdot\ldots\cdot d_k$. 

 \item $H_{(F_1,\ldots, F_i)}(t)
=(1-t^{d_1})\cdot\ldots\cdot(1-t^{d_i})(1-t)^n=H_{(x_1^{d_1},\ldots, x_i^{d_i})}(t)$ for any $i=1,\ldots, k$.	
 \item If $K$ is an infinite field then after some generic change of coordinates, 
$$
K[x_1,\ldots,x_k]/(F_1,\ldots, F_i)=(K[x_1,\ldots,x_k]/(F_1,\ldots, F_i,x_{i+1},\ldots,x_k))[x_{i+1},\ldots,x_k)]
$$
is finite over $K[x_{i+1},\ldots,x_k]$ of degree  $d_1\cdot\ldots\cdot d_i$.
 \end{enumerate}
\end{theorem}

It follows from the Macaulay definition of resultant that invertibility of the generalized Jacobians implies invertibility of the resultant. It is natural to ask whether the converse is true.

\begin{conjecture} 
Let $K$ be an infinite field, and $F_1,\ldots,F_k$ be homogeneous polynomials in $K[x_1,\ldots,x_k]$. Then the following conditions are equivalent:
\begin{enumerate}

\item $\Res(F_1,\ldots,F_k)$ is invertible (or $F_1,\ldots,F_k$ is a regular sequence).
\item There exists a generic coordinate  change such that
$$
J^s(F_1,\ldots,F_k: x_1^{d_1},\ldots,x_k^{d_k})
$$ 
is invertible for  all $s$. 	
\item There exists a generic linear coordinate 
change
such that 
 $K[x_1,\ldots,x_k]/(F_1,\ldots,F_i,x_{i+1},\ldots,x_k)$ is generated
 by $x^{\alpha}$, with $0\leq \alpha_i<d_i$.
\end{enumerate}
\end{conjecture}

It is quite easy to show the conjecture for $k\leq 2$.
The problem is related to a famous Eisenbud-Green-Harris conjecture.

The following results can be considered as a generalization of Theorem \ref{Rem} on the Weierstrass isomorphism for  Cohen-Macaulay singularities.
 
\begin{theorem}[Cohen-Macaulay-Weierstrass isomorphism] \label{Rem4}
Let $\cE_{n+k}(x,y)/\cI$ be a Cohen-Macaulay local ring (of dimension $k$) with a regular sequence defined by coordinates $y_1,\ldots,y_k$.
Let $\Delta=\expe_T(\cI(x,0))\subset \NN^n$ (for a normalized total order $T$) be the induced diagram. Then there is an isomorphism (quasi-isomorphism in the nonreduced category) of  free  $\cE_k(y)$-modules
$$
\phi: \bigoplus_{\alpha\in \Gamma} \cE_k(y)x^\alpha\to \cE_{n+k}/\cI.
$$	
\end{theorem}

\begin{proof} 
Consider the ideal $\cJ=(y_1,\ldots,y_k)\subset \cE_k$, and set $R=\cE_n/\cI$. Let $\cJ_R\subset R$ denote the ideal generated by $y_1,\ldots,y_k$. Then $\cJ_R^i=\cJ^i\cdot R$.
The  homomorphism $\phi$ is surjective by  Weierstrass-Hironaka division, and likewise its restriction
$$
\phi: \bigoplus_{\alpha\in \Gamma} \cJ^i\cdot x^\alpha\to \cJ^i_R.
$$
 By a theorem of Matsumura \cite{Matsumura}, 
$$
\gr_{\cJ_R}(R)=\bigoplus \cJ_R^i/\cJ_R^{i+1}\simeq (R/\cJ_R)[y_1,\ldots,y_k].
$$	
The epimorphism $\phi$ defines a gradation preserving epimorphism
	$$
\overline{\phi}: \bigoplus_{\alpha\in \Gamma} \cJ^i/\cJ^{i+1}\cdot x^\alpha=\bigoplus_{\alpha\in \Gamma} K[y_1,\ldots,y_k]\cdot x^\alpha\to \cJ^i_R=(R/\cJ_R)[y_1,\ldots,y_k],
$$
which is an isomorphism. If a function $f\in \bigoplus_{\alpha\in \Gamma} \cE_k(y)x^\alpha $ is in the kernel of ${\phi}$, and 
$$
f\in \cJ^i\cdot\Big(\bigoplus_{\alpha\in \Gamma} \cE_k(y)x^\alpha\Big)\setminus\cJ^{i+1}\cdot \Big(\bigoplus_{\alpha\in \Gamma} \cE_k(y)x^\alpha\Big),
$$ 
then it defines a nonzero element in the kernel of $\overline{\phi}$, which is impossible. 
 Thus  $f\in \cJ^\infty\cdot(\bigoplus_{\alpha\in \Gamma} \cE_kx^\alpha)$.
\end{proof}

\begin{theorem}[Weierstrass isomorphism 2] \label{Rem2} Let $\pi:X\to Y$ be a smooth of dimension $k$ of smooth schemes of dimension $n+k$ and $n$ over $K$ (respectively $X, Y$ are open subsets  of $\CC^{n+k}$ or 
$\RR^{n+}$ and of $\CC^{n}$ or 
$\RR^{n}$, and  $\pi:X\to Y$ is the restriction of the natural projection $\pi_0:\CC^{n+k}\to \CC^{n}$ or $\pi_0:\RR^{n+k}\to \RR^{n}$.

Let $z=(0,0)$ be  a ($K$-rational) point of $X$, and  $(x,y)=(x_1,\ldots,x_k,y_1,\ldots,y_n$ be a local coordinate system at $z=(0,0)$  with projection $\pi$ defined by $(x_1,\ldots,x_k)$.  Let $\cI\subset \cO_{X}$ be an ideal sheaf of finite type , and suppose   the local ring $\cO_{X,z}/\cI_z$ is Cohen-Macaulay.
Assume that $\cO_X({x},0)/\cI(x,0)$ is a finite $K$-space, with $\cI(x,0)$ defining a finite diagram $\Delta$ in $\NN^n$. 

Then there exist \'etale neighborhoods 
$X'$ of $X$ and $Y'$ of $Y$ preserving the residue field $K$ (respectively open neighborhoods) with the induced projection $X'\to Y'$ and the Weierstrass 
isomorphism (surjection in the differential setting with kernel contained in $m_z^\infty$)
of  free $\cO_{Y'}$-modules 
$$
\bigoplus_{\alpha\in \Gamma} \cO_{Y'}\cdot x^\alpha\to \pi_*(\cO_{X'}/\cI), 
$$
and $\cO({Y'})$-modules 
$$
\bigoplus_{\alpha\in \Gamma} \cO(Y')\cdot x^\alpha\to \cO(X')/\cI, 
$$
\end{theorem}

\begin{proof} Put $\overline{y}:=\pi(z)\in Y$.
The sheaf $\cF:=\cO_X/\cI$ is  of finite type.
Moreover the vector space $\cF_z/(y_1,\ldots,y_k)=\cF_z/m_{\overline{y},Y}$ is of finite dimension.
By Theorems \ref{Mi2} (in the algebraic setting)
\ref{Mi1} (in the analytic situation), and  \ref{projection} (in the differential setting) there exist neighborhoods $X'\times Y'$ of  $X\times Y$ such that $\pi_*(\cO_{X'}/\cI)$ is coherent on $Y$ (or of finite type in the differential situation).

 Then in  the analytic case, by Theorem \ref{Rem4},  $\pi_*(\cO_{X'}/\cI)_z\simeq (\cO_{X'}/\cI)_z $ is a free $\cO_y$-module with  basis $x^\alpha$, $\alpha\in \Gamma$. This implies that  after possibly shrinking $Y'$, $\pi_*(\cO_{X'}/\cI)$ is a free $\cO_{Y'}$-module with the same basis.

 In the algebraic setting, by Theorem \ref{Rem4}, after passing to the Henselianizations of the stalks we get an isomorphism of  free $\cO^h_{Y',\overline{y}}$-modules
\begin{equation}\label{p}  
\bigoplus_{\alpha\in \Gamma} \cO^h_{Y',\overline{y}}\cdot x^\alpha\simeq \cO^h_{X',z}/(\cI\cdot \cO^h_{X',z}) .\end{equation} 

In a certain \'etale neighborhood 
$X'$ of $X$ 
the generators of the coherent modules can be expressed in terms of $x^\alpha$. In other words, by modifying $X'$ and $Y'$ we may assume additionally that $x^\alpha$, $\alpha\in \Gamma$, generate the coherent module $\pi_*(\cO_{X'}/\cI)$. This defines an epimorphism of sheaves
$$
\phi: \bigoplus_{\alpha\in \Gamma}\cO_{Y'}\cdot x^\alpha\to \pi_*(\cO_{X'}/\cI).
$$
The corresponding homomorphism of  stalks
 $$
\phi_z: \bigoplus_{\alpha\in \Gamma}\cO_{Y',\overline{y}}\cdot x^\alpha\to \pi_*(\cO_{X'}/\cI)_z \simeq (\cO_{X'}/\cI)_z= \cO_{X',z}/\cI_z
$$ 
is also injective as $\cO_{Y',\overline{y}}\to \cO^h_{Y',\overline{y}}$ is injective, and we get an injective homomorphism, in fact an isomorphism, of the Henselianizations as in  (\ref{p}). 
 
 Thus $\phi_z$ is an isomorphism of  stalks which defines  an isomorphism of  coherent free $\cO_Y$-modules in an open  Zariski neighborhood. 

 In the differential setting, by shrinking $X$ and $Y$  we can assume that $\phi$ is  a morphisms of sheaves of modules of finite type. 
 By Theorem \ref{Rem4}, it defines an epimorphism $\phi_z$ of stalks with local generators $x^\alpha$, $\alpha\in \Gamma$.
 Since $\pi_*(\cO_{X'}/\cI)$ is of finite type we can assume by further shrinking that they generate $\pi_*(\cO_{X'}/\cI)$, and $\phi$ is epimorphism. Since 
 $\phi_z$ is a quasi-isomorphism of stalks, and  its  kernel is contained in $m_z^\infty$.
 \end{proof}

\section{Marked ideals and  standard basis along Samuel stratum}
In the next chapters we give  some applications of the previous results to  desingularization and description of the Samuel stratum. The main goal of Chapter 6
is to study the  notion  (introduced here) of   standard basis along Samuel stratum (Definition \ref{Dweak2}). It allows one to describe and modify singularities controlled by the Hilbert-Samuel function.

\subsection{Resolution of marked ideals}
Recall the definition of Hilbert-Samuel function of an ideal sheaf $\cI$ at a closed point $x\in X$, where $X$ is a manifold or a smooth scheme is definded as
$$H_{x,\cI}(k)=\dim \cO_{X,x}/(m_{x,X}^{k+1}+\cI).$$
As before the order of $\cI$ at $x\in X$ is denoted by $\ord_x(\cI):=\max\{k\in\NN\mid \cI\subset m_{x,X}^k\}$

\begin{definition}(Hironaka (see \Hi,),  Bierstone-Milman (see \cite{BM1}),Villamayor (see \Vi))
A {\it marked ideal} (respectively an H-marked ideal)   is a collection $(X,{\cI},E,\mu)$, (respectively $(X, {\cI},E,H)$) where $X$ is a smooth scheme  of finite type over a field $K$ (or an analytic/differentiable manifold), ${\cI}$ is a  sheaf of ideals on $X$ of finite type,  $E$ is a totally ordered collection of  divisors with SNC,
whose  irreducible components are smooth pairwise disjoint and all have multiplicity one, and $\mu$ is a nonnegative integer (respectively $H$ is a function $H: \NN \to \NN$ with integral values). Moreover the irreducible components of divisors in $E$ have simultaneously
simple normal crossings.

A collection of marked ideals $\{(X,{\cI_i},E,\mu_i)\}$ will be called a {\it multiple marked ideal}.
Marked functions $(f,\mu)$ are pairs of regular functions on $X$ and $\mu\in \NN$.

\end{definition}
The functions $H$ can be identify with infinite sequence  of nonnegative integers ordered lexicographically.

\begin{definition}(Hironaka (\Hi, ), Bierstone-Milman (see \cite{BM1}),Villamayor (see \Vi))
By the {\it cosupport} (originally {\it singular locus}) of $(X,{\cI},E,\mu)$ we mean
$$\cosupp(X,{\cI},E,\mu):=\{x\in X\mid \ord_x(\cI)\geq \mu\}.$$
Similarly $$\cosupp\{(X,{\cI_i},E,\mu_i)\}:=\{x\in X\mid \ord_x(\cI_i)\geq \mu_i\}=\bigcap_i \cosupp(X,{\cI_i},E,\mu_i).$$

By the {\it cosupport}  of $(X,{\cI},E,H)$ we mean
$$\cosupp(X,{\cI},E,H):=\{x\in X\mid \cH_x(\cI)\geq H \}.$$

\end{definition}

\begin{remarks}
\begin{enumerate}
\item In most of the applications $\cosupp(\cI,H)$ coincides with so called Samuel stratum.
\item For any sheaf of ideals ${\cI}$ on $X$ we have
 $\cosupp({\cI},1)=\cosupp({\cI})$.
\item For any marked ideals $({\cI},\mu)$ on $X$,
 $\cosupp({\cI},\mu)$ is a closed subset of $X$ (Lemma \ref{le: Vi1}).
\end{enumerate}
\end{remarks}
Let $u_1,\ldots,u_n$ be a local coordinate system of  a smooth variety (or an analytic or differentiable manifold) $X$, and  $C\subset X$  be  a closed smooth subspace (submanifold) of $X$ defined by $u_1=\ldots=u_r=0$, with $r\leq n$. Denote by $\PP^{r-1}$ the projective space with homogenous coordinates $y_1,\ldots,y_r$
Recall that {\it the blow-up of $C$} is defined (locally on $X$) as the map $$X':=\{(x,y)\in X\times\PP^{r-1}\mid u_iy_j=u_jy_i\}\to X$$
Then there are open neighborhoods $U_i$ of $X'$ where $y_i\neq 0$ with coordinate system $u_j'=y_j/y_i=u_j/u_i$ for $j\neq i$, $j\leq r$ and $u_i'=u_i$ otherwise.

\begin{definition} The   blow-ups with the smooth centers $C\subset \cosupp(X,\cI,E,\mu)$ (respectively \\$C\subset \cosupp(X,\cI,E,H)$ or $C\subset \cosupp\{(X,{\cI_i},E,\mu_i)\}$ will be  called {\it admissible for} $(X,\cI,E,\mu)$, (respectively for $(X,\cI,E,H)$ or for $\{(X,{\cI_i},E,\mu_i)\})$ if the centers are contained in the cosupport of  marked ideals and have SNC with $E$. Likewise we called the centers {\it admissible } for the marked ideals. 
	
\end{definition}
\begin{lemma} \label{bb}\label{center} Let $C\subset \cosupp({\cI},\mu)$ be a smooth center of the blow-up $\sigma: X\leftarrow X'$ and let $D$ denote the exceptional divisor.  Let ${\cI}_C$ denote the sheaf of ideals defined by $C$. Then
\begin{enumerate}
\item $\cI \subset {\cI}_C^\mu$.
\item  $\sigma^*({\cI})\subset ({\cI}_D)^\mu$.
\end{enumerate}
\end{lemma}
\begin{proof} (1) We can assume that the ambient variety $X$ is affine. Let $u_1,\ldots,u_k$ be parameters generating ${\cI}_C$
Suppose $f \in \cI \setminus {\cI}_C^\mu$. Then we can write $f=\sum_{\alpha}c_\alpha u^\alpha$, where either $|\alpha|\geq \mu$ or  $|\alpha|< \mu$ and $c_\alpha\not\in {\cI}_C$.
By the assumption there is $\alpha$ with $|\alpha|< \mu$ such that $c_\alpha\not \in {\cI}_C$. Take  $\alpha$  with  the smallest $|\alpha|$. There is a point $x\in C$ for which $c_\alpha(x)\neq 0$ and in the Taylor expansion of $f$ at $x$ there is a term $c_\alpha(x) u^\alpha$. Thus $\ord_x(\cI)< \mu$. This contradicts to the assumption $C\subset \cosupp({\cI},\mu)$.

(2) $\sigma^*({\cI})\subset \sigma^*({\cI}_C)^\mu=({\cI}_D)^\mu$.
\end{proof}
\begin{definition} Let $\sigma: X'\to X$ be an admissible blow-up  for $(X,\cI,E,\mu)$ with the exceptional divisor $D$ then a marked ideal
$(X',\cI',E',\mu)=\sigma^c(X,\cI,E,\mu)$ is called  {\it the controlled transform} of $(X,\cI,E,\mu)$ if
\begin{enumerate}
\item ${\cI}'=\cI(D)^{-\mu}\sigma^*({\cI})$. 
\item $E'=\sigma^{\rm c}(E)\cup \{D\}$, where  $\sigma^{\rm c}(E)$ is the set of strict transforms of divisors in $E_{i-1}$.
\item The order  on $\sigma^{\rm c}(E)$ is defined by the order on $E$ while $D$ is the maximal element of $E$.	
\end{enumerate}

Similarly {\it the controlled transform} of $\{(X,{\cI_i},E,\mu_i)\}$ is given as the collection of  the controlled transforms of  $(X,\cI_i,E,\mu)$.

\end{definition}

\begin{definition} 
Let $\cI$ be any ideal sheaf of finite type 
 on a smooth variety $X$ over a field $K$ (or an analytic or differentiable manifold) , and let  $C\subset X$ be a smooth subvariety.
 Consider the blow-up
$\sigma: X'\to X$ at a smooth closed center $C\subset X$ contained in the Samuel stratum. By the \emph{strict transform} of $\cI$ we mean 
here the ideal generated locally by 
%
$(1/y^{c(f)})\sigma^*(f)$, 
where $y$ is a local equation of the exceptional divisor, and $c(f)$ is the maximal exponent for which $y^{c(f)}$ divides $\sigma^*(f)$.	

By the {\it the strict transform} of $(X,\cI,E,H)$ under a admissible  blow-up $\sigma: X'\to X$ we mean $H$-marked ideal $(X,\cI,E,H)$, where $\cI'$ is the strict transform  and $E'$ satisfies (2) and (3). 
\end{definition}

By 
   \cite[Proposition 3.13]{BM2} the  condition of the strict transform is equivalent in the analytic and algebraic setting to the following:

{\it The strict transform is generated by all the functions $f\in\cO(U')$  
for which  $y^kf$, for some $k$, is  in the ideal generated by
$\sigma^*({g})$, where $g\in \cO(U)$.}

\begin{definition}(Hironaka (see \Hi,),  Bierstone-Milman (see \cite{BM2}),Villamayor (see \Vi))
By a  {\it admissible sequence of blow-ups} of $(X, {\cI},E,\mu)$ (respectively $(X,{\cI},E,H)$) we mean
a sequence of blow-ups $\sigma_i:X_i\to X_{i-1}$ of  of smooth centers $C_{i-1}\subset
 X_{i-1}$,

$$
X_0=X\buildrel \sigma_1 \over\longleftarrow X_1
\buildrel \sigma_2 \over\longleftarrow X_2 \buildrel
\sigma_3 \over \longleftarrow\ldots
X_i\longleftarrow \ldots \buildrel \sigma_{r}  \over\longleftarrow X_r,$$

\noindent which defines a sequence of  marked ideals
$(X_i,{\cI}_i,E_i,\mu)$ (respectively $(X_i,{\cI}_i,E_i,H)$,  such that the centers $C_{i-1}$ are admissible for $(X_{i-1},{\cI}_{i-1},E_{i_1},\mu)$ (respectively for $(X_{i-1},{\cI}_{i-1},E_{i-1},H)$), and $(X_i,{\cI}_i,E_i,H)$ are controlled transforms of  $(X_{i-1},{\cI}_{i-1},E_{i_1},\mu)$ (respectively $(X_i,{\cI}_i,E_i,H)$ are the strict transforms of $(X_{i-1},{\cI}_{i-1},E_{i_1},H)$). 
If additionally 
$$\cosupp(X_r,{\cI}_r,E_r,\mu)=\emptyset$$ (resp. $\cosupp(X_r,{\cI}_r,E_r,H)=\emptyset$)
then we call the sequence a {\it resolution} of  $(X, {\cI},E,\mu)$.

\end{definition}
The definition of admissible sequence  and a resolution sequence applies also to multiple marked ideals .

\subsection{Ideals of derivatives}\label{Deri}
Ideals of derivatives were first introduced and studied in the resolution context by Giraud.

\begin{definition}(Giraud, Villamayor) Let ${\cI}$ be a  sheaf of ideals of finite type on a smooth variety $X$ (or an analytic/differentiable manifold). For any $i\in \NN$, by  the {\it i-th derivative} ${\cD}^i({\cI})$ of $\cI$ we mean the  sheaf of ideals generated by all functions
$f\in {\cI}$ with their (Hasse) derivatives of $D_{u^\alpha}=\frac{1}{\alpha!}\frac{\partial^{|\alpha|} f_j}{\partial u^\alpha}$ for all multi-indices $\alpha=(\alpha_1,\ldots,\alpha_n)$, where $|\alpha|:=\alpha_1+\ldots+\alpha_n\leq i$.

 If $({\cI},\mu)$ is a marked ideal and $i\leq \mu$ then we define
$${\cD}^i({\cI},\mu):=({\cD}^i({\cI}),\mu-i).$$
\end{definition}

\begin{lemma}(Giraud, Villamayor) \label{le: Vi1}
For any  $i\leq\mu-1$,  $$\cosupp({\cI},\mu)\subset\cosupp({\cD}^i({\cI}),\mu-i).$$
(with equality in characteristic zero).
In particular case $$\cosupp({\cI},\mu)=\cosupp({\cD}^{\mu-1}({\cI}),1)=V({\cD}^{\mu-1}({\cI}))$$ is a closed subspace of $X$.
 \qed
\end{lemma}
\begin{proof} If $\ord_x(\cI)\geq \mu$ then ${\cD}^i({\cI})\geq \mu-i$. If $\ord_x(\cI)<\mu$ then 
${\cD}^{\mu-1}({\cI})$ is invertible.
	
\end{proof}

\bigskip
We write $({\cI},\mu)\subset ({\cJ},\mu)$ if $\cI\subset {\cJ}$.
\begin{lemma}(Giraud,Villamayor) \label{le: inclusions} Let $({\cI},\mu)$ be a marked ideal
and $C\subset\cosupp ({\cI},\mu)$ be a smooth center and $r\leq \mu$.
Let $\sigma: X\leftarrow
X'$ be a blow-up at $C$. Then $$\sigma^{\rm c}({\cD}^r({\cI},\mu)) \subseteq {\cD}^r(\sigma^{\rm c}({\cI},\mu)).$$
\end{lemma}
\noindent{\bf Proof.}   See simple computations using chain rule in \cite{Vi2}, \cite{Wlodarczyk}.
\begin{lemma}\label{diff} Let $\phi$ be any \'etale morphism. Then $\phi^*({\cD}^a({\cI}))={\cD}^a(\phi^*({\cI}))$ for any $a\in \NN$.
 \end{lemma}

\subsection{Standard basis along Samuel stratum}

The implicit function theorem proven in the previous sections   allows  relaxing the condition of the Hironaka standard basis. This  idea was first used in the papers of Bierstone-Milman \cite{BM2}, \cite{BM22} and applied to  functions in  formal coordinate charts. Our construction is expressed in a different language which is closed to Hironaka's Henselian Theorem. 

Recall that the \emph{Samuel stratum} $S$ through a closed point $x\in X$ on a scheme (or an analytic or differentiable manifold) $X$ is a locally closed subset $S\subset X$ 
%
consisting of  all the closed points  $y\in X$ with the same Hilbert-Samuel function 
$H_{x,X}=H_{y,Y}$,

If $\cI$ is an ideal sheaf of finite type on a smooth  scheme (or an analytic/differentiable  manifold) $X$ then  Samuel stratum of $\cI$  on $X$ is a locally closed subset $S$ of $X$, such that $H_{x,\cI}=H_{y,\cI}$  for any two closed points $x, y \in S$.

By using the singular implicit function theorem we are going to construct a standard  basis of any ideal sheaf of finite type on $X$ along Samuel stratum.

Consider a monotone diagram $\Delta$ with vertices $\alpha_1,\ldots,\alpha_k$. 

\begin{lemma} \label{sub} 
If $\Delta$ is monotone in $\NN^n$ and $\NN^s\subset \NN^n$ is the smallest ``sublattice'' containing all the vertices $\alpha_1,\ldots,\alpha_k$ then the vertices span $\NN^s$. Moreover, if there is a vertex $\alpha_i$ with $s$-coordinate nonzero, then for any $s'\leq s$ there is a vertex $\alpha_j$ with $s'$-coordinate not zero and $|\alpha_j|\leq |\alpha_i|$.
\end{lemma}

\begin{proof}
  Let $\alpha_i=(\alpha_{i1},\ldots,\alpha_{is})$ be a vertex with $\alpha_{is}\neq 0$. Such a vertex exists by the minimality of $\NN^s$. Then we need to show that for any $s'<s$ there exists a vertex with nonzero $s'$-coordinate.
  It follows by monotonicity that $\beta_{i}:=(\alpha_{i1},\ldots,\alpha_{i,s'-1}, \alpha_{is'}+\ldots+\alpha_{is},0,\ldots,0)\in \Delta.$ 
  Then $\beta_i=\alpha_j+\gamma$ for some vertex $\alpha_j$ with $|\alpha_j|<|\alpha_i|$ and $\gamma\in\NN^n$. If $\alpha_{js'}=0$ it follows that $\alpha_j$ has all coordinates not greater than the  vertex $\alpha_i$. This implies that $\alpha_i=\alpha_j+\gamma'$, where $\gamma'\in \NN^n$, which contradicts the definition of vertex. 
\end{proof}

 Let $u_1,
\ldots, u_n$ be a coordinate system on a smooth scheme (or a manifold) $X$. 
Consider now a monotone diagram $\Delta\subset \NN^s$ , for $s\leq n$ with vertices $\alpha_1,\ldots,\alpha_k$ corresponding  to the functions $f_1,\ldots,f_k$ at a point $x\in X$, and suppose each $f_i$ has the form  $f_i=u^{\alpha_i}+r_i$ (in other words, the monomial $u^{\alpha_i}$ occurs in the expansion of $f_i$ at $x$).

Assume now that $\alpha_1,\ldots,\alpha_k$ span $\NN^s$. For any coordinate $u_i$ consider a vertex $\alpha_{j(i)}$ 
%
with $i$-th coordinate $\alpha_{j(i)i}$ nonzero, and suppose $\alpha_{j(i)i}=p^{k_i}b_i$, where $p$ is the characteristic of $K$, and $p$ does not divide $b_i$. Set 
$$
a_i:=\left\{
	\begin{array}{ll}
		p^{k_i}  & \mbox{if } {\rm char}(K)=p, \\
		1 & \mbox{if } {\rm char}(K)=0,
	\end{array}
\right.  \quad \quad \beta_{i}:=\alpha_{j(i)}-a_ie_i.
$$  
Then $D_{u^{\beta_{i}}}f_{j(i)}$ has the form  $D_{u^{\beta_{i}}}f_{j(i)}=u^{a_i}_i+ \text{other  terms}$.

For a given coordinate system $u_1,
\ldots, u_n$ on a smooth scheme over a field $K$ or on a manifold and a  sequence of natural numbers $\overline{a}=(a_1,\ldots, a_s)$ one can introduce the resultant Jacobian differential operator
which plays the role of the ``main Jacobian'' in the Cohen-Macaulay case: 
%
$$
JR^{\overline{a}}(f_1,\ldots,f_s)(x)
:=\Res\Big(\sum_{\alpha\in \NN^s, |\alpha|=a_1} D_{u^\alpha}(f_i)(x)X^\alpha,\ldots,\sum_{\alpha\in \NN^s, |\alpha|=a_s} D_{u^\alpha}(f_i)(x)\cdot X^\alpha\Big).
$$
(Here $X:=(X_1,\ldots,X_s)$ are formal unknowns, and $\Res$ denotes the resultant as in the previous section. We shall often skip the index ${\overline{a}}$ in $JR^{\overline{a}}=JR$)
	
\begin{example} 
If $\overline{a}=(1,\ldots,1)$ then $JR^{\overline{a}}(f_1,\ldots,f_s)$ is the usual Jacobian determinant.	
\end{example}


\begin{definition} \label{Dweak2} 
Let $x\in X$ be a point on  a smooth scheme of dimension $n$ over a field $K$ (respectively a complex or real analytic or  differentiable manifold) with a coordinate system 
$u_1,
\ldots,u_k $.
Let $\cI$ be an ideal of finite type on $X$.
 Let $\Delta\subset \NN^s\subset\NN^n$ be a monotone diagram  with vertices $\alpha_1,\ldots,\alpha_k$ ordered reverse-lexicographically, which span $\NN^s$. 

A set of functions $f_1,\ldots, f_k\in \cI(U)$ on an \'etale neighborhood $U$ of $x$ will be called a  \emph{ standard basis of $\cI$ at $x$ with respect to $\Delta$} if:
\begin{enumerate}
 \item $H_x(\cI)=H(\Delta\times \NN^{n-s})$. 
  \item $\ord_x(f_i)=|\alpha_i|$. 
 \item $\supd(f_i)\subset \{\alpha_i\}\cup (\Gamma\times \NN^{n-s})$ and $D_{u^\alpha}(f)\equiv 1$ in a neighborhood of $x$.
\item   $J^s(f_1,\ldots,f_k: u^{\alpha_1},\ldots,u^{\alpha_k})(x)\neq 0$ for all  $s\leq d(\Delta)+1$.
\item $JR^{\overline{a}}(D_{u^{\beta_1}}f_{j_1},\ldots,D_{u^{\beta_k}}f_{j_k})(x)\neq 0$.
\end{enumerate}
We shall call $f_1,\ldots, f_k\in \cI(U)$ a \emph{ standard basis of $\cI$ on  $U$} along Samuel stratum $S$ if it is a standard basis at any (closed) point $y\in S$ with Hilbert-Samuel function.
$H_y(\cI)=H(\Delta\times \NN^{n-s})$.

We shall call a coordinate system for which the above conditions hold \emph{compatible} with the standard basis of $\cI$.

In the algebraic situation we call   functions $f_1',\ldots,f_k'$ on an open Zariski neighborhood $V\subset X$ of $x\in X$ a \emph{ standard pre-basis of $\cI$ at $x$} if there is an \'etale neighborhood $U\to V$ of $x$ and invertible functions $c_1,\ldots,c_k$ on $U$ such that 
$$
f_1:=c_1\cdot \sigma^*(f'_1),\ldots,f_k:=c_k\cdot\sigma^*(f'_k)
$$ 
is a  standard basis of $\cI$ at $x$.
\end{definition}

\begin{remark} 
The notion of  standard pre-basis satisfies all the conditions 
of a standard basis 
except for $(3)$ which is replaced with a weaker form. 

It can be easily shown that the standard basis in the above sense determines at any $x\in S$ a (formal analytic) standard basis relative to a diagram of $\widehat{\cI}_x\subset\widehat{\cO_{X,x}}=K[[u_1,\ldots,u_n]]$ in the sense of Bierstone-Milman \cite{BM2},\cite{BM22}. On the other hand the construction is conceived in the language which is related to Hironaka's Henselian approach (\cite{Hir2}). \end{remark}

\begin{example} \label{Weak3}
 If $C$ is  a smooth center on $X$ of codimension $s$, then we consider the diagram $\Delta\subset \NN^s$ generated by the standard basis $e_1,\ldots,e_s$. Then $C$ coincides with the Samuel stratum for $\cI_C$. The    standard basis of $\cI_C$ along $S=C$ with respect $\Delta$ is a set of generators
of the form $f_1,\ldots,f_s\in \cI_C$, where $$
f_i=u_i+h_i(u_{s+1},\ldots,u_n)$$ by condition (3) . In this case  conditions (4) and (5) and are (essentially) equivalent to $\det[\frac{\partial{f_i}}{u_j}(x)]_{i,j=1,\ldots,s}\neq 0.$ (See also Example \ref{Jacobian})
\end{example}

Condition (1) of Definition \ref{Dweak2} together with other conditions implies the following
\begin{lemma}  
There is a natural isomorphism of vector spaces over a field $K$   %
$$
\overline{r}: K[x_1,\ldots,x_n]^{\Gamma\times\NN^{n-s}}=\{f\in K[x_1,\ldots,x_n]\mid \supp(f)\subseteq \Gamma\times \NN^{n-s}\}\to  K[x_1,\ldots,x_n]/\inn_x{\cI_x},
$$
induced  by inclusion 
$K[x_1,\ldots,x_n]^{\Gamma\times\NN^{n-s}}\subset K[x_1,\ldots,x_n]$.

\end{lemma}
\begin{proof} By Lemma \ref{main012}, there is a basis of $K[x_1,\ldots,x_n]$ of the form
$$ \{\inn_x(f_\beta)\mid  \beta\in \Delta\times \NN^{n-s}\}
 \cup 
\{u^\alpha\mid  \alpha \in  \Gamma\times \NN^{n-s}\}
$$ Since $\{\inn_x(f_\beta)\mid  \beta\in \Delta\times \NN^{n-s}\}$ is contained in $\inn_x(\cI)$ we get that $\overline{r}$ is an epimorphism of vector spaces preserving the degrees of polynomials. But since $H_x(\cI)=H(\Delta\times \NN^{n-s})$ we see that $\overline{r}$ defines an isomorphism in each gradation and thus it is an isomorphism.

\end{proof}

\begin{corollary}\label{main013} The subset $$\cI_x^{\Gamma\times \NN^{n-s}}:=\{f\in \cI_x\mid \supd(f)\subseteq \Gamma\times \NN^{n-s}\}\subset \cO_{X,x}$$ is zero  in algebraic and analytic setting and is contained in $m_{x,X}^\infty$ in the differential setting.

\end{corollary} \begin{proof}
 Assume that  $f\in \cI_x^{\Gamma\times \NN^{n-s}}$ then $\inn_x(f)\in \inn_x(I_x)$ and $\supd(\inn_x(f)	)\subset {\Gamma\times \NN^{n-s}}$. By the previous Lemma, $\inn_x(f)=0$. Then the germ  $f_x=0$ in algebraic and analytic setting and it 
 is flat in the differential setting.
 \end{proof}
 
 \begin{corollary}\label{main014} Given a coordinate system and a monotone diagram $\Delta$. The standard basis (with respect to $\Delta$ and the coordinate system) of ideal sheaf is unique if exists in algebraic and analytic setting, and it is unique up to flat functions $m_{x,X}^\infty$ at each point $x$ in the Samuel stratum.
 	
 \end{corollary}

\begin{proof} If $(f_i)$ and $(f'_i)$ are two standard bases then by condition (3), $f_i-f'_i\in \cI^{\Gamma\times \NN^{n-s}}$ and we can  use the previous corollary.
	
\end{proof}

It follows from Lemma \ref{le: Vi1} that Condition $(2)$ of Definition \ref{Dweak2} is equivalent to
\begin{enumerate}
\item[$(2')$] $D_{u^\alpha}(f_i(x))=0$, $i=1,\ldots,k$, $|\alpha|<|\alpha_i|$.
\end{enumerate}

Condition (3) of Definition \ref{Dweak2}  is coherent in the sense that it is satisfied in a neighborhood. 

\begin{lemma}  Condition $(3)$ is equivalent to each of the following:
\begin{enumerate}
 \item[\rm(3A)]  ${\supd}(f_i-u^{\alpha_i})\subset \Gamma\times \NN^{n-s}$.
 \item[\rm(3B)] $D_{u^{\alpha_i}}(f_i)\equiv 1$ and $D_{u^{\alpha}}(f_i)\equiv 0$ for $\alpha\in \Delta\setminus \{\alpha_i\}$.
\item[\rm(3C)] $f_i=u^{\alpha_i}+r(f_i)$, where $\supp(r(f_i))\in \Gamma\times \NN^{n-s}$ (with respect to a coordinate system vanishing at $x$).
\end{enumerate}
\end{lemma}

\begin{proof} 
This follows from the relation between the support and differential support  in Lemma \ref{support}. Observe that  the set $\NN^n\setminus (\{\alpha\}\cup \Gamma)$ is $\NN^n$-invariant and   condition (3) of Lemma \ref{support} is satisfied at $x$. 
\end{proof}

Condition (4) of Definition \ref{Dweak2}  is coherent 
and  can be stated in the form 
$$
J^s(f_1,\ldots,f_k: u^{\alpha_1},\ldots,u^{\alpha_k})(x)
%
=\det\big[D_{u^{\alpha-\beta+\alpha_{i(\beta)}}}{(f_{i(\beta)})}\big]_{\alpha,\beta\in \Delta(s)}\neq 0
$$ 
for all  $s\leq d(\Delta)+1$. It follows from Lemma \ref{easy} that
this condition is satisfied for any standard basis at a
point $x$.

Condition (5) of Definition \ref{Dweak2}  is coherent. It implies that the variables $u_1,\ldots,u_s$ are \emph{essential} for the presentation of the initial forms
%
$F_i=\inn_x(f_i(u_1,\ldots,u_s,0\ldots,0))$
 in the sense that there
are no translations $u_i+a_it$, where $a_i\in K$, preserving $F_i$. In other words, there is no linear coordinate change $u'_1,\ldots,u'_s$ such that all $F_i(u'_1,\ldots,u'_s)=F_i(u'_1,\ldots,u'_{s-1},0)$
depend only on $s-1$ coordinates.
Suppose such a translation exists for $F_i$. Then the forms $D_{u^{\beta_{i}}}F_{j(i)}$
also depend on $u'_1,\ldots,u'_{s-1}$. In particular they have common nontrivial zeroes, and thus their resultant
$\Res(D_{u^{\beta_{j(i)}}}F_i)=JR^{\overline{a}}(D_{u^{\beta_{i}}}f_i)(x)$ is zero.

On the other hand, this condition is automatically satisfied for the standard basis with $\exp(f_i)=\alpha_i$.
In this case $\exp(D_{u^{\beta_{i}}F_i})=a_ie_i$,  which implies that 
%
$K[u_1,\ldots,u_s]/(D_{u^{\beta_{i}}F_i})$ 
is finite as it corresponds to a finite diagram contained in $\Gamma=\prod [0,a_i-1]$. Consequently, the forms have no nontrivial zeroes and their resultant $\Res(D_{u^{\beta_{j(i)}}}F_i)=JR(D_{u^{\beta_{i}}}f_i)(x)$ does not vanish.
 
 This also implies 
 
\begin{lemma} \label{stan} 
The standard basis of $\cI\subset \cE_n$ with respect to a monotone order is a  standard basis with respect to the induced diagram.\qed
\end{lemma}
\begin{remark} Note, however that the monotone order defines different diagrams along Samuel stratum.
	
\end{remark}



\subsection{Description of the Samuel stratum}
Let $\cI$ be  a sheaf of ideals of finite type on a smooth scheme $X$ over $K$, or a complex analytic or  differentiable manifold.   Let $\cI^\Gamma$ denote the subsheaf generated by the functions $f\in \cI$  with  $\supd(f)\subset \Gamma$.
Then, by section \ref{Deri},  the set $\cosupp(\cI^\Gamma,\infty)$ is a closed
subspace  described as the vanishing locus 
$V(\cD^\infty(\cI^\Gamma))$.

The  standard basis (and the standard pre-basis) along Samuel stratum  is a coherent notion, and gives a local description of the Samuel stratum.

\begin{theorem}[Existence of a weak standard basis along the Samuel stratum] \label{weak}
   Let $X$ be a smooth scheme  of finite type over a field $K$ (or an analytic/differentiable manifold) with a given coordinate system, and let $x\in X$ be a closed point.  Let $\cI$ be a sheaf of ideals  of finite type on $X$.   
   There  is  an  \'etale (respectively open) neighborhood $U\subset X$ of $x$  and  regular functions $f_1,\ldots,f_k\in \cI(U)$  which form
a  standard basis of $\cI$ along the Samuel stratum through $x$ with respect to a monotone diagram $\Delta$
     and a certain coordinate system $u_1,\ldots,u_n$ on $U$. Moreover:
\begin{enumerate} 
\item The Samuel stratum through $x\in U$  can be described as 
$$
S=S_x=\{y\in U\mid H_{y,\cI}=H_{x,\cI}\}=
\{y\in U\mid \ord_y(f_i)=|\alpha_i| \}.
$$
In the differential setting,
%
$$
S=S_x=
\{y\in U\mid \ord_y(f_i)=|\alpha_i| \} \cap  \cosupp(\cI^\Gamma,\infty).
$$
\item $H(\Delta\times \NN^{n-s})=\max\{H_y(\cI)\mid y\in U\}$ is the maximum value of the Hilbert-Samuel function on $U$.

\item In the algebraic and complex analytic setting any function $f\in \cI(U)$ can be uniquely written as $f=\sum h_if_i$ where $h_i\in \cO(U)^{\Gamma_i}=\{f\in\cO(U)\mid \supd(f)\subset \Gamma_i\}$.
\begin{itemize}
\item In the differential setting  $f=\sum h_if_i+r(f)$, where $r(f)\in m_S^\infty\cap \cI^\Gamma$. 

\item In the real analytic case there is a presentation $f=\sum h_if_i$ in a neighborhood $U_f$ of $S$ (depending on $f$).
\end{itemize}  
\end{enumerate}
   \end{theorem}

\begin{proof} 
Consider a monotone order $\overline{T}$ on $\NN^n$
and an \'etale neighborhood, possibly extending the residue field, (respectively an open neighborhood) $U$ of $x$ for which there exists a generic coordinate system $u_1,\ldots,u_n$ 
defining a monotone diagram $\Delta=\exp_x(\cI)$ at $x$ and a standard basis $f_1,\ldots, f_k$ of $\cI_x$ with respect to $\overline{T}$ (see Theorems \ref{etale}, and \ref{standard basis 3})). By Theorem \ref{standard basis 3}, and \ref{standard basis 2}, and Lemma \ref{stan} the functions $f_1,\ldots,f_k$ generate $\cI_x$ satisfy the conditions of Definition \ref{Dweak2} at $x$.
Since $\cI$ is of finite type, by shrinking $U$ if necessary we can 
assume that:
\begin{itemize}
\item	$\cI$ is generated by $f_1,\ldots, f_k\in \cI(U)$ on $U$.
\item $f_1,\ldots,f_k$ satisfy the differential conditions (3) through (5) on $U$ of Definition \ref{Dweak2}. 
\item In the algebraic, complex analytic and differential situation there exists Weierstrass-Hironaka division on $U$ by $f_1,\ldots, f_k\in \cI(U)$. 
\end{itemize}

  Let 
$$
S_0:=\{y\in U\mid \ord_y(f_i)=|\alpha_i| \}.
$$
 In the real analytic case we also assume that $U$ contains  a single connected component of $S_0$.
Let $y$ be a closed point in $X$. 
If $y\not \in S_0$ then  we can find a largest integer $d$  such that $\ord_y(f_j)=|\alpha_j|$ for all $f_j$ with $\ord_y(f_j)<d$.

Then  $d=\ord_y(f_i)<|\alpha_i|$ for a certain  $f_i$.
Consider a coordinate system $\overline{u}_1,\ldots,\overline{u}_n$, where $\overline{u}_i: =u_i-u_i(y)$, vanishing at  $y$ induced by $u_1,\ldots,u_n$ (possibly after passing to an \'etale neighborhood, and extending the base field).
It follows from Theorem \ref{main012}
 that a basis of $\gr(\cO_{X_y})/m^d_{y}\simeq \cO_{X_y}/m^d_{y}$ is given by 
$$
\Psi^d_{y}:=\{\overline{u}^\alpha\mid \alpha \in \Gamma, |\alpha|<d \} \cup \{\overline{u}^\alpha \inn_y(f_i)\mid \alpha \in \Gamma_i, |\alpha|+|\alpha_i|<d \}
$$
 with the subsets $\Psi^d_{y,0}:=\{\overline{u}^\alpha\mid \alpha \in \Gamma, |\alpha|<d \}$ and $\Psi^d_{y,1}:=\{\overline{u}^\alpha \inn_y(f_i)\mid \alpha \in \Gamma_i, |\alpha|+|\alpha_i|<d \}$ being in bijective correspondence with $\Gamma\times\NN^{n-s}$ and $\Delta\times\NN^{n-s}$.
Since 
%
$\supd(f_i)\subset \{\alpha_i\} \cup (\Gamma\times \NN^{n-s}) $ and $\ord_y(f_i)<d$, we conclude that
 $\supd(\inn_y(f_i))$ is contained in  $\Gamma\times \NN^{n-s}$ and thus it is linearly independent of  $\Psi^d_{y,1}$,
  which implies that $H_y(\cI)<H(\Delta\times \NN^{n-s})$ with respect to the lexicographic order. 

If $y\in S_0$  but $y \not\in \cosupp(\cI^\Gamma,\infty)$ (in the differential setting),  then there is  a nonzero $g\in \cI^\Gamma$, of a certain order $e$, with $\supd(\inn_y(g))\subset \Gamma\times\NN^{n-s}$. Then $\inn_y(g)$ is  not in the vector space $\Psi^{e+1}_y$.
 This implies that $H_{y,\inn_y(\cI)}= H_{y,\cI}<H({\Delta})$  (proving condition (2)). 

If $y\in S_0$ (or $y\in S_0\cap \cosupp(\cI^\Gamma,\infty)$ in the differential setting) then in the algebraic and the complex analytic setting $\cI$ is coherent, and for  any $g\in \cI$ (also in the differential setting)  there is Weierstrass-Hironaka division by $f_1,\ldots,f_k$ on $U$, yielding $g=\sum h_if_i+h_0$, where $\supd(h_0)\subset \Gamma$ (by Theorem \ref{main01}), and thus, by Corollary \ref{main014}, $h_0\in \cI^\Gamma$ is flat at $y$ or $h_0\equiv 0$  on $U$ in the algebraic/analytic case (as  $(h_0)_x\equiv 0$). We conclude that $\inn_y(g)=\sum H_i\inn_y(f_i)$. 

In the real analytic situation
the division exists at any point of $y\in S_0$. Using uniqueness of the extension, we can define the  functions $h_i$ in the neighborhood of $S_0$. Since $h_0$ is zero in a  neighborhood of $x\in S_0$, it is  zero in a neighborhood of the  connected component of $S_0$ through $x$, and  again $g=\sum h_if_i$ with $\inn_y(g)=\sum H_i \inn_y(f_i)$.

This shows that the functions in $\Psi^{\infty}_y$ form  a basis of  $\inn_y(\cI)$. And since they are in bijective correspondence with the elements of  $\Delta\times \NN^{n-s}$, we conclude that  
 $H_{y,\cI}=H(\Delta\times \NN^{n-i})$. 
\end{proof}

\begin{remark} 
The theorem implies existence of  a standard pre-basis of $\cI$ on a Zariski open neighborhood $V$ of $x$, as the functions $f_i$ on $U$ define principal divisors whose images determine a pre-basis  on $V$.
\end{remark}

\begin{lemma} \label{center2} 
Let $\cI$ be a sheaf of ideals on $X$, and 
 $C\subset S$ be a smooth center contained in the Samuel stratum $S$ of dimension $c$. Consider a  standard basis $f_1,\ldots,f_k$ in a neighborhood of $x\in S$ defined for a monotone diagram $\Delta\subset \NN^s\subset\NN^n$, with vertices spanning $\NN^s$. Then 
 there exists a  coordinate system 
$v_1,
\ldots,v_n$ which is compatible with $f_1,\ldots,f_k$, and such that $v_1,\ldots,v_{c}$ with $s\leq c\leq n$ describe the ideal $\cI_C$ in a neighborhood of $x\in X$. 
\end{lemma}

\begin{proof} Denote by 
$u_1,
\ldots,u_n$  a given  coordinate system compatible with the standard basis $f_1,\ldots,f_k$. Consider any  monotone order $T$ on $\NN^n$.
Let $v_1,\ldots,v_c$ be  the standard basis of the ideal $\cI_C$ at $x$ defined for $T$. 
Then $v_1,\ldots,v_c$ is a set of parameters describing $\cI_C$, and $\exp(\cI)$ is a diagram spanned by some vectors $e_{j_1}=\exp(v_1),\ldots,e_{j_c}=\exp(v_k)$  which is part of the basis
of $\NN^n$ with $j_1<\ldots<j_c$. That is, we can express $v_i$ as 
$$v
_i=u_{j_1}+r_i(u_{j_1+1},\ldots,u_n).
$$
Denote by $\mu_i=\ord_x(f_i)$ the multiplicity of $f_i$. Then $f_i\in \cI_C^{\mu_i}$, and the initial form 
$$
F=\inn_x(f_i)=\sum c_\alpha v^\alpha
$$ 
at $x$ is a function of $v_1,\ldots,v_k$, where $v_i=v_i(u_1,\ldots,u_n)$. 
Since the coordinates $u_1,\ldots,u_s$ are essential, $$
F(u_1,\ldots,u_s,0,\ldots,0)=\inn_x(f_i(u_1,\ldots,u_s,0,\ldots,0))=\sum c_\alpha v^\alpha(u_1,\ldots,u_s,0,\ldots,0)
$$ 
is a function depending upon the $s$ variables  $u_1,\ldots,u_s$ with no translations. This implies that  $(j_1,\ldots,j_s)=(1,\ldots,s)$,
and since $v_1,\ldots,v_c$ is the standard basis of $\cI_C$, it can be written in the form 
$v_i=u_i+f(u_{s+1},\ldots,u_n)$ for $i=1,\ldots,s$, $v_i=v_i(u_{s+1},\ldots,u_n)$ for $i=s+1,\ldots,c$.
This set can be extended to a coordinate system 
$v_1,
\ldots,v_n$ 
with  $v_i=v_i(u_{s+1},\ldots,u_n)$ for $i=s+1,\ldots,n$.
The coordinate change $u_i\mapsto v_i$ does not change the derivatives $D_{u_i}=D_{v_i}$ for $i=1,\ldots,s$, defining the differential support and used in  conditions (4) and (5) of  a standard basis. Consequently, the conditions for  a standard basis are satisfied for the coordinate system 
$v_1,
\ldots,v_n$.
\end{proof}

We shall refer to such a coordinate system 
$v_1,
\ldots,v_n$ as \emph{compatible} with the center $C$ and the weak standard basis $f_1,\ldots,f_k$.

Let $C\subset X$ be any smooth center on a smooth scheme (or an analytic or differentiable manifold $X$). Suppose $u_1,\ldots,u_c$ are the coordinates describing $C$ on a neighborhood $U$ of $X$. Consider a graded locally free $\cO_C$-sheaf 
$$
\gr_C(\cO_X)=\bigoplus \cI_C^s/\cI_C^{s+1}
$$
over $C$  which can be described over $U\cap C$ as
$$
\gr_C(\cO_U)\simeq \cO_{C\cap U}[u_1,\ldots,u_c].
$$
 Then any ideal $\cI$ defines the ideal of the initial forms 
$$
\inn_C(\cI):=\bigoplus (\cI\cap\cI_C^s)/(\cI\cap\cI_C^{s+1})= \bigoplus ((\cI\cap \cI_C^s)+\cI_C^{s+1})/\cI_C^{s+1} \subset \gr_C(\cO).
$$ 

\begin{theorem}[Hironaka's normal flatness theorem]\label{Flat}
Let $X$ be a smooth  scheme of finite type over a field $K$ (or  an analytic or differentiable  manifold).
The following conditions are equivalent for a smooth subvariety (submanifold)
$C\subset X$:
\begin{enumerate}
\item $H_{x,\cI}$ is locally constant along $C$ (defined for closed points).
\item $\gr_C(O_X)/\inn_C(\cI)$ is a locally free $\cO_C$-module.
\end{enumerate}

Moreover,  if $\{f_1,\ldots,f_k\}$ is a  standard basis of $\cI$ in a neighborhood of $x\in C$ and 
$u_1,
\ldots,u_n$ is a compatible coordinate system   with the  standard basis and the center so that $\cI_C=(u_1,\ldots,u_c)$ for  $c\geq s$,  then the set  
$$
\{u^\alpha \mid \alpha\in \Gamma\times \NN^{c-s}\}
$$ 
is a basis of the free $\cO_C$-module $\gr_C(\cO_X/\cI_C)$  on $U\cap C$ for some \'etale (respectively open) neighborhood $U$ of $x\in X$. On the other hand,  the set 
$$
\{u^\alpha\inn_C(f_i) \mid \alpha\in \Gamma_i\times \NN^{c-s}\}
$$ is a basis of the free $\cO_C$-module $\inn_C(\cI)$ on $U\cap C$.
\end{theorem}

\begin{proof}
$(1)\Rightarrow (2)$. Let $C\subset S$ be any smooth center contained in the Samuel stratum on  $X$, and let $x\in C$. Let $f_1,\ldots, f_k$ be a  standard basis of $\cI$ in a neighborhood of $x$ corresponding to a certain monotone diagram $\Delta\in \NN^{s}$ (it exists by Theorem \ref{weak}).

By  Lemma \ref{center2}, we can find a coordinate system    $u:=(u_1,\ldots,u_n)$  compatible with  $f_1,\ldots, f_k$ and  such that $(u_1,\ldots,u_c)$   define (locally) the ideal $I_C$ of  $C$. Denote by $v:=(u_{c+1},\ldots,u_n)$ the remaining coordinates.  By  Theorem \ref{weak},  each $f_i$ has a constant  multiplicity $\mu_i$ along $S$ and thus along $C$. Consequently,  by Lemma \ref{center},   $f_i\in \cI^{\mu_i}_C$. Write $\inn_C(f_i)=F_i(u_1,\ldots,u_c)\in \cO_C(v)[u_1,\ldots,u_c]$ as the form of degree $\mu_i$ with coefficients in $\cO_C(v)$.
Thus  $\inn_x(f_i)=\inn_C(f_i)(x)$,  where  $\inn_C(f_i)(x)=F_i(v(x))(u_1,\ldots,u_c)$  denotes  the evaluation  of the $\cO_C$-form $\inn_C(f_i)=F_i(u_1,\ldots,u_c)$ at $x\in C$. 

In particular, $\inn_x(f_i)\in K[u_1,\ldots,u_c]$,
and  $\{u^\alpha\mid \alpha\in \Gamma\times \NN^{c-s}\}\cup \{u^\alpha F_i(v(x))\mid \alpha\in \Gamma_i\times \NN^{c-s}\}$  is a basis of the evaluation $K[u_1,\ldots,u_c]$ of $\gr_C(\cO_X))=\cO_C[u_1,\ldots,u_c]$ at $x\in C$.

By Lemma \ref{open}, the set $\{u^\alpha\mid \alpha\in \Gamma\times \NN^{c-s}\}\cup \{u^\alpha F_i\mid \alpha\in \Gamma_i\times \NN^{c-s}\}$  is a basis over  $\cO_{C\cap U}$ of $\gr_C(\cO_U))=\cO_{C\cap U}[u_1,\ldots,u_c]$   in an open neighborhood $U$ of $x\in X$. Then for any $F\in \inn_C(\cI_{|U})$ we can write uniquely 
$$
F=\sum H_i\inn_C(f_i)+H_0,
$$ 
where the functions $H_i$ are homogeneous in $\gr_C(\cO_U))$, and $\supd(H_i)\subset \Gamma_i\times \NN^{c-s}$. Since $\inn_C(F_i)\in\inn_C(\cI)$, the set $\{u^\alpha \mid \alpha\in \Gamma\times \NN^{c-s}\}$ generates $\cO_{C\cap U}[u_1,\ldots,u_c]/\inn_C(\cI)$, and there is an epimorphism 
$$
\phi: \bigoplus_{\alpha\in \Gamma\times \NN^{c-s}} \cO_{C\cap U}\cdot u^\alpha\to\cO_{C\cap U}[u_1,\ldots,u_c]/\inn_C(\cI).
$$
The kernel of $\phi$  consists of 
%
all
$\cO_C$-forms  $F$ with  $\supd(F) \subset \Gamma\times \NN^{c-s}$.
 If  there is a nonzero $\cO_C$-form $F=\inn_C(f)\in\inn_C(\cI)$ with $\supd(F)\subset \Gamma\times \NN^{c-s}$ then  its evaluation $F(y)$ at  some closed point $y\in C\cap U$ is not zero. But this implies that $\supd(\inn_y(f))\subset \Gamma$ and is linearly independent of 
 $$
\{u^\alpha\cdot  \inn_y(f_i)(u) \mid \alpha\in \Gamma_i\times \NN^{n-k}\},
$$ 
and $H_y(\cI)<H(\Delta\times\NN^{n-k})=H_x(\cI)$, 
 which contradicts the assumption. 
 
  Consequently, the kernel of $\phi$ is trivial, and $\phi$ is an isomorphism. This also implies that any $F\in \inn_C(\cI)$ can be uniquely written as
$F=\sum H_i\inn_C(f_i)$, where $\supd(H_i)\subset \Gamma_i\times \NN^{c-s}$. This proves the implication $(1)\Rightarrow (2)$  and the ``moreover'' part of the theorem. 

$(2)\Rightarrow (1)$. Observe that  $\gr_C(\cO_X)=\cO_c[u_1,\ldots,u_c]$  is a graded $\cO_c$-module (with the standard grading), and its evaluation at $x\in C$ is  a graded $K$-module $K[u_1,\ldots,u_c]$ for a base field $K$. 
Let $\gr_C(\cI)_x\subset K[u_1,\ldots,u_c]$ be the evaluation of the ideal $\gr_C(\cI)$ at $x\in C$.
Consider  a monotone monomial order $ \overline{T}$ such that after a generic change of coordinates the diagram 
%
$\Delta:=\expe_{\overline{T}}(\gr_C(\cI)(x))$
  is  monotone.

Let $f_1,\ldots,f_k \in \inn_C(\cI)$ be the homogeneous polynomials over $\cO_C$ such that for their evaluations  $\overline{f_1},\ldots,\overline{f_k}$  at $x$, the initial exponents  $\expe_{\overline{T}}(\overline{f_i})=\alpha_i$  are 
the 
vertices of $\Delta$. Then the corresponding set  
$$
\{u^\alpha\mid \alpha\in \Gamma\times \NN^{c-s}\}\cup \{u^\alpha \overline{f_i}(x)\mid \alpha\in \Gamma_i\times \NN^{c-s}\}
$$
is a  basis of $K[u_1,\ldots,u_c]$.  Thus, by Lemma \ref{open},
$$
\{u^\alpha\mid \alpha\in \Gamma\times \NN^{c-s}\}\cup \{u^\alpha {f_i}(x)\mid \alpha\in \Gamma_i\times \NN^{c-s}\}
$$ 
is a basis of  the free $\cO_{C\cap U}$-module $\cO_{C\cap U}[u_1,\ldots,u_c]$.
On the other hand, the module $\cO_C[u_1,\ldots,u_c]/\inn_C(\cI)$ is locally free (so we may assume it is free on $U$).
Moreover, as before there is an epimorphism 
$$
\phi: \bigoplus_{\alpha\in \Gamma\times \NN^{c-s}} \cO_{C\cap U}\cdot u^\alpha\to\cO_{C\cap U}[u_1,\ldots,u_c]/\inn_C(\cI),
$$
which is an isomorphism after evaluating at $x$, which means that both the free $\cO_C$-modules have the same rank in each grading. Thus the epimorphism $\phi$ is an isomorphism. Then as before each element in $\gr_C(\cI)$ can be written as 
$f=\sum h_if_i$, where $\supp(h_i)\subset \Gamma_i$. This implies that each function in $\cI$ can be written as $f\simeq \sum h_if_i\ ({\rm mod}\ \cI_C^s)$ up to a power $\cI_C^s$ for any $s\gg 0$.
Consequently, the initial form $\inn_y(f)\in \inn_y(\cI)$ can be written as $\inn_y(f)=\sum \inn_y(h_i)\inn(f_i)$, where $\supd(\inn_y(h_i))\subset \Gamma_i\times \NN^{n-c}$. 

Thus the  Hilbert-Samuel function  $H_y(\cI)=H(\Delta\times \NN^{n-c})$  at any point $y\in C$ is determined by the diagram $\Delta\times \NN^{n-c}$  and is the same for all closed points $y\in U\cap C$.
\end{proof}

\subsection{Hilbert-Samuel function}
To deduce the important Bennett theorem, we shall need a useful  extension of Bierstone-Milman result \cite{BM} Corollary 5.2.2. In the Bierstone-Milman paper \cite{BM} a (stronger) pointwise order is considered for the set of Hilbert-Samuel functions. They show that the set of Hilbert-Samuel functions with pointwise order has d.c.c. property (see below). 
 Here we need a similar result for a (weaker) lexicographic order. The main idea of the proof to  use  diagrams of initial exponents, remains the same.

\begin{theorem}[Descending chain condition of the Hilbert-Samuel function]\label{dcc}
The set of values of the Hilbert-Samuel function (ordered lexicographically) 
%
$$
\cH(n):=\{H_\cI(k)=\dim(K[x_1,\ldots,x_n]/(\cI+m^{k+1}))
 \mid  \cI\subset K[x_1,\ldots,x_n])\}
$$ 
is d.c.c. (satisfies the descending chain condition). In other words, any decreasing sequence of functions 
\begin{equation} \label{HS} 
H_1\geq \ldots \geq H_n\geq\ldots 
\end{equation}
stabilizes: $H_s=H_{s+1}=\ldots$ for sufficiently large $s$.
\end{theorem}

\begin{proof} 
Consider any normalized and total order $\overline{T}$ on $\NN^n$.
For any  finite subset $\alpha:=\{a^1,\ldots,a^s\}$ in $\NN^n$ denote  by   $\Delta(\alpha):=\Delta(a^1,\ldots,a^s):=\bigcup a^i+\NN^n$ the  diagram of the initial exponents defined by $a^i$.

Let 
$H(\alpha)(k):=H(\Delta(a^1,\ldots,a^s))$
  denote the corresponding Hilbert-Samuel function. Then let 
%
$$
\cH(n):=\{H(\alpha)(k)\mid \alpha:=(a^1,\ldots,a^s)\in \NN^n\}
$$ 
be the set of all possible Hilbert-Samuel functions obtained that way.
We shall assume that the elements of $\alpha$ are ordered:  $a^1<\ldots<a^s$ with respect to $\overline{T}$, and  that $a^i\not \in \bigcup_{j=1}^{i-1} a^j+\NN^n$.  Denote by $S$ the set of all finite subsets of $\NN^n$ 
%
of that form, and by $S_k$ those subsets in $S$ for which all elements have multiplicity $|a^i|\leq k$.

Write $(H_i)$ from the sequence (\ref{HS}) as $(H(\alpha_i))$, where 
$(\alpha_i)$ is the corresponding sequence  of finite subsets of 
$\NN^n$. For any $\alpha=(a^1,\ldots,a^s)$ set  $\alpha+ \NN^n:=\bigcup a^i+\NN^n$. 

Immediately from the definition we see that if $\alpha\subset \beta$ then $\alpha+ \NN^n\subset \beta+ \NN^n$ and  $H(\alpha)\geq H(\beta)$. Also if $\alpha\subseteq \beta$ then for $b\in \beta\setminus \alpha$, we have  $b\not \in \alpha+ \NN^n$.

For any finite subset $\alpha$ of $\NN$ we define its ``restriction'' to be the subset 
$$
\res_k(\alpha)=\{a\in \alpha\mid |a|\leq k\}\in S_k.
$$
By the \emph{truncated Hilbert-Samuel function} $ H^{\leq k}$ we mean the truncation of $H:\NN\to \NN$ to  the set $\{1,\ldots,k\}$. 

We see that 
$$
H(\res_k(\alpha))\geq  H(\alpha)\quad  \text{and} \quad H^{\leq k}(\alpha)=H^{\leq k}(\res_k(\alpha)).
$$ 
Since each  set $S_k$ is finite, we can construct by induction
a sequence $\beta_k\in S_k$ such that: 
\begin{enumerate}
 \item $\res_{k-1}(\beta_k)=\beta_{k-1}$. 
\item For any $\beta_k$ there exist infinitely many  $\alpha_i$ from the sequence such that $res(\alpha_i)=\beta_k$.
\end{enumerate}
Thus we get a sequence $\beta_1\subset \beta_2\subset\ldots.$ But the relevant sequence 
$$
\beta_1+\NN^n\subset \beta_2+\NN^n\subset \beta_m+\NN^n=\beta_{m+1}+\NN^n=\ldots
$$ 
stabilizes since it corresponds to an increasing chain of  monomial ideals in the Noetherian ring $K[x_1,\ldots,x_n]$. This implies that the sequence of the sets of vertices $\beta_m=\beta_{m+1}=\ldots$ stabilizes. In other words, there is an infinite sequence $\alpha_{i_j}$ such that $\res_j(\alpha_{i_j})=\beta_m$ for $j\geq m$. We show that  $\alpha_{i_j}=\beta_m$ for  $j\geq m$. Observe that if  $\beta_m \subsetneq  \alpha_{i_{j_0}}$ then $\beta_m \subsetneq  \res_j(\alpha_{i_{j}})$ for sufficiently large $j>j_0$ and 
$$
H^{\leq j}(\alpha_{i_{j}})>   H^{\leq j}(\beta_m).
$$ 
Since $\beta_m=\res_j(\alpha_{i_j})$ we have the contradiction:
 $$
H^{\leq j}(\beta_m)=H^{\leq j}(\alpha_{i_j})\geq  H^{\leq j}(\alpha_{i_{j_0}})>    H^{\leq j}(\beta_m).
$$
 This shows that $\alpha_{i_{j_1}}= \alpha_{i_{j_2}}=\ldots $ stabilizes, and consequently $H_i=H(\alpha_i)$ stabilizes for $i\geq i_{j_m}$.
\end{proof}

\begin{corollary}\label{dcc2}  There exist only finitely many diagrams of initial exponents $\Delta$ in $\NN^k$ having the same Hilbert-Samuel function.
	
\end{corollary}

\begin{proof} We use the same notation as in the previous proofs. Suppose there exist infinitely many diagrams $\Delta_i$ with the same Hilbert-Samuel function. For each $s$ one can find the subsets $\alpha_s$ of $\NN^k_{\leq s}:=\{x\in \NN^k\mid |x|\leq s\}$, such that $\alpha_s$ is the restriction of infinitely many of $\Delta_i$. 
Each of the subsets generates the diagram  $\beta_s$
Then as before it leads to an infinite sequence of the subsets 
$\beta_1\subset\beta_2\subset\ldots\subset \beta_i\subset\ldots$, such that $res_{i-1}{\beta_i}=\beta_{i-1}$ for all $i=1,\ldots$.
But then the set $\beta:=\bigcup \beta_i$ corresponds to a finitely generated ideal $I_\beta:=\{ x^a\mid a\in\beta\}$ in $K[x_1,\ldots,x_n]$. Thus the sequences $(I_{\beta_i}($ and $(\beta_i)$ stabilize   $\beta_i=\beta_{i+1}=\ldots$ and define the same diagram.
	
\end{proof}

The following theorem extends Bennett's result to the differential (and  analytic) setting.

\begin{corollary}[Bennett \cite{Bennett}]
   Let $X$ be a smooth  scheme of finite type over a field $K$ (or a compact analytic or differentiable manifold).  Let $\cI$ be a sheaf of ideals of finite type on $X$.  Then the Hilbert-Samuel function   $H_{x,\cI}(k)$ of $\cI$ on $X$ is upper semicontinuous and attains only finitely many values. Consequently, there is a finite \emph{Samuel decomposition} into  locally closed strata such that two closed points
are in the same stratum if they have the same Hilbert-Samuel function 
%
$H_{\cI,x}(s)=\dim_{K_x}\cO_X/(m_{x}^{s+1}+\cI_x)$.
\end{corollary}

\begin{proof} 
Let us consider the case of a perfect field in the algebraic setting.
It follows from  Theorem \ref{weak} that the set  
$$
\{x\in X\mid H_{x,\cI}\leq H\}
$$ 
is open. Since $X$ is quasi-compact in Zariski topology (or in the case of compact manifold),  we conclude that the Hilbert-Samuel function attains its maximum value on $X$, and consequently, by d.c.c., it has finitely many values.
This proves the theorem over a perfect field. 

In general, we can pass to an algebraic closure $\overline{K}$, where we find the finite Samuel decomposition of $\overline{X}:=X\times_{\Spec{K}}\Spec{\overline{K}}$ which is stable under the action of the Galois group $\Ga(\overline{K}/K)$. This implies that it descends to the Samuel decomposition of $X$.
\end{proof}

\begin{corollary}[Bennett \cite{Bennett}] 
Let $X$ be any scheme of finite type over a field $K$ (or a compact analytic space or a differentiable space). Then there exists a finite \emph{Samuel decomposition} of $X$ into  a locally closed Samuel strata such that two closed points
are in the same stratum if they have the same Hilbert-Samuel function 
$H_{X,x}(s)=\dim_{K_x}\cO_X/m_{x}^{s+1}$.
\end{corollary}

\begin{proof} We can locally embed $X$ into a smooth scheme over $\bA^n_K$ or into $K^n$, where $K=\RR,\CC$, and apply the previous theorem.
\end{proof}


The following theorem shows that the problem of resolution of singularities controlled by the Hilbert-Samuel function can be reduced to the desingularization of marked ideals $(f_i,\mu_i)$ defined for the standard basis. Thus the  standard basis is a counterpart of Hironaka's distinguished data or  Bierstone-Milman's semicoherent presentation of ideals (see 
\cite{Hir2}, \cite{BM2}, \cite{BM22}).

\begin{theorem}[Stability of  standard basis under blow-ups]\label{weak2}
 Let $\cI$ be any ideal sheaf of finite type on  
  a smooth variety (or an analytic or differentiable manifold) $X$, and let  $C\subset X$ be a smooth subvariety.
 Consider the blow-up
$\sigma: X'\to X$ at a smooth closed center $C\subset X$ contained in the Samuel stratum, and let $\cI'$ be its strict transform. Let $U'\subset \sigma^{-1}(U)$ be an open subset where a coordinate $y$ on $U$ describes the exceptional divisor of $\sigma$.
 Then:
\begin{enumerate}
\item If $f_1,\ldots,f_k$ is a  standard basis  of $\cI$ on an \'etale (respectively open) neighborhood $U$ of $X$ with respect to a monotone  diagram $\Delta$   
  then 
$$
f_1':=\sigma^*(f_1)/y^{\mu_1},\ldots,f'_k:=\sigma^*(f_k)/y^{\mu_k}
$$ 
is a  standard basis of $\cI'$  on $U'$ with respect  to $\Delta$ and the induced coordinate systems. 
  
\item {\rm (Bennett)} $H_{x,\cI}\geq H_{x',(\cI')}.$
\end{enumerate}
\end{theorem}

\begin{proof}  
Let 
$u_1,
\ldots,u_n$  
be local parameters compatible with the  standard basis  and the center $C$ (Lemma \ref{center2}). Here the coordinates  $u_1,\ldots,u_c$ define $C$ locally.

 Let $\mu_i$ denote the multiplicity of the function $f_i$. It follows from Lemma \ref{center} that $f_i\in \cI^{\mu_i}\setminus \cI_C^{\mu_i+1}$ and correspondingly $\sigma^*(f) \in \cI_D^{\mu_i}\setminus \cI_D^{\mu_i+1}$, where $D=\sigma^{-1}(C)$ is  the exceptional divisor of $\sigma$.
Then $f_i':=(1/y^{\mu_i})\sigma^*(f_i)$. Consider the effect of the blow-up of $C$ near a point $x\in X$. 

The points $x'$  in $\sigma^{-1}(x)$ are defined by the lines in the normal space $N_C$, where $N^*_C=\cI_C/\cI_C^2$. For  the line $t[a_1,\ldots,a_k]$, where $a_i\in K$, we  consider the dual hyperplane in $N^*_C$ and the corresponding linear system. More precisely, let $r=\max\{i: a_i\neq 0\}$ and 
consider the change of coordinates 
\begin{equation} \label{f} 
\overline{u}_1=u_1-(a_1/a_r)u_r,\ldots,\overline{u}_{r-1}=u_{r-1}-(a_{r-1}/a_r)u_r, \overline{u}_r=u_r,\overline{u}_{r+1}=u_{r+1},\ldots,\overline{u}_n=u_n.
\end{equation}
The effect of the blow-up at $x'$ (in new coordinates) is described by 
$$
u'_1=\overline{u}_1/y,\ldots,
u'_{1}/y=\overline{u}_{1}, u'_r=\overline{u}_{r}=y,u'_{r+1}=\overline{u}_{r+1}/y,\ldots,u'_n=\overline{u}_n/y
$$
with the exceptional divisor $y=\overline{u}_{r}$.
Since $f_i\in \cI_C^{\mu_i}\setminus \cI_C^{\mu_i+1}$,  we can write  $f_i=\sum_{|\alpha|\geq \mu_i} c_{i\alpha}(v)\overline{u}^{\alpha}$, where $\overline{u}:=(\overline{u}_1,\ldots,\overline{u}_c)$ and $v:=(\overline{u}_{k+1},\ldots,\overline{u}_n)$.  Then the initial form with respect to the $\cI_C$-grading is 
$$
F_i(v,\overline{u})=\inn_C(f_i)=\sum_{|\alpha|= \mu_i} c_{i\alpha}(v)\overline{u}^{\alpha},
$$
and $f_i=F_i+G_i$, where $G_i\in \cI_C^{\mu_i+1}$ and
$\sigma^*(G_i)=y^{\mu_i+1}G'$. For any 
$\alpha=(\alpha_1,
\ldots,\alpha_c)\in \NN^c$ define 
$$
\alpha':=(\alpha_1,\ldots,0_r,\ldots,\alpha_c)
$$ 
with $0$ as the $r$-th coordinate.
The transformed function will have the form 
$$
f_i':=(1/y^{\mu_i})\sigma^*(f_i)=(1/y^{\mu_i})\sum \sigma^*(c_i)(u')^{\beta'}\cdot y^{|\beta|-\mu_i}=F_i(v)(u'_1,\ldots,1_r,\ldots,u'_c)+yG'_i(v,u'_1,\ldots,u'_c).
$$

 Let 
$$
\overline{F}_i(\overline{u}_1,\ldots,\overline{u}_k):=\inn_x(f_i)=F_i(v(x))=\sum_{|\alpha|= \mu_i} c_{i\alpha}(v)(x)\overline{u}^{\alpha}
$$ 
be the initial form at $x$. 
 Then $\overline{F}_i(\overline{u}_1,\ldots,\overline{u}_c)$ is a form of degree $\mu_i$. 
 Consider the  transform of $\overline{F}_i$ under the blow-up:
$$
F'_i(u'_1,\ldots,u'_c)=(1/y^{\mu_i})\sigma^*(F_i(\overline{u}_1,
\ldots,\overline{u}_c))=\overline{F}_i(u'_1,\ldots,1_r,\ldots,u'_c).
$$
This implies that  
%
$\ord_{x'}(f_i')=\ord_{x'}((1/y^{\mu_i})\sigma^*(f_i))\leq \ord_x(\overline{F}_i(u'_1,\ldots,1_r,\ldots,u'_c))
\leq \mu_i$.

Suppose that $\ord_{x'}(f'_i)= \mu_i$ for all $i$. Then $\deg(F'_i)=\deg(F_i)$ and we conclude that $\overline{F}_i$ does not contain $\overline{u}_r$, that is, 
$$
\overline{F}_i(u'_1,\ldots,0_r,\ldots,u'_c)=\overline{F}_i(\overline{u}_1,\ldots,\overline{u}_r,\ldots,\overline{u}_c)=F_i(\overline{u}_1,\ldots,0_r,\ldots,\overline{u}_c)=\overline{F}_i({u}_1,\ldots,0_r,\ldots,{u}_c).
$$
 Since the variables $u_1,\ldots,u_s$ are essential, this   implies that $r\geq s+1$. The set of coordinates 
$\overline{u}_1,
\ldots,\overline{u}_n$ 
is then compatible with the  standard basis, and the derivatives $D_{u_i}=D_{\overline{u}_i}$ are the same for $i=1,\ldots,s$. Moreover  by the chain rule 
%
$D_{u_i}=D_{\overline{u}_i}=(1/y) D_{u'_i}$ 
(see for instance \cite{Villamayor}). The latter implies that the differential conditions (3) through (5) in the definition of a  standard basis are preserved after the blow-up in the new coordinate system $u'_1,\ldots,u'_n$.  For instance for the 
second
part of  condition (3) we can write 
%
$$
D_{{u'}^{\alpha_i}}(f'_i)=y^{|\alpha_i|}D_{{\overline{u}}^{\alpha_i}}y^{-|\alpha_i|}\sigma^*(f_i)\equiv 1
$$
with the natural identification of $\sigma^*(f_i)$ with $f_i$. The other differential conditions (4) and (5) follow in the same way.
In other words, $\supd(f'_i)\subset \Gamma\times \NN^{n-s}$. To prove that $f'_i$ is a  standard basis of $\cI'$ at $x'$ with respect to the diagram $\Delta$, we need to show that  $H(\Delta)=H_{x',X'}(\cI')$. We show this in a series of lemmas below.

Now suppose  $\ord_{x'}(f'_i)< \mu_{i}$ for some  $i$ and let 
$$
d=\mu_{j}:=\min\{\mu_i:\ord_{x'}(f'_i)\}<\mu_{i}\}.
$$ 
If such a $d$ does not exist, that is,  $\ord_{x'}(f'_i)= \mu_{i}$ for all  $i$, then set $d=\infty$.

\begin{lemma} 
If $r\leq s$ then all the vertices $\alpha_i$ with $|\alpha_i|<d$ are in $\NN^{r-1}$.	
\end{lemma}

\begin{proof}
Let $\alpha_i$ be a vertex of $\Delta$ with $|\alpha_i|=\mu_i<d$. Suppose it has a coordinate $s\geq r$ which is not zero. Then by Lemma \ref{sub} there is 
another vertex $\alpha_{i_1}$ with $r$-th coordinate not zero and $|\alpha_{i_1}|\leq |\alpha_i|<d$.
But since 
%
$D_{u^{\alpha_{i_1}}}(f_{i_1})=1$, the initial form $\overline{F}_{i_1}$ depends on $u_r$, and $\ord(\overline{F}_{i_1}(u'_1,\ldots,1_r,\ldots,u'_c))<\mu_{i_1}<d$, which contradicts the assumption on $d$.
\end{proof}
 
 Consider the monotone diagram $\Delta^{r-1}$ generated by the  vertices with $|\alpha_i|<d$. Then $\Delta^{r-1}$ is contained  in $\NN^{r-1}$, and let $\Gamma^{r-1}=\NN^{r-1}\setminus \Delta^{r-1}$.
By the lemma,  both sets $\Gamma\times \NN^{n-s} \subset \Gamma^{r-1}\times \NN^{n-r+1}$ 
coincide for the exponents $\alpha\in \NN^s$ of degree $|\alpha|<d$.

\begin{lemma} \label{span4}
The set 
$$
\{(u')^\alpha\inn_{x'}(f'_i)\mid  \alpha\in \Gamma^r_i\times \NN^{n-r+1},\,
 |\alpha|+|\alpha_i|<d \}
\cup \{(u')^{\alpha}\mid \alpha \in \Gamma^r\times \NN^{n-r+1},\, |\alpha|<d  \}
$$
is a 
basis  of the $K[u'_r,\ldots,u'_n]$-module 
  $K[u'_1,\ldots,u'_n]/(u'_1,\ldots,u'_{r-1})^d$ if $d$ is finite and of $K[u'_1,\ldots,u'_n]$ if $d=\infty$. 
 \end{lemma}
 
 \begin{proof} 
If $r\leq s$ then for any $i$, the  vertex $\alpha_i$ with $\mu_i<d$ is  in $\NN^{r-1}$. If  $r>s$ then the vertices belong to   $\NN^s$ and we shall simply replace $\Gamma^{r-1}_i\times \NN^{n-r+1}$ with $\Gamma^{s}_i\times \NN^{n-s}$ in the considerations below.
 
 For $\mu_i<d$ we have $\inn_{x'}(f_i)=\overline{F}_i+yG_i$, and 
$$
\inn_{x'}(f_i)(u'_1,\ldots,u'_{r-1},0,u'_{r+1},\ldots,u'_s)
=\overline{F_i}(\overline{u}_1,\ldots,{\overline{u}_{r-1}},0,{\overline{u}_{r+1}},\ldots,\overline{u}_{c})
=\overline{F_i}(\overline{u}_1,\ldots,\overline{u}_{c}).
$$ 
By the assumption,  
$$
\{u^\alpha\overline{F_i}\mid  \alpha\in \Gamma^{r-1}_i\times \NN^{n-r+1},\, |\alpha|+|\alpha_i|<d \}\cup \{u^{\alpha}\mid \alpha \in \Gamma^{r-1}\times \NN^{n-r+1},\, |\alpha|<d  \}
$$ 
is a  basis  of the $K[u_r,\ldots,u_n]$-module  $K[u_1,\ldots,u_n]/(u_1,\ldots,u_{r-1})^d$.

Since $\overline{u}_i-u_i$ is divisible by $y=u_r$ we  conclude, by the proof of Lemma \ref{restriction3}, that  
$$
\{\overline{u}^\alpha\overline{F_i}\mid  \alpha\in \Gamma^{r-1}_i\times \NN^{n-r+1},\, |\alpha|+|\alpha_i|<d \}\cup \{\overline{u}^{\alpha}\mid \alpha \in \Gamma^r\times \NN^{n-r+1},\, |\alpha|<d  \}
$$ 
is also a 
basis  of the $K[u_r,\ldots,u_n]$-module $K[u_1,\ldots,u_n]/(u_1,\ldots,u_{r-1})^d$. 

Similarly the differences
 $$
(u')^\alpha\inn_{x'}(f_i)(u_1',\ldots,u_c')-(u')^\alpha\overline{F}_i(u_1',\ldots,u'_c)
$$ 
are  divisible by $y=u'_r$, and thus  again 
$$
\{(u')^\alpha\inn_{x'}(f'_i)\mid  \alpha\in \Gamma^r_i\times \NN^{n-r+1},\, |\alpha|+|\alpha_i|<d \}\cup \{(u')^{\alpha}\mid \alpha \in \Gamma^r\times \NN^{n-r+1},\, |\alpha|<d  \}
$$
is  a 
basis  of the  $K[u'_r,\ldots,u'_n]$-module   $K[u'_1,\ldots,u'_n]/(u'_1,\ldots,u'_{r-1})^d$.
 \end{proof}
 
\begin{lemma}\label{span3}
The set 
$$
\{u^{\alpha}\inn_{x'}(f'_i)\mid |\alpha|+\mu_i<d,\,  \alpha\in \Gamma_i\times \NN^{n-s}  \}
$$
 is a basis of the $K$-space  $\{\inn_{x'}(f)\mid f\in \cI'_{x'},\, \ord_{x'}(f)<d\}$ if $d$ is finite and of $\inn_{x'}\cI'$ if $d=\infty$. 
 \end{lemma}
 
\begin{proof}
Any function  $f\in \cI'_{x'}$ of order $e$ can be written as a combination   
$f=\sum c_i\sigma^*(g_i)/y^{k_i}$ of $\sigma^*(g_i)/y^{k_i}$, where $g_i\in \cI_x$, $y^{k_i}|\sigma^*(g_i)$ and $c_i\in \cO_{X',x'}$. One can approximate   $c_i$ up to $m_{x'}^{e+1}$ by a polynomial 
$p_i(u_1',\ldots,u_n')$ of the form
$$
p_i(u_1',\ldots,u_n')=\sigma^*(p_i(\overline{u_1}/\overline{u_r},\ldots,\overline{u_r},\ldots,\overline{u_c}/u_r,\ldots,\overline{u_n})
%
=(1/y^{\overline{k_i}})\sigma^*\overline{p}_i(\overline{u_1},\ldots,\overline{u_n})
$$
for suitable polynomials $\overline{p}_i$.

Then $f$ can be  approximated by 
$$
\sum p_i\sigma^*(g_i)/y^{k_i}=\sum  \sigma^*(\overline{p}_i)\sigma^*(g_i)/y^{k_i+\overline{k}_i}
%
=(1/y^k) \sigma^*\Big(\sum \overline{p}_ig_i\overline{u_r}^{k-(k_i+\overline{k}_i)}\Big)
$$
for sufficiently large $k$.
This implies that 
%
$\inn_{x'}(f)=\inn_{x'}((1/y^k)\sigma^*(\overline{f}))$ with $\overline{f}:=\sum \overline{p}_ig_i\overline{u_r}^{k-(k_i+\overline{k}_i)}\in \cI_x$.

Write $\overline{f}=\sum {h_i}{f_i}$ with $\supd({h_i})\subset \Gamma_i$.
By Lemma \ref{Flat},  we get  
$\ord_C(\overline{f})=\min \{\ord_C({h_i}f_i)\}\geq k$.
Consequently, 
$$
\sigma^*(\overline{f})/y^k=\sum \sigma^*(h_i)/y^{k-\mu_i}\sigma^*(f_i)/y^\mu_i= \sum (\sigma^*(h_i)/y^{k-\mu_i})f'_i. 
$$
Note that since $\mu_i<d$ and $\supd({h_i})\subset \Gamma_i\times \NN^{n-s}=\Gamma^{r-1}_i\times \NN^{n-r+1}$,  we get $D_{u^\alpha}(h_i)=0$ for $\alpha\in \Delta_i^{r-1}$, and consequently $D_{(u')^\alpha}(\sigma^*({h_i})/y^{k-\mu_i})=0$ for $\alpha\in \Delta_i^{r-1}$, which means that 
%
$$
\supd(\sigma^*({h_i})/y^{k-\mu_i})
\subset \Gamma^{r-1}_i\times \NN^{n-r+1}
$$ 
as well. Since $\Gamma^{r-1}_i\times \NN^{n-r+1}$ coincides with $\Gamma_i\times \NN^{n-s}$  for the exponents $\alpha<d$, we see
  that for  $d':=\ord_{x'}(f)< d$ we have
$$
\inn_{x'}({f})= \sum_{ \ord_{x'}(h)+\mu_i=d'} \inn_{x'}(\sigma^*({h_i})/y^{k-\mu_i})\inn_x(f'_i)=\sum H_i\inn_x(f'_i)
$$
for some $H_i$ with $\supd(H_i)\subset \Gamma_i\times \NN^{n-s}$, which completes the proof of the lemma.
\end{proof}

Lemma \ref{span3} implies that $H_{x',\cI'}=H(\Delta\times \NN^s)$ if $\ord_{x'}(f'_i)=\mu_i$ for all $i$ 
($d=\infty$).
Since all other properties were proven before, we conclude that $f'_i$ is a  standard basis at such points.

Now suppose that 
%
$d=\min\{\mu_i:\ord_{x'}(f'_i)<\mu_{i}\}$ is finite. We will show that $H_{x',X'}(\cI')<H(\Delta\times \NN^{n-s})$.
The argument splits into two  quite similar cases:

Case 1. Suppose that $\ord_{x'}(\overline{F}_j(u'_1,\ldots,1_r,\ldots,u'_s,0,\ldots,0))<d=\mu_j$.
In this case  $r\leq s$ defines the essential unknown. 
By  condition (3) of Definition \ref{Dweak2}, 
$$
\supd(f_j)\subset \{\alpha_j\}\cup (\Gamma\times \NN^{n-s}) \supset  \Gamma^{r-1}\times \NN^{n-r+1}.
$$
This implies that $D_{u^\alpha}(f_j)=D_{\overline{u}^\alpha}(f_j)=0$ for $\alpha\in \Delta^r\subset\Delta\setminus \{\alpha\}$. 
Then, by the chain rule,  $D_{u'^\alpha}(\inn_{x'}(f'_j))=0$ for $\alpha\in \Delta^r$, $|\alpha|<d$, or equivalently $\alpha\in \Delta$, $|\alpha|<d$. Thus, by Lemmas \ref{span4} and \ref{span3} we see that
 $\supd(\inn_{x'}(f'_j)) \subset \Gamma\times \NN^{n-s}$ and $\inn_{x'}(f'_j)$ is independent of  
$$
\{u^{\alpha}\inn_{x'}(f'_i)\mid\alpha\in \Gamma_i\times \NN^{n-s}\}
$$ 
and $H_{x,X}(\cI)=H(\Delta\times \NN^{n-s})>H_{x',X'}(\cI')$.

Case 2. Suppose that $r\geq s+1$ and  
$\ord(\overline{F}_i(\overline{u}_1,\ldots,\overline{u}_s,0,\ldots,0))
=\mu_i$  
for all $i$, and   let  $\ord_{x'}(f'_j)=d<\mu_j$ for a certain $j$. In this 
situation  $D_{u'^\alpha} f'_j\equiv 0$ for $\alpha \in \Delta\times \NN^{n-s}$, $|\alpha|<\mu_j$ as in Case 1. Then  again $\supd({\inn_{x'}(f'_j)})\subset \Gamma\times \NN^{n-s}$,  and $\inn_{x'}(f'_j)$ is independent of $u^{\alpha}\inn_{x'}(f'_i)$, where $\alpha\in \Gamma_i\times \NN^{n-s}$.
Consequently, $H_{x,X}(\cI)=H(\Delta\times \NN^{n-s})>H_{x',X'}(\cI')$ as in Case 1.

\begin{remark} 
The fact that one can find a basis of an ideal which preserves its optimal form under blow-ups (mainly condition (3) of Definition \ref{Dweak2})  is very important. It reduces a problem of the lowering Hilbert-Samuel function an thus strong desingularization to resolution of marked ideals. To ensure canonicity of the reduction one considers the relevant equivalence relation as in \cite{BM2} (Section 4).

In characteristic zero, however the situation is even simpler. We show that the  standard basis, though not unique, generates a unique canonical Rees Algebra  giving  a canonical reduction of Hironaka desingularization to resolution of marked ideals. Thus the reduction is automatic and requires no additional glueing relations.\end{remark}

\end{proof}

\section{Canonical Rees algebra and  standard basis}
\subsection{First properties of Rees algebras}


\begin{definition} Let $X$ some smooth space  over a field $K$ (or analytic or differentiable manifold).

By {\it  the  Rees Algebra} $$\cR^{\cdot}=\bigoplus_{\mu\in \NN} (\cR^\mu,\mu)$$  we mean a graded algebra satisfying  the conditions

\begin{enumerate}
\item  $\cR^0=\cO(U)$
\item    $\cR^\mu\subset \cO(U)$ is an ideal sheaf of $\cO_X$.
\item $\cR^{\mu}\subset  \cR^{\mu'}$ if $\mu'\geq \mu$
\item $\cR^{\mu}\cdot \cR^{\mu'}\subset \cR^{\mu+\mu'}$

\noindent By {\it  the  Rees Algebra generated by $\{\overline{\cI}\}:=(\cI_i,\mu_i)$}  we mean the  smallest graded algebra  $\cR^{\cdot}=\bigoplus_{\mu\in \NN} (\cR^\mu,\mu)$  such that 
$\cI_i\subseteq  \cR^{\mu_i}$.

\noindent A Rees algebra will be called  {\it  a differential Rees Algebra} if it satisfies
 
\item $\cD^a(\cR^{\mu})\subset \cR^{\mu-a}$ if $a\in \ZZ_{>0}$, $\mu\geq a.$

\end{enumerate}

Similarly   the differential Rees Algebra $\cR^{\cdot}=\cR^{\cdot}(\{\overline{\cI}\})$ is {\it diff-generated by $\overline{\cI}:=\{(\cI_i,\mu_i)\}$} if it is the smallest differential Rees algebra for which  
$\cI_i\in \cR^{\mu_i}$.

\end{definition}

\begin{remark} Different notions of Rees algebras defined by marked ideals were studied in the context  of resolution by Giraud, Hironaka, Oda, and more recently Kawanoue-Matsuki, and Villamayor. The above definition is essentially equivalent to the one used Villamayor's papers.
(See \cite{Giraud2},\cite{H},\cite{O},\cite{Vi3},\cite{Kawanoue}.) 

The differential Rees algebras are  natural extensions of marked ideals. They possess important properties generalizing the notion of coefficient ideals and homognization used in the simple proofs of the (weaker) desingularization in characterist zero \cite{Wlod}, \cite{Kollar}

 In this paper the notion will be  used mainly to study more subtle properties related to the Hilbert-Samuel function, and strong resolution (see Definition \ref{Rees0}).

\end{remark}

It follows from the definition that 
$$\cosupp(\cR^{\cdot}(\overline{\cI})=\bigcap_{\mu\in \NN} \cosupp(\cR^\mu,\mu)=\cosupp(\overline{\cI}),$$
for any multiple marked ideal $\overline{\cI}$.

Moreover, an immediate consequence of the definition is the foloowing:
\begin{proposition} Let $\overline{\cI}=\{(\cI_i,\mu_i)\}$ be  a finite collection of marked ideals on a smooth scheme of finite type $X$ over a field $K$. Let $u_1,\ldots,u_n$ be a system of cooordinates on $X$. Denote by $f_{i1},\ldots,f_{ij_i}$ the finite sets of generators of $\cI_i$. Then the differential Rees algebra $R^\cdot(\overline{\cI})$ is (finitely) generated by marked functions $(D_{u^\alpha}f_{ij},\mu_i-|\alpha|)$, where $\alpha\in \NN^n$, $|\alpha|\leq \mu_i$.

\end{proposition}

Also, by Lemma \ref{diff}, for any \'etale morphism $\phi: X\to Y$ of smooth varieties over a field $K$ we have that
 $$\phi^*({\cR^\cdot}(\overline{\cI})={\cR^\cdot}(\phi^*(\overline{\cI})).$$

\subsection{Canonical Rees Algebra along Samuel stratum and essential variables}
\bigskip
 Let $F\in K[x_1,\ldots,x_n]$ be a form of degree $d$. For any nonnegative integer $k\leq d$ denote by $\overline{D}^d(F)$ the vector space spanned by
the derivatives of order $d$. This definition does not
depend upon a linear change of coordinates. Using 
this operation one can define a homogenous counterpart of Rees algebra and Rees ideal.

\begin{definition} By {\it  the homogenous Rees Algebra} generated by the homogenous polynomials $F_i\in K[x_1,\ldots,x_n]$ of degree $d_i$  we mean the  smallest graded subalgebra  $${R}^{\cdot}=R^{\cdot}(F_1,\ldots, F_r)=\bigoplus_{d\in \NN} R^d$$ containing 
$F_i\in R^{d_i}$ and which is $\overline{D}^\cdot$-stable, that is
$$\overline{D}^a(\cR^d)\subset R^{d-a}$$ if $a\in \ZZ_{>0}$, $d\geq a.$

  The  graded ideal $$I^{\cdot}=I^{\cdot}(F_1,\ldots, F_r)\subset R^{\cdot}=R^{\cdot}(F_1,\ldots, F_r)$$ generated over $R^{\cdot}$ by $(F_i)$ will be called the {\it Rees ideal} generated by $(F_i)$.
\end{definition}

\begin{definition} Let $\cI=\bigoplus \cI_{a\in \NN} \subset K[x_1,\ldots,x_n]$ be a homogenous ideal. By the {\it essential set of coordinates} we mean a set of lineally independent linear forms $u_1,\ldots, u_k$ such that 
\begin{enumerate}

\item There  exists  a set of homogenous generators $F_1,\ldots, F_r\in \cI$, such that $F_i=F_i(u_1,\ldots, u_k)$.
\item The vector space $V:=\spa(u_1,\ldots, u_k)\subset \spa(x_1,\ldots, x_n)$ is minimal for all sets $u_1,\ldots, u'_{k'}$ satisfying the condition (1).
The vector space $V$ will be called the {\it essential space} of $\cI$.
\end{enumerate}

\end{definition}
The notion of essential coordinates makes sense for homogenous polynomials or their sets.
\begin{lemma}  (\cite{BM}, Lemma 6.2a)
In characteristic $0$ the vector space $V:=\overline{D}^{d-1}(F)$ is essential for $F$.\end{lemma}
\begin{proof} If $u_1,\ldots,u_k$ is a basis of $V$ then after extending the set to  a complete coordinate system $u_1,\ldots,u_n$ and $F$ does not depend upon $u_{k+1},\ldots,u_n$ thus $F=F(u_1,\ldots,u_k)$. On the other hand if $u'_1,\ldots,u'_l$ are essential unknowns then $F=F(u'_1,\ldots,u'_l)$ and $\spa(u'_1,\ldots, u'_l)\supseteq V=\overline{D}^{d-1}(F)$. 
\end{proof}

In  characteristic $p$ we shall use homogenous Rees algebras to isolate essential variables.
Let us first reformulate the above lemma first in characteristic zero
\begin{proposition}Let $K$ be  a  field of the characteristic $0$ and $F_1,\ldots, F_r\in K[x_1,\ldots,x_n]$ denote homogenous polynomials.
The homogenous Rees algebra $R^\cdot(F_1,\ldots,F_r)$ is generated by the essential unknowns $u_1,\ldots u_s$ for  $F_1,\ldots, F_r$. That is 
$$R^\cdot:=R^\cdot(F_1,\ldots,F_r)= K[u_1,\ldots,u_k]$$
\end{proposition}
\begin{proof} The result follows from the fact that $u_1,\ldots,u_k\in R^1$ and any form $G\in R^d$ can be expressed as function in  $u_1,\ldots,u_k$. This is true since the generators have this form
and the property is preserved by the derivations, sums, and products. 

\end{proof}

\begin{proposition} Let $K$ be  a perfect field  of $char(K)=p$. Let $F_1,\ldots, F_r\in K[x_1,\ldots,x_n]$ denote homogenous polynomials.
 generating  $R^\cdot=(R^\cdot(F_1,\ldots,F_r)$

Then  the essential space $V$ for $F_1,\ldots,F_r$ is unique.  Moreover there exists a unique filtration $$V_0\subset V_1\subset \ldots\subset V_\ell=V$$
 such that  $V_i^{(p^i)}=\{G \in K[x_1,\ldots,x_n]_1\mid G^{p^i}\in R^{p^i}\}$.
Moreover consider a coordinate system
   $$u_1,\ldots, u_{k_1},u_{k_1+1}^p,\ldots, u_{k_2}^p,\ldots,u_{k_{\ell-1}+1}^{p^\ell},\ldots, u_{k_\ell}^{p^\ell}=\overline{u}_1,\ldots,\overline{u}^{p^\ell}_\ell$$
with $V_i=\spa(u_1,\ldots, u_{k_i})=\spa(\overline{u}_1,\ldots,\overline{u}_i)$
Then 

 \begin{enumerate}
 \item $R^\cdot(F_1,\ldots,F_r)=K[\overline{u}_0,\ldots,\overline{u}^{p^\ell}]$.
\item  $J^\cdot(F_1,\ldots,F_r)=K[\overline{u}_0,\ldots,\overline{u}^{p^\ell}]\cap (F_1,\ldots,F_r)$.
\end{enumerate}
\end{proposition}
\begin{proof}
Consider $V_0:=R^1$. By rearranging coordinates we can assume that  $V_0=\spa(u_1,\ldots,u_{k_1})$. This implies that $D_{u_j}(G)=0$ for any $G\in R^\cdot$ and $j>k_1$.
Thus it can be expressed as $G(u_1,\ldots,u_{k_1}, u^{p_{k_1+1}},\ldots,u^{p_n})$. Consider the canonical subalgebra $R^\cdot_1\subset R^\cdot$ defined by the conditions $D_{u_j}(G)=0$ for all $j$.
Then $R^\cdot_1$ consists of the homogenous polynomials of the forms $G(u^p_1,\ldots,u^p_n)=G'(u_1,\ldots,u_n)^p$. This defines the subalgebra $$\sqrt[p]{R^\cdot_1}:= \{G\in K[u_1,\ldots,u_n]\mid  G^p\in R^\cdot_1\}.$$ By the inductive assumption there exists a unique essential space $V^1$ for $\sqrt[p]{R^\cdot_1}$, with the induced canonical filtration $\overline{V_1}\subset\ldots\overline{V}_{\ell-1}$. Since $V_1\subset \overline{V_1}$, It defines a unique filtration $V_1\subset\ldots {V}_{\ell}$, where $V_i:=\overline{V_{i-1}}$ for $i=2,\ldots,r$.

\end{proof}


The above proposition leads to a more subtle notion than essential space in characteristic $p$.
\begin{definition} Let $K$ be  a perfect field of characteristic $p>0$.
The {\it essential flag} $$V_0\subset V_1\subset \ldots\subset V_\ell=V$$ of a homogenous ideal $\cI\subset 
	[x_1,\ldots,x_n]$ is a filtration of vector subspaces of the essential space $E$ such that
	 for any partitioned basis $$u_1,\ldots,u_{k_0},u_{k_0+1},\ldots,u_{k_1},\ldots,u_{k_\ell}$$ of $E$ such that $u_1,\ldots,u_{k_i}$ is a basis of $V_i$ there exist homogenous generators $$F_i:=F_i(u_1,\ldots,u_{k_0},u^p_{k_0+1},\ldots,u^p_{k_1},\ldots,u^{p^\ell}_{k_\ell}).$$
of $I$. Moreover the sequence of numbers $(k_0,\ldots,k_\ell,0,\ldots)$ is minimal
\end{definition}

\begin{proposition} Let $K$ be a perfect field and  $I\subset K[x_1,\ldots,x_n]$ be a homogenous ideal, and let $F_1,\ldots, F_r\in I$ denote a  standard basis of $\cI$ (at $0$). Then 
\begin{enumerate}
\item $F_1,\ldots, F_r$ are homogenous  and moreover $F_i=F_i(u_1,\ldots, u_k)$, where $u_1,\ldots, u_k$ is essential set of coordinates for $I$.  
\item $R^\cdot(I):=R^\cdot(F_1,\ldots,F_r)$ is independent of the choice of the  standard basis of $I$
\item If $char(K)=0$ then $R^\cdot(I)=K[u_1,\ldots,u_k]$
and $J^\cdot(I)=I\cap K[u_1,\ldots,u_k].$
\item If $char(K)=p$ then $R^\cdot(I)=K[\overline{u}_0,\ldots,\overline{u}^{p^\ell}]$ and $ J^\cdot(I)=I\cap K[\overline{u}_0,\ldots,\overline{u}_\ell^{p^\ell}] $
\item The essential space and the essential flag for $I$ are unique.

\end{enumerate}

\end{proposition}
\begin{remark} The  parts (1) and (2) are proven for a  standard basis with respect to a certain monotone order in char. $0$ in (\cite{BM}) (Lemma 6.7(2)). 
\end{remark}

 
\begin{proof} We give here a proof in positive characteristic. The case of characteristic zero is the same but sightly simpler.
Let $F_1,\ldots,F_r$ be any basis of $I$. The algebra $R^\cdot(F_1,\ldots,F_r)=K[\overline{u}_0,\ldots,\overline{u}_r^{p^\ell}]$ is characterized uniquely by the properties $D_{\overline{u}_j^{p^t}}(G)=0$ where $0\leq j<t\leq \ell$, or $t>\ell$, and $j$ is arbitrary. This property is valid thus also for generators $F_i$.
This implies that  for any $G=\sum H_iF_i\in I$, $D_{\overline{u}_j^{p^t}}(G)\in I$. Thus the property holds for a  standard basis $(G_i)$ of $I$ associated with some coordinate system $(v_i)$ and diagrams $\Delta$ and $\Gamma$.

But, by the property of the  standard basis $G_i=v^{\alpha_i}+r_i$ with $\supd(r_i)\subset \Gamma$. More precisely, by Condition (3) of Definition \ref{Dweak2}, $D_{v^\alpha}(G_i)=0$ for $\alpha\in \Delta\setminus \{\alpha_i\}$, and $D_{v^{\alpha_i}}(G_i)=1$. In both cases, that is, if $\alpha\in \Delta$ we have $D_{v^\alpha}(D_{\overline{u}_j^{p^t}})(G_i))=D_{\overline{u}_j^{p^t}})(D_{v^\alpha}(G_i)=0$ which means that
 $\supd(D_{\overline{u}_j^{p^t}})(G_i)\subset \Gamma$.

 Since additionally $(D_{\overline{u}_j^{p^t}})(G_i)\in I$  we conclude that  $\supp(D_{\overline{u}_j^{p^t}})(G_i)=0$. The latter implies that $G_i\in R^\cdot(F_1,\ldots,F_r)$, and $R^\cdot(G_1,\ldots,G_s) \subseteq R^\cdot(F_1,\ldots,F_r)$. 
This implies that if $(F_i)$ is another standard basis then
by symmetry we get the equality $R^\cdot(F_1,\ldots,F_r)= R^\cdot(G_1,\ldots,G_s)$. 

This also shows that any standard basis determines a unique essential flag.

Now write $G_j=\sum H_iF_i$, with $\cosupp(H_i)\subset \Gamma_i$. Again we easily see by induction on $i$ that $0=D(G_i)=\sum D(H_iF_i)=\sum D(H_i)F_i$, with $\cosupp(D(H_i))\subset \Gamma_i$. Thus $G_i\in J^\cdot(F_1,\ldots,F_r)$ showing inclusion $J^\cdot(G_1,\ldots,G_m)\subseteq J^\cdot(F_1,\ldots,F_r)$, which implies (by symmetry) the equality of both algebras. One can choose a standard basis with respect to the coordinate system $(u_i)$. This implies that $J^\cdot(I)=I\cap K[\overline{u}_0,\ldots,\overline{u}_r^{p^\ell}].$
\end{proof}

The result on homogenous Rees algebras can be extended to any ideal sheaf. We will show it here  in characteristic zero (Theorem \ref{Rees}). The case of positive characteristic is more subtle and will be dealt in a separate paper.

Let $f\in (\cI,\mu)$ be a marked function, and $x\in \cosupp(f,\mu)$. Then 
we define the {\it initial form} $\inn_x(f)$
of $(f,\mu)$ at $x$ to be the  class  of $f$ in $(\cI+m^{\mu+1}_{x,X})/m^{\mu+1}_{x,X}$. Thus $\inn_x(f)$  can be identified with the initial $\mu$-form of $f$ if $\ord_x(\cI)=\mu$ and is $0$ otherwise (if $\ord_x(\cI)>\mu$).

We shall need the following result:
\begin{lemma}
 Let  $(f_i,d_i)$ be marked functions of maximal orders and consider the generated Rees algebra $\cR^\cdot=\cR^\cdot(f_1,\ldots,f_r)$  and the homogenous Rees algebra $R(\inn_x(f_1),\ldots,\inn_x(f_1))$.
Then $$R^\cdot (\inn_x(f_1),\ldots,\inn_x(f_1))= \inn_x(\cR^\cdot(f_1,\ldots,f_r)$$
\end{lemma}

\begin{proof} It follows by definition that $\inn_x(Df)=D(\inn_x(f))$, for any $D\in \overline{D}^a$, and $f\in\cR$.  Moreover $\inn_x$ preserves products and sums of the marked functions in Rees algebra $\cR$.

\end{proof}


\begin{theorem} \label{Rees} Assume $X$ is a smooth scheme of finite type over the ground field $K$ of the characteristic $0$ (respectively a complex manifold). Let $\cI$ be a coherent sheaf of ideals on $X$, and $(f'_i,d_i)$ be a  standard pre-basis of $\cI$ along a Samuel Stratum $S$. Then the Rees algebra $\cR^\cdot(\cI)=\cR^\cdot(f'_1,\ldots,f'_k)$ and the Rees ideal $\cJ^\cdot(\cI):=\cJ^\cdot(f'_1,\ldots,f'_k)$ are independent of  choice of   standard basis of $\cI$ in a neighborhood of the Samuel stratum $S$.
\end{theorem}
\begin{proof} Let $f'_1,\ldots,f'_r$, and $g'_1,\ldots,g'_m$ be two  standard  pre-bases of $\cI$ along $S$ on a Zariski open neighborhood of $x$ corresponding to  two  standard bases  
 $f_1,\ldots,f_r$, and $g_1,\ldots,g_m$ of $\cI$ along $S$ on a common \'etale neighborhood $X'$ preserving the residue field of $x\in X$ (see  Definition \ref{weak2}). Denote by $u_1,\ldots,u_n$ , and $v_1,\ldots,v_n$ the  corresponding compatible coordinates. We can assume here that $u_1,\ldots,u_s$ are distinguished , and $u_1,\ldots,u_k$ are essential, with $s\leq k\leq n$.
By symmetry it suffices to show that $f_i\in \cJ^{\mu_i}(g_1,\ldots,g_s)$. We will show that $f_i$ is in the completion of $\widehat{\cJ^{\mu_i}}(g_1,\ldots,g_s)_x \subset \widehat{\cO_{X',x}}=\widehat{\cO_{X,x}}$ in a  neighborhood of any $x\in S=\cosupp(\cJ^\cdot)$.

First observe that the initial forms $\inn_x(f_1),\ldots,\inn_x(f_r)$, and $\inn_x(g_1),\ldots,\inn_x(g_s)$  form  two different  bases of
the initial ideal $\inn_x(\cI)$. Then we can find the essential linear forms $$\overline{u}_1=\inn_x(u_1)=\inn_x(\widetilde{u}_1),\ldots, \overline{u}_s=\inn_x(\widetilde{u}_s)$$ in the grading $$R^1(\inn_x(I))=R^1(\inn_x(g_1),\ldots,\inn_x(g_s)),$$ 
for a certain coordinates $\widetilde{u}_1,\ldots,\widetilde{u}_k\in \cR^1(g_1,\ldots,g_s)$, and
$\widetilde{f}_i\in \cR^{d_i}(g_1,\ldots,g_s)$, such that $$(\inn_x(\widetilde{f}_i))(\overline{u}_1,\ldots,\overline{u}_s,0,\ldots,0)=\inn_x({f}_i)({u}_1,\ldots,{u}_s,0,\ldots,0).$$
This implies that $(\widetilde{f}_i)$ satisfies (in particular) the condition (4) of the weak standard basis. 
Using it one can perform in the completion ring $\widehat{\cO_{X,x}}$ the following  Euclidean divison algorithm.
Consider the function $\widetilde{f}_j$. 
Its initial form coincides with that of $f_j$. We shall modify $\widetilde{f}_j$ to get $f_j$ with all intermediate steps performed in $\widehat{\cJ^{\mu_i}}(g_1,\ldots,g_s)_x$. We just
need to eliminate all the higher degree monomials	in $\widetilde{f}_j$ which are not in $\Delta\times \NN^{n-s}$.
Set $h_0:=\widetilde{f}_j$

For any natural $s$ consider the vector space $V_s$ spanned by the ordered set of forms $$V_s:=\{ u^\alpha\mid \alpha \in \Delta=\bigcup \Delta_i, |\alpha|=s\},$$
and the natural projection $$\pi_s: K[[u_1,\ldots,u_n]]\to V_s.$$
 
 Let $$T^s:=[t_{\alpha, \beta}]=[\ldots,\pi_s(\widetilde{f}_\alpha),\ldots]$$ be the square matrix whose subscripts are labeled  by  the ordered set $$\Delta^s=:\{ \alpha \in \Delta , |\alpha|=s\}$$ and containing as $\alpha=\alpha'+\alpha_i$- column the vector $\pi_s(\widetilde{f}_\alpha)$, where $\widetilde{f}_\alpha:=\widetilde{u}^{\alpha'}\widetilde{f}_i$, for  $\alpha\in \Delta^s, \alpha=\alpha_i+\alpha'\in \Delta^s_i\subset \Delta^s$, with $\alpha_i$ the vertex of $\Delta_i$.  Then it follows from condition (4) of Definition \ref{weak2} and Stabilization Theorem that the matrices $T^s$ are invertible and let  ${(T^s)}^{-1}:=[r^s_{\alpha, \beta}]$.

For any $h=\sum_{(\alpha,\gamma)\in \NN^s\times \NN^{n-s}} \, \, c_{\alpha,\gamma} u^{\alpha,\gamma}\in K[[u_1,\ldots,u_n]]=\widehat{\cO_{X,x}}$ put $$\mu(h):=\inf\{|\alpha|+|\gamma|, |\gamma|,\gamma) \mid \quad  \alpha \in \Delta, \quad c_{\alpha,\gamma} \neq 0\} = (\beta,\gamma)$$ (with lexicographic order)
and let $s:=|\beta|$, and $t=|\gamma|$. Then for $$\overline{h}_1:=h_0-c_{\beta,\gamma} u^{0,\gamma}\sum_{\alpha\in \Delta^s} \,  r_{\beta, \alpha} \,\, \widetilde{f}_\alpha $$  we have   $\mu(\overline{h_1})> \mu(h)$. 
  
  Likewise since $\inn_x(D_{u_\beta}h_0)=c_{\beta,\gamma} u^{0,\gamma}+ \sum_{\gamma'>\gamma} a_{\gamma'}u^{0,\gamma}$
  the same is true for the function $$h_1:=h_0-D_{u_\beta}h_0\sum_{\alpha\in \Delta^s} \,  r_{\alpha, \beta} \,\, \widetilde{f}_{\alpha}.$$ The latter function remains to be in $\widehat{\cJ^{\mu_i}}(g_1,\ldots,g_s)_x$ since $D_{u_\beta}h_0\in \cJ^{\mu_i-s}$ and  $\widetilde{f}_{\alpha}:=\widetilde{u}^{\alpha'}\widetilde{f}_i\in  \cJ^s(g_1,\ldots,g_s)_x$ for any $\alpha\in \Delta^s$.
  
   This defines a  convergent sequence $(h_n)\to h_\infty\in \widehat{\cJ^{\mu_i}}(g_1,\ldots,g_s)_x$. Moreover the function $h_\infty\in \widehat{\cI_x}$ has a form $$h_\infty=u^{\alpha_j}+r_\infty$$, with $\cosupp(r_\infty)\in \Gamma\times \NN^s$.
By  uniqueness of the  standard basis (Corollary \ref{main014}) we deduce that $$f_j=u^{\alpha_j}+r_\infty=h_\infty\in \widehat{\cJ^{d_j}}(g_1,\ldots,g_s)_x,$$
Likewise $f'_j\in \widehat{\cJ^{d_j}}(g_1,\ldots,g_s)_x=\widehat{\cJ^{d_j}}(g'_1,\ldots,g'_s)_x$
and that $f'_j\in {\cJ^{d_j}}(g'_1,\ldots,g'_s)_x$.

\end{proof}

\begin{definition} \label{Rees0}  Let $\cI$ be a coherent sheaf of ideals on a smooth scheme $X$ over a field $K$. Consider a Samuel stratum $S$ of $\cI$ on $X$, and  a  weak standard pre-basis $(f'_i,d'_i)$ of $\cI$.

 We shall call the (multiple) marked ideal  $\cR^\cdot(\cI)=\cR^\cdot(f'_1,\ldots,f'_k)$ (respectively $\cJ^\cdot(\cI)$) {\it the canonical Rees algebra algebra along Samuel stratum} (respectively {\it canonical Rees ideal}) of $\cI$ along a Samuel stratum $S$. 
\end{definition}

\begin{remark} Different approaches to Rees algebras were considered were considered by Hironaka, Villamayor, and Kawanoue-Matsuki. They considered, in particular, an additional saturation given by the integral closure of Rees algebra.   Bravo-Garcia Escamilla-Villamayor show in \cite{BV} that the integral (and differential) closure determines a unique canonical Rees algebra defined by its equivalence class. This was then applied to Hironaka's construction of distinguished data as in \cite{Hir2}.
On the other hand Hironaka \cite{Hir4}, Bravo-Garcia Escamilla-Villamayor \cite{BV} and also Kawanoue-Matsuki \cite{Kawanoue}, \cite{KM} show that such an algebra is  finitely generated.
Finally Bierstone-Milman \cite{BM2}, \cite{BM22} do not consider any saturation and use equivalence relation instead in their inductive arguments.

\end{remark}

\bigskip
\subsection{Relative Rees algebras} From now on we consider the case of ground field of characteristic $0$.

\begin{theorem}\label{central} Let $\cI$ be any reduced ideal on $(X,E)$. Let $H=H_x(\cI)$ be the maximal value of the Hilbert function, and let $\cR^\cdot(\cI)$ denote  canonical Rees ideal.
 Then  any resolution of $\cR^\cdot(\cI)$ defines a resolution of $(\cI,H)$. Moreover 
 \begin{enumerate}
\item  The ideal $\cR^\cdot(\cI)$ is stable with respect to any \'etale morphisms and field extensions.
\item The centers of the admissible blow-ups of $\cR^\cdot(\cI)$ are contained in the Samuel strata of the ideals of the strict transforms of $\cI$, and thus are normally flat.
\end{enumerate}
	
\end{theorem}

\begin{proof} The Rees $\cR^\cdot(\cI)$ ideal is generated locally by the  standard basis $(f_i,d_i)$. The support of the standard basis defines the Samuel stratum of $\cI$ for the  value $H$ of the Hilbert-Samuel function. Moreover the controlled transforms $(f'_i,d_i):=\sigma^c(\{(f_i,d_i)\})$ of the  standard basis $(f_i,d_i)$ remain the  standard basis of the strict transform $\cI'$ of $\cI$. This implies that
  resolution of marked ideal $\{(f_i,d_i)\}$ defines
  a resolution of $(\cI,H)$. On the other hand by Lemma \ref{le: inclusions} $$\sigma^c(\{(f_i,d_i)\})\subseteq \sigma^c(\cR^\cdot(\cI))\subseteq \cR^\cdot(\sigma^c(\{(f_i,d_i)\})$$
  The latter implies that the Samuel stratum $S'$ of $\cI'$ can be described as $$S'=\cosupp\{(f'_i,d'_i)\}=\cosupp(\cR^\cdot(f'_i,d'_i))=\cosupp(\sigma^c(\cR^\cdot(\cI))$$
\end{proof}

\begin{definition}  Let $\cI$ be any ideal on a smooth $X$, and 
$\cR^\cdot=\cR^\cdot(I)$ be its Rees Algebra along a Samuel stratum $S$. Let $E=\{D_1,\ldots,D_k\}$ be a set of SNC divisors. By the {\it relative Rees} we shall mean the graded algebra $\cR^\cdot(\cI,E)$ generated by $\cR^\cdot$ and all $(\cI_{D_i},1)$.

	\end{definition}
It follows immediately that
$$\cosupp(\cR^\cdot(\cI,E))=\cosupp (\cR^\cdot(\cI))\cap E$$

Moreover this relation is preserved by the blow-ups with the center in $\cosupp (\cR^\cdot(\cI))\cap E$.

 \begin{proposition} \label{restric} Let  $X_1\hookrightarrow X_2$ be a closed embedding of a smooth schemes over $K$, and $Y\hookrightarrow X_1$ be a closed embedding of  reduced schemes. Assume that there exists a (possibly empty) set $E_2$ of $SNC$ divisors on $X_2$ which is transversal to $X_1$, and denote by $E_1$ its restriction to $X_2$. 

Let $\cI_1:=\cI_{Y,X_1}$, and   $\cI_2:=\cI_{Y,X_2}$ be the sheaves of ideals of $Y$ on $X$ and $X_2$, and suppose that $X_1$ is locally described by the vanishing locus of the set of parameters $u_1,\ldots,u_k$. Then there is a certain \'etale extension of $X_2\supset X_1$  the  such that
\begin{enumerate}
\item There is an inclusion of the sheaves $\cO_{X_1}\subset \cO_{X_2}$

\item There is an inclusion of the relative Rees algebras  $\cR^\cdot_1:=\cR^\cdot(\cI_1,E_1)\subset \cR^\cdot_2:=\cR^\cdot(\cI_2,E_2)$.

\item The Rees algebra $R_2$  is locally generated by $(u_1,1),\ldots,(u_k,1)$ and $R_1\subset R_2$, where $u_1=0,\ldots,u_k=0$ is a set of parameters describing $X_1$ on $X_2$.
 
\item The restriction of Rees algebra $R_2$ to $X_1$ coincides with $R_1$.
\item The relations above are preserved for the controlled transforms of $R_2$, $R_1$ and $u_1,\ldots,u_k$  under admissible blow-ups of   $R_2$.
\end{enumerate}

\end{proposition}
\begin{proof}

First we show the proposition in the case $E=\emptyset$. 	

Consider a system of parameters $u_1,\ldots,u_k,u_{k+1},\ldots,u_n$ on $X_2$, where $u_1,\ldots,u_k$ describe $X_1$. By considering a
division by $u_1,\ldots,u_k$ with respect to the standard monotone diagram $\Delta$ generated by the basis $e_1,\ldots,e_k$ we see that in the certain \'etale neighborhood $\cO_{X_2}$ can be identified with $$\cO_{X_1}^\Gamma=\{f\in \cO_{X_1}\mid D_{u_i}(f)=0, i=1,\ldots,k\}\subset \cO_{X_1}.$$ (In the characteristic $p$, we shall use the condition is $D_{u_i^{p_j}}(f)=0$).

Choosing the monotone order for the coordinates $u_{k+1},\ldots,u_n$ one can find a monotone diagrams $\Delta_1\in \NN^{n-k}$ and $\Delta_2=[0,1]^k\times \Delta_1\in \NN^{n}$ defined by $\cI_1$ and $\cI_2$.
Then one can find a  standard basis (passing to \'etale neighborhood)  of $\cI_2$  of the form
$$u_1,\ldots,u_k,f_1(u_{k+1},\ldots,u_n),\ldots,f_r(u_{k+1},\ldots,u_n),$$ with respect to $\Delta$, such that  $f_1(u_{k+1},\ldots,u_n),\ldots,f_r(u_{k+1},\ldots,u_n)$ is a  standard basis of $\cI_1$. Then it follows that the Rees $\cR^\cdot_1:=\cR^\cdot(\cI_2)$ algebra of $\cI_2$ is  generated by $(u_1,1),\ldots,(u_k,1)$ and $\cR^\cdot_1:=\cR^\cdot(\cI_1)\subset \cR^\cdot_2$. Then $\cR^\cdot_2$ is nothing but the restriction of $\cR^\cdot_2$ to $X_1$ with the subscheme $X_1\subset X_2$  descried exactly by the support of $(u_1,1),\ldots,(u_k,1)$. Moreover theses relations between the standard bases and the Rees algebra are  preserved under the blow-ups contained in the Samuel strata. The controlled transform $(u'_1,1),\ldots,(u'_k,1)$ describes the
strict transform $X'_1$ on $X'_2$, and the controlled  transform of the Rees algebra $(\cR^\cdot_1)'$ is the restriction of $(\cR^\cdot_2)'$ to $X'_2$. Moreover  $(\cR^\cdot_2)'$ is generated by $(u'_1,1),\ldots,(u'_k,1)$ and $(\cR^\cdot_1)'$.

The general case follows
since $\cR^\cdot(\cI_i,E_i)$, where $i=1,2$, is generated by $\cR^\cdot(\cI_i)$ and $(\cI_{D_j},1)$. Moreover by   part (1), the functions $x_j\in \cO_{X_2}$ defining $D_j$ on $X_2$ and their restrictions to $X_1$ can be identified under the  inclusion $\cO_{X_1}\subset \cO_{X_2}$.

	\end{proof}

\subsection {Minimal embedding spaces  and Samuel stratum} 
When considering Samuel strata of an ideal $\cI_Y$ on a smooth variety $X$ it is convenient to consider locally minimal embedding spaces. They define  ''optimal embedding'' of the smallest dimension and are canonical in the sense of the following lemma 
\begin{definition} Let $\cI$ be any ideal on a smooth $X$, with a SNC divisor $E$ and defined for a Samuel stratum $S$ and $E$. 

The maximal set  of parameters $u_1,\ldots,u_k\in \cR^\cdot_1$ transversal to $E=\{D_1,\ldots,D_k\}$ defines locally a smooth {\it minimal embedding space} $T$ containing $S$ (of codimension $k$ which is locally constant on $S\cap\bigcap D_i$ ). 	
\end{definition}

The following lemma is a direct extension of Lemma \ref{le: homo}.

\begin{lemma}  \label{le: homo2} Let $\cI$ be any ideal on a smooth $X$,  with SNC divisor $E$ and let 
$\cR^\cdot=\cR^\cdot(I,E)$ be its canonical Rees Algebra along a Samuel stratum $S$. 
	For any point in $x\in S$ consider a maximal set of parameters $u_1,\ldots,u_k\in \cR^1$ which is a part of the  standard basis in a neighborhood of $x$ and is transversal to $E$. Then for any two sets of $u_1,\ldots,u_k\in \cR^1$ and $v_1,\ldots,v_k\in \cR^1$
  there  exist \'etale neighborhoods $\phi_{u},\phi_v: \overline{X}\to X$ of $x=\phi_u(\overline{x})=\phi_v(\overline{x}) \in X$,  where $\overline{x}\in \overline{X}$, such that
\begin{enumerate}
\item  $\phi_{u}^*({\cR^\cdot})=\phi_{v}^*({\cR^\cdot})$.
\item  $\phi_{u}^*(E)=\phi_{v}^*(E)$.
\item  $\phi_{u}^*(u_i)=\phi_{v}^*(v_i)$.

\end{enumerate}
\end{lemma}
\begin{proof} (0) (0) {\bf Construction of  \'etale neighborhoods  ${\phi}_{u}, {\phi}_{v}: U\to X$.}

Let $U\subset X$ be an open subset for which there exist $u_{k+1},\ldots, u_n$ which are transversal to $u_1,\ldots,u_k$ and $v_1,\ldots,v_k$ on $U$ such  that
 $u_1,u_2,\ldots, u_n$  and $v_1,\ldots,v_k,u_{k+1},\ldots, u_n$ form  two sets of parameters on $U$ and divisors in $E$ are described by some $u_i$, where $i\geq 2$.
 Let ${\bf A}^n$ be the affine space with coordinates $x_1,\ldots,x_n$.
 Construct first \'etale morphisms $\phi_1,\phi_2: U\to {\bf A}^n$ with
 \begin{equation} {\phi}^*_{1}(x_i)=u_i \quad \textrm{for all } i \quad\quad \mbox{and} \quad
{\phi}^*_{2}(x_i)=v_i, \quad\mbox{for} \quad i\leq k\quad  {\phi}^*_{2}(x_i)=u_i\quad\mbox{for} \quad i>k.\nonumber \end{equation}

\noindent Then \begin{equation} \overline{X}:=U\times_{{\bf A}^n}U \nonumber \end{equation} \noindent is a fiber product for  the morphisms $\phi_1$ and $\phi_2$. The morphisms $\phi_u$, $\phi_v$ are defined to be the natural projections $\phi_u, \phi_v: \overline{X}\to U$ such that $\phi_1\phi_u=\phi_2\phi_v$. Set
\begin{equation} w_i:=\phi_u^*(u_i)=(\phi_1\phi_u)^*(x_1)=(\phi_2\phi_v)^*(x_i)=\phi_v^*(v_i),\quad  \textrm{for $i\leq k$} \nonumber\end{equation}
 \begin{equation} w_i=\phi_u^*(u_i)=\phi_v^*(u_i)\quad  \textrm{for $i\geq k+1$}. \nonumber\end{equation}

One can extend the ground field algebraically and assume that the residue fields of the points above $x$ and at $x$ are the same. Then we can construct an automorphism $\widehat{\phi}_{uv}=\widehat{\phi}_{uv}\widehat{\phi}^{-1}_{u}$ such that  $\widehat{\phi}_{uv}(u_i)=v_i$ for $i\leq k$. 

(1)  Let $h_i:=v_i-u_i\in {\cR^1({\cI})}$. For any  $f\in \widehat{\cR^s}$,
 \begin{equation} \widehat{\phi}_{uv}^*(f)=f(u_1+h_1,\ldots,u_k+h_k, u_{k+1},\ldots,u_n)= f(u_1,\ldots,u_n)+\sum_{i\leq k} \frac{\partial{f}}{\partial{u_i}}\cdot h_i+ \sum_{ij\leq k}\frac{1}{2!}
 \frac{\partial^2{f}}{\partial{u_iu_j}}\cdot h_ih_j+\ldots + \nonumber \end{equation}
 The latter element belongs to \begin{equation}\widehat{\cR^s}+\widehat{\cR}_{s-1}\cdot \widehat{\cR^1}+\ldots +\widehat{\cR}_{s-2}\cdot \widehat{\cR^2}+ \ldots =\widehat{\cR^s}.\nonumber\end{equation}

Hence $\widehat{\phi}_{uv}^*(\widehat{\cR})\subset {\cH}\widehat{\cR}$. \bigskip
which implies that $\widehat{\phi}_{u}^*(\widehat{\cR})=\widehat{\phi}_{v}^*(\widehat{\cR})$, and locally ${\phi}_{u}^*({\cR})={\phi}_{v}^*({\cR})$.

 (2), (3) follow from the construction

(4) Let $h_i:=v_i-u_i$ for $i\leq k$. By the above the morphisms $\phi_u$ and $\phi_v$ coincide on $\phi_u^{-1}(V(h_1,\ldots,h_k))=\phi_v^{-1}(V(h_1,\ldots,h_k))$.

By (4) the blow-ups of the centers $C\subset \cosupp(\cH(\cI))$  lifts to the  blow-ups at the 
same center $\phi_u^{-1}(C)=\phi_u^{-1}(C)$. Thus (5), (6) follow.
\end{proof}

\subsection{Equivalence  for marked ideals and Capacitors}

Let us introduce the following equivalence relation for marked ideals:
\begin{definition} We say that two (multiple) marked ideals  $\{(X,{\cI}_i,E,\mu_i)\}$ and $\{(X,{\cJ}_j,E,\mu_j)\}$ 
on  a smooth variety $X$ with SNC collection of divisor $E$, are equivalent:
$$\{(X,{\cI}_i,E,\mu_i)\}\simeq \{(X,{\cJ}_j,E,\mu_j)\}\quad \mbox{if}$$ 
\begin{enumerate}
\item $\cosupp\{(X,{\cI}_i,E,\mu_i)\}=\cosupp\{(X,{\cJ}_j,E,\mu_j)\}$
\item The sequences of admissible   blow-ups $(X_i)_{i=0,\ldots,k}$ are the same for both marked ideals and \\
$\cosupp\{(X,({\cI}_i)_k,E,\mu_i)\}=\cosupp\{(X,({\cJ}_j)_k,E,\mu_j)\}$
\end{enumerate}
\end{definition}
\begin{definition}
We write $$\{(X,{\cI}_i,E,\mu_i)\}\subseteq \{(X,{\cJ}_j,E,\mu_j)\}\quad \mbox{if}$$ 

\begin{enumerate}
\item $\cosupp\{(X,{\cI}_i,E,\mu_i)\}\supseteq \cosupp\{(X,{\cJ}_j,E,\mu_j)\}$
\item The sequences of admissible   blow-ups $(X_i)_{i=0,\ldots,k}$ of $\{(X,{\cJ}_j,E,\mu_j)\}$ is admissible for $\{(X,{\cI}_i,E,\mu_i)$  and 
$\cosupp\{(X,({\cI}_i)_k,E,\mu_i)\}\supseteq \cosupp\{(X,({\cJ}_j)_k,E,\mu_j)\}$

\end{enumerate}
\end{definition}

\begin{example}
For any $k\in {\bf N}$, $({\cI},\mu)\simeq ({\cI}^k,k\mu)$.
\end{example}
\begin{remark} The marked ideals considered in this paper satisfy a stronger equivalence condition:
For any smooth morphisms $\phi: X' \to X$, $\phi^*(\cI,\mu)\simeq\phi^*(\cJ,\mu)$.
This condition will follow and is not added in the definition.
 \end{remark}

Assume now that all marked ideals are defined for the smooth variety $X$ and the same
set of exceptional divisors $E$.
Define the following operations of addition and multiplication of marked ideals:
\begin{enumerate}
\item 

$({\cI}_1,\mu_1)+\ldots+({\cI}_m,\mu_m):=({\cI}_1^{\mu_2\cdot\ldots\cdot\mu_m}+
{\cI}_2^{\mu_1\mu_3\cdot\ldots\cdot\mu_m}+\ldots+{\cI}_m^{\mu_1\ldots\mu_{k-1}} ,\mu_1\mu_2\ldots\mu_m).$
\item $({\cI}_1,\mu_1)\cdot\ldots\cdot({\cI}_m,\mu_m):=({\cI}_1\cdot\ldots\cdot{\cI}_m,\mu_1+\ldots+\mu_m)$
\end{enumerate}
It follows from definition 
\begin{lemma} \label{le: operations}
\begin{enumerate} 
\item $({\cI}_1,\mu_1)+\ldots+({\cI}_m,\mu_m))\simeq \{({\cI}_1,\mu_1),\ldots,({\cI}_m,\mu_m))\}$


\item $({\cI}_1,\mu_1)\cdot\ldots\cdot({\cI}_m,\mu_m)\subseteq \{({\cI}_1,\mu_1),\ldots,({\cI}_m,\mu_m))\}$
 
\end{enumerate}
\end{lemma}

\begin{lemma}\label{le: coeff} For any multiple marked ideal $\overline{\cI}=\{\cI_i,\mu_i)\}$  the induced Rees algebra
 $\cR^\cdot(\overline{\cI})$ is equivalent to $\overline{\cI}$.

\end{lemma}
\begin{proof} If $\ord_x(\cI_i)\geq \mu_i$ then $\ord_x(\cR^i)\geq i$, since the operations of differentiation and product preserve the relevant inequalities. Moreover by Lemma \ref{le: inclusions} the inequalities are preserved by the controlled transforms of Rees algebras. 
\end{proof}

One of the disadvantages  of the canonical Rees algebra is that it consists of infinitely many marked ideals which is slightly inconvenient for
our presentation of  the resolution algorithm.

One can easily remedy this by introducing the capacitors of Rees algebras.
\begin{definition}\label{capacitor} Let $R^\cdot=\bigoplus R^i$ be a finitely generated Rees algebra on a smooth scheme $X$. By its {\it capacitor} we mean any gradation ideal $R^i\subset \cO_X$, such that the marked ideal $(R^i,i)$ is equivalent  to $R^\cdot$.
\end{definition}

In practice it means that, when convenient, we can translate the problems of resolution of Rees algebras to their capacitors containing the essential information about Rees algebras.
\begin{proposition} The following equivalence holds true:

{\bf Existence of resolution of finitely generated Rees algebras on smooth schemes}

\centerline{ $\Updownarrow$}

{\bf Existence of resolution of marked ideals  on smooth schemes.}
	
\end{proposition}

The following rather obvious but useful observation was made, in particular by Bravo-Garcia Escamilla-Villamayor in \cite{BV}:
\begin{proposition}(\cite{BV}) Any finitely generated Rees algebra admits its capacitor.
	
\end{proposition}
\begin{proof} If $R^\cdot$ is generated by a finitely many marked functions in gradation $i_j$ then it is generated by $R^{i_j}$ and thus equivalent to 
$\sum_j R^{i_j}$. But $\sum_j R^{i_j}\subset R^a$ , where $a$ is the product of of $i_j$. This implies that $R^\cdot$ is equivalent to $R^a$.

\end{proof}

There are various methods of finding some canonical or minimal capacitor of Rees algebras. 
In case of the Rees algebras along strata let $n$ be the dimension of the minimal embedding spaces of $\cI$ along the stratum $S$.
Let $a$ be the maximal multiplicity of all possible vertices of all possible diagrams $\Delta\in \NN^n$. It is finite by Corollary \ref{dcc2}.
Then its {\it canonical capacitor} is defined as $\overline{R}=R^{a!}$.


\section{Strong Hironaka desingularization in characteristic zero}

\subsection{Formulation of Hironaka's resolution theorems}.

We give a proof of the following version of the Hironaka non-embedded resolution

\begin{enumerate}
\item{\bf Strong Canonical Hironaka's Resolution with normally flat centers}
\begin{theorem} \label{th: 3} Let $Y$ be an algebraic variety over a field of characteristic zero.

There exists a canonical desingularization of $Y$ that is
a smooth variety $\widetilde{Y}$ together with a proper birational morphism $\res_Y: \widetilde{Y}\to Y$ such that
\begin{enumerate}
\item  $\res_Y$ is a composition  of blow ups $Y=Y_0\leftarrow Y_1\leftarrow\ldots\leftarrow Y_k=\widetilde{Y}$ with smooth centers $C_i$, and exceptional divisors $E_i$.

\item The centers  $C_i$ are   either contained in the set of  singular  points  $Sing(Y_i)$ of $Y_i$, or  if $Y_i$ are smooth  in the exceptional divisor $D_i$. 

\item The centers $C_i$ are normally flat on $Y_i$ that is are contained in the Samuel stratum of $Y_i$.

\item The variety $\tilde{Y}=Y_k$ is nonsingular and the inverse image of the singular locus $\Sing(Y)$ is a simple normal crossing divisor exceptional divisor is SNC divisor on  $\tilde{Y}$.

\item $\res_Y$ is functorial with respect to smooth morphisms, and the field extensions, and it  is equivariant with respect to any group action not necessarily preserving the ground field.



\end{enumerate}
\end{theorem}


\item {\bf  Strong Hironaka's Embedded  Desingularization} 
\begin{theorem} \label{th: emde} \label{th: 2} Let $Y$ be a closed subvariety of a smooth variety
$X$ over a field of characteristic zero, and  $E$ be a (possibly zero) SNC divisor on $X$.
There exists a  sequence $$ X_0=X \buildrel \sigma_1 \over\longleftarrow X_1
\buildrel \sigma_2 \over\longleftarrow X_2\longleftarrow\ldots
\longleftarrow X_i \longleftarrow\ldots \longleftarrow X_r=\widetilde{X}$$ of
blow-ups  $\sigma_i:X_{i-1}\longleftarrow X_{i}$ of smooth centers $C_{i-1}\subset
 X_{i-1}$ such that

\begin{enumerate}

\item The union $E_i$ of the exceptional divisor  of the induced morphism $\sigma^i=\sigma_1\circ \ldots\circ\sigma_i:X_i\to X$ and of the strict transform of the divisor $E$ has only  simple normal
crossings and $C_i$ has simple normal crossings with $E_i$.

\item Let $Y_i\subset X_i$ be the strict transform of $Y$.  The centers  $C_i$ are   either contained in the set of  singular  points  $Sing(Y_i)$ of $Y_i$, or if $Y_i$ are smooth, in the exceptional divisor $E_i$.

\item  The strict transform $\widetilde{Y}:=Y_r$ of  $Y$  is smooth and
has only simple normal crossings with the  divisor $E_r$.

\item The morphism $(X,{Y})\leftarrow (\widetilde{X},\widetilde{Y})$ defined by the embedded desingularization commutes with smooth morphisms, field extensions, and embeddings of ambient varieties. It is equivariant with respect to any group action not necessarily preserving the ground $K$.

 \end{enumerate}

\end{theorem}
\item{\bf Canonical Hironaka's Principalization} 
\begin{theorem} \label{th: 1} Let ${\cI}$ be a sheaf of ideals on a smooth algebraic variety $X$ over a field of characteristic zero, $E$ be a (possibly zero) SNC divisor on $X$, and  $Y\subset X$ be any closed subvariety of $X$.
There exists a principalization of ${\cI}$  that is, a sequence

$$ X=X_0 \buildrel \sigma_1 \over\longleftarrow X_1
\buildrel \sigma_2 \over\longleftarrow X_2\longleftarrow\ldots
\longleftarrow X_i \longleftarrow\ldots \longleftarrow X_r =\widetilde{X}$$

of blow-ups $\sigma_i:X_{i-1}\leftarrow X_{i}$ of smooth centers $C_{i-1}\subset
 X_{i-1}$
such that

\begin{enumerate}

\item The union $E_i$ of the exceptional divisor  of the induced morphism $\sigma^i=\sigma_1\circ \ldots\circ\sigma_i:X_i\to X$ and of the strict transform of the divisor $E$
has only  simple normal
crossings and $C_i$ has simple normal crossings with $E_i$.

\item The
total transform $\sigma^{r*}({\cI})$ is the ideal of a simple normal
crossing divisor
$\widetilde{E}$ which is  a natural  combination of the irreducible components of the divisor ${E_r}$.

\end{enumerate}
The morphism $(\widetilde{X},\widetilde{\cI})\rightarrow(X,{\cI}) $ defined by the above principalization  commutes with smooth morphisms and  embeddings of ambient varieties. It is equivariant with respect to any group action not necessarily preserving the ground field $K$.

\end{theorem}

\end{enumerate}

\begin{remark} Note that the blow-up of codimension one components  is an isomorphism. However it defines a nontrivial transformation of marked ideals. In the actual desingularization process this kind of blow-up may occur for some marked ideals induced on subvarieties of ambient varieties. Though they define isomorphisms of those subvarieties they determine blow-ups of ambient
varieties which are not isomorphisms.
\end{remark}

\begin{remarks} \begin{enumerate}
\item By the exceptional divisor of the blow-up $\sigma: X'\to X$ with a smooth center $C$ we mean the inverse image $E:=\sigma^{-1}(C)$ of the center C. By the exceptional divisor of the composite of blow-ups $\sigma_i$ with  smooth centers $C_{i-1}$ we mean the union of the strict transforms of the exceptional divisors of $\sigma_i$. This definition coincides with the standard definition of the exceptional set of points of the birational morphism in the case when $\codim(C_i)\geq 2$ (as in Theorem  \ref{th: 2}). If $\codim(C_{i-1})=1$ the blow-up of $C_{i-1}$ is an identical isomorphism and defines a formal operation of converting a subvariety $C_{i-1}\subset X_{i-1}$ into a component of the exceptional divisor $E_i$ on $X_i$. This formalism is convenient  for  the proofs. In particular it indicates that $C_{i-1}$ identified via $\sigma_i$ with  a component of $E_i$ has simple normal crossings with other components of $E_i$. 
\item In the Theorem \ref{th: 2} we blow up centers of codimension $\geq 2$ and both definitions coincide.
\end{enumerate}
\end{remarks}



\subsection{Hironaka resolution principle}

The proof of the strong embedded and nonembedded Hironaka resolution with normal crossing centers builds upon the  following theorem connecting resolution of marked ideals and strong desingularization of varieties. 

\begin{definition} \label{th: 4} Let $K$ be a base field. By  a {\it canonical resolution of marked ideals} we mean a functor which  associate with any marked ideal
$(X,\cI,E,\mu)$ over $K$  a  unique resolution $(X_i)$. such that 
\begin{enumerate}
\item For any smooth 
morphism $\phi: X'\to X$ the induced resolution $\phi^*(X_i)$  defines  the canonical resolution of  $\phi^*(X,\cI,E,\mu)$.

\item $(X_i)$ commutes with (separable) ground field extensions. 
\end{enumerate}
\end{definition}
\begin{remark} The canonical resolution of marked ideals constructed here does not commute with embedding of ambient varieties. Such a commutativity holds only for the marked ideals of the special type $(X,\cI,\emptyset,1)$.
\end{remark}

\begin{theorem} \label{Red1} Assume $char(K)=0$. Then the following implications hold true: 
\begin{eqnarray}  &\textbf{(1) Canonical resolution of  marked ideals}\nonumber\\ 
&\textbf{$\Downarrow$}\nonumber\\
&\textbf{(2) Canonical strong  Embedded
Desingularization with normally flat centers}\nonumber\\
&\textbf{$\Downarrow$}\nonumber\\
&\textbf{(3) Canonical strong Desingularization with normally flat centers}\nonumber \end{eqnarray}
\end{theorem}
\begin{remark} The proof of the first implication is more straightforward and shorter if we do not impose that the desingularization commutes with embeddings. This property is useful in the proof of the second implication. To ensure this condition  we use the concept of the minimal embedding spaces.
	
\end{remark}

\begin{proof} 
(1)$\Rightarrow$(2) In the considerations below by resolving multiple marked ideal defined by Rees algebra $R^\cdot$ we shall mean the canonical resolution of their canonical capacitors $\overline{R}$. 
Let $Y$ be a reduced closed subscheme of a smooth scheme $X$ over a field $K$, and $E=\{D_1,\ldots,D_k\}$ be SNC divisor on $X$. Denote by $\cI:=\cI_Y\subset \cO_X$ the coherent sheaf of ideals of $Y$. Let $H_1:=H_\cI$ be the maximal value of the Hilbert-Samuel function of $\cI$ on $X$, and let $S_1$ be the corresponding stratum.
 Consider a minimal embedding space $T_1$ for $Y$ on $X$ (corresponding to $S_1$) and  let $R(\cI)$ be the canonical Rees algebra.
 The  divisors $E=\{D_1,\ldots,D_k\}$ on $X'$  might not be transversal to the locally defined minimal embedded space $T_1$ 
We consider the relative Rees algebra $\cR^\cdot(\cI,E)$, and its restriction to its minimal embedding space. Its resolution creates a strict transform $\cI'$ of $\cI$
with $$S'_1\cap \bigcup_{i=1}^k E'_i=\cosupp(\cR^\cdot(I'))\cap \bigcup_{i=1}^k E'_i=\emptyset$$
In other words it moves the intersection of $k$ divisors (the strict transfroms $E'_i$ away from the the Samuel stratum $S_1\cosupp(\cR^\cdot(\cI)$.  
 By Lemma \ref{le: homo2} and canonicity, the resolution is independent of the choice of the minimal embedded space.
The strict transform of $\cI'$ describes the strict transform $Y'$ of $Y$ on $X'$ and $T'$. That is $\cI'=\cI_{Y'}$. In particular all the centers are contained in $Y'\subset T'$.

Now, consider   all possible intersections of $k-1$ divisors in $E'$. Note that those intersections  are disjoint and ordered. Then for each subset of $k-1$ divisors $E''$ of $E'$ apply
 the canonical resolution of the restriction of $\cR^\cdot(\cI,E'')$ to its minimal embedding space.
 After that all the intersection of any $k-1$  strict transforms of the divisors will be moved away from the Samuel stratum.
 Repeating the procedure for $k-2,k-3,\ldots,1$ will result in the variety $X_{1A}$ with the Samuel stratum having no intersection with the strict transforms of $E$.
 Let $T_{1A}$ be the  strict transform of $T_1$ on $X_{1A}$ and $R_{1A}$ be the controlled transform of $R_2$. Note that passing from $X_1$ to $X_{1A}$ produces "new" exceptional  divisors $E_{1A}$ from the centers contained in the strict transform of $T_1$. They are transversal to $T_{1A}$, and their restrictions to $T_{1A}$ define SNC divisors $E_{1A}^T$.
 
 This defines a (multiple) marked ideal
 $(X_{1A}, R_{1A}, E_{1A})$ and its restriction $(T_{1A}, R^T_{1A}, E^T_{1A})$
 Resolving canonically $R_{1A}$ restricted to $T_{1A}$ defines a 
 a canonical resolution of $(X_1,\cI_1,H_1)$, that is 
variety $X_2$ with the maximal value of the Hilbert-Samuel function $H_2<H_1$. The centers are normally flat by Theorem \ref{central}.


After the resolution the maximal value of the Hilbert-Samuel function $H_{\cI_2}$ of the strict transform $\cI_2$ of $\cI_1$ on the variety $X_2$ drops  to the value $H_2<H_1$ with the Samuel stratum $S_2$. 
 
 We repeat the procedure for the the strict transform $\cI_2$ on $X_2$ and $H_2$ with the induced divisor $E_2$ and continue the process until the Hilbert-Samuel attains its minimum on the strict transform $Y'$ (or, in a reducible case on each component of the strict transform). This will be done in finitely many steps since by Theorem \ref{dcc} the Hilbert-Samuel function has a d.c.c. property and the sequence $\ldots<H_3<H_2<H_1$ shall stabilize. In the terminal case the Hilbert-Samuel function will attain the minimal value $H_\infty:=H_{smooth}$ for a generic smooth point of $Y'$ for all points of $Y'$. In other words
 $Y'$ becomes smooth.
 
  To ensure that the strict transform $Y'$ has a SNC with the exceptional divisors we shall continue to run the algorithm for resolution of 
 $(\cI,H_{smooth})$  or alternatively its Rees ideal.
 At some point of the procedure the strict transform of $Y$ will vanish. At this moment the algorithm prescribes to blow up the center which coincides (locally) with the (components) of the strict transform of $Y$. That is the strict transform has only
 SNC with the exceptional divisors. We shall stop the algorithm at this stage, and obtain the smooth subscheme $Y'$ having SNC with the exceptional divisors.
 
 It follows from Proposition \ref{restric}, and Lemma \ref{le: homo2} that the resulting resolution commutes with closed embedding \'etale morphisms and field extensions.

(2)$\Rightarrow$(3) A  nonembedded canonical desingularization is obtained by gluing  defined locally embedded canonical desingualarizations.

We shall need the following:
\begin{proposition} For any affine variety $U$ there is a smooth variety $\widetilde{U}$ along with a birational morphism $\res: \widetilde{U}\to U$ subject to the conditions:
\begin{enumerate}
\item For any closed embedding $U\subset X$ into a smooth affine variety $X$, there is a closed embedding $\widetilde{U}\subset \widetilde{X}$ into a smooth variety $\widetilde{X}$ which is a canonical embedded desingularization of $U\subset X$.
\item For any open embedding $V\hookrightarrow U$ there is an open embedding of resolutions $\widetilde{V}\hookrightarrow \widetilde{U}$ which is a lifting of $V\to U$ such that $\widetilde{V} \to \res_U^{-1}(V)$ is an isomorphism over $V$.

\end{enumerate}
\end{proposition}
\begin{proof}
 (1) 
 Consider a closed embedding of $U$ into a smooth affine variety $X$ (for example $X={\bf A}^n$). The canonical embedded desingularization $\widetilde{U}\subset \widetilde{X}$ of $U\subset X$ defines 
the desingularization $\widetilde{U}\to U$. This desingularization is independent of the ambient variety $X$.
Let $\phi_1: U\subset X_1$ and $\phi_2: U\subset X_2$  be two closed embeddings and
let $\widetilde{U}_i\subset \widetilde{X}_i$ be two embedded desingularizations.
Find  embeddings $\psi_i: X_i\to {\bf A}^n$ into the affine space ${\bf A}^n$. They define the embeddings 
$\psi_i\phi_i: U\to {\bf A}^n$. Recall 
\begin{lemma} \label{le: l}(see \cite{Wlodarczyk} or \cite{Jelonek}). For any closed embeddings $\phi_1,\phi_2: Y\subset {\bf A}^{n}$ there exist closed embeddings $\psi_1,\psi_2: {\bf A}^{n}\to {\bf A}^{2n}$ such that $\psi_1\phi_1=\psi_2\phi_2$.
\end{lemma}
By  Lemma \ref{le: l},
there are embeddings $\Psi_i: {\bf A}^n\to {\bf A}^{2n}$ such that
$\Psi_1\psi_1\phi_1=\Psi_2\psi_2\phi_2: U\to {\bf A}^{2n}$. Since embedded desingularizations commute with closed embeddings of ambient varieties we see that the $\widetilde{U}_i$ are isomorphic over $U$. 

(2) Let $V\to U$ be an open embedding of affine varieties. Assume first that $V=U_f=U\setminus V(f)$, where $f\in K[U]$ is a regular function on $U$.
Let $U\subset X$ be a closed embedding into an affine variety $X$. Then
$U_f\subset X_F$ is a closed embedding into an affine variety $X_F=X\setminus V(F)$ where $F$ is a regular function on $F$ which restricts to $f$.
Since embedded desingularizations commute with smooth morphisms the open embedding
$X_F\subset X$ defines the open embedding of embedded desingularizations $(\widetilde{X_F},\widetilde{U_f})\subset (\widetilde{X},\widetilde{U})$ and the open embedding of desingularizations $\widetilde{U_f}\subset \widetilde{U}$.

Let $V\subset U$ be any open subset which is an affine variety. Then there are desingularizations $\res_V: \widetilde{V}\to V$ and $\res_U:\widetilde{U}\to U$.
Suppose the natural birational map $\widetilde{V} \to \res_U^{-1}(V)$ is not an isomorphism
over $V$. Then we can find an open subset $U_f\subset V$ such that 
$\res_V^{-1}(U_f) \to \res_U^{-1}(U_f)$ is not an isomorphism over $U_f$. But $U_f=V_f$ and by the previous case $\res_V^{-1}(U_f)\simeq \widetilde{U_f}=\widetilde{V_f}\simeq \res_U^{-1}(V)$.
\end{proof}

To finish the proof of the second implication let $Y$ be an algebraic variety over $K$. 
By the compactness of $Y$ we find a cover of affine subsets $U_i$ of $Y$ such that each $U_i$ is embedded in an affine space ${\bf A}^{n}$ for $n\gg 0$. We can assume that the dimension
$n$  is the same for all $U_i$  by taking if necessary embeddings of affine spaces ${\bf A}^{k_i}\subset {\bf A}^n$. 

Let $U_i$ be an open affine cover of $X$. For any two open subsets $U_i$ and $U_j$
set $U_{ij}:=U_i\cap U_j$. For any $U_i$ and $U_{ij}$ we find  canonical resolutions $\widetilde{U_i}$ and $\widetilde{U}_{ij}$ respectively.  By the proposition $\widetilde{U}_{ij}$  can be identified with an open subset of $\widetilde{U}_i$. We define $\widetilde{X}$ to be a variety obtained by glueing $\widetilde{U_i}$ along $\widetilde{U}_{ij}$. Then $\widetilde{X}$
is a smooth variety and $\widetilde{X}\to X$ defines a canonical desingularization independent of the choice of $U_{ij}$.
(For more details see \cite{Wlodarczyk}).
\end{proof}

The proof of principalization follows from the following implication
\begin{theorem} The following implication holds true:
\begin{eqnarray}\nonumber  &\textbf{(1)  (Canonical) Resolution of relative marked ideals $(X,\cI,E,\mu)$}\nonumber\\ &\textbf{$\Downarrow$}\nonumber\\
& \textbf{(2) (Canonical)  Principalization of the sheaves $\cI$ on $X$\nonumber}\nonumber
\end{eqnarray}\nonumber
\end{theorem}

\begin{proof}
  It follows immediately from the definition that a resolution of $(X,{\cI},E,1)$
determines a principalization of ${\cI}$. Denote by $\sigma: X\leftarrow \widetilde{X}$ the morphism defined by a resolution of $(X,\cI,E,1)$. The controlled transform $(\widetilde{\cI},1):=\sigma^{\rm c}(\cI,1)$ has  empty support. Consequently, $V(\widetilde{\cI})=\emptyset$,  and thus $\widetilde{\cI}$ is equal to the structural sheaf $\cO_{\widetilde{X}}$. This implies that the full transform $\sigma^*(\cI)$ is principal and generated by the sheaf of ideal of a divisor whose components are the exceptional divisors.
The actual process of desingularization is often achieved before $(X,{\cI},E,1)$ has been resolved (see \cite{Wlodarczyk})
\end{proof}

\section{Canonical resolution of marked ideals}

\subsection{Hypersurfaces of maximal contact}

The concept of the {\it hypersurfaces of maximal contact} is one of the key points of this proof. It was
originated by  Abhyankhar, Hironaka, and Giraud and developed in the papers of Bierstone-Milman and Villamayor.

In our terminology we are looking for a smooth hypersurface containing the supports of marked ideals and whose strict transforms under multiple  blow-ups contain the supports of the induced marked ideals. Existence of such hypersurfaces allows a reduction of the resolution problem to  codimension 1.

First we introduce marked ideals which locally admit  hypersurfaces of maximal contact.

\begin{definition}(Villamayor (see \cite{Vi}))
We say that a marked ideal $({\cI},\mu)$  is of {\it maximal order} (originally {\it simple basic object}) if $\max\{\ord_x({\cI})\mid x\in X\}\leq \mu$ or equivalently ${\cD}^\mu({\cI})=\cO_X$. We say that a multiple marked ideal (in particular a Rees algebra) $\overline{\cI}=\{(\cI_i,\mu_i)\}$ is of {\it maximal order} if for any point $x\in X$ at least one of the ideals $(\cI_i,\mu_i)$ is of maximal order in a neighborhood of $x$.
\end{definition}
Any  marked ideal  of maximal order generates or diff-generates the Rees algebra of maximal order.

\begin{lemma}(Villamayor (see \cite{Vi})) Let $({\cI},\mu)$ be a marked ideal of maximal order
and $C\subset\cosupp ({\cI},\mu)$ be a smooth center. Let $\sigma: X\leftarrow
X'$ be a blow-up at $C\subset\cosupp({\cI},\mu)$. Then $\sigma^{\rm c}({\cI},\mu)$ is of maximal order.
\end{lemma}
\noindent {\bf Proof.} If $({\cI},\mu)$ is a marked ideal of maximal order then ${\cD}^\mu({\cI})=\cO_X$.
Then by  Lemma \ref{le: inclusions},  ${\cD}^\mu(\sigma^{\rm c}({\cI},\mu))\supset \sigma^{\rm c}({\cD}^\mu({\cI}),0)=\cO_X$. \qed

\begin{lemma}(Villamayor (see \cite{Vi})) If $({\cI},\mu)$ is a marked ideal of maximal order and $0\leq i \leq \mu$ then ${\cD}^{i}({\cI},\mu)$
is of maximal order.
\end{lemma}
\noindent {\bf Proof.} ${\cD}^{\mu-i}({\cD}^{i}({\cI},\mu))={\cD}^{\mu}({\cI},\mu)=\cO_X$.\qed

\begin{lemma}(Giraud (see \cite{Giraud}))\label{le: Gi} Let $({\cI},\mu)$ be the marked ideal of maximal order.  Let $\sigma: X\leftarrow
X'$ be a blow-up at a smooth center $C\subset\cosupp({\cI},\mu)$. Let $u\in {\cD}^{\mu-1}({\cI},\mu)(U)$ be a function such that, for any $x\in V(u)$, $\ord_x(u)=1$.
Then \begin{enumerate}
\item $V(u)$ is smooth.
\item $\cosupp({\cI},\mu)\cap U \subset V( u)$
\end{enumerate}

Let $U'\subset \sigma^{-1}(U)\subset X'$ be an open
set where the exceptional divisor is described by $y$.  Let $u':=\sigma^{\rm c}(u)=y^{-1}\sigma^*(u)$ be the controlled transform of $u$.
Then
\begin{enumerate}
\item $u'\in {\cD}^{\mu-1}(\sigma^{\rm c}({\cI}_{|U'},\mu)).$
\item  $V(u')$ is smooth.
\item $\cosupp({\cI'},\mu)\cap U' \subset V( u')$
\item $V(u')$ is the restriction of the strict transform of $V(u)$ to $U'$.

\end{enumerate}
\end{lemma}
\noindent {\bf Proof.} (1) $u'=\sigma^{\rm c}(u)=u/y\in\sigma^{\rm c}({\cD}^{\mu-1}({\cI}))\subset {\cD}^{\mu-1}(\sigma^{\rm c}({\cI}))$.

(2) Since $u$ was one of the local parameters describing the center of blow-ups, $u'=u/y$ is a parameter, that
is, a function of order one.

 (3) follows from (2). \qed
\begin{definition} We shall call a function $$u\in T({\cI})(U):={\cD}^{\mu-1}({\cI}(U))$$ \noindent of multiplicity one a {\it tangent direction} of $({\cI},\mu)$ on $U$.
\end{definition}
As a corollary from the above we obtain the following lemma:
\begin{lemma}(Giraud)\label{le: Giraud}
Let $u\in T({\cI})(U)$ be a tangent direction of $({\cI},\mu)$ on $U$. Then  for any multiple  blow-up  $(U_i)$ of $({\cI}_{|U},\mu)$ all the supports  of the induced marked ideals  $\cosupp(\cI_i,\mu)$ are contained in the strict transforms $V(u)_i$ of $V(u)$.  \qed
\end{lemma}
\begin{remarks}
\begin{enumerate}
\item Tangent directions are functions defining locally hypersurfaces of maximal contact.
\item  The main problem leading to complexity of the proofs is that of noncanonical choice of
the tangent directions. We overcome this difficulty by introducing {\it homogenized ideals}.
\end{enumerate}
\end{remarks}
\begin{corollary} If $R^\cdot$ is a differential Rees algebra of maximal order then for any point $x\in X$, there is a tangent direction $u\in R^1(U)$ in a  neighborhood $U$ of $x$. 

\end{corollary}

\subsection{Properties of Rees alebra: Graded homogenization}

An important properties of Rees algebras of maximal order is that for any tangent directions it "looks the same". This property allows to run induction and assures that the restriction of  the Rees algebra
to the maximal contact does not depend on the choice of tangent directions. From that perspective it can be considered as a more general "homogenization"- a technique which allows to run induction in a canonical way without using, so called,  Hironaka's trick.
On the other hand the technique of Rees algebra saturation is better suited for the case of multiple generators like in the case of ideals along stratum or  in positive characteristic.
Thus the notion of Rees algebra can be thus understood as graded homogenization. The graded homogenization was considered by Kawanoue in his approach to Rees algebra. Another version of homogenization was considered by Kollar \cite{Kollar}.

The following result is a cosmetic modification of the original "glueing lemma" which was stated for the standard "nongraded" homogenization: 
\begin{lemma} \label{le: homo} ({\bf Glueing Lemma}) (\cite{Wlodarczyk}) 
Let ${R}^\cdot$ be a differential Rees algebra of maximal order on a smooth $X$. Let $u,v\in R^1(X)$ be two tangent directions at  $x\in\cosupp(R^\cdot)$ which are transversal to the set of exceptional divisors $E$.
Then there exist \'etale neighborhoods $\phi_{u},\phi_v: \overline{X}\to X$ of $x=\phi_u(\overline{x})=\phi_v(\overline{x}) \in X$,  where $\overline{x}\in \overline{X}$, such that
\begin{enumerate}
\item  $\phi_{u}^*({R}^\cdot)=\phi_{v}^*({R}^\cdot)$.
\item  $\phi_{u}^{-1}(E)=\phi_{v}^{-1}(E)$.
\item  $\phi_{u}^*(u)=\phi_{v}^*(v)$.

\bigskip
Set $\overline{R}^\cdot:=\phi_{u}^*({R}^\cdot)=\phi_{v}^*({R}^\cdot)
,\quad \overline{E}:=\phi_{u}^{-1}(E)=\phi_{v}^{-1}(E)$

\item  For any $\overline{y}\in \cosupp(\overline{X},\overline{R}^\cdot,\overline{E})$, $\phi_u(\overline{y})=\phi_v(\overline{y})$.
\item  For any multiple  blow-up $(X_i)$ of  $({X},{R}^\cdot,\emptyset,\mu)$ the induced multiple  blow-ups $\phi_u^*(X_i)$ and $\phi_v^*(X_i)$ of $(\overline{X},\overline{\cR}^\cdot,\overline{E})$ are the same (defined by the same centers). 

 Set $(\overline{X}_i):=\phi_u^*(X_i)=\phi_v^*(X_i)$.
 \item For any  $\overline{y}_i\in \cosupp(\overline{X}_i,\overline{R}^\cdot_i,\overline{E}_i)$ and the induced morphisms
$\phi_{u i},\phi_{v i}:\overline{X}_i\to X_i$, $\phi_{u i}(\overline{y}_i)=\phi_{v i}(\overline{y}_i).$

\end{enumerate}
\end{lemma}

\begin{proof}
We use the same  strategy as applied already to Rees algebras in Lemma \ref{le: homo2}. (See more details in \cite{Wlodarczyk})
\end{proof}

\subsection{Properties of Rees alebra: Coefficient ideals}

A differential Rees algebra has another important feature which is used in the inductive scheme.
It has a good restriction properties and can serve as a coefficient ideal. This can be expressed by the following proposition:

\begin{proposition} \label{le: S} Let $(X,R^\cdot,E)$ be a differential Rees algebra of maximal order.  Assume that $S\subset X$ is a smooth subvariety which has only simple normal crossings with $E$. Then
$$\cosupp(R^\cdot)\cap S= \cosupp((\cR^\cdot)_{|S}).$$
Moreover  let  $(X_i)$ be  a multiple blow-up with  centers $C_i$   contained in the strict transforms $S_i\subset X_i$ of $S$. Then

\begin{enumerate}

\item The restrictions  $\sigma_{i|S_i}: S_i\to S_{i-1}$ of the morphisms $\sigma_i: X_i\to X_{i-1}$ define
  a multiple  blow-up $(S_i)$
of ${\cC}({\cI},\mu)_{|S}$.

\item $\cosupp(\overline{\cI}_i,\mu)\cap S_i= \cosupp[((\cR^\cdot(\overline{\cI}))_{|S})]_i.$
\item Every multiple  blow-up $(S_i)$  of $\cosupp((\cR^\cdot(\overline{\cI}))_{|S})$  defines  a multiple  blow-up $(X_i)$ of $({\cI},\mu)$ with  centers $C_i$ contained in the strict transforms $S_i\subset X_i$ of $S\subset X$.

\end{enumerate}

\end{proposition}

\begin{proof}
First observe that $\cosupp((\cR^\cdot\cap {S}))\subseteq\cosupp((\cR^\cdot)_{|S}))
$.

 Let $x_1,\ldots,x_k,y_1,\ldots, y_{n-k}$ be local parameters at $x$ such that $\{x_1=0,\ldots,x_k=0\}$ describes $S$. Then any function $f\in {\cI}$ can be written as
$$f=\sum c_{\alpha f}(y) x^\alpha,$$\noindent where $c_{\alpha f}(y)$ are formal power series in $y_i$.

Now $x\in \cosupp(\cR^\cdot\cap S$ iff $\ord_x(c_\alpha)\geq \mu-|\alpha|$ for all $f\in R^\mu$ and $|\alpha|\leq \mu$. Note that $$c_{\alpha f|S}=\bigg(\frac{1}{\alpha!}\frac{\partial^{|\alpha|}(f)}{\partial x^\alpha}\bigg)_{|S}\in {\cD}^{|\alpha|}({\cI})_{|S}$$ and consequently  $\cosupp(R^\mu,\mu)\cap S=\bigcap_{f\in {\cI}, |\alpha|\leq \mu}\cosupp({c_{\alpha f|S}}, \mu-|\alpha|)\supseteq  \cosupp((\cR^\cdot)_{|S}))$.

The above relation is preserved by admissible multiple  blow-ups of $\overline{\cI}$. For the details 
see \cite{Wlodarczyk}.\end{proof}

One can reformulate a useful collorary for the capacitor
\begin{corollary}
 Let $(X,{\cI},E,\mu)$ be a marked ideal of maximal order. 
Assume that $S$ has only simple normal crossings with $E$. Then
$R^c(\cI,\mu)$ is capacitor for $R^\cdot(\cI,\mu)$, where $c:=\mu!$, and 
$$\cosupp({\cI},\mu)\cap S= \cosupp(R^c({\cI},\mu)_{|S}).$$
Moreover  let  $(X_i)$ be  a multiple blow-up with  centers $C_i$   contained in the strict transforms $S_i\subset X_i$ of $S$. Then

\begin{enumerate}

\item The restrictions  $\sigma_{i|S_i}: S_i\to S_{i-1}$ of the morphisms $\sigma_i: X_i\to X_{i-1}$ define
  a multiple  blow-up $(S_i)$
of ${\cC}({\cI},\mu)_{|S}$.

\item $\cosupp({\cI}_i,\mu)\cap S_i= \cosupp[{\cR^c}({\cI},\mu)_{|S}]_i.$
\item Every multiple  blow-up $(S_i)$  of ${\cR^c}({\cI},\mu)_{|S}$  defines  a multiple  blow-up $(X_i)$ of $({\cI},\mu)$ with  centers $C_i$ contained in the strict transforms $S_i\subset X_i$ of $S\subset X$.

\end{enumerate}

\end{corollary}

These properties allow one to control and modify the part of support of $(\cI,\mu)$ contained in $S$ by applying multiple  blow-ups of  $\cosupp((\cR^c(\cI,\mu))_{|S})$.

\begin{proof}
$\cR^c(\cI,\mu)$ is generated as Rees algebra by 
$(\cD^i(\cI),\mu-i)$. The means that $R^c(\cI,\mu)\subset \sum (\cD^i(\cI),\mu-i)$.
As in the precious proof for any $f=\sum c_{\alpha f}(y) x^\alpha \in \cI$, we have that $x\in \cosupp(\cI^\cdot\cap S$ iff $\ord_x(c_\alpha)\geq \mu-|\alpha|$ for all $f\in R^\mu$ and $|\alpha|\leq \mu$.  The rest of the reasonining is the same.

\end{proof}


\subsection{Resolution algorithm}

The presentation of the following Hironaka resolution algorithm  builds upon
 Bierstone-Milman's, Villamayor's and W\l odarczyk's algorithms which are simplifications of the original Hironaka proof. We also use Koll\'ar's trick allowing to completely eliminate the use of invariants.
 As was observed by Kawanoue, it is possible to rewrite the algorithm in the language of Rees algebras. 
 The relevant notion of Rees algebra of marked ideals  replaces the concept of the coefficient and homogenized ideals.
  The direct constructions lead to the multiple fractional marked ideals and Rees algebras with fractional grading.  
  To avoid the technicalities  one can consider, for instance an integral subfiltration of Rees algebras, or simply capacitor as in our case (see Definition \ref{capacitor}).  The latter is  simpler conceptually and allows to work with (single) marked ideals, while using Rees algebras as auxilliary tool for the constructions.

\begin{remarks}\begin{enumerate} \item Note that the blow-up of codimension one components  is an isomorphism. However it defines a nontrivial transformation of marked ideals. The inverse image of the center is still called the exceptional divisor.

\item In the actual desingularization process this kind of blow-up may occur for some marked ideals induced on subvarieties of ambient varieties. Though they define isomorphisms of those subvarieties they determine blow-ups of ambient
varieties which are not isomorphisms. 
\item The  blow-ups of the center $C$ which coincides with the whole variety $X$ is an empty set.
The main feature which characterizes is given by the restriction property:

\bigskip
If $X$ is a smooth variety containing a smooth subvariety $Y\subset X$, which contains the center $C\subset Y$
then the blow-up $\sigma_{C,Y}:\tilde{Y}\to Y$ at $C$ coincides with the strict transform of $Y$ under the blow-up $\sigma_{C,X}:\tilde{X}\to X$, i.e
$$\tilde{Y}\simeq \overline{\sigma_{C,X}^{-1}(Y\setminus C)}$$
\end{enumerate}
\end{remarks}

\begin{theorem}  For any  marked ideal $(X,{\cI},E,\mu)$ such that $\cI $ there is an associated  resolution $(X_i)_{0\leq i\leq m_X}$, called \underline{canonical},
 satisfying the following conditions:
\begin{enumerate}
\item For any  surjective  \'etale morphism $\phi: X '\to X$ the induced sequence 
 $(X'_i)=\phi^*(X_i)$ is  the canonical resolution of $(X',{\cI}',E',\mu):=\phi^*(X,{\cI},E,\mu)$. 
\item Any  \'etale morphism $\phi: X '\to X$  induces the sequence 
 $(X'_i)=\phi^*(X_{i})$ consisting of blow-ups   of the canonical resolution of  $(X',{\cI}',E',\mu):=\phi^*(X,{\cI},E,\mu)$ and (possibly) some identical transformations. 
\end{enumerate}
\end{theorem}
\begin{proof} If $\cI= 0$ and $\mu>0$ then $\cosupp(X, {\cI},\mu)=X$, and the blow-up of $X$ is the empty set and thus it defines a unique resolution. Assume that $\cI\neq 0$.

 We shall use the induction  on the dimension of $X$. If $X$ is $0$-dimensional, $\cI\neq 0$ and $\mu>0$ then $\cosupp(X, {\cI},\mu)=\emptyset$ and all resolutions are trivial. 

{\bf Step 1}  {\bf Resolving a marked ideal $(X,{\cJ},E,\mu)$ of maximal
order.}

Before performing the  resolution algorithm  for the marked ideal $({\cJ},\mu)$ of maximal order in Step 1 we shall replace it with the equivalent  to the capacitor ideal $\cR^c({\cJ},\mu))$.
Resolving the ideal $\cR^c({\cJ},\mu)$ defines a resolution of $({\cJ},\mu)$ at this step.
To simplify notation we shall denote $\cR^c({\cJ},\mu))$ by $(\overline{\cJ},\overline{\mu})$.

 {\bf Step 1a}  {\bf Reduction to the nonboundary case. Moving $\cosupp({\overline{\cJ}},\overline{\mu})$ and $H^s_\alpha$ apart .}
 For any multiple  blow-up $(X_i)$ of $(X,{\overline{\cJ}},E,\overline{\mu})$ we shall identify (for simplicity) strict transforms of $E$ on $X_i$ with $E$.

For any $x\in X_i$, let $s(x)$ denote the number of divisors in $E$ through $x$ and
set $$ s_i=\max\{ s(x) \mid x \in \cosupp({\overline{\cJ}}_i)\}.$$

Let $s=s_0$. By  assumption the intersections of any $s> s_0$ components of the exceptional divisors are disjoint from
$\cosupp({\overline{\cJ}},\overline{\mu})$.
Each intersection of divisors in $E$ is locally defined by intersection of some irreducible components of these divisors.
Find all   intersections $H^s_\alpha, \alpha\in A$, of $s$ irreducible components of divisors $E$ such that $\cosupp({\overline{\cJ}},\overline{\mu})\cap H^s_\alpha\neq\emptyset$. By the maximality of $s$, the supports $\cosupp({\overline{\cJ}}_{|H^s_\alpha})\subset H^s_\alpha$ are disjoint from $H^s_{\alpha'}$, where $\alpha'\neq\alpha$.

Set  $$H^s:=\bigcup_\alpha H^s_\alpha,\quad U^s:=X \setminus H^{s+1}, \quad \underline{H^s}:=H^s\setminus H^{s+1}.$$

Then $\underline{H^s}\subset U_s$ is a smooth closed subset $U_s$. Moreover $\underline{H^s}\cap \cosupp(\cI)={H^s}\cap \cosupp(\cI)$ is closed.

Construct the canonical resolution of ${\overline{\cJ}}_{|\underline{H^s}}$. By Lemma \ref{le: S}, it 
defines a multiple  blow-up of $({\overline{\cJ}},\overline{\mu})$ such that
$$\cosupp(\overline{\cJ}_{j_1},\overline{\mu})\cap{H^s_{j_1}}=\emptyset.$$
In particular the number of the strict tranforms of $E$ passing through a single point of the support drops $s_{j_1}<s$. Now we put $s=s_{j_1}$ and repeat the procedure. We continue the above process till $s_{j_k}=s_r=0$.
Then $(X_j)_{0\leq j\leq r}$  is a multiple  blow-up of $(X,{\overline{\cJ}},E,\overline{\mu})$
such that $\cosupp({\overline{\cJ}}_r,\overline{\mu})$ does not intersect any divisor in $E$.

Therefore $(X_j)_{0\leq j\leq r}$ and further longer multiple  blow-ups $(X_j)_{0\leq j\leq m}$ for any $m\geq r$   can be considered as multiple  blow-ups of $(X,{\overline{\cJ}},\emptyset,\overline{\mu})$ since starting from $X_r$ the strict transforms of $E$ play no further role in the resolution process since they do not intersect  $\cosupp({\overline{\cJ}}_j,\overline{\mu})$ for $j\geq r$. We reduce the situation to the "nonboundary case".

\bigskip
{\bf Step 1b}.   {\bf Nonboundary case}

Let $(X_j)_{0\leq j\leq r}$ be the multiple  blow-up of $(X,{\overline{\cJ}},\emptyset,\overline{\mu})$ defined in Step 1a.

For any $x\in \cosupp({\overline{\cJ}},\overline{\mu}) \subset X$ find a tangent direction $u_\alpha \in {\cD}^{\overline{\mu}-1}({\overline{\cJ}})$  on some neighborhood $U_\alpha$ of $x$.
Then $V(u_\alpha)\subset U_\alpha$ is a hypersurface of maximal contact.
By the quasicompactness of $X$ we can assume that the covering defined by $U_\alpha$ is finite.
Let $U_{i\alpha}\subset X_i$ be the inverse image of $U_{i\alpha}$ and let $H_{i\alpha}:=V(u_\alpha)_i\subset U_{i\alpha}$ denote the strict transform of $H_\alpha:=V(u_\alpha)$.

Set (see also \cite{Kollar}) $$\widetilde{X}:= \coprod U_\alpha\quad \quad  \widetilde{H}:=\coprod H_\alpha\subseteq \widetilde{X}.$$

The closed embeddings $H_\alpha \subseteq U_\alpha$ define the closed embedding $\widetilde{H}\subset \widetilde{M}$ of a hypersurface of maximal contact $\widetilde{H}$.

Consider the surjective \'etale morphism 
$$\phi_U: \widetilde{X}:= \coprod U _\alpha\to X$$
 
 Denote by $\widetilde{J}$ the pull back of the ideal sheaf $\overline{\cJ}$ via $\phi_U$.
The multiple  blow-up  $(X_{i})_{0\leq i\leq r}$ of $\overline{J}$ defines a multiple  blow-up $(\widetilde{X}_{0\leq i\leq r})$ of $\widetilde{J}$  and a multiple  blow-up $(\widetilde{H}_i)_{0\leq i\leq r}$ of $\widetilde{J}_{|H}$.

Construct the canonical resolution of $(\widetilde{H}_{i})_{r\leq i\leq m}$ of the marked ideal ${\widetilde{\cJ}}_{r|\widetilde{H}_{r}}$ on $\widetilde{H}_r$. It defines, by Lemma \ref{le: coeff}, a resolution
$(\widetilde{X}_{r\leq i\leq m})$ of ${\widetilde{\cJ}}_{r}$ and thus also a resolution  $(\widetilde{X}_i)_{0\leq i\leq m}$
of $(\widetilde{X},{\widetilde{\cJ}},\emptyset, \overline{\mu})$. Moreover both resolutions are related by the property $$\cosupp(\widetilde{\cJ}_{i})=\cosupp(\widetilde{\cJ}_{i|\widetilde{H}_i}).$$

Consider a  (possible) lifting  of $\phi_U$: $$\phi_{iU}: \widetilde{X}_i:= \coprod U _{i\alpha}\to X_i,$$ which is a  surjective locally \'etale morphism. 
The lifting is constructed for $0\leq i\leq r$.

 For ${r\leq i\leq m}$ the resolution $\widetilde{X}_i$ is induced by the canonical resolution $(\widetilde{H}_i)_{r\leq i\leq m}$ of ${\overline{\cJ}}_{r|\widetilde{H}_r}$

We show that the resolution $(\widetilde{X_i})_{r\leq i\leq m}$
descends to the resolution $({X_i})_{r\leq i\leq m}$. 

Let $\widetilde{C}_{j_0}=\coprod C_{j_0\alpha}$ be the center of the blow-up $\widetilde{\sigma}_{j_0}:\widetilde{X}_{j_0+1}\to\widetilde{X}_{j_0}$. The closed subset $C_{j_0\alpha}\subset U_{j_0\alpha}$ defines the center of an extension of the canonical resolution $(H_{j\alpha})_{r\leq j\leq m}$.

 If  ${C}_{j_0\alpha}\cap   U_{j_0\beta}\neq \emptyset $ then by the canonicity and condition (2) of the inductive assumption, the subset
${C}_{j_0\alpha\beta}:={C}_{j_0\alpha}\cap   U_{j_0\beta}$ defines  the center  
of an extension of  of the canonical resolution $H_{j\alpha\beta}:=((H_{j\alpha}\cap U_{j\beta}))_{r\leq j\leq m}$.
On the other hand ${C}_{j_0\beta\alpha}:={C}_{j_0\beta}\cap   U_{j\alpha}$ defines
 the center of an extension of  the canonical resolution $((H_{j\beta\alpha}:=H_{j\beta}\cap U_{j\alpha}))_{r\leq j\leq m}$.


By Glueing Lemma \ref{le: homo} for the tangent directions $u_\alpha$ and $u_\beta$ we find 
 there exist \'etale neighborhoods  $\phi_{u_\alpha},\phi_{u_\beta}: \overline{U}_{\alpha\beta}\to {U}_{\alpha\beta}:=U_\alpha\cap U_\beta $ of $x=\phi_u(\overline{x})=\phi_v(\overline{x}) \in X$,  where $\overline{x}\in \overline{X}$, such that
\begin{enumerate}
\item  $\phi_{u_\alpha}^*({\cJ})=\phi_{u_\beta}^*(\cJ)$.
\item  $\phi_{u_\alpha}^*(E)=\phi_{u_\beta}^*(E)$.
\item  $\phi_{u_\alpha}^{-1}(H_{j\alpha\beta})=\phi_{u_\beta}^{-1}(H_{j\beta\alpha})$.
\item $\phi_{u_\alpha}(\bar{x})=\phi_{u_\beta}(\bar{x})$ for $\overline{x}\in \cosupp(\phi_{u_\alpha}^*({\cJ}))$.
\end{enumerate}
Moreover all the properties lift to the relevant \'etale morphisms    $\phi_{u_{\alpha i}},\phi_{u_{\beta i}}: \overline{U}_{\alpha\beta i}\to {U}_{\alpha\beta i}$. 
Consequently, by canonicity $\phi_{u_\alpha j_0}^{-1}({C}_{j_0\alpha\beta})$ and $\phi_{u_\beta j_0}^{-1}(C_{j_0\beta\alpha})$ define both the next center of the extension of the canonical resolution 
$\phi_{u_\alpha}^{-1}(H_{j_0\alpha\beta})=\phi_{u_\beta}^{-1}(H_{j_0\beta\alpha})$ of $\phi_{u_\alpha j_0}^*({\cJ}_{|H_{\alpha\beta}})=\phi_{u_\beta}^*({\cJ}_{|H_{\beta\alpha}})$.

Thus $$\phi_{u_\alpha}^{-1}({C}_{j_0\alpha\beta})=\phi_{u_\beta}^{-1}({C}_{j_0\beta\alpha}),$$ and finally, by property (4),
 $${C}_{j_0\alpha\beta}={C}_{j_0\beta\alpha}.$$
Consequently $\widetilde{C}_{j_0}$ descends to the smooth closed center $C_{j_0}=\bigcup {C}_{j_0\alpha} \subset X_{j_0}$ and
the resolution $(\widetilde{X_i})_{r\leq i\leq m}$
descends to the resolution $({X_i})_{r\leq i\leq m}$.

\bigskip

{\bf Step 2}.   {\bf Resolving of marked ideals $(X,{\cI},E,{\mu})$.}

For any  marked ideal $(X,{\cI},E,\mu)$ write
 $$I=\cM({\cI}){\cN}({\cI}),$$  \noindent where $\cM({\cI})$ is the {\it monomial
part} of ${\cI}$, that is, the product of the principal ideals
defining the irreducible components of the divisors in $E$, and ${\cN}({\cI})$ is a {\it nonmonomial
part} which is not divisible by any ideal of a divisor in $E$.
Let $$\ord_{{\cN}({\cI})}:=\max\{\ord_x({\cN}({\cI}))\mid x\in\cosupp(\cI,\mu)\}.$$

\begin{definition} (Hironaka, Bierstone-Milman,Villamayor, Encinas-Hauser)
By the {\it companion ideal} of $({\cI},\mu)$  where $I={\cN}({\cI})\cM({\cI})$
we mean the marked ideal of maximal order
\begin{displaymath}
O({\cI},\mu)= \left\{ \begin{array}{ll}
({\cN}({\cI}),\ord_{{\cN}({\cI})})\quad + \quad(\cM({\cI}),\mu-\ord_{{\cN}({\cI})})& \textrm{if   $\ord_{{\cN}({\cI})}<\mu$},\\ ({\cN}({\cI}),\ord_{{\cN}({\cI})})&\textrm{if   $\ord_{{\cN}({\cI})}\geq \mu$}.
\end{array} \right.
\end{displaymath}
\end{definition}

In particular $O({\cI},\mu)=({\cI},\mu)$ for ideals $({\cI},\mu)$ of maximal order.

{\bf Step 2a}. {\bf Reduction to the monomial case by using companion ideals}

 By Step 1 we can resolve the marked ideal of maximal order $(\cJ,\mu_\cJ):=O({\cI},\mu)$.  By Lemma \ref{le: operations}, for any multiple  blow-up of  $O({\cI},\mu)$,

\centerline{$
\cosupp(O({\cI},\mu))_i= \cosupp[{\cN}({\cI}),\ord_{{\cN}({\cI})}]_i\,\, \cap \,\, \cosupp[M({\cI}),\mu-\ord_{{\cN}(H{\cI})}]_i =$}
 \centerline{$
 \cosupp[{\cN}({\cI}),\ord_{{\cN}({\cI})} ]_i\,\,\cap \,\, \cosupp({\cI}_i,\mu).$}
 Consequently, such a resolution leads to the ideal $({\cI}_{r_1},\mu)$ such that $\ord_{{\cN}({\cI}_{r_1})}<\ord_{{\cN}({\cI})}$. Then we repeat the procedure for $({\cI}_{r_1},\mu)$.
 We find marked ideals $({\cI}_{r_0},\mu)=(\cI, \mu), ({\cI}_{r_1},\mu), \ldots, ({\cI}_{r_m},\mu)$ such that
$\ord_{{\cN}({\cI}_0)}>\ord_{{\cN}({\cI}_{r_1})}>\ldots>\ord_{{\cN}({\cI}_{r_m})}$.
  The procedure terminates after a finite number of steps when we arrive at the ideal $({\cI}_{r_m},\mu)$ with $\ord_{{\cN}({\cI}_{r_m})}=0$ or with $\cosupp({\cI}_{r_m},\mu)=\emptyset$. In the second case we get the resolution. In the first case $\cI_{r_m}=\cM({\cI}_{r_m})$ is monomial.

{\bf Step 2b}. {\bf Monomial case $\cI=\cM({\cI})$}.

Let $x_1,\ldots,x_k$ define equations of the components ${D}^x_1,\ldots,{D}^x_k\in E$ through $x\in \cosupp(X,{\cI},E,\mu)$ and
 ${\cI}$  be generated by the monomial $x^{a_1,\ldots,a_k}$ at $x$. In particular $$\ord_x(\cI)(x):=a_1+\ldots+a_k.$$

 Let
 $\rho(x)=\{D_{i_1},\ldots,D_{i_l}\}\in \Sub(E)$  be the maximal subset satisfying the properties

\begin{enumerate}
\item  $a_{i_1}+\ldots+a_{i_l}\geq \mu.$
\item For any $j=1,\ldots,l$, $a_{i_1}+\ldots +\check{a}_{i_j}+\ldots +a_{i_l} < \mu. $
\end{enumerate}

Let $R(x)$ denote the subsets in $\Sub(E)$ satisfying the properties (1) and (2).
The maximal components of the $\cosupp({\cI},\mu)$ through $x$ are described by the intersections
$\bigcap_{D\in A} D$ where $A\in R(x)$. The maximal locus of $\rho$ determines at most o one maximal component of $\cosupp({\cI},\mu)$ through each $x$.

  After the blow-up at the maximal locus $C=\{x_{i_1}=\ldots=x_{i_l}=0\}$ of $\rho$, the ideal $\cI=(x^{a_1,\ldots,a_k})$ is equal to $\cI'=({x'}^{a_1,\ldots,a_{i_j-1},a,a_{i_j+1},\ldots,a_k})$ in the neighborhood  corresponding to $x_{i_j}$, where $a=a_{i_1}+\ldots+a_{i_l}-\mu<a_{i_j}$. In particular the invariant $\ord_x(\cI)$ drops for all  points of some maximal components of $\cosupp({\cI},\mu)$. Thus the maximal value of $\ord_x(\cI)$ on the maximal components of  $\cosupp({\cI},\mu)$ which were blown up is bigger than the maximal value of $\ord_x(\cI)$  on the new maximal components of $\cosupp({\cI},\mu)$.
The algorithm
terminates after a finite number of steps. 
\end{proof}

\end{document}